\documentclass[11pt, onesided, english, reqno,]{amsart}

\usepackage[T1]{fontenc}
\usepackage{babel}
\usepackage{mathrsfs}
\usepackage{amsthm}
\makeindex

\usepackage[dvips,top=3.8cm,left=4cm,right=3.8cm,
foot=3.8cm,bottom=4.0cm]{geometry}
\usepackage{amsfonts,amssymb,amsmath,amsthm, booktabs, 
latexsym}
\usepackage{centernot}
\usepackage{enumerate}

\newtheorem{theorem}{Theorem}[section]
\newtheorem{lemma}[theorem]{Lemma}
\newtheorem{proposition}[theorem]{Proposition}
\newtheorem{corollary}[theorem]{Corollary}
\theoremstyle{definition}
\newtheorem{definition}[theorem]{Definition}
\newtheorem{example}[theorem]{Example}
\theoremstyle{remark}
\newtheorem{remark}[theorem]{Remark}

\numberwithin{equation}{section}

\numberwithin{equation}{section}
\newcommand{\be}{\begin{equation}}
\newcommand{\ee}{\end{equation}}

\newcommand{\N}{{\mathbb N}}
\newcommand{\Z}{{\mathbb Z}}
\newcommand{\R}{{\mathbb R}}

\def\va{\varphi}
\def\la{\lambda}
\def\csi1{\circ\sigma^{-1}}
\def\cs{\circ\sigma}
\def\lacr{\lambda\circ R}
\def\ol{\overline}
\def\wt{\widetilde}
\def\wh{\widehat}

\def\mc{\mathcal}

\newcommand{\B}{{\mathcal B}}
\newcommand{\sms}{(X, \mathcal B, \mu)}
\newcommand{\FXB}{{\mathcal F(X, \B)}}
\newcommand{\sB}{{\sigma^{-1}(\B)}}

\begin{document}

\title[Infinite-dimensional transfer operators]{Infinite-dimensional
 transfer operators, endomorphisms, and measurable partitions}

\author{Sergey Bezuglyi}
\address{Department of Mathematics, University of Iowa, Iowa City,
52242 IA, USA}
\email{sergii-bezuglyi@uiowa.edu}

\author{Palle E.T. Jorgensen}
\address{Department of Mathematics, University of Iowa, Iowa City,
52242 IA, USA}
\email{palle-jorgensen@uiowa.edu}




\maketitle

\newpage

\tableofcontents
  
\newpage

\vskip2cm

\section*{Foreword and Acknowledgments}

While the mathematical structures of positive operators, endomorphisms, 
transfer operators, measurable partitions, and Markov processes arise in a 
host of settings, both pure and applied, we propose here a unified study. 
This is the general setting of dynamics in Borel measure spaces. Hence the 
corresponding linear structures are infinite-dimensional. Nonetheless, we 
prove a number of analogues of the more familiar finite-dimensional settings, 
for example, the Perron-Frobenius theorem for positive matrices, and the 
corresponding Markov chains.

\vskip2cm

The first named author is thankful to Professors Jane Hawkins, Olena Karpel, 
Konstantin Medynets, and Cesar Silva for useful discussions on properties of
endomorphisms. 
The second named author gratefully acknowledge discussions, on the 
subject of this book, with his colleagues, especially helpful insight from 
Professors Daniel Alpay, Dorin Dutkay, Judy Packer, Erin Pearse, 
Myung-Sin Song, and Feng Tian.

\newpage

\section{Introduction and examples}\label{sect Introduction}

We develop a new duality between endomorphisms  
$\sigma$ of measure spaces $(X, \B)$, on the one hand, and a
 certain family of positive operators  $R$ acting in spaces of 
 measurable functions on $(X, \B)$, on the other. A framework of 
 standard Borel spaces $(X, \B)$ is adopted; and this generality is wide
  enough to cover a host of applications.

    In detail, from a given pair $(R, \sigma)$ on $(X, \B)$, a positive 
operator $R$, and an endomorphism $\sigma$, we define the notion of transfer 
operator.  At the outset, measures on $(X, \B)$ are not specified, but 
they will be introduced, and adapted to the questions at hand; in fact, 
a number of convex sets of measures on $(X, \B)$ will be analyzed in 
order for us to make precise the desired duality correspondences 
between the two parts, operator and endomorphism, in a fixed transfer 
operator pair $(R, \sigma)$. The theorems we obtain in this setting are 
motivated in part by recent papers dealing with stochastic processes 
(especially in joint work between D. Alpay et al and the second named 
author), applications to 
physics, to path-space analysis, to ergodic theory, and to dynamical 
systems and fractals. A source of inspiration is a desire to find an 
infinite-dimensional setting for the classical Perron-Frobenius theorem 
for positive matrices, and for the corresponding infinite Markov chains. 
Indeed, recent applications dictate a need for such infinite-dimensional 
extensions.

   Tools from the theory of operators in Hilbert space of special 
significance to us will be the use of a certain universal Hilbert space, as 
well as classes of operators in it, directly related to the central theme 
of duality for transfer operators. From ergodic theory, we address such 
questions as measurable cross sections, partitions, and Rohlin analysis 
of endomorphisms of measure spaces. While there are classical 
theorems dealing with analogous questions for \textit{automorphisms} of 
measure spaces, a systematic study of \textit{endomorphisms} is of more 
recent vintage;--  in its infancy. In order to make the exposition 
accessible to students and to researchers in neighboring areas, we 
have included a  number of explicit examples and applications.

The notion of transfer operators includes settings from statistical 
mechanics where they are often referred to as \textit{Ruelle 
operators} (and 
we shall use the notation $R$ for transfer operator for that reason), 
from harmonic analysis, including spectral analysis of wavelets, from 
ergodic theory of endomorphisms in measure spaces, Markov random 
walk models, transition processes in general; and more. 
The terminology ``transfer operator'' is from statistical mechanics; 
used for example in consideration of the action of a dynamical system 
on mass densities. The idea is that for chaotic systems, it is not 
possible to predict individual ``atoms'', or molecules,  only the density 
of large collections of initial conditions \cite{Ruelle1978}. Or in mathematical language, 
``transfer operator'' refers to the transformation of individual 
probability distributions for systems of random variables.
 There are further a number of parallels between our present 
infinite-dimensional theory and the classical Perron-Frobenius theorem 
for the special case of finite positive matrices.

To make the latter parallel especially striking, it is helpful to restrict 
the comparison to the case of the Perron-Frobenius for finite matrices 
in the special case when the spectral radius is 1 (see e.g., 
\cite{Baladi2000, BaillifBaladi2005, BaladiJiangLanford1996, 
Keane1972, MayerUrbanski2015, NicoloRadin1982, Parry1969, 
Radin1999}).

As we hint at in the title to our book, in our infinite-dimensional 
version of Perron-Frobenius  transfer operators, we include theorems 
which may be viewed as analogues of many points from the classical 
finite-dimensional Perron-Frobenius  case, for example, the classical 
respective left and right Perron-Frobenius eigenvectors, now take the 
form in infinite-dimensions of a 
\textit{positive $R$-invariant measure} (left) 
and the infinite-dimensional right Perron-Frobenius  vector becomes a 
\textit{positive harmonic function}.

Of course in infinite-dimensions, we have more non-uniqueness than is 
implied by the classical matrix version, but we also have many parallels. 
We even prove infinite-dimensional versions of the Perron-Frobenius 
limit theorem from the classical matrix case.

 In recent research (detailed citations below) in infinite-dimensional 
 analysis, a number of frameworks have emerged that involve positive 
 operators, but nonetheless, a \textit{unified} infinite-dimensional setting is 
 only slowly taking shape. While these settings and applications involve 
 researchers from diverse areas, and may on the surface appear quite 
different, they, in one way or the other, all involve generalizations of 
the classical Perron-Frobenius theory which in turn has already found 
many applications in ergodic theory, in the study of Markov chains, and 
more generally in infinite-dimensional dynamics. 

Motivated by recent research, it is our aim here to address and unify 
these infinite-dimensional settings. Our work in turn is also motivated 
by many instances of the use of classes of positive operators which by 
now go under the name “transfer operators,” or Ruelle operators, (see 
below for precise definitions). The latter name is after David Ruelle who 
first used such a class of these operators in the study of phase 
transition questions in statistical mechanics. Subsequent research on 
such questions as symbolic dynamics, spectral theory, endomorphisms 
in measure spaces, and diffusion processes, further suggest the need 
for a unifying infinite-dimensional approach.
    In fact the list of applications is longer than what we already hinted 
at, and it includes recent joint research involving the second named 
author with Daniel Alpay, and collaborators; details and citations are 
included below (see e.g., \cite{AlpayJorgensenLewkowicz2013, 
AlpayJorgensenLewkowicz2016, AlpayLewkowicz2013}). This collaborative research also makes use of positive 
operators and transfer operators in several infinite-dimensional 
settings, specifically in the study of such stochastic processes as 
infinite-dimensional Markov transition systems, analysis of Gaussian 
processes, and in the realization of wavelet multiresolution 
constructions for a host of probability spaces, and their associated 
$L^2$  
Hilbert spaces, all of which go beyond the more traditional setting of 
$L^2(\mathbb R^d)$ from wavelet theory. Indeed the last mentioned 
multi-scale wavelet constructions are applicable to a general framework 
of self-similarity from geometric measure theory (see e.g., 
\cite{Keane1972, KutzFuBrunton2016, 
Hutchinson1981, HutchinsonRueschendorf1998}).

  Important points in our present consideration of transfer operators 
  are as follows: We formulate a general framework, a list of precise 
  axioms, which includes a diverse host of cases of relevance to 
  applications. In this, we separate consideration of the transfer 
  operators as they act on \textit{functions on Borel spaces}
   $(X, \B)$ on the   one hand,  and their \textit{Hilbert space 
   properties }on the other. When a 
  transfer operator is given, there is a variety of measures compatible 
  with it, and we will discuss both the individual cases, as well as the 
  way a given transfer operator is acting on a certain universal Hilbert 
  space. The latter encompasses all possible probability measures on 
  the given Borel space $(X, \B)$. This approach is novel, and it helps 
  us organize our discussion of a host of ergodic theoretic properties 
  relevant to the theory of transfer operators.
  
  The sections in the book are organized as follows: The early sections are 
  in the most general setting, and the framework is restricted in the later 
  more specialized sections. Each specialization in turn is motivated by 
  applications. To make the book accessible to a wider readership, including 
  non-specialist, at the end of these sections we have cited some papers/
  books which may help by discussing foundations, applications, and 
  motivation.
 
A detailed summary of our main theorems is given in Subsection 
\ref{subsect Main theorems}  below.

\subsection{Motivation} This work is devoted to the study of transfer 
operators, see Definition \ref{def transfer operator from Intro}. This
 kind of operators, acting in a functional space, has been 
studied in numerous research papers and books. They are 
also known  by the name of \textit{Ruelle operators} or \textit{Perron-
Frobenius operators} that are used synonymously.  One of the first 
instances of 
the use of transfer operators in the sense we address here was papers 
by Ruelle in the 1970ties (see, e.g., \cite{Ruelle1978}) dealing with
 phase transition in statistical mechanics.
Since then the subject has branched off in a variety of new 
directions, and new applications. Our present  aim is to give a 
systematic and general setting for the study of transfer operators, and 
to offer some key results that apply to this general setting. 
Nonetheless, by now, the literature dealing with transfer operators and 
their diverse applications is large. For readers interested in the many 
settings in dynamical systems where some version, or the other, of a 
transfer operator arises, we have cited the papers below 
\cite{Keane1972, Ruelle1978, Ruelle1989,BaillifBaladi2005, 
BaladiEckmanRuelle1989, 
BaladiJiangLanford1996, Baladi2000, Ruelle1992, Ruelle2002, 
Dutkay2002, DutkayRoysland2007, Jorgensen2001, Kato2007, 
MayerUrbanski2010, MayerUrbanski2015, Radin1999,  
Stoyanov2012, Stoyanov2013}. Non-singular transformations of 
measure spaces are of a special interest. We refer to the following
 papers in this connection \cite{BezuglyiGolodets1991, BruinHawkins2009, 
 DajaniHawkins1994, EigenSilva1989, HawkinsSilva1991, Silva1988}. 
 Invariant measures on  Cantor sets are studied, in particular, in the following
 papers \cite{BezuglyiKwiatkowskiMedynetsSolomyak2010,
 BezuglyiKwiatkowskiMedynetsSolomyak2013, BezuglyiHandelman2014,
 BezuglyiKarpel2016}.

Our present results are motivated in part by applications. These 
applications include  Markov random walk models, problems 
from statistical mechanics, and from dynamics. While our setting here,  
dealing with transfer operators and endomorphisms in general measure 
spaces, is of independent interest, there are also a number of more 
recent applications of this setting to problems dealing with generalized 
multi-resolution analysis, relevant to the study of wavelet filters 
which require the use of solenoid analysis for their realization. In fact,  
the following is only a sample of  research papers devoted to 
these problems \cite{BaggettFurstMerrillPacker2009, 
BaggettMerrillPackerRamsay2012, FarsiGillaspyKangPacker2016, 
LatremoliereFredericPacker2013}.

Since our work touches rather different areas of Analysis, we give
here a list of principal references  in the corresponding fields.
While there is a rich literature on \textit{endomorphisms} of 
non-Abelian 
algebras of operators, both C*-algebras, and von Neumann algebras, 
the nature of endomorphisms of Abelian measure spaces presents 
intriguing new questions which are quite different from those studied 
so far in the corresponding non-Abelian situations. Our present analysis  
deals with endomorphisms of Abelian measure spaces. 
(The interplay between the Abelian vs the non-Abelian case is at the
 heart of the Kadison-Singer problem/now theorem, see 
 \cite{CasazzaTremain2016, MarkusSpielmanSrivastava2015}, but this 
 direction will not be addressed here.) For readers 
interested in the non-Abelian cases, we offer the following references 
\cite{Longo1989, BratteliElliottKishimoto1993, 
BratteliJorgensenPrice1996, BratteliJorgensen1997, 
BratteliKishimoto2000, PowersPrice1993, Powers1999, Jones1994, 
BischoffKawahigashiLongo2015}.

The study of \textit{transfer operators}, and more generally 
\textit{positivity-preserving operators},  are both of independent 
interest in their 
own right. This is in addition to its use in numerous applications; both 
within mathematics, and in neighboring areas; for example in physics, in 
signal analysis, in probability, and in the study of stochastic processes. 
While we shall cite these applications inside the book, we already now 
call attention to the following recent papers 
\cite{AlpayJorgensenLewkowicz2013, AlpayJorgensenLewkowicz2013,
AlpayKipnis2013, AlpayLewkowicz2013, AlpayJorgensenSalomon2014,
AlpayJorgensenVolok2014, AlpayJorgensenKimsey2015, 
AlpayKipnis2015, AlpayJorgensen2015, 
AlpayJorgensenLewkowiczMartziano2015,  
AlpayJorgensenLewkowiczVolok2016, 
AlpayColomboKimseySabadini2016}.
 
We cite some papers on the \textit{multiresolutions} that are 
related  to our work
\cite{BratteliJorgensen2002, BaggettJorgensenMerrillPacker2005, 
KutzFuBrunton2016,
BandaraRubergCirak2016, SuarezGhosal2016,
AlpayJorgensenLewkowiczVolok2016}.

\textit{Iterated function systems (IFS) }are used to describe the 
properties of 
fractal sets, and have close relations to transfer operators. Here we 
cite  papers on IFS and their connections to various aspects 
of transfer operators: \cite{BarnsleyHutchinsonStenflo2008,
 BarnsleyHutchinsonStenflo2012, Beardon1991,
Hutchinson1981,  Hutchinson1995, 
HutchinsonRueschendorf1998, Ruelle1978, 
Ruelle1989, Ruelle1992, Ruelle2002, YanLiuZhu1999}.

\subsection{Examples of transfer operators}
Our goal is to study \emph{transfer operators} in the framework of 
various functional spaces. To be more specific, we briefly mention 
several typical examples of
transfer operators. They will illustrate our results proved below. The 
rigorous definitions of used  notions are given in the next section, see
also Definitions \ref{def transfer operator from Intro} and \ref{def transfer 
operator}.

Our approach to the theory of transfer operators can be briefly 
described as follows. We first define and study these operators 
in the most abstract setting, aiming 
to find out what general properties they have. By \textit{abstract setting}, 
we mean the space of Borel real-valued functions $\mc F(X, 
\B)$ over a \textit{standard Borel space}
 \index{standard Borel space} $(X,\ B)$. 
Such spaces being 
endowed with a topology, or a Borel measure, are used in  most 
interesting classes of transfer operators such as Frobenius-Perron 
operators, or operators corresponding to iterated function systems, 
or operators acting in a Hilbert space, etc. 

Let  $\sigma$ be a fixed \textit{surjective Borel  endomorphism} 
\index{endomorphism ! Borel} of $(X, \B)$,  
 and let $M(X)$ be the set of all Borel (finite or sigma-finite) measures 
 on $(X, \B)$. In general, $\sigma^{-1}(\B): = \{\sigma^{-1}(A) : A 
 \in \B\}$ is a proper nontrivial 
 subalgebra of $\B$ where $\sigma^{-1}(A) = \{x \in X : \sigma(x) 
 \in A\}$. In fact, an endomorphism $\sigma$ defines a 
 sequence of filtered subalgebras $\B \supset \sigma^{-1}(\B)
 \supset \cdots \sigma^{-n}(\B) \supset \cdots$. An important 
 property of $\sigma$, called \textit{exactness}, is characterized by 
 the triviality of the subalgebra $\B_{\infty}= \bigcap_{n \in \N}
 \sigma^{-n}(\B)$, see Definition \ref{def ergodic and exact}.  
 We note that a Borel function  $f$ on $(X, \B)$ is
  $\sigma^{-1}(\B)$-measurable if and only if there exists a Borel 
  function $g$ such that $f =  g\cs$. 
 
When a measure $\lambda \in M(X)$ is fixed, then we get into the 
framework of a \textit{standard measure space} $(X, \B, \lambda)$ 
(see, e.g., \cite{CornfeldFominSinai1982}), 
and, in this situation, we  use measurable sets from the complete 
sigma-algebra\footnote{We reserve the symbol $\sigma$ for an 
endomorphism of a standard Borel space $(X, \B)$, so that to avoid 
any confusion we write sigma-algebra and sigma-finite measure instead 
of such more common terms $\sigma$-algebra and $\sigma$-finite} $
\B(\la)$ and functions measurable with respect to 
$\B(\la)$  instead of Borel ones. With some abuse of notation, 
we will also use the same symbol $\B$ for the sigma-algebra of 
measurable sets.

Having these data defined, we now give the following main  definition.

\begin{definition}\label{def transfer operator from Intro}
Let  $\sigma : X \to X$ be a surjective endomorphism of a standard 
Borel space $(X, \B)$. We say that $R$ is a \emph{transfer 
operator} \index{transfer operator} 
if $R : \mathcal F(X) \to \mathcal F(X)$ is a linear operator satisfying 
the properties:

(i) $R$ is a positive operator \index{positive operator}, 
that is $f \geq 0 \ \Longrightarrow \ Rf \geq 0$;

(ii)  for any Borel functions $f, g \in \mathcal F(X)$,
\be\label{eq char property of r via sigma}
R((f\circ \sigma) g) = f R(g).
\ee
If $R(\mathbf 1)(x) > 0$ for all $x\in X$, then we say that $R$ is a
 \textit{strict} transfer operator \index{transfer operator ! strict} 
 (here and below 
$\mathbf 1$ means the constant function that takes value 
1). If $R( \mathbf 1) = \mathbf 1$, then 
$R$ is called a \textit{normalized} transfer operator.  
\index{transfer operator ! normalized}
\end{definition}

Relation (\ref{eq char property of r via sigma}) is called the 
\textit{\index{pull-out property}}.

In what follows we describe several classes of transfer operators and 
then give a universal approach to these classes based on the notion of 
a \textit{measurable partition}, see Subsection \ref{subsect standard 
spaces}. More examples of transfer operators 
will be also given in  subsequent  sections.

\begin{example}[Transfer operators defined by finite-to-one 
endomorphisms]\label{ex R forIFS}
Let $X = [0,1]$ be the unit interval with Lebesgue measure $dx$. 
Take the endomorphism $\sigma$ of $X$ into itself defined by
$$
\sigma(x) = 2x \ \mathrm{mod}\; 1.
$$
Then $\sigma$ is onto, and $|\{\sigma^{-1}(x) : x \in X\}| = 2$. 
Consider a functional space $\mc F$ of real-valued functions over $X$. 
We do not need to specify this space here. For instance, it can be
 either $L^p(X, dx)$, or the space of all Borel functions, or the space 
 of continuous functions, etc. Set
\be\label{eq example R_sigma IFS}
R_\sigma (f)(x) := \frac{1}{2}\left(f(\frac{x}{2}) + f(\frac{x+1}{2})
\right), \ \ \ f \in \mc F.
\ee
Relation (\ref{eq example R_sigma IFS}) gives an example of a transfer 
operator that is well studied in the theory of iterated function systems 
(IFS). 

Based on this elementary example, we can use  a more general 
approach to the definition of $R_\sigma$. Suppose that $\sigma$ is 
an $n$-to-one endomorphism of a measurable space $(X, \B)$, and $
\mc F(X)$ is an appropriate functional space of real-valued functions. 
Let $W$ be a nonnegative function  on $X$ (it is called a weight 
function). We define
a transfer operator  $\mc F(X)$ by the formula
\be\label{eq def Ruelle operator classic}
R_\sigma (f)(x) = \sum_{y\in \sigma^{-1}(x)}W(y) f(y).
\ee
Clearly, $R_\sigma f \geq 0$ whenever $f \geq 0$, i.e., $R_\sigma$ is 
a positive operator.  Moreover, if $\mathbf 1$ denotes the constant 
function that takes the value 1, then the condition $R_\sigma
(\mathbf 1) =\mathbf 1$ holds if and only if $\sum_{y\in \sigma^{-1}
(x)}W(y) =1$ for all $x$.
The most important fact about $R_\sigma$ is that if $R_\sigma$ 
satisfies the 
pull-out property: for any functions $f$ and $g$ from $\FXB$,
\be\label{eq charact property of R_sigma}
R_\sigma((f\circ \sigma)g)(x) = f(x) (R_\sigma g)(x).
\ee

In case of the transfer operator given in (\ref{eq example R_sigma 
IFS}), it can be modified by considering a nontrivial weight function $W
$. Illustrating our further results, we will deal with $R_\sigma$ defined 
by (\ref{eq example R_sigma IFS}), or more generally by
\be\label{eq IFS with tri weight}
R'_\sigma (f)(x) := \cos^2(\frac{\pi x}{2})f(\frac{x}{2}) + \sin^2 
(\frac{\pi x}{2}) f(\frac{x+1}{2}), \ \ \ f \in \mc F,
\ee
as well. 

As we will see below, any normalized 
transfer operator defines an action on the 
set of probability measures. It is interesting to note that $R_\sigma$
 and $R'_\sigma$ have different properties relating to the 
 corresponding  
 invariant  measures. We  present them in the following  table. More
 detailed exposition of these results is given in Section 
 \ref{sect examples}.

\begin{table}[h!]
  \centering
  \caption{Invariant measures for $R_\sigma$ and $R'_\sigma$}
  \label{tab:table1}
  {\renewcommand{\arraystretch}{1.2}
  \begin{tabular}{c  c  c}
    \toprule
   Transfer operator & Lebesgue measure $\mu$  & Dirac measure $
   \delta_0$ \\
    \midrule
    \hline
   $R_\sigma$ (see (\ref{eq example R_sigma IFS}))
     & $\mu R_\sigma = \mu$  & $\delta_0 R_\sigma = 1/2(\delta_0 
     + \delta_{1/2}) \centernot\ll \delta_0$\\
    $R'_\sigma$ (see (\ref{eq IFS with tri weight})) & $\mu R'_\sigma \ll \mu$ and $d(\mu R'_\sigma) =
    2\cos^2(\pi x)dx$ & $\delta_0 R'_\sigma =\delta_0$\\
    \bottomrule
  \end{tabular}}
\end{table}
 
\end{example}

The following class of transfer operators is a continuous analogue of 
the operators defined by (\ref{eq def Ruelle operator classic}). 

\begin{example}[Frobenius-Perron operators]\label{ex F-P operator}
We follow here \cite{AbramovichAliprantis2001, LasotaMackey1994, 
DingZhou2009}. Suppose we have 
a standard measure space $\sms$ and a surjective 
non-singular endomorphism $\sigma$  acting on the space $\sms$. 
Let $P$ be a positive 
operator on $L^1\sms = L^1(\mu)$. It is said that $P$ is a
 \textit{Frobenius-Perron operator} 
\index{Frobenius-Perron operator} if for any 
$f \in L^1(\mu)$, and any set $A \in \B$,
\be\label{eq F-P operator}
\int_A P(f)\; d\mu = \int_{\sigma^{-1}(A)} f \; d\mu.
\ee
It can be easily checked that this Frobenius-Perron operator satisfies
the pull-out property (\ref{eq char property of r via sigma}).
 Furthermore, it follows from (\ref{eq F-P 
operator}) that $P$ preserves the measure $\mu$, i.e., $\mu P = 
\mu$ where $\mu P$ is defined by the formula:
$$
(\mu P)(A) = \int_X P(\chi_A)\; d\mu.
$$

More generally, we can define a ``non-singular'' Frobenius-Perron 
operator, meaning that $\mu P \ll \mu$:
\be\label{eq F-P operator with W}
\int_A P(f)\; d\mu = \int_{\sigma^{-1}(A)} Wf \; d\mu.
\ee
Then $W$ is the \textit{Radon-Nikodym derivative} 
\index{Radon-Nikodym derivative}
 of $\mu P$ with respect to $\mu$.
\end{example}

\begin{example}[Transfer operators on densities]\label{ex example 2} 
Let $\sigma$ be an onto endomorphism of a standard Borel space $(X, 
\B)$. Fix a Borel measure $\lambda$ on $(X, \B)$ such that $\la\circ 
\sigma^{-1} \ll \la$. Define a linear operator $ R = R_\la$ acting on 
non-negative functions $f$ from $L^1(\la)$ by the formula
\be\label{eq TO on densities}
R_\la( f)(x) = \frac{(fd\la)\circ \sigma^{-1}}{d\la},
\ee
where the right-hand side is the Radon-Nikodym derivative of the 
measure 
$(fd\la) \circ \sigma^{-1}$ with respect to $\la$. Then $R_\la$ is 
called a Ruelle transfer operator. It can be easily checked that $R_\la$
 satisfies the conditions of Definition \ref{def transfer operator from 
 Intro}:  (i) $R_\la$ is positive, (ii) $R_\la((f\circ \sigma)g) = fR_\la (g)$
  for any $f, g \in L^1(\la)$. We note that this operator $R_\la$ 
simultaneously acts on the set of Borel measures $M(X)$. 
The pull-out property of $R_\la$ (\ref{eq char property of r via 
sigma})  can be written in the equivalent form
$$
\int_X g (Rf) \; d\la = \int_X (g\circ\sigma) f \; d\la.
$$
Then one sees that $\la R_\la = \la$.
\end{example}

It turns out that the transfer operators defined in Examples \ref{ex F-P 
operator} and \ref{ex example 2} are related in a simple way.

\begin{lemma}
Let $P$ be  a Frobenius-Perron operator on $L^1\sms$ given by 
(\ref{eq F-P operator}). Then $P(f) = R_\mu(f)$ for any $f \in 
L^1(\mu)$.  If $P$ is defined by (\ref{eq F-P operator with W}), then
$P(f) = R_\mu(Wf)$.
\end{lemma}

\begin{proof} Indeed, relation (\ref{eq F-P operator with W}) can be written
in an equivalent form as
$$
\int_X P(f) g\; d\mu = \int_X (g\cs)fW\; d\mu.
$$
Then the lemma follows. 
\end{proof}

The next example is important and will be used later, see Sections 
\ref{sect integrable operators} and \ref{sect examples}. 

\begin{example}[Transfer operators via systems of conditional 
measures]\label{ex TO by cond meas}
 This example of  a transfer operator is of different nature and is based 
 on the notion of a \textit{system of conditional measures}. 
 \index{system of conditional measures}  The definitions of
  used terms can be found in Section \ref{section Basics}.

Let $\sms$ be a standard measure space with finite measure, and let 
$\sigma$ be an endomorphism onto $X$. Consider  the 
\textit{measurable partition} \index{measurable
 partition} $\xi$ of $X$ into preimages of $\sigma$, $\xi = 
 \{\sigma^{-1}(x) : x \in X\}$. Take the system of conditional 
 measures $\{\mu_C\}_{C \in \xi}$ corresponding to the partition 
 $\xi$ (see Definition \ref{def system cond measures}).

We define a  transfer operator $R$ on the standard 
probability measure space  $(X, \B, \mu)$ by setting 
\be\label{eq TO via cond syst meas Intro}
R(f)(x) := \int_{C_x} f(y)\; d\mu_{C_x}(y)
\ee
where $C_x$ is the element of $\xi$ containing $x$, i.e., $C_x = 
\sigma^{-1}(x)$. The domain of $R$ is $L^1(\mu)$ in this example.

\begin{lemma}\label{lem R via cond syst meas}
The operator $R : L^1(\mu) \to L^1(\mu)$ defined by (\ref{eq TO via 
cond syst meas Intro}) is a transfer operator. 
\end{lemma}

\begin{proof}
 Clearly, $R$ is a positive operator. To see that 
 (\ref{eq char property of r via sigma}) holds, we simply calculate
\begin{eqnarray*}
  R((f\circ\sigma) g)(x) & = & \int_{C_x} f\circ\sigma(y)g(y)\; d
  \mu_{C_x}\\
  \\
   &=& f(x) \int_{C_x} g(y)\; d\mu_{C_x}(y)\\
   \\
  &=& f(x) (Rg)(x).
\end{eqnarray*}
Here we used the fact that $f(\sigma(y)) = f(x)$ for $y \in C_x = 
\sigma^{-1}(x)$.
\end{proof}
\end{example}

More results about this type of transfer operators are discussed in 
Section \ref{sect examples}.

\subsection{Directions and motivational questions}

In this subsection, we formulate, in a rather loose manner, a few 
problems that could be considered as directions of further 
work in this area. 

If a transfer operator $R$ is defined by an endomorphism $\sigma$, 
then, as we will see below,  it is convenient to view at $R$  as a pair 
$(R, \sigma)$. This notation makes sense because the set of such 
pairs forms a semigroup, and moreover it  emphasizes that these two
 objects are closely  related to each other, according to 
the ``pull-out property'' given in (\ref{def transfer operator from Intro}) and
 (\ref{eq char property of r via 
sigma 1}). Next, this point of view is useful for the problem of 
classification of transfer operators. Clearly, the set $\mc R_\sigma$ 
of transfer operators defined by the same endomorphism $\sigma$
 can be vast, as we have seen in  the examples given in this  section.

To understand better the research directions of our approach,  we
 mention here a few questions which are not rigorously formulated 
but nevertheless  serve as motivational questions. 
Obviously, the study of possible relations between 
 \textit{Borel dynamical systems }$(X, \B, \sigma)$, or  
 \textit{measurable 
dynamical systems} $(X, \B, \mu, \sigma)$, and \textit{transfer 
operators} $R$, is a big multifacet problem, and we do not try to 
discuss all aspects of it here. 

In detail: 
(A) Suppose an endomorphism $\sigma$ is given in a standard Borel 
space $(X, \B)$. Denote by $\mc R_\sigma$ the set of all transfer 
operators $(R, \sigma)$ on $\mc F(X)$. What can be said about the 
properties of the set $\mc R_\sigma$? Clearly,  $\mc R_\sigma$  is a 
convex set. How can one find its extreme points? This 
 question becomes clearer when a measure $\mu$ is fixed on $(X, \B)
 $ and the operators $(R, \sigma)$ are considered in 
 $L^p(\mu)$-spaces. 

(B) The interaction between dynamical properties of endomorphisms 
and transfer operators, such as ergodicity, mixing, etc.,  has been
 discussed in many papers, see e.g., 
 \cite{LasotaMackey1994, DingZhou2009}. Our main interest is  the 
 study of the set of  measures which are quasi-invariant for both 
 transformations,  $\sigma$  and $R$. This approach  has been 
  productive for the Frobenius-Perron operators defined in Section
  \ref{sect Introduction}.  

(C) In treating positive, and transfer operators, $R$ as 
infinite-dimensional 
analogues of positive matrices, it is natural to raise the questions 
about spectral properties of such operators. If $h$ is a harmonic 
function for  $R$ and a measure $\mu$ is invariant (or ``quasi-
invariant'')  with  respect to $R$, then the relations $Rh = h$
and $\mu R = \mu$  ($\mu R \ll \mu$, respectively) are 
infinite dimensional analogues of eigenvectors in the matrix case 
for $R$. It would be 
interesting to find out how far the analogue with positive matrices 
can be extended to transfer operators.

(D) In the definition of a transfer operator, it is required that $R$ is an 
operator defined on the set of functions. In some cases, this action 
generates a ``dual'' action of $R$ on the set of all Borel measures 
$M(X)$. For instance, this is true for transfer operators defined on 
continuous functions over a compact Hausdorff space. How can one 
find, say, measures invariant with respect to $R$? Is there an 
interaction between actions of $R$ on functions and on measures?
In particular, we can define an equivalence relation on the set of all
transfer operators. Given $(R, \sigma)$,  let $\mc I(R,\sigma)$ be the 
set of all probability measures which are invariant with respect to $R$
and $\sigma$. It is said that $(R_1, \sigma_1)$ and $(R_2, 
\sigma_2)$ are \textit{measure equivalent} if $\mc I(R_1,\sigma_1)
= \mc I(R_2,\sigma_2)$. How can transfer operators be classified 
with respect to the measure equivalence relation?

(E) We will study transfer operators $R$ acting in various functional 
spaces. The same transfer operator $R$ and endomorphism $\sigma$ 
can be considered in 
different  frameworks depending on the choice of  its domain.  For 
instance, if $X$ is a  compact Hausdorff space and $\sigma$ is 
a continuous map on $X$, then it is natural 
to consider a transfer operator $(R, \sigma)$ as acting on 
continuous functions 
$C(X)$. At the same time, $(R, \sigma)$ can be viewed as a transfer
operator on the space of Borel functions $\FXB$, or on the
space $L^p(X, \B, \la)$. It would be interesting to understand how
 properties of $R$ depend on the choice of an underlying space. 

\subsection{Main results}\label{subsect Main theorems}
    A common theme is as follows: Given a transfer operator $(R, 
\sigma)$, what are the properties and the interplay between the 
following dual actions, action of $R$ on functions vs its action on 
measures? What is the interplay between the action of $R$ and that of 
an associated endomorphism $\sigma$? What are the important 
classes of quasi-invariant measures? These questions are answered in 
Sections 4 - 6, see especially Theorems 
\ref{thm  lambda R abs cntn wrt lambda}, \ref{thm erg decomp for TO},
\ref{thm measures vs operators}, \ref{thm from measure equivalence}, 
\ref{thm adjoint of R},  
\ref{thm L(R) is Q plus (sigma)}, and \ref{thm t_R is 1-1 and more}.
We also mention our main Theorems 
\ref{thm Wold}, \ref{thm S isometry}, \ref{thm K_1(R) tfae}, \ref{thm 
existence of p_k},
\ref{thm p_k in terms alpha}, \ref{thm measure mu_pi}, \ref{thm csm}, 
\ref{prop harmonic f-ns} from other sections (more important results are 
obtained in Corollaries \ref{cor support of mu_x}, and \ref{cor delta}). 

The necessary preparation and preliminary results are in Sections 2 - 3.
   
     For each of the classes of quasi-invariant measures, when do we 
have existence? This is the Perron-Frobenius setting, and now made 
precise in the general infinite-dimensional setting, and involving 
harmonic functions and measurable partitions. Our answers here are in 
Theorems \ref{thm mu R-invariance},  \ref{lem H(la) is S-inv}, 
\ref{thm R is adjoint of S}, and in 
Proposition \ref{prop harmonic}.

When does a given transfer operator $(R, \sigma)$ induce a 
multiresolution, i.e., a filtered system of subspaces, or of measures? 
And under what conditions does exactness hold? (See Theorem 
\ref{thm t_R is 1-1 and more}).
    
  In Theorems 4.13 and 10.6, we establish explicit measurable 
  partitions, co-boundary analysis, ergodic properties, and ergodic       
  decompositions.  In Theorem 8.10, we show that there is  a  
  \textit{universal Hilbert space} which realizes every transfer operator 
  $(R, \sigma)$.
  
  \newpage

\section{Endomorphisms and measurable partitions}
\label{section Basics}

In this section, we collect definitions and some basic facts about
 the underlying spaces, endomorphisms, measurable partitions, etc., 
 which are  used throughout the book.  Though these notions are 
 known in ergodic theory, we discuss them for the reader's
  convenience. The main references are the original works by Rohlin 
  \cite{Rohlin1949}, \cite{Rohlin1961}.  We refer also to
   \cite{Bogachev2007}, \cite{CornfeldFominSinai1982}
   \cite{Kechris1995}, 
    \cite{Renault1987}, \cite{Vershik2001}.
 \\

\subsection{Standard Borel and standard measure spaces} 
\label{subsect standard spaces}
Let $X$ be a separable completely metrizable topological space (a 
\textit{Polish space}, \index{Polish space} for short), and let $\mathcal B$ 
be  the \textit{sigma-algebra of Borel subsets} \index{sigma-algebra of Borel subsets}
 of $X$. Then  $(X, \mathcal B)$ is
called a \emph{standard Borel space}. \index{standard Borel space}
 If $\mu$ is a continuous (i.e., 
 non-atomic) Borel  measure \index{measure} 
 \index{measure ! sigma-finite} \index{measure ! Borel}  on 
 $(X, \mathcal B)$, then 
 $(X, \mathcal B, \mu)$ is called a \emph{standard measure space}.
  \index{standard measure space} 
In this book, we will use the same notation, $\B$, for Borel sets, and
 measurable  sets, of a standard measure space. It will be clear from 
 the context in what settings we are. Dealing with the sigma-algebra
  of measurable sets, we will assume that  $\B$ is  \textit{complete}
   with   respect to the measure $\mu$.  We will consider the set  
   $M(X)$ of all sigma-finite complete Borel measures on $(X, \B)$. Let
 $M_1(X) \subset M(X)$ denote the subset of probability measures.
For short,  an element of $M(X)$ will be called a measure.
If $\mu, 
\nu$ are two measures from $M(X)$, then $\mu$ is absolutely
 continuous with respect to $\nu$, $\mu \ll \nu$, if $\nu(A) = 0$
 implies $\mu(A) =0$.  
  Two measures  $\mu$ and $\nu$ on $(X, \B)$ are called
   \textit{equivalent} \index{measure ! equivalent},  $\mu \sim \nu$, 
   if they share the same sets  of measure zero, i.e., 
   $\mu\ll \nu$ and $\nu \ll \mu$. 

We denote by $\mathcal F(X)$ (or by $\mathcal F(X, \B)$)  the
 vector space of Borel functions. If a Borel measure $\mu$ is defined 
 on $(X, \B)$, we will work with $\mu$-measurable functions.

All objects considered in the context of measure spaces (such as sets,
 partitions, functions, transformations, etc) are considered by modulo
  sets of zero measure (they are also called null sets).  In most cases,
 we will implicitly use this $\mod\ 0$ convention.

 It is a well known fact that all uncountable standard Borel spaces
are Borel isomorphic,  and that all standard measure spaces 
are measure isomorphic. This means that  results do not depend 
on a specific realization of an underlying space. 
We will discuss this issue
in the context of \textit{isomorphic transfer operators} in Section \ref{sect TO
on measurable spaces}. 
\\

\subsection{Endomorphisms of measurable spaces}\label{subsect
Endomorphisms}

The notion of an endomorphism is a central concept of ergodic theory
and endomorphisms are studied extensively in many books and 
research papers. We mention only a few of them to present a wide
spectrum of research directions:  \cite{Rohlin1961}, 
\cite{Hawkins1994}, \cite{CornfeldFominSinai1982}, 
\cite{Beneteau1996},  \cite{BruinHawkins2009},  
\cite{PrzytyckiUrbanski2010}.

Let $\sigma$ be a Borel map of $(X, \B)$ onto itself. Such a map 
$\sigma$ is called an onto \textit{endomorphism} of $(X, \B)$. 
 \index{endomorphism}
In particular, $\sigma$ may be injective; in this case, we have a Borel  
\textit{automorphism} \index{automorphism} of $(X, \B)$. 
Since the  cardinality of  the 
set $\sigma^{-1}(x)$ is a Borel function on $X$, we can
  independently   consider the following classes:  $\sigma$ is either a
   finite-to-one or  countable-to-one map, or  $\sigma^{-1}(x)$ is an
   uncountable  Borel subset for any $x\in X$. 
 In general, we do not require that the  set $\sigma(A)$ is Borel but if 
  $\sigma$ is at most countable-to-one, then this property holds
automatically. 

We denote by $End(X, \B)$ the semigroup (with respect to the 
composition) of all surjective endomorphisms of the standard Borel 
space $(X, \B)$.

Given an endomorphism $\sigma$ of $(X, \B)$,  we denote by 
$\sigma^{-1}(\B)$ the proper subalgebra of $\B$ consisting of 
 sets $\sigma^{-1}(A)$ where $A$ is any set from $\B$.

We will use endomorphisms mostly in the context of  standard 
measure spaces $\sms$ with a finite (or sigma-finite) 
measure $\mu$.  Any endomorphism $\sigma$ of $(X, \B, \mu)$ 
defines an action on the set of measures $M(X)$ by 
$$ 
\mu \mapsto \mu\circ\sigma^{-1} : M(X) \to M(X),
$$
where $(\mu\circ\sigma^{-1})(A) := \mu(\sigma^{-1}(A))$. For a 
fixed measure $\mu$, it is said that $\sigma$ is a 
\textit{non-singular endomorphism} \index{endomorphism ! non-singular} 
(or equivalently that $\mu \in M(X)$ is 
a  (backward)  \emph{quasi-invariant measure} \index{measure ! 
quasi-invariant} 
 with respect to $\sigma$) if   $\mu\csi1$ is equivalent to $\mu$, i.e., 
  $$
\mu(A) = 0 \  \Longleftrightarrow \ \mu(\sigma^{-1}(A)) = 0,\ \ \
 \forall A\in \B.
$$
 Let $End\sms$ denote the set of all non-singular endomorphisms of 
 $\sms$.
 
\textit{ In this book, we consider only non-singular endomorphisms of 
 standard measure spaces.} We will also assume that $(X, \sB, 
 \mu_\sigma)$
 is again a standard measure space where  $\mu_\sigma$ is  the 
 restriction of $\mu$ to  $\sigma^{-1}(\B)$.

If $\mu(\sigma^{-1}(A))   = \mu(A)$ for any measurable set $A$, 
then $\sigma$ is called a  \textit{measure 
preserving endomorphism} 
\index{endomorphism ! measure preserving}, and $\mu$ is 
called  a \textit{$\sigma$-invariant measure}. \index{measure ! invariant} 

In some cases, we will also need the notion of a 
\textit{forward  quasi-invariant measure} \index{measure ! forward  
quasi-invariant} $\mu$. 
This means that, for every $\mu$-measurable  set 
$A$, the set  $\sigma(A)$ is measurable and $\mu(A) = 0 \ 
  \Longleftrightarrow  \ \mu(\sigma(A)) = 0$. For an at most   
  countable-to-one non-singular  endomorphism $\sigma$, this property 
  is automatically true. On the other hand, it it is not hard to construct 
an endomorphism  $\sigma$ of a  measure space $\sms$ such that $
  \sigma$ is not  forward quasi-invariant   with respect to $\mu$.

It is worth noting that, for standard
 measure spaces $\sms$ and non-singular $\sigma$, $\sigma(A)$ 
 is measurable when $\sigma$ satisfies the condition: $\mu(B) =0 
 \Longrightarrow \mu(\sigma(B)) = 0$ for any Borel set $B$.

\begin{lemma}\label{lem existence of q-inv measure for  sigma}
Let $\sigma$ be a surjective endomorphism of a standard Borel 
space $(X, \B)$. Then $M(X)$ always contains a 
$\sigma$-quasi-invariant measure $\mu$.
\end{lemma}

\begin{proof}
A proof is factually contained in  
\cite[Proposition 3.1]{DoughertyJacksonKechris1994} that can be
 easily adapted to the case of an endomorphism.
\end{proof}

We will keep the following notation for a surjective endomorphism 
$\sigma$ of $\FXB$:
$$
\mc Q_{-} = \{\mu \in M(X) : \mu \csi1 \sim \mu\}, \ \ 
\mc Q_{+} = \{\mu \in M(X) : \mu \cs \sim \mu\}. 
$$  

It is known that there are Borel endomorphisms $\sigma$ of $(X, \B)$
for which there exists no \textit{finite} $\sigma$-invariant measure,
see e.g. \cite{DoughertyJacksonKechris1994}.

\begin{remark} \label{rem RN derivatives for sigma}
Quasi-invariance of $\mu$ with respect to an endomorphism
 $\sigma$ of $\sms$ (backward and forward) allows us to
 define the notion of Radon-Nikodym derivatives 
 \index{Radon-Nikodym derivative} of measures $\lambda
 \csi1$ and $\lambda \cs $ with respect to $\lambda$:
$$
\theta_\lambda(x) = \frac{d\lambda\csi1}{d\lambda}(x) \ \ \
 \mbox{and} \ \ \ \omega_\lambda(x) = 
 \frac{d\lambda\circ\sigma}{d\lambda}(x).
$$
In other words, for any function $f\in L^1(\lambda)$, one has
$$
\int_X f\circ\sigma \; d\lambda = \int_X f \theta_\lambda\; d\lambda
$$
 and
$$
 \int_X (f\circ\sigma) \; \omega_\lambda\; d\lambda = 
 \int_X f \; d\lambda.
$$
To justify these relations, we observe that  $\lambda$ and 
$\lambda\circ \sigma$ are well defined measures when they are
 considered on the subalgebra  $\sigma^{-1}(\B)$. When $\sigma$ 
 is forward  quasi-invariant with respect to $\lambda$, we can 
 uniquely  define   the $\sigma^{-1}(\B)$-measurable function 
 $\omega_\lambda(x)$.  Since $\theta_\lambda\circ\sigma$ is also 
   $\sigma^{-1}(\B)$-measurable, then, by uniqueness of the 
   Radon-Nikodym derivative, we obtain that
$$
\omega_\lambda(x) = \frac{1}{\theta_\lambda}(\sigma x).
$$
\end{remark}

The following fact is obvious. 

\begin{lemma}
 Suppose that $\sigma \in End\sms$ and $\nu$ is a measure
  equivalent to $\mu$, i.e., there exists a measurable function $\xi$
  such that 
$d\nu(x) = \xi(x) d\mu(x)$. Then $\sigma$ is also non-singular with
 respect to $\nu$, and  $\theta_\nu$ is cohomologous to 
 $\theta_\mu$,  i.e., $\theta_\nu(x) = \xi (\sigma x) 
 \theta_\mu(x) \xi(x)^{-1}$.
\end{lemma} 

Here we define the most important dynamical properties of 
endomorphisms. 

\begin{definition}\label{def ergodic and exact}
Let If $\sigma \in End\sms$. 

(i) The endomorphism $\sigma$ is called
 \textit{conservative} \index{endomorphism ! conservative} 
 if for any set $A$ of positive measure there 
 exists $n >0$ such that $\mu(\sigma^n(A) \cap A) > 0$.

(ii) The endomorphism $\sigma$  is called 
\textit{ergodic} \index{endomorphism ! ergodic}
 if whenever $A$ is $\sigma$-invariant, i.e. $
\sigma^{-1}(A) = A$,  then either $A$ or $X \setminus A$ is of 
measure zero.  

(iii) For $\sigma \in End\sms$, one associates the sequence of 
subalgebras generated by $\sigma$:
$$
\B \supset \sigma^{-1} (\B) \ \cdots \ \supset \sigma^{-i} (\B)
\supset \ \cdots
$$
Then $\sigma \in End\sms$ is called \textit{exact} \index{endomorphism ! 
 exact} if 
 $$
 \B_\infty := \bigcap_{k \in \N} \sigma^{-k}(\B) = \{\emptyset, X\} \ 
 \ \mod 0.
 $$
 Clearly, every exact endomorphism is ergodic. 
\end{definition}

The nested (filtered) family of sigma-algebras from Definition \ref{def 
ergodic and exact} is a 
recurrent theme in symbolic dynamics, and in ergodic theory;-- for 
details, see, for instance, \cite{Kakutani1948}, \cite{Rohlin1961}, 
\cite{Ruelle1989},  \cite{CornfeldFominSinai1982}, 
\cite{Jorgensen2001}, \cite{Jorgensen2004}, \cite{Hawkins1994}.
 A main theme in our work is to point out that this basic filtered 
 system has three incarnations in our 
analysis, each important in a systematic study of transfer operators.

In more detail: The starting point for our study of infinite-dimensional 
analysis of transfer operators is a fixed system $(X, \B, R, \sigma)$ as 
specified above, i.e., a fixed transfer operator $R$, subject
to the pull-out property for $\sigma$, as in Definition \ref{def transfer 
operator from Intro}. The three 
incarnations we have in mind of the scale of sigma-algebras from 
Definition  \ref{def 
ergodic and exact} are: (i) measure-theoretic (Sections \ref{sect TO
 on measurable spaces} and \ref{sect integrable operators}), (ii) 
geometric/symbolic (Sections \ref{sect TO
 on measurable spaces}, \ref{sect TO on densities},  
 and \ref{sect Example IFS}), and (iii) operator 
theoretic (Sections  \ref{sect L1 and L2}, \ref{sect Wold}, and
  \ref{sect Universal HS}). In each of these settings, we show 
that when $(X, \B, R, \sigma)$ is given, then the system from 
Definition \ref{def 
ergodic and exact} induces corresponding scales of measures, of 
certain closed 
subspaces in a suitable universal Hilbert space, and in geometric 
systems of self-similar scales; referring to (i)-(iii), respectively. The 
details and the applications of these three correspondences will be 
presented systematically in in the respective sections (below), inside 
the body of the book.

\subsection{Measurable partition and subalgebras}\label{subsect meas 
partitions} 
We give here a short overview of the theory of measurable partitions,
developed earlier in a series of papers by V.A. Rohlin 
(see his pioneering 
article  \cite{Rohlin1949} and the book \cite{CornfeldFominSinai1982} 
for further references). Later on, the ideas and methods of this 
theory were used in many papers. We refer to the works 
  \cite{Vershik_Fedorov1985,  Vershik1994, Vershik2001} where the 
  orbit theory of dynamical  systems was studied in the framework of 
  sequences of measurable   partitions. 

 Let $\xi = \{C_{\alpha} : \alpha \in I\}$ be a \textit{partition} 
 \index{partition}  of a standard
probability  measure space $(X, \mathcal B, \mu)$ into measurable 
sets.  We will focus on  the most interesting case when all sets $C_
\alpha$ and the  index set $I$ are uncountable (though some 
endomorphisms, arising  in the  examples considered below, have finite 
preimages). 

One says that a set $A = \bigcup_{\alpha \in I'} C_\alpha$ is a
 \emph{$\xi$-set} where $I'$ is any subset of $I$. Let $\mathcal
  B(\xi)$ be the sigma-algebra of $\xi$-sets. Clearly,  $\mathcal B(\xi)
   \subset \mathcal B$. 

By definition, a partition $\xi$ is called \emph{measurable} \index{partition 
! measurable} if  
$\mathcal B(\xi)$  contains a countable subset $(D_i)$ of $\xi$-sets
such that it separates any two elements $C, C'$ of $\xi$:   there
exists $i \geq 1$ such that either $C \subset D_i$ and $C' \subset 
X \setminus D_i$ or $C' \subset D_i$ and $C \subset 
X \setminus D_i$. 
   
Any partition $\xi$ defines the \textit{quotient space}
 $X/\xi$ whose points
are elements of $\xi$. Let $\pi$ be the natural projection from $X$ 
  to $X/\xi$. For $\mu \in M_1(X)$, we define the probability measure
   $\mu_\xi$ on   $\mathcal B(\xi)$ by setting $\mu_\xi
= \mu \circ \pi^{-1}$.

It can be proved that \textit{$\xi$ is measurable if and only
if $(X/\xi, \mathcal B(\xi), \mu_\xi)$ is a standard measure space.}
More generally, suppose $(X, \mathcal B, \mu)$ and $(Y, \mathcal C,
 \nu)$ are two standard measure spaces. Let $\varphi : X \to Y$ be a
 measurable map. Then the partition $\zeta := \{\varphi^{-1}(y) : y
 \in Y\}$ is obviously measurable.  In particular,  $\varphi$ can be
a surjective non-singular endomorphism of $(X, \B, \mu)$. In this case,
we see that the partition $\zeta(\varphi) := \{\varphi^{-1}(x) : x \in X
\}$ has the following  properties 
\be\label{eq quotient for endo}
X/\zeta(\varphi) = X, \ \  \B(\varphi) = \varphi^{-1}(\B), \ \ 
\mu_{\varphi} = \mu|_{\varphi^{-1}(\B)}
\ee
Hence, the partition  $\zeta(\varphi)$ is indexed by 
points of the space $X$, that is the quotient space 
$X/\zeta$ is identified with $X$. 

Let $Orb_\varphi(x) := \{y \in X : \varphi^m(y) = \varphi^n(x) \ 
\mbox{for\ some\ } m, n \in \N_0\}$ be the orbit of $\va$  through
$x\in X$. Then, in contrast to the above 
partition $\zeta$, the partition of $X$ into orbits of $\varphi$ is not 
measurable, in general.

 We recall here a few facts and definitions about measurable partitions
  that will be used below. It is said that a partition $\zeta$ 
  \emph{refines} $\xi$
 (in symbols, $\xi\prec \zeta $) if every element $C$ of $\xi$ is a 
 $\zeta$-set. If $\xi_\alpha$ is a family of measurable partitions, then
  their product $\bigvee_\alpha \xi_\alpha$ is a measurable partition 
  $\xi$ which is uniquely determined by the conditions: (i) $\xi_\alpha
  \prec \xi$ for all $\alpha$, and (ii) if $\eta$ is a measurable partition
   such that $\xi_\alpha \prec \eta$, then $\xi \prec \eta$. Similarly,
    one defines the intersection $\bigwedge_\alpha\xi_\alpha$ of
     measurable partitions.

It turns out that every partition $\zeta$ has a \emph{measurable hull},
 that is a measurable partition $\xi$ such that $\xi \prec \zeta$ and 
 $\xi$ is a maximal measurable partition with this property. In order to
 illustrate this fact, we consider a measurable automorphism  $T$ of a
  measure space  $(X, \mathcal B, \mu)$ and define the partition 
  $\zeta(T)$ of $X$ into orbits of $T$, $\zeta(T)(x) = \{T^i x : i \in
   \Z\}$. In general, $\zeta(T)$ is not measurable. There exists a 
  measurable partition $\xi$, the measurable hull of $\zeta(T)$, which
  is known as the partition of $X$ into \textit{ergodic components} of 
  $T$. If $T$ is ergodic, then $\xi$ is the trivial partition.

\begin{lemma} 
There is a one-to-one correspondence between the set of measurable
 partitions of a standard measure space $(X, \mathcal B, \mu)$ and
  the set of complete sigma-subalgebras $\mathcal A$ of $\mathcal 
  B$. This correspondence is defined by assigning to each partition 
  $\xi$ the sigma-algebra $\B(\xi)$ of $\xi$-sets. Moreover, 
  $$
  \mathcal A(\bigwedge_\alpha\xi_\alpha) = 
  \bigcap_\alpha\mathcal A(\xi_\alpha),  \qquad 
\mathcal A(\bigvee_\alpha \xi_\alpha) = \bigvee_\alpha \mathcal
 A( \xi_\alpha)
 $$ 
 where the latter is the minimal sigma-subalgebra that contains all 
 $\mathcal A(\xi_\alpha)$.
 \end{lemma}

We need the following classical result due to Rokhlin 
\cite{Rohlin1949} about the disintegration of probability measures.

\begin{definition}\label{def system cond measures}
For a  standard probability  measure space  $(X, \mathcal B, \mu)$
 and a measurable partition $\xi$ of $X$, we say that a  collection of
  measures $(\mu_C)_{C \in X/\xi}$ is a \emph{system of conditional
   measures}  \index{system of conditional measures} with respect to 
$((X, \mathcal B, \mu), \xi)$ if

(i) for each $C \in X/\xi$,  $\mu_C$ is a measure on the
sigma-algebra $\mathcal B_C := \mathcal B \cap C$ such that $(C,
 \mathcal B_C, \mu_C)$ is a standard probability measure space;

(ii) for any $B \in \mathcal B$, the function $C \mapsto \mu_C(B
 \cap C)$ is $\mu_\xi$-measurable;

(iii)  for any $B \in \mathcal B$,
\be\label{eq mu(B) integral}
\mu(B) = \int_{X/\xi} \mu_C(B \cap C)\; d\mu_\xi(C).
\ee
\end{definition}

\begin{theorem}[\cite{Rohlin1949}] \label{thm Rokhlin disintegration}
For any measurable partition $\xi$ of a standard probability measure
 space $(X, \mathcal B, \mu)$, there exists a unique system of
  conditional measures \index{system of conditional measures} 
  $(\mu_C)$. Conversely, if  $(\mu_C)_{C \in 
  X/\xi}$ is a system of conditional measures with respect to $((X,
   \mathcal B, \mu), \xi)$, then $\xi$ is a measurable partition.
\end{theorem}

We notice that relation (\ref{eq mu(B) integral}) can be written
 as follows:
for any $f \in L^1(X, \B, \mu)$,
\be\label{eq function integration csm}
\int_X f(x)\; d\mu(x) = \int_{X/\xi} \left( \int_C f_C(y)\; 
d\mu_C(y)\right)\, d\mu_\xi(C)
\ee
where $f_C = f|_C$.

We can apply this theorem to the case of a surjective endomorphism $
\varphi \in End(X, \B, \mu)$ 
as described above. Let $\zeta(\varphi)$ be the 
measurable partition of $\sms$  into preimages $\varphi^{-1}(x)$
of points $x \in X$ (see (\ref{eq quotient for endo}). 
Let $(\mu_C)$ be the system of conditional 
measures defined by $\zeta(\varphi)$. Then relation (\ref{eq function
 integration csm}) has the form 
\be\label{eq function disintegration for endo}
\int_X f(x)\; d\mu(x) = \int_X \left( \int_C f_C(y)\; 
d\mu_C(y)\right)\, d\mu_{\varphi}(C)
\ee

In most important cases, the disintegration of a measure is applied to
 probability (finite) measures. The case of an infinite sigma-finite
  measure was considered by several authors. We refer here to 
   \cite{Simmons2012}. The result is formulated in a slightly more 
  general terms, in comparison with probability measures.

Let $\sms$ and $(Y, \mathcal A, \nu)$ be standard measure spaces
 with sigma-finite measures, and suppose that $\pi : X \to Y$ is a
  measurable map.  By definition, a \emph{system of conditional
 measures}  is a collection of measures $(\nu_y)_{y \in Y}$ such that

i) $\nu_y$ is a measure on the standard measure space $(\pi^{-1}(y), 
\B \cap \pi^{-1}(y))$, $y \in Y$;

ii) for every $B \in \B$,
$$
\mu(B) = \int_Y \nu_y(B) \; d\nu(y).
$$

\begin{theorem}[\cite{Simmons2012}]\label{thm simmons}
Let $\sms$ and $(Y, \mathcal A, \nu)$ be as above. Suppose that 
$\widehat \mu= \mu\circ \pi^{-1} \ll \nu$. Then there exists a unique     
 system of conditional measures $(\nu_y)_{y \in Y}$ for $\mu$. For 
 $\nu$-a.e., $\nu_y$ is a sigma-finite measure.
\end{theorem}

The structure of countable-to-one endomorphisms is described in 
the following result.

\begin{theorem}[\cite{Rohlin1949}]\label{thm Rohlin partition}  For a 
countable-to-one endomorphism $\sigma$ of $\sms$, there exists a 
partition $\zeta = (A_1, A_2, ... )$ of $X$ into at most countably 
many elements such that

(i) $\mu(A_i) > 0$ for all $i$;

(ii) $\sigma_i := \sigma|_{A_i}$ is one-to-one and $A_i$ is of maximal 
measure in $X \setminus \bigcup_{j <i}A_j$ with this property. In 
particular, $\sigma_1$ is one-to-one and onto $X$.
\end{theorem}

Clearly, the Rohlin partition  $\zeta$ is finite if $\sigma$ is bounded-
to-one. Let $\tau_i$ be a one-to-one Borel map with domain 
$\sigma_i(A_i)$ such that $\sigma\circ \tau_i = \mbox{id}$. Then 
the collection of maps $\tau_i$'s represents the inverse branches of
$\sigma$. They are used in explicit constructions of positive operators
related to iterated function systems. This type of endomorphisms 
arises also as shifts on stationary Bratteli diagrams. 

\subsection{Solenoids and applications} To finish this section we recall the 
construction of natural extension of endomorphisms (or a solenoid in
 other terms).

Let $\sigma$ be an endomorphism of a standard Borel space 
$(X, \mathcal B)$. We associate to $( (X, \mathcal B), \sigma)$ a 
\emph{solenoid } \index{solenoid} $Sol_\sigma(X)$ as follows. By definition,
$$
Sol_\sigma (X) :=\{ y = (x_i) \in \prod_{i=0}^\infty X  : \sigma(x_{i
+1}) = x_i,\ i \in \N_0\}.
$$
Since $Sol_\sigma (X)$ is a Borel subset of $\prod_{i=0}^\infty (X, 
\mathcal B) $, any solenoid is a standard Borel space in its turn. If $X$ 
is a compact space,
then $Sol_\sigma(X)$ is also a  compact subset. Furthermore, $Sol_
\sigma(X)$ is an invariant subset of $ \prod_{i=0}^\infty X $ with 
respect to the shift $\sigma_0 (x_i) = (\sigma x_i)$. We use the 
notation $\pi_i, i \in \N_0,$ for the projection from $ Sol_\sigma(X)$ 
onto $X$, $\pi_i((x_i)) = x_i$.

\begin{lemma}
Let $\lambda$ be a Borel measure on $(X, \mathcal B)$. In the above 
notation, the partition of $Sol_\sigma(X)$ into the fibers 
$\{\pi_0^{-1}(x) : x \in X\}$ is measurable.
\end{lemma}

Starting with a transfer operator system $(X, B, \sigma, R)$, there is a 
general procedure for extending to an invertible dynamical system, now 
realized on an associated solenoid; see the outline here in Lemma 2.9. 
As documented in the literature (see, for example, 
\cite{BratteliJorgensen1997, BratteliJorgensen2002, 
BaggettJorgensenMerrillPacker2005, BaggettMerrillPackerRamsay2012,
 DutkayJorgensen2006, DutkayRoysland2007, Dutkay2002, 
 FarsiGillaspyKangPacker2016, Jorgensen2001, Jorgensen2004, 
 JorgensenTian2015}), 
there are many applications of this construction: (i) the given 
endomorphism $\sigma$ lifts in a canonical fashion to an 
automorphism on the solenoid; (ii) under suitable assumption, the 
given transfer operator system $(X, B, \sigma, R)$ then admits a 
realization by unitary operators, again realized on suitable $L^2$ 
spaces 
and realized on the solenoid; and (iii) the construction in (ii) includes 
families of generalized wavelets. These wavelet families in turn include 
as special cases more traditional multi-resolution wavelet constructions 
considered earlier in the standard Hilbert space $L^2(R^d)$. Under 
suitable restrictions, in fact, $L^2(R^d)$ embeds naturally in an $L^2$ 
space on the solenoid. We shall refer to the cited literature for details 
regarding (i)-(iii), but see also \cite{AlpayJorgensenLewkowicz2016}.

For the solenoid $Sol_\sigma(X)$, we define a Borel map  $\wt\sigma
$ of the solenoid by setting
\be\label{eq automorphism wt sigma}
\wt\sigma(x_0, x_1, x_2, ... ) = (\sigma(x_0), x_0, x_1,...)
\ee

\begin{lemma}
The transformation $\wt\sigma: Sol_\sigma(X) \to Sol_\sigma(X)$ 
is a one-to-one and onto map, i.e. $\wt\sigma$ is a Borel 
automorphism of the solenoid. 
\end{lemma}

\begin{proof}
To see this, we set
$$
\wt\sigma^{-1}(x_0, x_1, x_2, ...  ) = ( x_1, x_2, ...  ).
$$
Then, the relation $\wt\sigma^{-1}\wt\sigma = \mbox{id}$ is 
obvious. On the other hand, for any $(x_0, x_1, x_2, ...  ) \in Sol_
\sigma(X)$, we have
$$
\wt\sigma\wt\sigma^{-1} (x_0, x_1, x_2, ...  ) =  \wt\sigma( x_1, 
x_2, ...  ) =
\wt\sigma( y_0, y_1, ...  )
$$
$$
= (\sigma( y_0), y_0, y_1, ...  ) = (\sigma( x_1), x_1, x_2, ...  ) = 
(x_0, x_1, x_2, ...  )
$$
where $y_i = x_{i+1}, i \geq 0.$
\end{proof}

We will use this construction below. 

\begin{remark} It is worth noting that, based on 
the definition of a transfer operator built by a system of conditional 
measures, see Example \ref{ex TO by cond meas}, we can immediately 
extend the 
main results of \cite[Theorems 3.1, 3.4]{DutkayJorgensen2007} to
the case of an arbitrary surjective endomorphism $\sigma$.
\end{remark}

\newpage


\section{Positive, and transfer, operators on measurable spaces:
 general properties} \label{sect TO on measurable spaces}

The notions of positive operators and transfer operators are central 
objects in this book. We will discuss various properties of these 
operators and their specific  realization in the subsequent sections.
 Here we first focus on the most general properties  and basic 
 definitions related to these operators. We also refer to 
 \cite{Karlin1959} as one of the pioneering papers on positive 
 operators. 
 
While the setting for a study of transfer operators, and more general 
positive linear operators, is that of a set $Y$, and a fixed 
sigma-algebra $\mc A$, in order to get explicit characterizations, it is 
useful to 
restrict attention to \textit{standard Borel spaces}; so the case when 
$(Y, \mc A)$ is now a pair $(X,\B)$ given to be isomorphic to some 
separable complete metric space (a Polish space) with associated Borel 
sigma-algebra $\B$ ; or  $(X,\B)$ is isomorphic to some uncountable 
Borel subset of some separable complete metric space with the 
induced Borel sigma-algebra. Generally we allow for the possibility that 
$X$ is non-compact.

  By a transfer operator in $(X, \B)$ we mean a pair $(R, \sigma)$ 
satisfying the conditions in Definition \ref{def transfer operator} (i), 
and (\ref{eq char property of r via sigma 1}). The starting point in the 
present section is a fixed pair $(R, \sigma)$ on $(X, \B)$, defining a 
transfer operator; and we begin with a systematic study of various 
sets of measures on $(X, \B)$ which allow us to derive spectral 
theoretic information for the transfer operator $(R, \sigma)$ under 
consideration. For this purpose, we also make precise a notion of 
\textit{isomorphisms} of pairs of transfer operators $(R, \sigma)$; see 
Definition \ref{def iso for TO}. Our study of \textit{measure classes }
associated to a fixed $(R, \sigma)$ will be undertaken in the two 
sections to follow.
 
\subsection{Transfer operators on Borel functions}\label{sect Basics
 on PO and TO}
We will consider positive and transfer operators acting in some 
natural spaces  $\mathcal F$ of real-valued functions. Examples of
 such spaces are: $\FXB$, $L^p\sms\ (1\leq p \leq \infty)$, $C(X)$ 
 (if  $X$  is considered as a compact Hausdorff space), etc. 
 In all these spaces, the generating cone $\mathcal F_+$ of  
 non-negative  functions is  obviously defined.
  Hence, we can define a  \textit{positive  operator} $P$ 
  \index{positive operator}  as a  linear     operator that preserves 
 the cone of non-negative   functions: $f \in \mathcal F_+ \ 
 \Longrightarrow \ P(f) \in    \mathcal{F}_+$. If a Borel measure 
 $\mu$ is given on $(X, \B)$, then we can consider non-negative
 elements of the space $L^p\sms$ and define a positive operator $P$
 on $L^p(\mu)$ similarly. 
 
 Let $(X,\B)$ be a standard Borel space,  and let $\sigma$ be
a surjective Borel  endomorphism of $(X, \B)$. A  function $f$ is
called $\sigma^{-1}(\B)$-measurable if $f \in \mathcal 
F(X, \sigma^{-1}(\B))$.  For any function $f \in \mathcal F(X)$, 
the function $f\circ\sigma$ is
 constant on every element of the partition $\xi =\{\sigma^{-1}(x) : x 
 \in X\}$, and therefore $f\circ\sigma$ is measurable with respect to 
 $\sigma^{-1}(\B)$. 
Thus, it can be easily seen that a Borel function $g$ is
$\sigma^{-1}(\B)$-measurable if and only if there exists a 
Borel function $G$ such that $g = G\cs$. In this settings, the operator
$U : \FXB \to \mathcal F(X, \sigma^{-1}(\B)): f \mapsto 
f\cs$ is positive and called the \textit{composition operator}. 
\index{composition operator} In the
 framework of ergodic theory this operator $U$ being considered 
 on  the spaces $L^1(\mu)$ or $L^2(\mu)$  is  known by the
  name of \textit{Koopman operator.}, see e.g., \cite{Ruelle1978}. 

The set of positive operators contains an important class of operators
called \textit{transfer operators}. We find it useful to expand the
 definition of a transfer operator from Section \ref{sect Introduction},
 giving more details now.

\begin{definition}\label{def transfer operator}
(1) Let  $\sigma : X \to X$ be a surjective endomorphism of a
 standard Borel space $(X, \B)$. We say that $R$ is a \emph{transfer 
 operator} \index{transfer operator} if $R : \mathcal F(X) \to 
 \mathcal F(X)$ is a linear operator  satisfying the properties:

(i) $f \geq 0 \ \Longrightarrow \ R(f) \geq 0$ (i.e., $R$ is positive);

(ii)  for any Borel functions  $f, g \in \mathcal F(X)$,
\be\label{eq char property of r via sigma 1}
R((f\circ \sigma) g) = f R(g).
\ee

(2) For a non-singular endomorphism $\sigma$ on $\sms$, we define 
similarly a transfer operator acting in the space $L^p\sms, 1 \leq p
\leq \infty$.
 
(3) If $R(\mathbf 1)(x) > 0$ for all $x\in X$, then we say that $R$ is 
a \emph{strict} \index{transfer operator ! strict} transfer operator (here and below the expression 
$R(\mathbf 1)$ means the image of the constant function that takes 
value 1 under the action of $R$).

(4) If $R(\mathbf 1) = \mathbf 1$, then the transfer operator $R$ is 
called \textit{normalized}. \index{transfer operator ! normalized} 

(5) If $h$ is a non-negative function such that $Rh = h$, then $h$
is   called a\textit{ harmonic function.} \index{harmonic function}

\end{definition}

We use also the notation $(R, \sigma)$ for a transfer operator $R$ to 
emphasize that these two objects are closely related according to the
 ``pull-out property'' given in  (\ref{eq char property of r via sigma 
 1}). Moreover, this point of view is useful for the problem of
  classification of transfer operators (see the corresponding definitions 
  below in this section). It is worth
remarking that the set $\mathcal R(\sigma)$ of transfer operators $R
$ defined by the same endomorphism $\sigma$ can be vast.

\begin{remark}
(1) Since we work with standard Borel and measure spaces, the 
transfer 
operators do not depend on underlying space, in general. This means 
that  if $(X, \B)$ and $(Y, \mathcal A)$ are standard Borel spaces and
$\psi: (X, \B) \to (Y, \mathcal A)$ is a Borel map implementing the 
Borel isomorphism of these spaces, then, for every transfer operator 
$(R, \sigma)$ acting in $\FXB$, there exists an isomorphic transfer
operator $(R', \sigma')$ acting on the space $\mathcal F(Y, \mathcal 
A)$. We discuss the notion of isomorphism of transfer operators 
below in this section. 

(2)  When we discuss properties of a transfer operator $R$, we will 
mostly  work with non-negative Borel (or measurable) functions. The 
point is that if a transfer operator $R$ is defined on the cone of 
positive functions $\mathcal F(X)_+$, then $R$ is naturally extended 
to $\mathcal F(X)$ by linearity. The same approach is used in all 
statements related to integration with respect to a measure $\la$.
\end{remark}

The point of view on transfer operators as pairs $(R, \sigma)$ allows 
us to introduce a semigroup structure on such pairs. 

Let $\sigma \in End(X, \B)$, and let 
$$
\mathcal 
R(\sigma) := \{(R, \sigma) : R \ \mathrm{is\ a\ transfer\ operator\ 
w.r.t.}\ \sigma \}. 
$$
Denote 
$$
\mc R(X, \B) := \bigcup_{\sigma \in End(X, \B)} \mc R(\sigma).
$$ 

\begin{lemma}\label{lem semi-group}
(1) The set $\mc R(X, \B)$ is a semigroup with identity with respect 
to the product 
$$
(R_1R_2, \sigma_1\sigma_2) =  (R_1, \sigma_1)(R_2, \sigma_2).
$$
(Here the notation $R_1R_2$ and $\sigma_1\sigma_2$ means
the composition of mappings.)

(2) The set $\mc R(\sigma)$ is a vector space for each fixed $\sigma
\in End(X, \B)$. 
\end{lemma}

\begin{proof} 
Let   $(R_1, \sigma_1)$ and $(R_2, \sigma_2)$ be two transfer 
operators, where $R_i : \mathcal F(X, \B) \to \mathcal F(X, \B)$ and
 $\sigma_i$ is an onto endomorphism of $(X, \B)$, $i =1,2$.
 We need to check that $(R_1R_2, \sigma_1\sigma_2) $ is a well 
 defined transfer operator in $\FXB$. Since the range of any  transfer
  operator $(R, \sigma)$ is $\FXB$ (see Lemma \ref{lem TO is onto 
1-1}), the composition $R_1R_2$  is defined. It remains to check that
$(R_1R_2, \sigma_1\sigma_2) $ satisfies Definition \ref{def transfer 
operator}. The positivity is obvious and 
\begin{eqnarray*}
  R_1 R_2[f(\sigma_1\sigma_2(x)) g(x)] &=&
   R_1[f(\sigma_1(x))R_2(g(x))]\\
   &=& f(x) R_1R_2( g)(x)
\end{eqnarray*}

The second claim is clear because, for $a, b \in \mathbb R$,
$$
(a R_1 + b R_2)[(f\circ \sigma) g] = a fR_1(g) + b fR_2(g) =
 f(aR_1+ bR_2)(g).
$$
\end{proof}

The dynamical properties of  endomorphisms $\sigma$ such as 
ergodicity, mixing, etc can be described in terms of transfer operators,
see \cite{LasotaMackey1994, DingZhou2009}. We mention here 
several simple observation to motivate our future study.

\begin{remark}\label{rem links R and sigma}
(1) Suppose that $(R, \sigma)$ is a transfer operator acting on the
 space $(X, \B, \mu)$. If $\sigma$ is not an ergodic endomorphism of 
 $(X, \B, \mu)$, then for any $\sigma$-invariant set $A$ of positive 
 measure ($\sigma^{-1}(A) = A\ \mod 0$), we  can define the
  restriction of $R$ on $(A, \B|_A)$. For this, we set
$$
R_A (f) = R(\chi_A f), \qquad f \in \mc F(X, \B).
$$
\begin{lemma}\label{lem invariant set}
The operator $R_A : \mc F(A, \B|_A) \to \mc F(A, \B|_A)$ 
is a transfer operators corresponding to $\sigma_A = \sigma
 : A \to A$. 
\end{lemma}

\begin{proof} We need to check that $(R_A, \sigma_A)$ satisfies the
 Definition \ref{def transfer operator}: 
\begin{equation*}
\setlength{\jot}{10pt}
\begin{aligned}
R_A((f\circ \sigma) g) &= R(\chi_A(f\circ \sigma) g)\\
& = R((\chi_A)^2(f\circ \sigma) g)\\
&= R((\chi_A\circ\sigma)(f\circ \sigma) \chi_A g)\\
& = \chi_A f R(\chi_A g)\\
& =  \chi_A f R_A( g).
\end{aligned}
\end{equation*}
We used here the relation $\chi_{A} = \chi_{\sigma^{-1}(A)} =
\chi_{A}\cs$.
\end{proof}

(2) Suppose that $\sigma$ is periodic on $(X, \B)$ of period $p$, i.e., 
$\sigma^p(x) = x$ for all $x$. If $R$ is a transfer operator 
from $\mc R(\sigma)$, then $R$ is also periodic. Indeed, $R^p$ is 
a transfer operator corresponding $\sigma^p$. Hence, for any 
functions $f, g \in \FXB$, it satisfies the relation $R^p(f) g =
f R^p(g)$ which means that $R^p$ is the identity operator. 
\end{remark}

The statements, proved in Remark \ref{rem links R and sigma}, mean 
that the classes of all ergodic endomorphisms of a measure space, and
the aperiodic endomorphisms for Borel spaces, play the central role.  
Thus, we can avoid some trivialities by considering only ergodic and/or 
aperiodic  endomorphisms $\sigma$. 
\\

\subsection{Classification} The problem of classification of transfer operators has many aspects
and depends on the choice of equivalence relations on the set of all 
transfer operators. We consider only the definition of isomorphic 
transfer operators $(R, \sigma)$. 
For motivation, we begin with the following example.

\begin{example}
Suppose that $\sigma$ and $\sigma'$ are two surjective 
endomorphisms of $(X, \B)$ such that $\sigma' \tau(x)  =\tau\sigma 
(x)$ for some one-to-one Borel map $\tau$ and all $x \in X$. We 
define the operator $S := S_\tau$ acting on the set of Borel
functions $f \in \mathcal F(X, \B)$ by the formula
$$
(S f)(x):= f(\tau x).
$$

\begin{lemma} Let $(R, \sigma)$ be a transfer operator in $\FXB$. 
Then  $R' = (S^{-1}RS$ is  transfer operator corresponding to the 
endomorphism $\sigma'$.
\end{lemma}

This observation follows from the facts that $R'$ is positive and $(R', 
\sigma')$ satisfies the relation:
\begin{equation*}
\setlength{\jot}{10pt}
\begin{aligned}
S^{-1}RS[f(\sigma' x) g(x)] &=  S^{-1}R[f(\sigma' \tau x) 
g(\tau x)] \\
   &= S^{-1}R[f(\tau (\sigma x)) g(\tau x)]  \\
   &= S^{-1}[f\circ \tau(x) (Rg)(\tau x)] \\
   &=  f(x) (S^{-1}RSg)(x).
\end{aligned}
\end{equation*}
\end{example}

The next definition is a generalization of the above example. 

\begin{definition}\label{def iso for TO} Let $\sigma_i $ be an onto 
endomorphism of a standard Borel space $(X_i, \B_i), i =1,2$. Suppose 
that $(R_1, \sigma_1)$ and $(R_2, \sigma_2)$ are transfer operators 
acting on Borel functions defined on $(X_1, \B_1)$ and $(X_2, \B_2)
$, respectively. We say that $(R_1, \sigma_1)$ and $(R_2, \sigma_2)
$ are \emph{isomorphic} \index{transfer operator ! isomorphic} if there exists a Borel isomorphism  $T : 
(X_1, \B_1) \to (X_2, \B_2)$  such that
$$
T\sigma_1 = \sigma_2 T \quad \mbox{and} \quad T_* R_2 = R_1 
T_*,
$$
where $T_*$ is the induced map  $\mathcal F(X_2) \to \mathcal 
F(X_1)$:
$$
(T_* f)(x_1) = f(Tx_1), \ \ \forall f\in \mathcal F(X_2), \ \ x_1 \in 
X_1.
$$
\end{definition}

In order to justify this definition, we need to show that $T R_2 T^{-1}$ 
is a transfer operator corresponding to $\sigma_1$. It is obvious that  
$T R_2 T^{-1}$ is a positive operator. To verify the pull-out property  
(\ref{eq char property of r via sigma}), we calculate, for $g, h \in 
\mathcal F(X_1)$,
\begin{equation*}
\setlength{\jot}{10pt}
\begin{aligned}
 T R_2 T^{-1}[(g\circ\sigma_1) h]  &= T R_2 [g(\sigma_1T^{-1}x) 
 h(T^{-1}x)]\\
 &= T_* R_2 [(g\circ T_*^{-1}(\sigma_2 x)) (h \circ T^{-1})(x)]\\
   &=  T_* [(g\circ T^{-1}(x)) R_2(h \circ T^{-1})(x)]\\
   & =  g(x) (T_* R_2 T_*^{-1}h)(x)).\
\end{aligned}
\end{equation*}

If $T : X_1 \to X_2$ is a not invertible Borel map, then this definition
 gives the notion of a \emph{factor} map between two transfer
  operators.

\begin{example}\label{ex isomorphic R} In this example, we illustrate 
the definition of the isomorphism for the transfer operators defined by 
the formula
 \be\label{eq example of TO}
(R_if)(x) := \sum_{\sigma_i y = x} q_i(y)f(y), i =1,2.
\ee
Here $\sigma_i$ is a finite-to-one onto endomorphism of $X_i$. 
\textit{Under 
what conditions on $q_1, q_2$ are the transfer operators $(R_1, 
\sigma_1)$ and $(R_2, \sigma_2)$ isomorphic?} Let $T : X_1 \to 
X_2$ be as in Definition \ref{def iso for TO}.  If one rewrites the 
relation $(T_*R_2f)(x) = (R_1T_*f)(x)$ with $f\in \mathcal F(X_2)$, 
then it transforms to the identity
$$
\sum_{y : \sigma_2 y = Tx} q_2(y) f(y) = \sum_{z : \sigma_1 z = x} 
q_1(z) f(Tz)
$$
which holds for any $x \in X_1$ and any Borel function $f\in \mathcal 
F(X_2)$.

\begin{lemma} Let $(R_1, \sigma_1)$ and $(R_2, \sigma_2)$ be 
defined by (\ref{eq example of TO}). If they are isomorphic via a  
transformation $T$, then
\be\label{eq iso condition for q_1 q_2}
\sum_{a \in C_x} q_2(a) = \sum_{a \in C_x} (T_*^{-1}q_1)(a)
\ee
where $C_x = \{a \in X_2 : T^{-1}\sigma_2 a = x\}$.
\end{lemma}

\begin{proof} Take $f = \delta_a$ where $a$ is a point from $X_2$. 
Then
$$
(T_*R_2\delta_a)(x)  = \sum_{a : \sigma_2 a = Tx} q_2(a)
$$
and
$$
(R_1T_*\delta_a)(x) = \sum_{a : \sigma_1 T^{-1}a = x} q(T^{-1}a).
$$
We notice that $T^{-1}\sigma_2 a = \sigma_1T^{-1} a$, therefore 
the relation $(T_*R_2f)(x) = (R_1T_*f)(x)$ implies (\ref{eq iso 
condition for q_1 q_2}).
 \end{proof}
\end{example}

\subsection{Kernel and range of transfer operators}
In the next statements we discuss the structural properties of transfer
operators. 

\begin{lemma}\label{lem TO is onto 1-1}
Let $\sigma$ be an onto endomorphism of a standard Borel space
$X, \B)$. Suppose $R : \FXB \to \FXB$ is a strict transfer operator
with respect to $\sigma$. If $R|_\sigma$  is the restriction  of $R$
 onto $\mc F(X, \sigma^{-1}(\B))$, then  
$$
R|_\sigma : \mc F(X, \sigma^{-1}(\B)) \to \FXB
$$ 
 is a one-to-one and onto map. 
\end{lemma}

\begin{proof}
Since, for any function $f \in \FXB$, 
\be\label{eq simplest pull out prop}
R(f\cs) = f R(\mathbf 1)
\ee
and $R(\mathbf 1) >0$, we see that $R|_{\sigma}$ is onto. 

Suppose $f, g \in \FXB$ are two distinct Borel functions and set
$A = \{x \in X : f(x) \neq g(x)\}$. Then, for $x \in \sigma^{-1}(A)$, 
we have $(f\cs)(x) \neq (g\cs) (x)$. It follows that
$$
R(f\cs) = f R(\mathbf 1) \neq g R(\mathbf 1) = R(g\cs)
$$
and the proof is complete. 
\end{proof}

Denote by $\mathcal S(X)$ the set of real-valued simple functions on 
$(X, \B)$:
$$
\mathcal S(X) := \{s : X \to \R : s(x) = \sum_{i \in I} c_i\chi_{E_i}, \ |I|
< \infty \}
$$
where $\{E_i : i \in I\}$ is any finite partition of $X$ into Borel 
subsets.

\begin{lemma}\label{lem TO preserves simple fncts}
Suppose $\sigma $ is an onto endomorphism of a standard Borel 
space. Let $R : \FXB \to \FXB$ be a normalized transfer operator. 
Then $R$ sends the set of $\sB$-measurable simple functions onto 
the set of simple functions in $\FXB$ (by Lemma \ref{lem TO is onto 
1-1} this map is one-to-one and onto).
\end{lemma}

\begin{proof}
The result follows from the following observation: for any set $A \in 
\B$,
$$
R(\chi_{\sigma^{-1}(A)}) = R(\chi_A\cs) = \chi_A R(\mathbf 1) =
\chi_A.
$$
Hence, this relation is  extended to simple functions by linearity. 
\end{proof}

Based on the proved results, one can ask whether relation 
(\ref{eq simplest pull out prop}) determines the pull-out property. The 
affirmative answer is contained in the following lemma.

\begin{lemma}
If $(X, \B)$ and $\sigma$ are as above, then 
$$
R(f\cs) = f R(\mathbf 1) \ \Longleftrightarrow \ R((f\cs) g) =
fR(g), \ \ \ \forall g \in \mc F(X, \sB).
$$
\end{lemma}

\begin{proof}
We need to show only that $(\Longrightarrow)$ holds. Indeed, if 
$g \in \mc F(X, \sB$, then there exists $G \in \FXB$ such that $g 
= G\cs$. Then 
$$
R((f\cs) g) = R((fG)\cs) = fG R(\mathbf 1) = f R(G\cs) = f R(g).
$$
\end{proof}

Consider the kernel \index{transfer operator ! kernel} of $R$,
$$
Ker(R) := \{ f \in \FXB : R(f) = 0\}.
$$
It is clear that the pull-out property implies that
$$
f \in Ker(R) \ \Longrightarrow \ f(g\cs) \in Ker(R), \ \forall g \in 
 \mc F(X, \sB).
$$

The above relation shows that  $Ker(R)$ can be viewed as an
$\mc F(X, \sB)$-module. 

\begin{theorem}\label{thm f = f0 + ol f}
 Let $(R, \sigma)$ be a normalized transfer
 operator on $\FXB$ where $\sigma$ is an onto endomorphism.
   For any  Borel function $f \in \FXB$,  there exist
uniquely determined functions $f_0\in Ker(R)$ and $\ol f \in 
\mc F(X, \sB)$ such that 
\be\label{eq f = f_0 + ol f}
f = f_0 + \ol f.
\ee
\end{theorem}

\begin{proof}
We first show that, for any Borel function $f \notin \mc F(X,\sB)$, 
there exists a function $\ol f \in \mc F(X, \sB)$ such that 
$R(f) = R(\ol f)$. (If $f \in \mc F(X, \sB)$, the we take $\ol f = f$.)
Indeed, take $R(f)$ and set $\ol f= R(f) \cs$. Since $R(\mathbf 1) =  
\mathbf 1$, the function  $\ol f$ has the desired properties. 
Set $f_0 = f - \ol f$. Then $f_0 \in Ker(R)$ and (\ref{eq 
f = f_0 + ol f}) is proved. 

It remains to show that this representation is unique. If $f = f = f_0 + 
\ol f = g_0 +\ol g$ where $f_0, g_0 \in Ker(R)$, then $R(\ol g - 
\ol f) =0$. Since $R$ is one-to-one on $\mc F(X, \sB)$, we obtain
that $\ol f = \ol g$ and hence $f_0 = g_0$.
\end{proof}

\begin{corollary}\label{cor properties of E}
(1) For a normalized transfer operator $(R, \sigma)$ as above, let 
\be\label{eq E for R}
E : \FXB \to \mc F(X, \sB) : f \mapsto R(f) \cs. 
\ee
Then the operator $E$ has the following properties: $E$ is positive, 
$E(\FXB) = \mc F(X, \sB)$, 
$E^2 = E$, $E|_{\mc F(X, \sB)} = id$,  $E\circ R = R$, and 
$R\circ E = R$.

(2)  For $Ker(R^n)$ and $\mc F(X, \sigma^{-n}(\B))$ and any $f \in 
\FXB$ there exists a decomposition $f = f_0^{(n)} + \ol f^{(n)}$
 which is similar to (\ref{eq f = f_0 + ol f}). 
\end{corollary}

\begin{proof} Most of the properties formulated in (1) are 
obvious; we check only  that $E$ is an
 idempotent:
$$
E(E(f)) = E (R(f) \cs) = R[R(f)\cs]\cs = R(f)\cs = E(f).
$$
The other relation easily follow from the definition. 

For (2), we notice that 
$$
Ker(R) \subset Ker(R^2) \subset \ \cdots \ \subset Ker(R^n) \subset 
\ \cdots 
$$
and 
$$
\FXB \supset \mc F(X, \sigma^{-1}(\B)) \ \cdots \ \supset 
\mc F(X, \sigma^{-n}(\B)) \supset \ \cdots
$$
The proof of the existence of decomposition in (2) is analogous 
to that in Theorem \ref{thm f = f0 + ol f}. 

\end{proof}

For an onto endomorphism $\sigma $ of the space $(X, \B)$, we set 
\be\label{eq U_sigma defin in Borel}
U_\sigma : \FXB \to \mc F(X, \sB) : f(x) \mapsto f(\sigma(x)).
\ee

\begin{corollary}\label{cor U_sigma isometry in Borel}
For $U_\sigma$ defined in (\ref{eq U_sigma defin in Borel}), we have
$$
(RU_\sigma)(f) = f, \ \ \ \ \ (U_\sigma R)(f) = E(f). 
$$ 
If $R$ is not normalized, then the operator $RU_\sigma$ is the 
multiplication operator:
$$
(RU_\sigma)(f) =  R(\mathbf 1)f, \ \ \ f \in \FXB.
$$
The restriction of $R$ to $\mc F(X, \sB)$ is a multiplication operator 
itself.   
\end{corollary}

These formulas are obvious. We will use them in the framework 
of of our study of transfer operators in the Hilbert space $L^2\sms$ 
where $\mu$ is $\sigma$-invariant.
Then $U_\sigma$ would be an isometry, and $R$ could be treated 
as a co-isometry for $U_\sigma$. Thus, Corollary \ref{cor U_sigma 
isometry in Borel} presents a Borel analogue of the \textit{dual pair} 
$U_\sigma, R$. This observation is the basis for the further study of
the isometry $U_\sigma$. In particular, in Section \ref{sect Wold}, we 
discuss the Wold 
decomposition  generated by the sequence
of subalgebras $\{\sigma^{-n}(\B) : n \in N_0\}$. 
\\

We will also see later that the operator $E$ becomes the conditional 
expectation when $R$ is considered in the context of 
$L^p(\mu)$-spaces.

\subsection{Multiplicative properties of transfer operators}
\label{subsect multiplicative TO}

It turns out that any transfer operator possesses some multiplicative 
properties when it is restricted to an appropriate subset of Borel 
functions.

We begin with the following statement proved in 
\cite{ChoiEffros1977, BratteliJorgensen2002} (we formulate only
a part of the statement here because we need only the fact that $A$
is abelian).

\begin{lemma}\label{lem new product}
Let $A$ be an abelian $C^*$-algebra with unit 1, and let $E : A \to A$
be a linear map with the properties: (i) $E$ is positive, (ii) $E(1) =1$,
(iii) $E^2 = E$. Then the map 
$$
(a, b) \mapsto a \times b := E(ab)
$$  
is an associate product on the linear space $E(A)$. Moreover, for 
all $a \in E(A)$ and $b \in A$, 
\be\label{eq E(ab) = E(aE(b))}
E(ab) = E(aE(b)).
\ee
\end{lemma}

We can apply Lemma \ref{lem new product} for the operator $E$  
defined in (\ref{eq E for R}). It follows from  Corollary \ref{cor 
properties of E} that this operator $E$ satisfies the conditions of the
 lemma. 

\begin{theorem}\label{thm multiplicative props of R}
Let $R$ be a normalized transfer operator in $\FXB$. 
The positive  operator  $E(f) = R(f) \cs :  \FXB \to \mc F(X, \sB)$
 satisfies the conditions of Lemma  \ref{lem new product}. For the 
 product $f \times g := E(fg)$, the operator $R$ has the properties:
 \be\label{eq  R(f times g) = R(f) R(g)}
 R(f \times g) = R(f) R(g), \ \ f \in \mc F(X, \sB), g \in \FXB,
 \ee
 $$
 R(fg) = R(f) R(g), \ \ f, g \in \mc F(X, \sB).
 $$
\end{theorem}

\begin{proof} 
We observe first that $E(f) = R(f) \cs, f \in \FXB,$ is a positive
 normalized  idempotent map onto $\mc \mc F(X, \sB)$. Thus we can 
 use the conclusion of Lemma \ref{lem new product}. 
It follows from (\ref{eq E(ab) = E(aE(b))}) that, for $f \in 
\mc F(X, \sB)$ and $g \in \FXB$,
$$
f \times g = E(f E(g)) = R(f (R(g)\cs)\cs = [R(f)R(g)]\cs
$$
Applying $R$ we obtain
$$
R(f \times g) = R(f) R(g).
$$
Here the relation $R(\mathbf 1) =\mathbf 1$ has been repeatedly 
used. 

Since $\mc F(X, \sB) = \{ f\cs : f \in \FXB\}$, we see that
\begin{equation*}
\setlength{\jot}{10pt}
\begin{aligned}
(f\cs) \times (g\cs) &=  E((f\cs)(g\cs))\\
& =  R[(f\cs)(g\cs)] \cs \\
&= (fg)\cs = (f\cs) (g\cs). 
\end{aligned}
\end{equation*}
Hence, $R(fg) = R(f) R(g)$ holds on the space of Borel functions
measurable with respect to $\sB$. 
\end{proof}

\begin{remark}
(1) We emphasize that, though the product $f \times g$ is defined for
functions from $\FXB$, the multiplicative property of $R$ is 
only true when at least one function belongs to $\mc F(X, \sB)$.

(2) To illustrate Theorem \ref{thm multiplicative props of R}, we show
that (\ref{eq  R(f times g) = R(f) R(g)}) holds for the transfer 
operator $R(\sigma)$ (see (\ref{eq example R_sigma IFS})) under the
condition $\sum_{y : \sigma(y) = x} W(y) =1$ for all $x$.  
Then
\begin{equation*}
\setlength{\jot}{10pt}
\begin{aligned}
R_\sigma((f\cs) g) & =  \sum_{y : \sigma(y) = x} W(y) f(\sigma(y)) 
g(y)\\
& =  f(x) \sum_{y : \sigma(y) = x} W(y) g(y)\\
& =  \sum_{y : \sigma(y) = x} W(y) f(\sigma(y))
\sum_{y : \sigma(y) = x} W(y) g(y)\\
& =  R_\sigma (f \cs) R_\sigma(g).
\end{aligned}
\end{equation*}
\end{remark}

\subsection{Harmonic functions and coboundaries for 
transfer operators}\label{subsect harmonic and coboundaries}

Harmonic functions will be discussed in the book repeatedly. We
mention first  a couple of simple facts about characteristic functions.

In what follows, we will work with transfer operators. In this case, 
the study of harmonic functions is more interesting. 
We begin with a simple result about $\sigma$-invariant functions.

\begin{lemma} (1) Let $\sigma$ be a surjective  endomorphism of a
 standard Borel space $(X, \B)$. Suppose  $s = 
 \sum_{i} c_i\chi_{A_i}$ is a non-negative simple function such that 
 $\sigma^{-1}(A_i) = A_i$ for all $i$. Then $s$ is a harmonic function 
 for any normalized  transfer  operator 
 $(R, \sigma)  \in \mc R(\sigma)$.

(2) Let $G = \{g \in \mc F(X) : g\cs = g\}$ be the set of 
$\sigma$-invariant functions. Then every non-negative function $g \in
 G$ is harmonic with respect to any normalized transfer operator $(R,
  \sigma)$.  
\end{lemma}

\begin{proof} (1) We notice that if $\sigma^{-1}(A) = A$, then $
\chi_A(x) = \chi_{\sigma^{-1}(A)}(x) = \chi_A(\sigma x)$. Therefore,
$$
R(s) = R(\sum_{i} c_i\chi_{A_i}) = R(\sum_{i} c_i\chi_{A_i}\cs) = s.
$$

(2) The statement can be proved similarly to (1). 
\end{proof}

We have already noticed how important is the property $R(\mathbf 1)
= \mathbf 1$ in the study of  transfer operators.  
In the next lemma, we give simple conditions under which a transfer 
operator can be normalized. 

\begin{lemma}\label{lem reduce to R1=1} (1) If $(R, \sigma)$ is a
 strict transfer operator acting in the space $\FXB$, then $R_1 =
  (R\mathbf 1)^{-1}R$ is a normalized transfer operator.

(2) Let $(R, \sigma)$ be a transfer operator, and let $k$ be a
 non-negative Borel function. Then the  operator
\be\label{eq def of R_k}
R_k(f)(x) = \begin{cases}
                   \dfrac{R(fk)}{k}(x), & \mbox{if}\ \  x \in \{k \neq 0\} \\
                             0,  &  \mbox{if} \ \ x \in \{ k =0 \}
\end{cases}
\ee
is a well defined transfer operator. Moreover, $R_k$ is a normalized 
transfer operator if  and only if $k$ is a harmonic function for $R$. 
\end{lemma}

\begin{proof} The first statement is trivial. In order to prove the
 second one, we check that $(R_k, \sigma)$ satisfies the definition of
  a  transfer operator. Indeed, the positivity of $R_k$ is clear, and the 
 pull-out property follows from the relation
$$
R_k((f\circ\sigma) g) = \frac{R((f\circ\sigma) g k)}{k} = f\frac{R(g 
k)}{k} = fR_k(g).
$$
To finish the proof,  we observe that the property $R_k( \mathbf 1) = 
\mathbf 1$ holds if and only if $k$ is a harmonic function for $R$. 
\end{proof}

\begin{remark} (1) We note that the operator $R_k$ can be 
considered as an abstract \textit{Doob transform}. \index{Doob transform}
We refer, for  instance,  to
the papers   \cite{AlbeverioUgolini2015, AliliGraczykZak2015}  
 for more details.

(2) It is useful to justify the correctness of the definition 
of $R_k$ in (\ref{eq def of R_k}). For this, we notice that in case when
 $h$ is a harmonic function for $R$, then
$$
h(x) = 0 \ \Longrightarrow  \ \ R(fh)(x) =0.
$$
Indeed, since $R$ is positive, the Schwarz inequality 
\index{Schwarz inequality} shows that
$$
|R(fh)| \leq \sqrt{R(f^2)} \sqrt{R(h^2)} \leq \sqrt{R(f^2)} R(h) = 
\sqrt{R(f^2)} h,
$$
and the result follows.

In fact, one can prove even a stronger result. Write
$fh = f\sqrt{h}\sqrt{h}$, and apply the Scwarz inequality
$$
|R(fh)| \leq R(f^2h)^{1/2} R(h)^{1/2} \leq R(f^2h)^{1/2} h^{1/2} .
$$
Repeating this inequality $k$ times, we obtain
$$
|R(fh)| \leq R(f^{2^k}h)^{2^{-k}} h^{2^{-1}+ \cdots + 2^{-k}}.
$$
\end{remark}

\begin{remark}
 Suppose $(R_1, \sigma_1)$ and $(R_2, \sigma_2)$ are 
isomorphic transfer operators. Let $T: X_1\to X_2$ be a one-to-one 
Borel map that implements the isomorphism. This means, in particular, 
that $T_*R_2 = R_1T_*$ where $T_*$ is the induced map from $
\mathcal F(X_2)$ to $\mathcal F(X_1)$. Then $T_*$ realizes a one-
to-one correspondence between the sets of harmonic functions
$H(R_2)$ and $H(R_1)$. Indeed, let 
$h \in H(R_2)$ be a $R_2$-harmonic function. Then $T_* h = T_*R_2 
h = R_1T_* h$.
\end{remark}

If $Rh =h$ for a transfer operator $R$, then this fact  can be 
treated  as the existence of an 
eigenfunction $h$ corresponding to the eigenvalue 1. Similarly, we can 
represent $H(R)$ as the eigenspace corresponding to the eigenvalue 1. 

The following observation is based on Corollary \ref{cor U_sigma 
isometry in Borel}.

\begin{lemma} Let $(R, \sigma)$ be a transfer operator acting in
$\FXB$. Then  $\mc F(X, \sB)$ is the eigenspace corresponding to the 
eigenvalue $R(\mathbf 1)$, that is $\mc F(X, \sB) = \{h \in \mathcal 
F(X) : Rh = R(\mathbf 1) h\}$.
\end{lemma}

We say that two transfer operators $R_1$ and $R_2$ from
$\mc R(\sigma)$ are \textit{equivalent} 
\index{transfer operator ! equivalent}
($R_1 \sim R_2$ in symbols) 
if there exists a non-negative $k \in \FXB$ such that 
\be\label{eq equivalence of TO}
kR_2(f) = R_1(kf), \ \ \ \forall f \in \FXB. 
\ee
Let $\Gamma$  denote the multiplicative group $\FXB_+^0$
 of strictly positive
Borel functions. For the sake of simplicity, we will assume that the 
function $k$ from (\ref{eq equivalence of TO}) is taken from 
$\Gamma$ though this restriction can be, in general, omitted.

\begin{proposition}
(1) $R_1 \sim R_2$ is an equivalence relation on the set $\mc 
R(\sigma)$ of transfer operators.

(2) $(R_k)_g = (R_g)_k = R_{kg}$ where $k, g \in \FXB_+ $.

(3) The map $r(k, R) \mapsto R_k$ defines an action of  $\Gamma = 
\FXB_+^0$ on the set  $\mc R(\sigma)$. The equivalence relation
 $\sim$ coincides with the partition into orbits of the action $r$. 

(4) The action of $\Gamma$ on $\mc R(\sigma)$ is not free.  If
 $\Gamma_0 = \Gamma \cap \{g \in  \mc F(X, \B) : g\cs = g\}$,
  then $\Gamma_0$ belongs to the stabilizer of every $R \in 
  \mc R(\sigma)$. 
 
\end{proposition}

\begin{proof}
Statement (1) is verified directly. If $R_1(f) = \dfrac{R_2(kf)}{k}$, 
then, denoting $g = kf$, we get 
$$
R_2(g) = \frac{R_1(gk^{-1})}{k^{-1}};
$$
this proves that $\sim$ is symmetric. Clearly, if $R_1 \sim R_2$ and
$R_2 \sim R_3$, then $R_1 \sim R_3$.

(2) follows from (1).

To see that (3) holds, we use (2) and compute
$$
r(k_1, r(k_2, R)) = r(k_1, R_{k_2}) =  (R_{k_2})_{k_1} = 
R_{k_1k_2}.
$$
Hence, for a fixed $R \in \mc R(\sigma)$, the set $\{R_k : k \in 
\FXB_+^0\}$ is the orbit of the action of $\Gamma = 
\FXB_+^0$ on the set $\mc R(\sigma)$.

(4) The action of $\Gamma$ is not free: $k(I) = I, \forall k \in 
\Gamma$ where $I(f) = f$ (the identity map). Moreover, if 
$k \in \Gamma_0$, then   $k\cs = k$ and we have
$$
(kR)(f) = R_{k}(f) = \frac{R((k\cs) f}{k} = R(f)
$$
where $R\in \mc R(\sigma)$ is a fixed transfer operator and $f$ is any
Borel function.

\end{proof}

\begin{lemma}\label{lem Harmonic for R_k}
For $\sigma\in End(X, \B)$, let $R$ be  a transfer operator from 
$\mc R(\sigma0$. If $k \in \Gamma$, then the map $h \mapsto
kh$ sends the set $H(R_k)$ onto $H(R)$.  
\end{lemma}

\begin{proof} We need to show that $h \in H(R_k)$ if and only if 
$hk \in H(R)$. This result follows from the relation:
$$
R_k(h) = \frac{R(kh)}{k} = h  \ \Longleftrightarrow \  R(hk) = hk.
$$

\end{proof}

\begin{definition}\label{def R and sigma coboundaries}
Let $(R, \sigma)$ be a transfer operator. It is said that a function
$f \in \FXB$ is a \textit{$\sigma$-coboundary} if there exists some 
$g \in \FXB$ such that $(g\cs) f = g$. 

We say that $f \in \FXB$ is an \textit{$R$-coboundary} \index{coboundary}
if there exists $k \in \FXB$ such that $kf = R(k)$. The set of all 
 $R$-coboundaries is denoted by $Cb(R)$.
\end{definition}

From Definition \ref{def R and sigma coboundaries} we can 
deduce the following result.

\begin{proposition} 
Let $R$ be a normalized operator from $\mc R(\sigma)$. 

(1)  If $f$ is  a $\sigma$-coboundary, then $R(f)$ is an 
$R$-coboundary.

(2) 
$$
Cb(R) = Cb(R_k), \ \ \ \ \ \forall k \in \Gamma.
$$
\end{proposition}

\begin{proof}
For (1), we simply apply $R$ to the equality $(g\cs) f = g$ and obtain
the result. The converse is not true, in general.

Let $f$ be an $R$-coboundary, i.e., there exists some $g \in \FXB$
such that $gf = R(g)$. Fix $k \in \Gamma.$. We claim that $g \in
 Cb(R_k)$. Indeed, we need to show that there exists $h$ such that 
 $hf =  R_k(h)$ or, equivalently,
 $$
 hf = \frac{R(hk)}{k}.
 $$
 The latter means that $h$ must satisfy the equality
 $$
 hkf = R(hk).
 $$
 If we take $h= gk^{-1}$, then we get $gf = R(g)$ which is true. 
 Hence, $Cb(R) \subset  Cb(R_k)$. The converse inclusion is proved 
 similarly. 
\end{proof}

Let $f$ be a Borel function on $(X, \B)$, and let $\sigma : X \to X$ be 
an onto endomorphism of $(X, \B)$. Define the \textit{cocycle} 
\index{cocycle} generated by $(f, \sigma)$:
$$
\alpha_f(x, \sigma^k) := f(\sigma^{k-1}(x))f(\sigma^{k-2}(x)) \ 
\cdots \ f(x),\ \ \ k \in \N.
$$

We observe that the following fact holds.

\begin{lemma}\label{lem cocycles for harmonic}
 Suppose that $(R, \sigma)$ is a transfer operator acting in $\FXB$,
and let $h$ be an $R$-harmonic function. Then
$$
R(\alpha_h(x, \sigma^k)) = (h(x))^k, \ \ k \in \N.
$$
\end{lemma}

\begin{proof} We calculate
$$
R(\alpha_h(x, \sigma)) = R((h\circ \sigma) h)(x) = h(x)R(h)(x) = 
h^2(x).
$$
Then the result follows by induction.
\end{proof}

\newpage


\section{Transfer operators on measure spaces}\label{sect integrable 
operators}

 Our starting point is a fixed pair $(R, \sigma)$ on $(X, \B)$ making up 
 a transfer operator. In the next two sections we turn to a systematic 
 study of specific and important sets of measures on $(X, \B)$ and 
 actions of $(R, \sigma)$ on these sets of measures. These classes of 
 measures in turn lead to a \textit{structure theory} for our 
 given transfer 
 operator $(R, \sigma)$. Our corresponding structure results are 
 Theorems \ref{thm  lambda R abs cntn wrt lambda},  \ref{thm formula 
 for R in L2}, \ref{thm adjoint of R}, \ref{thm from measure 
 equivalence}, and \ref{thm mu R-invariance}.
 
\subsection{Transfer operators and measures}\label{subsect action 
on measures}
In general,  positive and transfer operators on the space $\FXB$ or 
$L^p\sms$ are not continuous. 
But one can use the notion of \textit{order convergence} to define 
the order
continuity of positive operators. Under this assumption we can define
an action of a positive operator on the set $M(X)$ of Borel measures. 
Order continuity of  a positive operator is commonly 
used, in particular,  for positive operators one Banach lattices. The 
literature devoted to this subject is very extensive; we 
refer to \cite{AbramovichAliprantis2001} for details and further 
references. 
 
In this section, we work with \textit{positive} and \textit{transfer}
 operators acting 
on the space of measurable (or integrable) functions on a standard
measure space $\sms$ where $\mu$ is a continuous Borel measure. 
We use the same notation and definitions 
as in case of Borel functions keeping in mind the $\mod\, 0$ 
convention.  By $P$ we denote a  positive operator 
acting on an $L^p(\mu)$-space, $1\leq p \leq \infty$. If 
$\sigma \in End\sms$ is a measurable surjective  endomorphism,  then 
we define a transfer operator $R = (R, \sigma)$ on 
$L^p(\mu)$ as in Definition \ref{def transfer operator}. 
We recall that in this case $\sigma$ is assumed to be
 a non-singular onto endomorphism.

The next definition is formulated in a setting which is suitable for
our purposes. We use the language of Banach lattices keeping in 
mind the functional spaces $\FXB$ and $L^p(\mu)$. 

\begin{definition}
Let $\mathcal E$ be  a Banach lattice and $P$ a positive operator
on $\mathcal E$. A sequence $(f_n)$ in  $\mathcal E$ is \textit{order 
convergent} \index{order ! convergent} to $g \in  \mathcal E$, written $f_n \stackrel{o }
\longrightarrow g$, if  there exists a 
sequence $(h_n)$ in $\mathcal E_+$ such that $h_n \downarrow 0$ 
and $| f_n - g| \leq h_n$ for all $n \geq N$, $N \geq 1$. 

It is said that a positive operator $P$ acting on $\mathcal E$ 
is \textit{order continuous} \index{order ! continuous} if for any 
sequence $(f_n)$ of Borel 
functions, the relation  $f_n \stackrel{o}\longrightarrow 0$ implies
$P(f_n) \stackrel{o}\longrightarrow 0$. 
\end{definition}

\begin{definition}\label{def action of P on mu}
Let  $\mu\in M(X)$ be a Borel measure on $(X, \B)$, and let $P$ 
be a positive  order continuous operator on $\FXB$. 
Define the \textit{action }of $P$ on the set $M(X)$ as follows: 
for every fixed $\mu \in M(X)$, we set 
\be\label{eq action of P on M(X)}
(\mu P)(f) := \int_X P(f) \; d\mu, 
\ee
where $f$ is a measurable  function.
\end{definition}

Applying (\ref{eq action of P on M(X)}) to characteristic functions $
\chi_A$ ($A\in \B$), we have the formula 
\be\label{eq mu P measure of A}
(\mu P)(A) = \int_X P(\chi_A)\; d\mu.
\ee

\begin{lemma}\label{lem prop of mu P} Suppose $P$ is a positive
order continuous operator on $\FXB$. Then

(1)  $\mu P$ is a sigma-finite Borel measure on $(X, \B)$ such that
$$
(\mu P)(X) < \infty \ \Longleftrightarrow \ P(\mathbf 1) \in 
L^1(\mu).
$$

(2) If $X$ is a Polish spaces and $P(C(X)) \subset C(X)$, then
the action $ \mu \mapsto \mu P$ is continuous with respect to 
the weak* topology on $M_1(X)$. 

(3) The set $M_1 P := \{\mu R : \mu \in M_1(X)\}$ is a closed 
subset om $M_1(X)$ in the weak* topology.
\end{lemma}

\begin{proof} (1) The assertion
 that $\mu P$ is a measure follows directly 
from  Definition \ref{def action of P on mu}. One can apply monotone
 convergence theorem, and the order continuity of $P$, to show that 
 $ \mu P$ is countably additive.

Since $P(\chi_A) \leq  P(\mathbf 1)$, we see that finiteness of 
measure  $\mu P$ is equivalent to the property $P(\mathbf 1)\in 
L^1(\mu)$.

(2) The assumption that $X$ is a Polish space is not restrictive 
because any standard Borel space is Borel isomorphic to a Polish space. 
Thus, the set of probability measures $M_1(X)$ can be endowed with 
the weak* topology. Let a sequence $(\mu_n)$ converge to a 
measure $\mu$. Then, for $f \in C(X)$, 
$$
(\mu_n P)(f) = \int_X P(f) \; d\mu_n \ \rightarrow \ \int_X P(f) \; d
\mu = (\mu P)(f)
$$
as $n \to \infty$.

(3) It follows from (2). 
\end{proof}

\textit{Assumption.} In the sequel, we assume (sometimes implicitly)
that the considered transfer operators $R$ are order continuous. This
means that we always have the well defined action of $R$ on the set 
of measures $R : \mu \mapsto \mu R$. We denote  $M_1(X) R = 
K_1$. This set of measures will play an important role in the next 
sections. 

In the next remark, we consider a particular case when an action
of positive operators can be defined on the set of measures without
additional assumptions. 

\begin{remark} \label{rem X lc Hausdorff}
Suppose $X$ is a locally compact finite Hausdorff space and $P$ is 
a positive 
operator in the space $\FXB$. Then, for every $x \in X$, the operator
 $P$ defines a positive linear functional on $C_c(X)$ by the formula 
 $f \mapsto P(f)(x)$. By Riesz' theorem, there exists a positive Borel 
 measure $\mu^x$ such that
$$
P(f)(x) = \int_X f(\cdot)\; d\mu^x(\cdot).
$$
Then the function $F : x \mapsto \mu^x(A)$ is measurable on 
$X$ for every $A \in \B$  because $F(x) = P(\chi_A)(x)$. Hence, for
every positive operator, we can associate a measurable family of
 measures $(\mu^x)$ defined on $(X, \B)$.  

Now, given a measure $\nu \in M(X)$,  define $\nu P$ as a function 
on Borel sets:
\be\label{eq nu P for lc hausdorff}
(\nu P)(A) := \int_X P(\chi_A)(x) \; d\nu(x).
\ee
It follows from the definition of $(\mu^x)$ that
$$
(\nu P)(A) = \int_X \left( \int_X \chi_A(\cdot)\; d\mu^x(\cdot)
\right) \; d\nu(x) = \int_X \mu^x(A)\; d\nu(x).
$$
It is obvious that $\nu P$ is a well defined complete Borel measure on
$(X, \B)$. 
\end{remark}

The concept, we considered in Remark \ref{rem X lc Hausdorff}, is 
known in  probability theory by the name of a \textit{random 
measure}. \index{measure ! random}  
More formally, let  $(X, \B, \mu)$ be a measure space. 
Then a random measure $\Phi$ defined with respect to this space is a 
function  $x \mapsto \nu_x : X \to M(Y)$ such that $\nu_x(A)$
is  $\B$-measurable for every $A \in \mc A$, see, e.g.,  
\cite{GalicerSagliettiSchmerkinYavicoli2016, Sureson2016}.

\begin{definition}\label{def R is p-integrable}
Let $P$ be a positive operator acting in $\FXB$. Given a measure $
\lambda\in M(X)$, we say that $P$ is \emph{$p$-integrable} 
\index{positive operator ! integrable} with 
respect to $\lambda$ if $P(\mathbf 1) \in L^p(\lambda)$. We use the 
terms ``1-integrable'' and ``integrable'' as synonyms.
\index{positive operator ! integrable}

For a fixed positive operator $P$, we denote
$$
I_p(P) := \{\lambda \in M(X) : \int_X P(\mathbf 1)^p \; d\lambda < 
\infty\}.
$$
\end{definition}

It is obvious that if $\lambda_1, \la_2 \in I_p(R)$, then $c_1\la_1+ 
c_2\la_2 \in I_p(R)$.

In the next statement, we prove that the set $\mathcal S(X)$ belongs 
to the domain of any $p$-integrable transfer operator $R$.

\begin{lemma}\label{lem simple f-n p-integrable}
Let $P$ be a positive operator on $\FXB$. Suppose $\la$ is a 
$p$-integrable measure. Then $P(s(x)) \in L^p(\lambda)$ for any
 simple function $s \in \mathcal S$ where $1\leq p < \infty$.
\end{lemma}

\begin{proof}
We  notice that, because $P$ is positive ($f \leq g$ implies $P(f) \leq 
P(g)$), then  the fact that $P(\mathbf 1)$ is in $L^p(\lambda)$ 
implies that $P(\chi_A)$ is in $L^p(\lambda)$. Then the statement  
follows from linearity of $R$, and we can conclude that
$$
P( \mathcal S(X) ) \subset L^p(\lambda) \ \Longleftrightarrow \ 
P(\mathbf 1) \in L^p(\la). 
$$
\end{proof}

Let $f \in \FXB$, then by $\mathrm{supp}(f)$ we denote the Borel 
set  $\{x \in X: f(x) \neq 0\}$.

\begin{lemma}\label{lem R(s) is integrable}
Let $(X, \B), R, \sigma$  be as in Definition \ref{def transfer operator}.

(1) For any Borel set $A$, 
$$
\mathrm{supp}( R(\chi_A)(x)) \subset \sigma(A),
$$ 
that is  $R(\chi_A)(x) = 0$ if $x \notin \sigma(A)$.

(2) If $R$ is a strict transfer operator in $\FXB$, then, for any Borel 
set $A$,
$$
\mathrm{supp}( R(\chi_A)(x)) = \sigma(A),
$$
$$
\la(\sigma(A) \setminus \mathrm{supp}( R(\chi_A)) = 0.
$$

(3) If $\la$  quasi-invariant with respect to $\sigma$,  
$\lambda \in \mathcal{Q_-}(\sigma)$, then $\la(A) > 0 $ implies
that $\la(\sigma(A)) > 0$. Then statements (1) and (2) hold for
$\la$-a.e. $x \in X$.
\end{lemma}

\begin{proof}
(1) We use the pull-out property (\ref{eq char property of r via sigma 
1}) to prove (1). Observe that the relation  $A \subset 
\sigma^{-1}(\sigma (A))$ holds for any endomorphism $\sigma$ and
 any measurable set $A$. Then
$$
\chi_A(x) = \chi_{\sigma^{-1}(\sigma (A))}(x)\chi_A(x) = \chi_{\sigma (A)}(\sigma(x)) \chi_A(x).
$$
Hence,
$$
R(\chi_A(x)) = R(\chi_{\sigma (A)}(\sigma(x)) \chi_A(x)) = \chi_{\sigma (A)}(x) R(\chi_A(x))
$$
and the result follows from
$$
R(\chi_A)(x) (1 - \chi_{\sigma (A)}(x)) = 0.
$$

(2) By assumption, we have that $R(\mathbf 1)(x) > 0$ for any $x
\in X$. Since $\mathbf 1(x) = \chi_A (x) + \chi_{X \setminus A}(x)$,
 we obtain that $R(\chi_A) (x) + R(\chi_{A^c})(x) > 0$ for all $x \in 
 X$. It follows from (1) that 
$$
\mathrm{supp}( R(\chi_A)(x)) \cap \mathrm{supp} (R(\chi_{X 
\setminus A})(x)) \subset \sigma(A) \cap \sigma(X\setminus A)
 = \emptyset.
$$
Since $R(\mathbf 1)(x) > 0$, we see that (2) holds. 
 
(3) It follows from the relation $A \subset \sigma^{-1}(\sigma(A))$ 
that  $\la(A) > 0 \ \Longrightarrow \ \la(\sigma(A)) >0$. 
Hence statement (3) is deduced from above. 
\end{proof}

\begin{remark}
Suppose $R$ is a transfer operators acting in $L^p(\la), 1\leq p < 
\infty$.For  $\lambda \in M(X)$, we set
$$
\mathcal S^p(\lambda) = \mathcal S(X) \cap L^p(\lambda).
$$
Then $\mathcal S^p(\lambda)$ is dense in $L^p(\lambda)$ with
 respect to the norm. 
\end{remark}

\begin{corollary}\label{cor R as an operator in Lp}
Suppose $R$ is $p$-integrable transfer operator defined on the 
space $\FXB$, $1 \leq p < \infty$. 
Then $R$ generates a transfer operator in $L^p(\lambda)$ with 
dense domain containing the set $\mathcal 
S^p(\lambda)$:
$$
R(\mathcal S^p(\lambda)) \subset L^p(\lambda).
$$
In general, $R$ is an unbounded linear operator in $L^p(\lambda)$.
\end{corollary}

The property of $1$-integrability for a transfer operator $R$ allows 
one to define a  map  $\lambda \mapsto \lambda R$ from the set
$I_1(R) =  I(R)$ to the set of \emph{finite} measures on $(X, \B)$.

\begin{proposition}\label{prop measure lambda R}
Let $(R, \sigma)$ be a transfer operator on $(X, \B, \lambda)$ such 
that $\la$ is backward and forward quasi-invariant with respect to $
\sigma$, i.e., $\la \in \mc Q_- \cap \mc Q_+$. Then, if $\lambda \in 
I(R)$,  the relation
\be\label{eq measure lambda circ R as integral}
(\lambda R)( f) := \int_X R(f)(x)\; d\lambda(x)
\ee
defines a finite Borel measure $\lambda R$ on $(X,\B)$.
\end{proposition}

\begin{proof} By the premise of the proposition, 
we have $R(\mathbf 1) 
\in L^1(\lambda)$. To justify relation (\ref{eq measure lambda circ R 
as integral}), we use the standard approach via approximation by 
simple functions. 

Given a $\lambda$-measurable non-negative function $f$, take a 
sequence  $(s_n)$ of simple Borel functions such that $s_n \leq s_{n
+1}$ and
$f (x)= \lim_n s_n(x)$ for $\lambda$-a.e. $x\in X$.
Then we can define
\be\label{eq Rf via simple f-ns}
\int_X R(f)(x) \; d\lambda(x) := \lim_{n\to\infty} \int_X R(s_n)(x) \; 
d\lambda(x), \ \ \ f \in \FXB.
\ee
The limit in (\ref{eq Rf via simple f-ns}) exists (it may be infinite) 
because the sequence $(R(s_n))$ is increasing and consists of 
integrable functions  by Lemma \ref{lem R(s) is integrable}. In fact, we 
can see that the definition in (\ref{eq Rf via simple f-ns}) does not 
depend on the choice of a sequence $(s_n)$ since
$$
\lim_{n\to\infty} \int_X R(s_n)(x) \; d\lambda(x) = \sup\left( \int_X 
R(s)(x) \; d\lambda(x) \ : \ s(x) \leq f(x) , \forall x \in X\right).
$$

We need to show relation (\ref{eq measure lambda circ R as integral}) 
defines a Borel measure on $(X,\B)$.  Set, for any $A\in \B$, 
$$
(\lambda R)(A) := \lambda(R(\chi_A)).
$$ 
 To see that $\lambda R$ is sigma-additive, it suffices to prove that if 
 $A_i \supset A_{i+1}$ and $\bigcap_i A_i =\emptyset$, then
\be\label{eq lambda RA_i tends to 0}
\lim_{i\to \infty}(\lambda R)(A_i) = 0.
\ee
It follows from Lemma \ref{lem R(s) is integrable} that $(\lambda  
R)(A_i) \leq \lambda(\sigma (A_i))$. Since $\sigma$ is forward quasi-
invariant, we see that $\lambda (\sigma A_i) \to 0$ as $i \to \infty$. 
Hence, relation  (\ref{eq lambda RA_i tends to 0}) holds, and $\lambda
 R$ is a sigma-additive measure.

Moreover, we see from the equality
$$
(\lambda R)(X) = \int_X R(\mathbf 1)\; d\lambda,
$$
that $\lambda R$ is a finite measure.
\end{proof}

In general, the measure $\lambda R$ is not absolutely continuous with 
respect to $\lambda$, see e.g., Example \ref{ex R forIFS} and Table 
1. One of our aims is to find conditions under 
which $\lambda R \ll \lambda$. We recall the following example 
(more details are in \cite{DutkayJorgensen2006}).

\begin{example}[Case of wavelets \cite{DutkayJorgensen2006},  
see also Example \ref{ex R forIFS}] \label{ex Riesz measure}
Let $\mathbb T = \{z \in \mathbb C : |z| =1\}$ be the unit circle, and 
let $\sigma_2 (z) = z^2$, $\sigma_3 (z) = z^3$ be two surjective 
endomorphisms of $\mathbb T$. Suppose the transfer operator $R_i$ 
on $C(\mathbb T)$ is defined as follows:
$$
R_i(f)(z) = \sum_{w: \sigma_i (w) = z}|m(w)|^2f(w),\ \ \ i =2,3,
$$
where $m(w) = \dfrac{1 + w^2}{\sqrt 2}$. It was proved in 
\cite{DutkayJorgensen2006} that:

(a) the measure $\delta_1$ (the Dirac point mass measure at $z =1$) 
is $R_i$-invariant, $\delta_1 R_i =\delta_1$;

(b) for the transfer operator $(R_3, \sigma_3)$, the Riesz measure
$$
d\nu(t) = \lim_{n\to \infty} \frac{1}{2\pi} \prod_{k=1}^n (1 + \cos 
(2\cdot 3^kt))
$$
is singular with respect to the Lebesgue measure, and satisfies the relation
$\nu R_3 = \nu$. 
\end{example}

\begin{remark}\label{rem justification R acts on measures}
We are interesting in the following problem:
Given a transfer operator $(R, \sigma)$ acting in $\FXB$, find a 
Borel measure $\la$ such that $\la R \ll \la$. We showed in 
Examples \ref{ex R forIFS}, and \ref{ex Riesz measure} that a
transfer operator may satisfy, or not satisfy, this condition of 
absolute
continuity. In order to formulate the problem correctly, an action
of $R$ on measures must be well defined. We know that this is always 
the case when $R$ is \textit{order continuous} (see Lemma \ref{lem 
prop of mu P}), or when $X$ is a locally compact Hausdorff space (see 
Remark \ref{rem X lc Hausdorff}). On the other hand, if we restrict our 
choice of measures to the subsets of $\sigma$-quasi-invariant
measures and  integrable functions  $R(\mathbf 1)$, we still have
a vast set of measures which includes interesting applications. 
We note that if $R$ is a normalized transfer operator, then the latter
condition automatically holds for all finite measures.  

Based on these observation, we will assume that, for a transfer 
operator $R$,  the map $ \la \mapsto \la R$ is defined on $M(X)$ 
(or on a subset of $M(X)$ in case of need).
\end{remark}

\begin{definition}\label{def R acts on measures}  For a fixed 
order-continuous transfer operator $(R, \sigma)$, we define the set 
\be\label{eq lambda R abs cntn wrt lambda}
\mc L(R) := \{ \la \in M(X) : \la R \ll \la\}.
\ee
In case when $R$ is integrable, we use the same notation 
$\mc L(R)$ for the set  of Borel measures $\lambda$ such that 
$\la R \ll \la$, $R(\mathbf 1) \in L^1(\la)$, and $\la$ is 
quasi-invariant with respect to $\sigma$. 
\end{definition}

We are interested in the following questions.
\medskip

\noindent
\emph{\textbf{Question}:} \emph{ (1) Under what conditions on $(R, 
\sigma)$ is the set $\mathcal L(R) $ non-empty?}

\emph{(2) What properties does the set $\mathcal L(R)$ have? In 
particular, can we iterate the map $\lambda \mapsto \lambda R$ 
infinitely many times?}
\medskip

We first give a partial answer to Question (1) in the following theorem.

\begin{theorem}\label{thm  lambda R abs cntn wrt lambda}
(1) Let  $(R, \sigma)$ be  a transfer operator defined on a standard
 measure space $(X, \B, \lambda)$ such that $R(\mathbf 1) \in 
 L^1(\la)$. Then
$$
\mathcal L(R) \supset  I(R) \cap \mc Q_+.
$$
In other words, if  $\sigma$ is backward and forward quasi-invariant
 with respect to $\lambda$, and $\la \in I(R)$, then $\lambda R \ll 
 \la$.

(2) If, additionally to the conditions in (1), we  assume that $(R, 
\sigma)$ is a strict transfer operator, then $\la  R $ is equivalent to 
$\la$ and 
$$
\{\la : \la R \sim \la\} =  I(R) \cap \mathcal{Q_-} \cap \mc Q_+.
$$

\end{theorem}

\begin{proof}
(1) Let $A$ be a Borel subset of $X$. By Lemma \ref{lem R(s) is 
integrable}, the function $R(\chi_A)$ is non-negative and integrable 
with respect to $\lambda$. As was shown in the proof of Lemma
 \ref{lem R(s) is integrable}, the relation
$$
\chi_A(x) =  \chi_{\sigma (A)}(\sigma(x)) \chi_A(x)
$$
holds. Based on this fact and the pull-out property for $R$, 
we  calculate
\begin{eqnarray}\label{eq formula for lambda R}
\nonumber
  (\lambda R)(A)  &=& \int_X R(\chi_A)\; d\lambda(x) \\
  \nonumber
  \\
\nonumber 
&=& \int_X R[\chi_{\sigma (A)}(\sigma(x)) \chi_A(x)] \; d
\lambda(x) \\
\\
 \nonumber &=& \int_X \chi_{\sigma (A)}(x) R(\chi_A)(x) \; d
 \lambda(x) \\
 \nonumber
 \\
  \nonumber &=& \int_{\sigma(A)} R(\chi_A)(x)\; d\lambda(x).
\end{eqnarray}
Now if $\lambda(A) =0$, then, by the assumption, 
$\lambda(\sigma(A)) =0$. Therefore $\lambda R(A) = 0$, and 
this proves that $\la R\ll \la$.

(2) Having statement (1) proved, we need to show that $\lambda (A)
 >0$ implies that $(\lambda R)(A) >0$. As $A \subset \sigma^{-1}
 (\sigma (A))$, and $\sigma$ is backward non-singular, we obtain that 
  $\lambda (A) > 0 \ \Longrightarrow \ \lambda (\sigma(A)) >0$. In 
  this case, we see from (\ref{eq formula for lambda R}) that
$$
\int_{\sigma(A)} R(\chi_A)(x)\; d\lambda(x) > 0
$$
because of Lemma \ref{lem R(s) is integrable}. Thus, it follows from 
 (\ref{eq formula for lambda R}) that $(\lambda R)(A) >0$.
\end{proof}

\begin{remark}  Note that the condition $\lambda\circ R \ll \la$
 is based on the 
 forward quasi-invariance of $\sigma$. On the other hand, If we 
 additionally assume that $R$ is a strict transfer operator, then 
our proof of  $\la \ll \lambda\circ R$ is based on the backward 
quasi-invariance of $\sigma$. In fact, the condition that $R$ is a strict 
 transfer operator can be slightly weakened as shown in the next
  statement.
\end{remark}  
  
\begin{corollary}\label{cor on equivalence of measures}
Let $(R, \sigma)$ be as in Theorem \ref{thm  lambda R abs cntn wrt 
lambda} and $\lambda \in I(R)$. Suppose that for every measurable 
set $A$ with $\lambda(A) > 0$, the function $R(\chi_A)$ is nonzero 
as a function in $L^1(\lambda)$. Then the measures $\lambda$ and $
\lambda R$ are equivalent.
\end{corollary}

\begin{proof} It suffices to show that $(\lambda R)(A) = 0$ implies
 that $\lambda(A) =0$ because the converse result was proved in 
 Theorem \ref{thm  lambda R abs cntn wrt lambda}. We use (\ref{eq 
 formula for lambda R}) and the fact that the support of the function 
 $R(\chi_A)$ belongs to $\sigma(A)$.
Hence, if $\lambda(A) > 0$, then, by quasi-invariance of $\sigma$, we 
conclude that $\lambda (\sigma(A)) > 0$. Therefore, by the 
assumption,
$$
\int_{\sigma(A)} R(\chi_A)\; d\lambda > 0.
$$
and finally we obtain that $(\lambda R) (A) >0$.
\end{proof}

The following result states that the measures $\lambda R$ and 
$\lambda$ are  equivalent when they are restricted on the 
sigma-subalgebra  $\sigma^{-1}(\B)$ of $\B$. This results agrees with the
fact proved in the setting of Borel dynamics. Namely, we showed in 
Lemma \ref{lem TO is onto 1-1} that
 $R$ is a one-to-one map from $\mc F(X, \sB)$ onto $\FXB$. 

\begin{proposition}\label{prop lambda R equiv to lambda}
Suppose that $(R, \sigma) $ is a transfer operator on a standard Borel 
space $(X, \B)$, and let $\la \in I(R)$. If  $\sigma$ is a non-singular 
endomorphism on $(X, \B)$ with respect to $\la$, then
$$
\lambda R |_{\sigma^{-1}(\B)} \sim \lambda |_{\sigma^{-1}(\B)}.
$$
\end{proposition}

\begin{proof}
For any set $A \in \B$, set $B = \sigma^{-1}(A)$. By non-singularity 
of $\sigma$, we have  
$$
(\lambda(B) = 0 )\ \Longleftrightarrow\ (\lambda(A) =0).
$$ 
On the other hand, we can apply the same method as in Theorem 
\ref{thm  lambda R abs cntn wrt lambda} and obtain that
\begin{eqnarray*}
  (\lambda R)(B) &=& \int_X R(\chi_B)(x)\; d\lambda(x)  \\
  \\
   &=& \int_X R(\chi_{\sigma^{-1}(A)})(x) \; d\lambda(x) \\
   \\
   &=& \int_X R((\chi_A \circ\sigma) (x))\; d\lambda(x)  \\
   \\
   &=& \int_X \chi_A(x) R(\mathbf 1)(x)\; d\lambda(x) \\
   \\
   &=& \int_A R(\mathbf 1)(x)\; d\lambda(x).
\end{eqnarray*}
Since $R(\mathbf 1)\in L^1(X, \B, \lambda)$, we conclude that 
$(\lambda R) (B) =0$ if and only if $\lambda(A) =0$ if and only if $
\lambda(B) =0$.
\end{proof}

\subsection{Ergodic decomposition of transfer operators}
Fix a Borel measure $\lambda\in M(X)$ and consider the dynamical 
system $(X, \B, \la, \sigma)$. Suppose $(R, \sigma)$ is a transfer
operator in $\FXB$. It follows from Lemma \ref{lem invariant 
set} that if $A$ is a $\sigma$-invariant set then the restriction of $R$
to the space of Borel functions on $A$ gives an \textit{induced transfer 
operator} $R_A$. It is well known that if $\sigma$ is non-ergodic, then 
one can use the standard procedure of ergodic decomposition for 
$\sigma$.   We show here that, in this case, any transfer operator $R$ 
related to $\sigma$ also admits a kind of  ergodic decomposition. 
\index{ergodic decomposition} This means that we give a justification of the 
relation $R = \int_{X/\xi} R_C \; d\mu_\xi(C)$ which is intuitively clear. 

Let $(X, \B, \lambda, \sigma)$ be a non-singular non-ergodic 
dynamical system. Consider the partition $\zeta$ of $X$ into orbits of 
$\sigma$. We recall that, by definition, $x$ and $y$ are in the same 
\textit{orbit} of 
$\sigma$ if there are positive integers $n, m$ such $\sigma^n(x) = 
\sigma^m(x)$. Let $\xi$ be the measurable hull of $\zeta$, i.e., $\xi
$ is a maximal measurable partition with property $\xi \prec \zeta$. 
We recall that $\xi$ is trivial when $\sigma$ is ergodic.
Denote by $(X/\xi, \B/\xi, \la_\xi)$ the quotient measure space. By 
Theorem \ref{thm Rokhlin disintegration}, there exists a unique system 
of conditional measures $(\lambda_C)_{C \in X/\xi}$ such that
$$
\la(A)  = \int_{X/\xi} \la_C (B) \; d\la_\xi(C),
$$
and, for any measurable function $f$ on $X$,
\be\label{eq integration in system of c m}
\int_X f \; d\la = \int_{X/\xi} \left(\int_{C}\chi_{C} f\; d\la_C\right) 
\; d\la_\xi(C).
\ee
Here $C$ is an arbitrary element of the partition $\xi$ and can be 
considered as a point in $X/\xi$. We refer to Subsection \ref{subsect 
meas partitions} for more information about conditional measures.

Let $(R, \sigma)$ be  a transfer operator in $\FXB$ and $\la \in I(R)$ 
be a fixed measure. Suppose that $\la \in \mc L(R)$. Then there
exists an integrable measurable function $W$ such that
\be\label{eq W as R-N derivative}
\int_X R(f) \; d\la = \int_X fW \; d\la, 
\ee
i.e., $W (x) = \dfrac{d(\la R)}{d\la}(x)$ is the \textit{Radon-Nikodym 
derivative.}

\begin{theorem}\label{thm erg decomp for TO}
Let $(X, \B, \la, \sigma)$ and $\xi$ be as above and $\la(X) =1$. 
Suppose 
$(R,\sigma)$ is a transfer operator defined on the space $L^1(\la)$ 
such that $\la R \ll \la$. Let $(\la_C)$ be the system of conditional 
measures defined by the measurable partition $\xi$ on the measure 
space $(X, \B, \la)$. Then there exists a measurable field of 
transfer operators $(R_C, \sigma_C)$ such that
$\la R_C \ll \lambda_C$,  and
$$
W_C := \frac{d(\la_C R_C)}{d\la_C} = W \chi_C,
$$
where $W d\la = d(\la R)$.
\end{theorem}

\begin{proof}
Let $(C, \B\cap C, \la_C)$ be the standard measure space obtained 
by restriction of Borel sets to $C \in X/\xi$. Then $C$ is 
$\sigma$-invariant, 
 and, by Lemma  \ref{lem invariant set},  we can define a transfer 
 operator $(R_C, \sigma)$ by the formula  $R_C(f) = R(f)|_C$.  It 
 follows from (\ref{eq integration in system of c m}) that $R_C$ is a 
 transfer operator in the space $L^1(\la_C)$. Then we  can compute
 \begin{eqnarray*}
 \int_X R(f) \; d\la & = & \int_X fW \; d\la \\
 \\
& = & \int_{X/\xi}  \left(\int_{C} fW\chi_C\; d\la_C\right) \; 
d\la_\xi\
 \end{eqnarray*}
On the other hand, we have
\begin{eqnarray*}
\int_X R(f) \; d\la & =& \int_{X/\xi} \left(\int_{C} R(f) \chi_C\; d
\la_C\right) \; d\la_\xi\\
\\
& = & \int_{X/\xi} \left(\int_{C} R[f (\chi_C\circ \sigma)]\; d\la_C
\right) \; d\la_\xi\\
\\
&= &\int_{X/\xi} \left(\int_{C} f \chi_C\; d(\la_C R_C)\right) \; 
d\la_\xi\\
\\
& = & \int_{X/\xi} \left(\int_{C} f W_C\; d\la_C\right) \; d\la_\xi.
\end{eqnarray*}

By uniqueness of the system of conditional measures, we have the 
result. 
\end{proof}

\subsection{Positive operators and polymorphisms}\label{subsect 
pairs of measures}

In this subsection, we discuss the following\textit{ problem}.
 Let $(X, \B)$ be  a standard Borel space, and let $\mu_1, \mu_2$ 
 be two probability measures on $(X,\B)$. 
 How can we characterize positive  operators $P$ such that 
 $\mu_1 P = \mu_2$? What can be said about transfer operators
 $(R, \sigma)$ satisfying the condition $\mu_1 R= \mu_2$ where
 $\sigma$ is an onto endomorphism of $(X, \B)$? 
 
 We denote by $\mc P$ the set of positive linear operators acting
 in the space of Borel functions $\FXB$.  We assume implicitly that 
 operators from $\mc P$ are order continuous so that the action 
 $\mu \mapsto \mu P$ is defined on $M(X)$.
 
For fixed measures $\mu_1$ and  $\mu_2$ on $(X, \B)$, we denote
   $$
 \mc P(\mu_1, \mu_2) := \{ P\in \mc P :  \mu_1 P = \mu_2,\ 
 P(\mathbf 1) = \mathbf 1\}.
 $$
Similarly, if $\mc R(\sigma)$ is the set of transfer operators 
corresponding to an endomorphism $\sigma$, then we denote
$$
 \mc R(\mu_1, \mu_2) := \{ R \in \mc R(\sigma) :  \mu_1 R
  = \mu_2,\  R(\mathbf 1) = \mathbf 1 \}.
$$ 

Given a standard Borel space $(X, \B)$, consider the product space
$(Y, \mc A) = (X \times X, \B \times \B)$, and let $\pi_1$ and $
\pi_2$ be the projections, $\pi_i(x_1, x_2) = x_i, i=1,2$. For
 convenience of notation, we will also write $Y = X_1 \times X_2$ 
 where $X_1 = X = X_2$. It will be clear from our next discussions that 
 all results remain true in the case when we have two distinct spaces 
 $(X_1, \B_1, \mu_1)$ and $(X_2, \B_2, \mu_2)$.
 
 Suppose that $\nu$ is a Borel probability measure on $X \times X$. Then 
 $\nu $ defines the marginal measures 
 \index{measure ! marginal} $\mu_1$ and $\mu_2$ on $X_1$
  and $X_2$, respectively:
$$
\mu_i(A)  = \nu(\pi_i^{-1}(A)), \ \ A\in \B.
$$
Denote by
$$
\frak M(\mu_1, \mu_2) :=  \{\nu \in M_1(Y) : \nu\circ \pi_1^{-1} = 
\mu_1,\ \nu\circ \pi_2^{-1} = \mu_2\}.
$$ 
We remark that if two measures $\mu_1, \mu_2$ are given on 
$X$, then the product measure $\nu =\mu_1 \times \mu_2$ gives an
 example of a measure from  $ \frak M(\mu_1, \mu_2)$.

\begin{lemma}\label{lem nu by nu_x}
Suppose $\nu$ is a probability measure on $Y = X_1 \times X_2$ from
the set $\frak M(\mu_1, \mu_2)$. Then $\nu$ is uniquely determined
 by the system of conditional measures $(\nu_x : x \in X_1)$ generated
 by the measurable partition $\xi_1 : = \{\pi_1^{-1}(x) : x \in X_1\} $,
 $$
 \nu = \int_{X_1} \nu_x \; d\mu_1
 $$
\end{lemma}

This results follows immediately from the uniqueness of the system of
conditional measures (see Subsection \ref{subsect meas 
partitions}). It can be interpreted as follows: if $\nu$ and $\nu'$
are two measures from $\frak M(\mu_1, \mu_2) $, and $\nu_x = \nu'_x$ 
for $\mu_1$-a.e. $x\in X_1$, then $\nu = \nu'$. 

For every measure $\nu \in \frak M(\mu_1, \mu_2)$, we define a positive
operator  $P_\nu : L^1(\mu_2) \to L^1(\mu_1)$ by setting
\be\label{eq pos operator P_nu}
P_\nu(f)(x) = \mathbb E_\nu (f\circ \pi_2  \ | \ \pi_1^{-1}(x)).
\ee
Equivalently, formula (\ref{eq pos operator P_nu}) can be written as follows
\be\label{eq pos op P_nu in integrals}
P_\nu(f)(x) = \int_{X_2} (f\circ \pi_2)\; d\nu_{x},\ \  x\in X_1,
\ee
where $\nu_x$ is the system of conditional measures defined in Lemma 
\ref{lem nu by nu_x}.

We observe that if $\nu = \mu_1 \times \mu_2$, then $P_\nu$  is 
a rank 1 operator such that
$$
P_\nu (f) = \int_{X_2} (f\circ\pi_2) \; d\mu_2,
$$
so that $P_\nu (f)(x), x \in X_1,$ is a constant function.

For the measure space $(Y, \mc A, \nu) = (X_1 \times X_2, \B \times \B,
\nu)$, the projections $\pi_1$ and $\pi_2$ define the \textit{isometries} 
$V_1$ and $V_2$, respectively, where 
\be\label{eq V_1}
V_1(f) = f \times 1 : L^2(X_1, \mu_1) \to L^2(Y, \nu),
\ee
\be\label{eq V_2}
V_2(f) = 1 \times f : L^2(X_2, \mu_2) \to L^2(Y, \nu).
\ee
With some abuse of notation, we denote by $f_1 \times f_2$ the function
 $f(x_1, x_2) = f_1(x_1)f_2(x_2)$. Equivalently, $(f \times 1)(x_1, x_2) = 
 f\circ\pi_1(x_1, x_2)$, and $(1 \times f)(x_1, x_2) = 
 f\circ\pi_2(x_1, x_2)$.

\begin{lemma}\label{lem P_nu as V_1V_2}
The operator $P_\nu$, considered as an operator acting from
 $L^2(\mu_2)$ into $L^2(\mu_1)$, satisfies the relation
 $$
 P_\nu = V_1^* V_2.
 $$ 
\end{lemma}

\begin{proof} We begin with finding the explicit formula for the adjoint
operator $V_1^* : L^2(\nu) \to L^2(\mu_1)$. For any functions $f \in 
L^2(\nu)$ and $g \in L^2(\mu_1)$, we have 
\begin{eqnarray*}
\langle f, V_1(g)\rangle_{L^2(\nu)} & = &\int_{X_1 \times X_2}
f(x_1, x_2) (g\circ \pi_1)(x_1, x_2)\; d\nu(x_1, x_2)\\
\\
&=& \int_{X_1} (g\circ \pi_1)(x_1, x_2)\left( \int_{X_2} 
f(x_1, x_2) \; d\nu_{x_1} \right) \; d\mu(x_1)\\
\\
& =& \langle V_1^*(f), g\rangle_{L^2(\mu_1)},
\end{eqnarray*}
where 
\be\label{eq adjoint of V_1}
V_1^*(f)(x_1) = \int_{X_2} f(x_1, x_2) \; d\nu_{x_1}.
\ee

The remaining part of the proof follows now from (\ref{eq adjoint of V_1}):
$$
V_1^* V_2 (f) = V_1^*(f\circ\pi_2) = \int_{X_2} 
(f\circ\pi_2)(x_1, x_2) \; d\nu_{x_1}(x_2) = P_\nu(f),
$$
and we are done.
\end{proof}

Our main result of this  subsection, Theorem 
\ref{thm measures vs operators},  contains several statements that clarify 
the relationship between the set of measures $\frak M(\mu_1, \mu_2)$
and the set of positive operators $P \in \mc P(\mu_1, \mu_2)$. We use 
here the notation introduced above. We also consider positive operators 
acting in the corresponding $L^2$-spaces.

\begin{theorem} \label{thm measures vs operators}
(1) Let $\nu \in \frak M(\mu_1, \mu_2)$. Then 
formula (\ref{eq pos operator P_nu}) defines an affine map 
$$
\Psi(\nu) = P_\nu : \frak M(\mu_1, \mu_2) \to \mc P(\mu_1, \mu_2).
$$

(2) Let $P \in \mc P(\mu_1, \mu_2)$ be a positive operator acting in
$\FXB$. Define a measure $\nu_P$ on $(X\times X, \B\times \B)$ 
by setting
\be\label{eq def of nu_P}
\nu_P(f_1 \times f_2) := \int_{X} f_1 P(f_2)\; d\mu_1.
\ee
Then, $\Phi(P) = \nu_P$ defines  an affine  map  
$$
\Phi : \mc P(\mu_1, \mu_2) \to \frak M(\mu_1, \mu_2).
$$

(3) The maps $\Psi : \nu \mapsto P_\nu $ and $\Phi : P  \mapsto 
\nu_P$ are affine  bijections between 
the sets  $\frak M(\mu_1, \mu_2)$ and $\mc P(\mu_1, \mu_2)$ such that $\Psi\circ \Phi (P) = P$, and $\Phi\circ\Psi (\nu) = \nu$.

\end{theorem}

\begin{proof} (1) We first  check  that the operator $P_\nu$ belongs
to the set $\mc P(\mu_1, \mu_2)$.  Since $\nu$ is a probability measure,
 $P_\nu$ is a normalized positive operator, we obtain
\begin{eqnarray*}
\int_{X_1} P_\nu(f)(x_1) \; d\mu_1(x_1) &=&  \int_{X_1} 
\left( \int_{X_2}
(f\circ \pi_2) (x_1, x_2)\; d\nu_{x_1}(x_2) \right)\; d\mu_1(x_1)  \\
\\
& = &\int_{X_1 \times X_2} (f\circ \pi_2) (x_1, x_2)\; d\nu(x_1, x_2)\\
\\
 & =&  \int_{X_2} \left(\int_{X_1} (f\circ \pi_2) (x_1, x_2)\; 
 d\nu_{x_2}(x_1)\right) \; d\mu_2(x_2)\\
 \\
 & =& \int_{X_2} f(x_2)\; d\mu_2(x_2).
\end{eqnarray*}
We used here the Fubini theorem, and  two properties: (i) the measure 
$\nu_x$ is probability for $\mu_1$-a.e
$x \in X_1$,  and (ii) $X_1 = X_2 = X$. Thus, we conclude that $\mu_1 
P_\nu = \mu_2$. 

Moreover, as we will see from (2), a positive normalized operator $P$ in 
$\FXB$ belongs to $\mc P(\mu_1, \mu_2)$ if and only if there exists a
 measure $\nu \in 
\frak M(\mu_1, \mu_2)$ such that $P = P_\nu$ where $P_\nu$ is defined
by (\ref{eq pos operator P_nu}). For this, we 

The fact that $\Psi(\alpha \nu_1 + (1-\alpha) \nu_2) = \alpha \Psi(\nu_1)
+ (1-\alpha)\Psi(\nu_2), \alpha \in (0,1),$ is obvious.

(2) We show that $\nu_P \in \frak M(\mu_1, \mu_2)$. Apply the definition 
of $\nu_P$ to characteristic functions $\chi_A$ and $\chi_B$ where 
$A\subset X_1$ and $B \subset X_2$:
$$
\nu_P(\chi_{A \times B}) = \int_X \chi_A P(\chi_B)\; d\mu_1. 
$$ 
Then we see that  $\nu\circ \pi_1^{-1} (A) = \mu_1$, and $\nu\circ 
\pi_2^{-1} (B) = \mu_1 P$. Because $P \in \mc P(\mu_1, \mu_2)$ we 
have $\mu_1 P = \mu_2$. Hence,  $\mu_1$ and $\mu_2$ are the
marginal measures for $\nu_P$. 

Next, we show that $\Phi$ is a one-to-one map. Suppose there are
positive operators $P, Q \in \mc P(\mu_1, \mu_2)$ such that 
$\nu_P = \nu_Q$. Then, for any functions $f_1$ and $f_2$, we have 
$$
\int_{X} f_1 P(f_2)\; d\mu_1 = \int_{X} f_1 Q(f_2)\; d\mu_1.
$$
It follows, by standard arguments, that $P = Q$. 

(3) It remains to check that $\Psi$ and $\Phi$ are inverses of each other. It 
can be done by direct computations:
\begin{eqnarray*}
(\Phi \circ \Psi (\nu))(f_1 \times f_2) & =& \int_{X_1} f_1P_\nu(f_2)
\; d\mu_1\\
\\
& =& \int_{X_1} f_1(x_1)\left( \int_{X_2} (f_2\circ \pi_2)(x_1, x_2)
\; d\nu_{x_1}(x_2)\right) d\mu_1(x_1)\\
\\
& =& \int_{X_1}  \int_{X_2} f_1(x_1) f_2( x_2)\; d\nu(x_1, x_2)\\\
\\
& =& \nu(f_1\times f_2).
\end{eqnarray*}

Similarly, we can show that $\Psi \circ \Phi (P) = P$ for any $P \in 
\mc P(\mu_1, \mu_2)$. This is equivalent to the equality $P_{\nu_P} 
= P$ where $\nu_P$ is defined by (\ref{eq  def of nu_P}). The latter 
relation can be proved by using the definitions of $\Phi$ and $\Psi$. 
\end{proof}

\begin{remark} Since the maps $\Phi$ and $\Psi$ are affine, we obtain from
Theorem \ref{thm measures vs operators} that 
they establish one-to-one correspondence between extreme points of the
sets $\frak M(\mu_1, \mu_2)$ and $\mc P(\mu_1, \mu_2)$. We observe 
that the measure $\mu_1\times \mu_2$ is an extreme point in  $\frak 
M(\mu_1, \mu_2)$ as well as the corresponding rank one operator  
$P_{\mu_1 \times \mu_2}$ is an extreme point in $\mc P(\mu_1, 
\mu_2)$.
\end{remark}

\begin{corollary}
Suppose the $L^2$-space of $(X_1 \times X_2, \B\times \B, \nu)$, 
$\nu \in \frak M(\mu_1, \mu_2)$, is 
isometrically embedded into $L^2(\Omega, \rho)$ where $(\Omega, \rho)$
is a standard measure space, $U : L^2(X_1 \times X_2, \nu) \to
 L^2(\Omega, \rho)$. Let $V_i : L^2(X_i, \mu_i) \to L^2(X_1 \times 
 X_2, \nu), i =1,2,$ be the isometries defined by (\ref{eq V_1}) and
 (\ref{eq V_2}). Set $\wt V_i = UV_i$. Then the positive operator 
 $\wt P$ defined by $\wt V_i$ as in Lemma \ref{lem P_nu as V_1V_2} 
 coincides with $P_\nu$.
\end{corollary}

The proof follows immediately from the relation 
$$
\wt P = \wt V_1^* \wt V_2 = V_1^*U^*UV_2 = P_\nu.
$$
\medskip

Suppose now an onto endomorphism  $\sigma$ is defined on a standard 
Borel space $(X, \B)$. Let $\mu $ be a probability measure on $(X, \B)$.
We know that the partition $\xi = \{\sigma^{-1}(x) : x\in X\}$ of $X$ is 
measurable, hence there exists a system of conditional measures
$(\mu_{C_x})$ defined by $\xi$, where $C_x$ is the element of $\xi$
that contains $x$, see Subsection \ref{subsect meas partitions}. In Example
\ref{ex TO by cond meas}, we used measures $(\mu_{C_x})$ to define a
transfer operator 
\be\label{eq R from cond meas}
R(f)(x) = \int_{C_x}f(y)\; d\mu_{C_x}(y).
\ee

We consider here another class of measures on the product space 
$(X \times X, \B \times \B)$ associated to $(X, \B, \mu, \sigma)$. 
For $\mu$, $\sigma$, and $(\mu_{C_x})$ s above,
take the partition $\xi_1$ of $X \times X$ into the fibers 
$\{\pi^{-1}_1(x) : x\in X\}$ and assign the measure $\mu_{C_x}$ to the
set $\{x\} \times \pi^{-1}_1(x)$ endowed the induced Borel structure. 
We see that, in fact, the measure $\mu_{C_x}$ is supported by the set 
$\{x\} \times C_x$. 

Let now $\nu$ be the measure on $X \times X, \B \times \B)$ such that, 
for a Borel function $f(x_1, x_2)$,
\be\label{eq measure nu from cond measures}
\nu(f)  = \int_{X_1} \left(\int_{\pi^{-1}_1(x_1)} f(x_1, x_2)\; 
d\mu_{C_{x_1}}(x_2)\right) \; d\mu(x_1).
\ee

Using the partition $\xi_1$, we can also disintegrate $\nu$ over $X_1$
 and get the family of conditional measures $\nu_x, x \in X_1$. 
 By definition of $\nu$, we have $\nu_x = \mu_{C_x}$.

\begin{lemma} Let $R$ and $\nu$ be defined by (\ref{eq R from cond 
meas}) and  (\ref{eq measure nu from cond measures}), respectively. 
Then $R = R_\nu$ where $R_\nu$ is  defined in 
(\ref{eq pos operator P_nu})
\end{lemma}
\begin{proof}
We have from the definition of $\nu$ 
\begin{eqnarray*}
R_\nu(f)(x)  &=& \mathbb E_\nu(f(\pi_2(x, y))\; | \; 
\pi_1^{-1}(x)) \\
\\
&=& \int_{\pi_1^{-1}(x)} f(y) \; d\nu_x(y)\\
\\
&=& \int_{C_x} f(y) \; d\mu_x(y)\\
\\
& =& R(f)(x), \ \ \ \qquad (x, y) \in X\times X.
\end{eqnarray*}
\end{proof}

This means, in particular,  that the transfer operator $R_\nu$ possesses 
the pull-out property. 

\begin{proposition}
Let the measure $\nu$ on $X \times X$  be defined by 
(\ref{eq measure nu from cond measures}). Then  the marginal measures 
for $\nu$ are $\mu_1 = \nu\circ \pi_1^{-1}  = \mu$ and $\mu_2 = 
\nu\circ \pi_2^{-1}$
 where 
$$
\mu_2(B) = \int_X \mu_{C_x}(B) \; d\mu(x), \ \ B\in \B.
$$

Moreover, $\mu_1 R_\nu = \mu_2$ and $R_\nu \in \mc 
R(\mu_1, \mu_2)$.
\end{proposition}

\begin{proof}
The fact that the marginal measure $\mu_1$ coincides with $\mu$ is 
obvious. To find $\mu_2(B) = \nu(X_1 \times B)$, we take
\begin{eqnarray*}
\mu_2(B) & = & \int_{X} \left(\int_{\pi_1^{-1}(x)} \chi_B(y) \; 
d\mu_{C_x}(y)\right) \; d\mu(x)\\
\\
& =& \int_X \mu_{C_x}(B) \; d\mu(x).
\end{eqnarray*}

The second statement is a reformulation of the first result. Indeed,  if we 
use (\ref{eq R from cond meas}) and the relation $R_\nu = R$, then we
can conclude that 
\begin{eqnarray*}
\mu_2(f) & = & \int_{X_1} \mu_{C_x}(f) \; d\mu(x)\\
\\
& =& \int_{X_1} R(f) \; d\mu_1(x).
\end{eqnarray*}
This completes the proof.
\end{proof}

We give one more example of a measure $\nu $ on the product space
$(X \times X, \B\times B)$. Let $\mu$ be  a measure on $(X, \B)$, and 
let $\sigma$ be an onto  endomorphism of $X$. Define the probability
 measure $\nu = \nu(\sigma)$ on $ \B \times \B$ as follows:
\be\label{eq nu via sigma}
\nu(\sigma)(A \times B) := \mu( A \cap \sigma^{-1}(B)), \ \ \ A, B \in \B.
\ee

\begin{lemma} In the above notation, the following properties hold:

(1) For $\nu = \nu(\sigma)$,
$$
\mu_1 = \nu \circ \pi_1^{-1} = \mu, \ \ \ \mu_2= 
\nu \circ \pi_2^{-1} = \mu\csi1.
$$

(2) The composition operator $S_\sigma : f \mapsto f\cs $ belongs
to $\mc P(\mu_1, \mu_2)$ where $\mu_2 = \mu_1 S_{\sigma}, 
\mu_1 = \mu$, i.e., \index{composition operator}
$$
\mu S_\sigma   = \mu\csi1.
$$ 

(3) Let $P_{\nu(\sigma)}$ be the positive operator defined by 
the measure $\nu(\sigma)$ according to (\ref{eq pos operator P_nu}). 
Then $P_{\nu(\sigma)} = S_\sigma$. Equivalently, the operator 
$S_\sigma$ is the only solution to the equation
\be\label{eq S is unique}
\nu(\sigma) (f_1 \times f_2) = \int_{X_1} f_1 S(f_2)\; d\mu_1.
\ee

\end{lemma}

\begin{proof}
The first two assertions of this lemma are rather obvious: (1) follows
 immediately from the definition of $\nu(\sigma)$, and (2) is verified 
  straightforward. 

To see that (3) holds, we note that, by (1), $\nu(\sigma) \in \frak 
M(\mu, \mu\csi1)$, and therefore, we can use Theorem \ref{thm measures 
vs operators}.  Since the maps $\Phi$ and $\Psi$ are one-to-one, we 
conclude that there exists only one operator satisfying 
(\ref{eq S is unique}).

\end{proof}

\

\begin{remark} (1) There is a clear connection between positive 
operators from  the set $\mc 
P(\mu_1, \mu_2)$ and  the notion of \textit{polymorphisms} which 
was introduced and studied in a series of papers by A. Vershik, see e.g.,
\cite{Vershik2000, Vershik2005}. 

By definition, a polymorphism \index{polymorphism}
 $\Pi$ of the standard measure space
$\sms$ to itself is a diagram consisting of an ordered triple of standard
measure spaces:
$$
(X, \B, \mu_1) \stackrel {\pi_1} \longleftarrow \ (X \times X, \B\times \B, \nu)
\stackrel {\pi_2} \longrightarrow \ (X,\B, \mu_2),
$$
where $\pi_1$ and $\pi_2$ are the projections to the first and second 
component of the product space $(X \times X, \B\times \B, \nu)$ such that
$\nu\circ\pi_i^{-1} = \mu_i$. 

This definition can be naturally extended to the case of two different
measure spaces $(X_i, \B_i, \mu_i), i =1,2$. Then we have  a polymorphism
defined between these measure spaces. 

(2) Our approach to the study of measures on product spaces is similar 
to the  study of \textit{joinings} \index{joining} in ergodic theory. 
We recall the definition
of this notion given for single transformations. Suppose  that two
dynamical systems, $(X, \B, \mu, \sigma)$ and $(Y, \mc A, \nu, \tau)$, 
are given.  Then a joining  is a measure $\la$ on $(X \times Y, \B \times 
\mc A)$ such that (i) $\la$  is invariant with respect to $\sigma \times 
\tau$, and (ii) the projections of $\la$ onto the $X$ and $Y$
 coordinates are $\mu$  and $\nu$, respectively. The theory of joinings
 is well developed in ergodic theory and topological dynamics
  and contains many impressive  results. We refer to \cite{Glasner2003, 
  Rue2006, Rudolph1990}  where the   reader can  find further references. 
\end{remark}

We finish this section by formulating a result that was proved in 
\cite{AlpayJorgensenLewkowicz2016}. 

Suppose a positive operator $R$, acting on measurable function over 
$\sms$, has the properties
\be\label{eq Rh =h muR =mu} 
R h = h, \ \ \ \ \ \mu R = \mu,
\ee
where $h$ is a harmonic function for $P$ and $\mu$ is a probability 
$R$-invariant measure.

\begin{theorem}\label{thm from AJL16}
Let $R$ be a positive operator satisfying (\ref{eq Rh =h muR =mu}). 
Suppose
$$
(\Omega, \B_\infty) = \prod_0^\infty (X, \B) 
$$
is the infinite product space. Then, on $(\Omega, \B_\infty)$, there exists
a unique probability measure $\mathbb P$, defined on cylinder functions
$f_0 \times f_1 \times \cdots \times f_n$ ($n \in \N, f \in \FXB$), as 
follows
$$
\int_\Omega f_0 \times f_1 \times \cdots \times f_n \; d\mathbb P =
\int_X f_0 P(f_1 P(\ \cdots \ P(f_{n-1}P(f_n h))\  \cdots \ ))\; d\mu.
$$

If $\{ \pi_i \  |\ i = 0,1,... \}$ denotes the coordinate random functions, then 
the following Markov property holds 
$$
\mathbb E_{\mathbb P}(f\circ \pi_{i+1} \ |\ \pi_i^{-1}(x)) =
R(f)(x)
$$
for all $i$, $x \in X$, and $f \in \FXB$. 

\end{theorem}

 Readers coming from other but related areas, may find the following
  papers/books useful for background \cite{DutkayJorgensen2015, 
  Fedotov2013, ZhangJorgensen2015, JorgensenSong2015, 
 LatremoliereFredericPacker2016, Maier2013,  Reveles2016}.
 
 \newpage

\section{ Transfer operators on $L^1$ and $L^2$}\label{sect L1 and 
L2}

Given a transfer operator $(R, \sigma)$,  it is of interest to find the 
measures $\mu$ such that both   $R$  and $\sigma$ induce operators 
in the corresponding $L^p$ spaces, i.e., in $L^p(X, \B, \mu)$. We turn 
to this below, but our main concern are the cases $p =1, p=2$, and $p 
= \infty$. When $R$  is realized as an operator in $L^2(X, \B, \mu)$, 
for a suitable choice of $\mu$,  then it is natural to ask for the adjoint 
operator $R^*$ where ``adjoint'' is defined with respect to the 
$L^2(\mu)$-inner product. 

We turn to this question in Subsections 
\ref{subsect Adjoint operator in L2} and \ref{subsect relations 
between R and sigma} below. Our operator theoretic results for these 
$L^2$-spaces will be used in  Section \ref{sect Universal HS} below 
where we introduce a certain 
 \textit{universal Hilbert space} $\mc H(X)$, or rather $\mc H(X, \B)$. 
  Indeed, when a transfer operator $(R, \sigma)$ is given, we show 
   that there is then a naturally induced \textit{isometry} in the
    universal Hilbert 
  space $\mc H(X)$, which we show offers a number of applications 
 and results which may be considered to be an infinite-dimensional 
 Perron-Frobenius theory. Our study of transfer operators in 
 $L^2$-spaces is motivated by \cite{AlpayJorgensenLewkowicz2016,
 Jorgensen2001}.

\subsection{Properties of transfer operators acting on $L^1$ and 
$L^2$} In this section, we will keep the following settings. Let $(X, \B,
 \la)$ be a standard measure space, and let $\lambda\in M(X)$ be 
  a Borel measure on $\B$.  Suppose that $\sigma$ is a  non-singular 
  surjective  endomorphism on $(X, \B, \la)$.  We will consider transfer
   operators $(R, \sigma)$ 
defined on the space of Borel functions $\FXB$.  It will be assumed
that the function $R(\mathbf 1)$ is either in $L^1(\lambda)$, or in 
$L^2(\lambda)$, depending on the context. 

\begin{lemma} Let $(X, \B)$ be  a standard Borel space, and $\sigma
\in End(X, \B)$. Set $S(f) = f\cs , f \in \mc F(X, \B)$. For a measure 
$\mu$  on $(X, \B)$, let the measure $\mu S$ be defined by 
\be\label{eq S acts on measures}
\int_X f\; d(\mu S) = \int_X S(f)\; d\mu = \int_X f\cs\; d\mu.
\ee
Then $\mu S = \mu\csi1$.
\end{lemma}

\begin{proof} It follows from (\ref{eq S acts on measures}) that, for any
$B \in \B$ and the characteristic function $\chi_B$, we have
$$
(\mu S)(B) = \int_X \chi_B \cs \; d\mu = \int_X \chi_{\sigma^{-1}(B)}
 \; d\mu = (\mu\csi1)(B),
$$
and the result follows.
\end{proof}

\begin{lemma}\label{lem Borel TO induces L-p TO}
Let $(R,\sigma)$ be a transfer operator in $\FXB$. Let $\la$ be a 
Borel measure on $(X, \B)$. Then $(R, \sigma)$ induces a transfer 
operator in the space $L^p(\la)$ if and only if $\la R \ll \la$ and 
$\la\cs$.
\end{lemma}

This observation is obvious and explains why we will work
with the $\sigma$-quasi-invariant measures $\la$ which belong to 
the set $\mc L(R)$.

Assume that $\lambda\in \mathcal L(R)$ denote the Radon-Nikodym
 derivative of $\la R$ with respect to $\la$ by
$$
W_\lambda(x) = W(x) := \frac{d(\lambda \circ R)}{d\lambda}(x).
$$
Since $R$ is integrable, we have $W_\lambda \in L^1(\lambda)$, and 
the following useful equality holds (see (\ref{eq W as R-N derivative}):
$$
\int_X R(\mathbf 1)\; d\la = \int_X W \; d\la, \ \ \lambda\in \mathcal 
L(R).
$$

\begin{lemma}\label{lem Ri as derivative}
In the above notation, the function $R(\mathbf 1)$ is represented as
 follows:
\be\label{eq R1 R-N derivative}
R(\mathbf 1)(x) = \frac{(Wd\lambda)\circ\sigma^{-1}}{d\lambda}(x) = 
\frac{d(\lambda R)\circ\sigma^{-1}}{d\lambda}(x)
\ee
 where  $\lambda$ is any measure from $\mathcal L(R)$.

If $R$ is a normalized transfer operator, $R(\mathbf 1) = \mathbf 1$,
then 
$$
(\la R) \csi1 = \la, \quad \ \forall \la \in \mc L(R).
$$

Moreover, a transfer operator $R$ is integrable with respect to $
\lambda$ if and only if $(\lambda R)(X) < \infty$.
\end{lemma}

\begin{proof}
By the definition of the Radon-Nikodym derivative $W$,  we have the 
relation
\be\label{eq R-N der W via integral}
\int_X R(f) \; d\lambda = \int_X fW\; d\lambda
\ee
which holds for any measurable function $f$.  In particular,
$f$ can be any simple function.
Substitute  $f\circ\sigma$ instead of $f$ in (\ref{eq R-N der W via 
integral}). Then
\begin{equation*}
\int_X (f\circ\sigma)W\; d\lambda =  \int_X R(f\circ\sigma) \; d
\lambda  = \int_X f R(\mathbf 1) \; d\lambda.
\end{equation*}
Since the last relation holds for every $f$, we have the equality
of measures
$$
(Wd\lambda)\circ \sigma^{-1} (x)= R(\mathbf 1)d\lambda(x),
$$
that proves the first statement.

The other two  assertions follow immediately from (\ref{eq R1 R-N
 derivative}).
\end{proof}

In the next remark we collect several direct consequences of Lemma 
\ref{lem Ri as derivative}. Though these results can be easily proved,
they contain some important properties of transfer operators that  are 
used below.  

\begin{remark}\label{rem properties 4} (1) Equality (\ref{eq R1 R-N 
derivative}) might be confusing because  the left hand side of the 
relation
$$
R(\mathbf 1)(x) = 
\frac{d(\lambda R)\circ\sigma^{-1}}{d\lambda}(x), \ 
\lambda\in \mathcal L(R),
$$
does not contain the measure $\lambda$. But we should remember
 that, in the setting introduced above,  $R(\mathbf 1)(x)$ is
  considered as a  function in $L^1(\lambda)$, so that $\lambda$ is 
  involved implicitly.  
 If we denote by $\theta_\lambda$ the Radon-Nikodym derivative for 
 a non-singular endomorphism $\sigma$,
$$
\theta_\lambda(x) = \frac{d\lambda\circ\sigma^{-1}}{d\lambda}(x),
$$
then the function $R(\mathbf 1)$ can be written as follows
\begin{eqnarray*}
R(\mathbf 1)(x) &= & \frac{d(\lambda R)\circ\sigma^{-1}}{d\lambda 
R}(x) \frac{d(\lambda R)}{d\lambda}(x) \\
\\
&=& \theta_{\lambda R}(x) W(x).
\end{eqnarray*}

(2) Let $\sigma$ be a non-singular endomorphism of $(X, \B, \la)$. 
It follows from Lemma \ref{lem Ri as derivative} that a transfer 
operator $(R, \sigma)$ on $L^1(\lambda)$ is strict, i.e., $(R1)(x) > 
0$ $\la$-a.e., if and only if $W(x) > 0$ $\la$-a.e., and $\lambda\circ 
\sigma^{-1} \sim \lambda$.  
 Moreover, it is seen from (\ref{eq R1 R-N derivative}) that we can 
 prove the following result. 

  \begin{lemma}\label{lem equivalence strict TO}
Let $\sigma$ be a non-singular endomorphism of $(X, \B, \la)$. Then
the following properties are equivalent:

i) $R$ is strict, i.e. $R(\mathbf 1)(x) > 0$ for $\lambda$-a.e. $x$;

ii) $W(x) > 0$ for $\lambda$-a.e. $x$;

iii) $\lambda  R \sim \lambda$;

iv) $\theta_\lambda (x) = R(W^{-1})$.
  \end{lemma}
  
\begin{proof} We prove iv) only and leave other assertions to the 
reader. For this, we check
\begin{eqnarray*}
\int_X f \theta_\la \; d\la  &=& \int_X (f\cs) \; d\la \\
\\
& = & \int_X \frac{1}{W} W (f\cs) \; d\la \\
\\
& =& \int_X R\left(\frac{1}{W}(f\cs) \right) \; d\la\\
\\
&=& \int_X  R\left(\frac{1}{W} \right) f \; d\la.
\end{eqnarray*}
Since $f$ is any function, we have the result.
\end{proof}  
  
(3) If $R(\mathbf 1) = \mathbf 1$, then we obtain from the 
statements  of Lemma \ref{lem equivalence strict TO} that
$$
\theta_{\la R} = \frac{1}{W}, \qquad \ \ \ R(\theta_{\la R}) 
= \theta_\la.
$$

Indeed, to see these, we find
\be\label{eq theta la R}
\theta_{\la R} = \frac{d(\la R) \csi1}{d(\la R)} =
\frac{d(\la R) \csi1}{d\la} \frac{d\la}{ d(\la R)} = R(\mathbf 1)
\frac{1}{W}.
\ee

(4) We notice that if $(R, \sigma)$ is a strict transfer operator acting 
on the space of measurable functions over $(X, \B, \lambda)$ with 
$\lambda \in \mathcal L(R)$,  then $\sigma$ is non-singular with 
respect to $\la R$. This fact follows from  (ref{eq theta la R}).

In other words, one has the properties
$$
(\la R)\circ \sigma^{-1} \ll \la R \ll \la.
$$

 (5) Another corollary of relation (\ref{eq R1 R-N derivative}) is 
 formulated as follows: for any two measures $\lambda, \lambda' \in 
 \mathcal L(R)$, we have
$$
\frac{W_\lambda}{W_{\lambda'}} = \frac{\theta_{\lambda' R}}
{\theta_{\lambda R}}.
$$
\end{remark}

In ergodic theory, it is extremely important to understand how 
 properties of a transformation depend on a measure.
More precisely, suppose a transformation $T$ acts on a measure space 
$\sms$. What can be said about dynamical  properties of $T$
if $\mu$ is replaced by an equivalent measure $\nu$? 
We discuss here this question in the context of transfer operators.

\begin{lemma}\label{lem lambda_1 equiv to lambda} Let 
$(R, \sigma)$ be a transfer operator and $\lambda \in \mathcal L(R)$. 
Suppose that a Borel measure $\lambda_1$ is equivalent to $\lambda
$, that is there exists a positive measurable function $\varphi(x) $ 
such that $d\lambda_1(x) = \varphi(x)d\lambda(x)$. Then $
\lambda_1\in \mathcal L(R)$, and the Radon-Nikodym derivative 
$W_1 = \dfrac{d\la_1 R}{d\la_1}$ is $\sigma$-cohomologous 
to $W$: \index{Radon-Nikodym derivative ! cohomologous}
$$
W_1(x) = \va (\sigma (x)) W(x) \va(x)^{-1}.
$$
\end{lemma}

\begin{proof} The proof is based on the direct calculation, the 
definition of the Radon-Nikodym derivative for $R$, and 
the pull-out property of $R$. We note that because $\la \sim \la_1$,
then $\va$ is positive a.e.  Let $f$ be 
any measurable function, then we compute
\begin{eqnarray*}
 \int_X R(f)\; d\la_1 &=& \int_X R(f)\va \; d\la \\
 \\
   &=&  \int_X R((\va\circ \sigma) f)\; d\la \\
   \\
   &=& \int_X (\va\circ \sigma) f\; d(\la R) \\
   \\
   &=&  \int_X (\va\circ \sigma) f W\; d\la \\
   \\
   &=& \int_X f (\va\circ \sigma)  W\va^{-1}\; d\la_1
  \end{eqnarray*}
Thus, we proved that $(\la_1 R)(f) = (\va\circ \sigma)  W\va^{-1}
\lambda_1(f)$. Hence,
$$
\dfrac{d\la_1 R}{d\la_1} = (\va\circ \sigma)  W\va^{-1}.
$$
\end{proof}

\begin{remark}
We observe that one can directly check the validity of the 
equality for the measure $\la_1$
$$
R\mathbf 1 = \frac{[ (\va\circ \sigma)  W\va^{-1}d\lambda_1]\circ 
\sigma^{-1}}{ d\lambda_1}.
$$
This confirms the conclusion  of Lemma \ref{lem Ri as derivative}.
\end{remark}

\begin{corollary}\label{cor coboundary implies inv meas}
Let $(R, \sigma)$ be a transfer operator such that $R\mathbf 1 \in 
L^1(\lambda)$ for a Borel measure $\la$. Suppose that  
 $\lambda  \in \mathcal L(R)$. The Radon-Nikodym derivative
 $W = \dfrac{d\la R}{d\la}$ is a coboundary with respect to $\sigma$ 
 if and only if there exists a measure $\la_1$ such that  $\la_1
  \sim \la$    and  $\la_1 R = \la_1$.
\end{corollary}

\begin{proof}
Suppose that $W$ is a coboundary, that is there exists a measurable 
function $q(x)$ such that
$W = (q\circ \sigma) q^{-1}$. Take a new measure $\la_1$ defined 
by $d\la_1 = q d\la$. Then $\la_1$ is equivalent to $\la$ and, for 
any integrable function $f $, we compute
 \begin{eqnarray*}
 \la_1 (Rf)  &=&  \int_X f \; d(\la_1R) \\
 \\
    &=&  \int_X R(f) q\; d\la\\
    \\
    &=&   \int_X R(f (q\cs))\; d\la\\
    \\
     &=&   \int_X f (q\cs))\; d(\la R)\\
     \\
      &=&   \int_X f (q\cs)) (q \cs)^{-1}q \; d\la  \qquad 
      \mbox{(because\ $W = (q\circ \sigma) q^{-1}$)}\\
      \\
       &=&   \int_X f  \; d\la_1 \\
       \\
       &=&  \la_1(f).
 \end{eqnarray*}
Hence $\la_1$ is $R$-invariant.

Conversely, suppose that a measure $\la_1 $ is $R$-invariant and $d
\la_1 = \varphi d \la$. By invariance with respect to $R$, we have
\begin{eqnarray*}
\int_X f \; d\la_1 &=& \int_X R(f) \; d\la_1 \\
\\
 &=& \int_X R(f) \va \; d\la \\
 \\
   &=& \int_X R(f(\varphi\cs))\; d\la  \\
   \\
   &=&   \int_X f(\varphi\cs)W\; d\la  \\
   \\
   &=& \int_X f(\varphi\cs)W \varphi^{-1}\; d\la_1
\end{eqnarray*}
Since $f$ is any function, we conclude that $W$ is a 
$\sigma$-coboundary:
$$
W = (\varphi\cs)^{-1} \varphi.
$$
\end{proof}

\begin{theorem}\label{thm from measure equivalence} Let $(R, 
\sigma)$ be a transfer operator acting in the space of Borel functions 
over $(X,\B)$.

(1) Suppose that $R$ is such that the action $\la \mapsto \la R$ is 
well defined on the set of all measures $M(X, \B)$. Then the partition 
of  $M(X, \B)$  into subsets $[\lambda] := 
\{\lambda' \in M(X, \B)  : \lambda' \sim \lambda\}$, consisting of 
equivalent measures, is invariant with respect to the action of $R$. 
Thus, the transfer operator $R$ sends equivalent
measures to equivalent ones. 
More generally, if $\lambda_1 \ll \la$, then $\la_1 R \ll \la R$
and
$$
\frac{d(\la_1 R)}{d(\la R)} = \va \cs
$$
where $d\la_1 = \varphi d\la$. 

(2) If $\la_1, \la_2 \in \mathcal L(R)$, then $\lambda_1 + \lambda_2 
\in \mathcal L(R)$. Moreover,
$$
W := \frac{d(\lambda_1 + \lambda_2)  R}{d(\lambda_1 + 
\lambda_2)} = W_1\frac{d\la_1}{d(\lambda_1 + \lambda_2)} + W_2 
\frac{d\la_2}{d(\lambda_1 + \lambda_2)}
$$
where $W_i$ is the Radon-Nikodym derivative of $R_i$  defined in 
(\ref{eq R-N der W via integral}), $ i =1,2$.
\end{theorem}

\begin{proof}
(1) The first part of the statement is obvious, see Lemma 
\ref{lem lambda_1 equiv to lambda} . To prove the other statements 
in (1), it suffices 
to check the fact that  $\la_1 \ll \la$ implies $\la_1  R \ll \la  R$. 
Since $d\la_1 = \varphi d\la$,  we have
$$
\int_X f \; d(\la_1 R) = \int_X R(f) \; d\la_1 = \int_X R(f) \varphi
\; d\la
$$
$$
= \int_X R[(\varphi\circ \sigma) f] \; d\la = \int_X (\varphi\circ 
\sigma) f\; d(\la R).
$$
Hence, we get
$$
\varphi\circ \sigma = \frac{d(\la_1 R)}{d(\la R)}.
$$

(2) We compute the Radon-Nikodym derivative of $(\la_1 + \la_2)R$ 
with respect to $(\la_1 + \la_2)$ as follows:
\begin{eqnarray*}
  \int_X fW\; d(\lambda_1 + \la_2) &=& \int_X R(f)\; d(\lambda_1 + 
  \la_2)  \\
  \\
   &=& \int_X f\; d(\lambda_1 R) + \int_X f\; d(\lambda_2 R) \\
   \\
   &=& \int_X fW_1\; d\lambda_1 + \int_X fW_2\; d\lambda_2 \\
   \\
   &=& \int_X f\; [W_1 d\lambda_1 + W_2d\la_2].
\end{eqnarray*}
Thus, $Wd(\lambda_1 + \la_2) = W_1 d\lambda_1 + W_2d\la_2$, 
and we are done.

\end{proof}

\begin{remark}\label{rem choice of equivalen measures}
Suppose that $(X, \B, \mu, T)$ is a  measurable dynamical system 
where $T$ is a measurable transformation of $X$. How do properties 
of $T$ depend on measure? Can $\mu$ be replaced by an equivalent 
measure? These questions are well known in ergodic theory, and 
many dynamical properties of $T$ do not depend on a choice of
 a measure in the class $[\mu]$. In particular, this happens in the 
 orbit equivalence theory. The importance of Lemma \ref{lem 
 lambda_1 equiv to lambda} and  Theorem \ref{thm from measure 
 equivalence} consists of explicit formulas relating Radon-Nikodym 
 derivatives of $R$  to equivalent measures.
\end{remark}

\begin{remark}
(1) It follows from Theorem \ref{thm from measure equivalence} any  
transfer operator acts not only on individual measures from 
$\mc L(R)$ but also it acts on the set of classes of  equivalent
 measures: $R[\la] = [\la R]$.
 
(2) We point out several formulas that relate the Radon-Nikodym 
derivatives for $R$ and $\sigma$. They are based on Lemma 
\ref{lem Ri as derivative} and Remark \ref{rem properties 4}. It is 
 assumed that a transfer operator $(R, \sigma)$ is defined on $(X,\B, 
 \la) $ where $\lambda$ is a measure from $\mathcal L(R)$.

(a) If $\lambda  R = \lambda$ (that is $W =1$ a.e.), then
\be\label{eq R1 is theta_la}
R(\mathbf 1)(x) = \dfrac{d\lambda\circ \sigma^{-1}}{d\lambda}(x) =
\theta_\lambda(x).
\ee

Let $Fix(R) := \{\la \in \mc L(R) : \la R = \la\}$ be the set 
of $R$-invariant measures. The above formula means that   the 
function  $(x, \lambda) \mapsto \theta_\lambda(x)$ does not depend 
on $\lambda \in Fix (R)$. As a confirmation of this observation, one  
can show directly that for $\lambda_1, \lambda_2 \in Fix(R)$ the 
condition $\theta_{\lambda_1 + \lambda_2}(x) = R(\mathbf 1)(x)$  
holds.

(b) If $\la R = \la$, then relation (\ref{eq R1 is theta_la}) implies that
$$
R(\mathbf 1) = \mathbf 1 \ \ \Longleftrightarrow \ \ \theta_\la = 1.
$$
Therefore, if $\la \in Fix(R)$, then $R$ is normalized if and only if
$\la$ is $\sigma$-invariant. 

(c) Let $\la, \la_1$ be two equivalent measures from $\mc L(R)$
where $(R, \sigma)$ is a transfer operator. Let the  function $\xi(x) >
 0$, be defined by the relation  $d\lambda_1(x) = \xi(x)
  d\lambda(x)$.  Then we can show that
$$
\frac{d(\lambda_1 R)\sigma^{-1}}{d\lambda_1}(x) = R(\xi\circ
\sigma)(x) \frac{d\lambda}{d\lambda_1}(x)
$$
\end{remark}

\subsection{The adjoint operator for a transfer operator}
\label{subsect
Adjoint operator in L2}

We recall briefly the notion of a symmetric pair of linear operators in a 
Hilbert space.

Suppose that $\mc H_1$ and $\mc H_2$ are Hilbert spaces and $A$ 
and $B$ are operators with dense domains $Dom(A) \subset \mc 
H_1$ and $ Dom (B) \subset \mc H_2$ such that $A : Dom(A)  \to 
\mc H_2$ and $B : Dom(B) \to  \mc H_1$. It is said that $(A; B)$ is a
\emph{ symmetric pair} if
$$
\langle Ax, \ y\rangle_{\mc H_2} = \langle x, \ By\rangle_{\mc H_1},
$$
where $x \in Dom (A), y \in Dom (B)$.
In other words, $(A; B)$ is a symmetric pair if and only if $A \subset  
B^*$ and $B \subset  A^*$.

If $(A; B)$ is a symmetric pair, then the operators $A$ and $B$ are  
closable. Moreover, one can prove that \index{symmetric pair}

(i) $A^*\ol A$ is densely defined and self-adjoint with $Dom (A^*\ol 
A) \subset Dom(\ol A) \subset \mc  H_1$,

(ii) $B^*\ol B$ is densely defined and self-adjoint with $Dom (B^*\ol 
B) \subset Dom(\ol B) \subset  \mc H_2$.

Therefore, without loss of generality, we can assume that $A$ and $B$ 
are closed operators, and we can work with self-adjoint operators 
$A^*A$ and $B^*B$.
\medskip

We will discuss below transfer operators $(R, \sigma)$ defined on 
$L^2(\lambda)$ where $\lambda$ is a $\sigma$-quasi-invariant 
measure. It turns out that one can explicitly describe various properties 
of $R$ and the adjoint operator $R^*$.

\begin{theorem}\label{thm adjoint of R}
Let $(R, \sigma)$ be a transfer operator considered on the space $(X, 
\B, \lambda)$. Suppose that $R(\mathbf 1) \in L^1(\lambda) \cap 
L^2(\lambda)$ and  $\lambda\circ R \ll \lambda$. Then $R$ is a 
densely defined linear operator in the Hilbert space $\mathcal H =  
L^2(\lambda)$ whose adjoint operator $R^*$ is determined by the 
formula \index{transfer operator ! adjoint}
$$
R^* f = W (f\circ \sigma), \ \ \ f \in \mathrm{Dom}(R^*).
$$
In particular, $W = R^*(\mathbf 1)$, and $W \in L^2(\lambda)$.
\end{theorem}

\begin{proof} We take simple functions $f$ and $g$ such that $f, g 
\in \mathcal S^2(\lambda)$. Then $(f\circ \sigma)g$ is also a simple 
function. Since $R(\mathbf 1) \in L^1(\lambda)$, we conclude that 
$R((f\circ \sigma)g ) \in L^1(\lambda)$ according to Lemma \ref{lem 
R(s) is integrable}.

It follows from (\ref{eq R-N der W via integral}) and (\ref{eq char 
property of r via sigma 1}) that the following equalities hold:
$$
\int_X (f\circ\sigma)gW\; d\lambda = \int_X R((f\circ\sigma)g)\; d
\lambda = \int_X f R(g)\; d\lambda.
$$
All integrals in these formulas are well defined. Indeed, we use that $f, 
g \in S^2(\lambda)$ and $R(\mathbf 1) \in L^2(\lambda)$ to 
conclude that $f, R(g)$ are in $L^2(\lambda)$. Hence, we have
\be\label{eq L2 inner product}
\int_X f R(g)\; d\lambda  = \langle  R(g), f \rangle_{L^2(\lambda)}.
\ee
We recall that $\lambda R$ is a finite measure, and the functions $f
\circ\sigma$, and $g$ are simple. Therefore, the integrals
$$
\int_X (f\circ\sigma)gW\; d\lambda = \int_X (f\circ\sigma)g\; 
d(\lambda R)
$$
are finite. Moreover,
\be\label{eq inner product S5}
\int_X (f\circ\sigma)gW\; d\lambda = \langle g, W(f\circ \sigma)
\rangle_{L^2(\lambda)}.
\ee

Thus, we have proved that, for any function $g \in \mathcal 
S^2(\lambda)$,
$$
 \langle  R(g), f\rangle_{L^2(\lambda)} = \langle g, W(f\circ \sigma)
 \rangle_{L^2(\lambda)}.
$$
This relation means that the adjoint operator $R^*$ is defined for 
every $f\in \mathcal S^2(\lambda)$ and
\be\label{eq adjoint of R}
R^*(f) = W (f\circ \sigma).
\ee
\end{proof}

It turns out that when a transfer operator $(R, \sigma)$ is considered 
as an operator in  the space $L^2(\la)$ with $\la \in \mc L(R)$, then  
this operator can be realized  explicitly, see Example 
\ref{ex example 2}.

\begin{theorem}\label{thm formula for R in L2} Suppose $R$ is a 
transfer operator acting in $L^2(\la)$ such that $d(\la R) = W d\la$. 
Then, for any $f \in L^2(\la)$,
\be\label{eq R in L^2}
R(f)(x) = \frac{(fWd\la)\circ \sigma^{-1}}{d\la}(x), \qquad \ \ \la
\mbox{-a.e.}.
\ee
\end{theorem}

\begin{proof} We first  note that formula (\ref{eq R in L^2}) defines a 
transfer operator that can be checked directly. Next, as it follows from 
(\ref{thm adjoint of R}), a function $R(f) \in L^2$ which satisfies the 
equation
$$
\int_X g R(f)\; d\la = \int_X (g\circ\sigma)f W\; d\la, \qquad \forall g 
\in L^2(\la), $$
is uniquely determined by this relation. Then, the right hand side is 
represented as
$$
\int_X (g\circ\sigma)f W\; d\la = \int_X g \; \frac{(fWd\la)\circ
\sigma^{-1}}{d\la}\; d\la,
$$
and this equality proves (\ref{eq R in L^2}).

\end{proof}

If $(R,\sigma)$ is a transfer operator acting in $\FXB$ and a measures 
$\la \in M(X)$ is such that $\la R \ll \la$, then the realization of $R$
in $L^(\la)$, given in (\ref{eq R in L^2}), is denoted by $R_\la$.

\begin{proposition}\label{prop Ri =1 for all measures}
Let a transfer operator $( R_\la, \sigma)$ be defined in $L^2(\la)$ 
and suppose that $R_\la (\mathbf 1) =\mathbf 1$. Then, for any
 measure $\la' \sim \la$, 
the operator $R_{\la'} \in L^2(\la)$ has the property $R_{\la'} 
\mathbf 1 = \mathbf 1$.
\end{proposition}

\begin{proof} Let $d\la' = \varphi d \la$. Then, as shown in Lemma 
\ref{lem lambda_1 equiv to lambda}, the corresponding Radon-
Nikodym derivatives $W_\la$ and $W_{\la'}$ are related by the 
formula
$$
W_{\la'} = (\varphi\circ\sigma) W_\la \varphi^{-1}.
$$
Based on the proof of Theorem \ref{thm formula for R in L2}, it 
suffices to show that,  for any $g \in L^2(\la')$, 
$$
\int_X g \; d\la' = \int_X (g\circ\sigma)  W_{\la'}\; d\la'.
$$
We  compute
\begin{eqnarray*}
  \int_X (g\circ\sigma)  W_{\la'}\; d\la' &=& \int_X   (g\circ\sigma) 
  (\varphi\circ\sigma) W_\la \varphi^{-1} \varphi \; d\la\\
  \\
   &=& \int_X [(g\varphi)\circ \sigma] W_\la\; d\la  \\
   \\
   & = &  \int_X R[(g\varphi)\circ \sigma]  \; d\la\\
   \\
   &=&  \int_X g\varphi\; d\la\\
   \\
   & = &\int_X g \; d\la'.
\end{eqnarray*}
This means that $R_{\la'}( \mathbf 1) = \mathbf 1$.
\end{proof}

The following corollary is basically deduced from Theorem \ref{thm 
adjoint of R}. 

\begin{corollary}\label{cor formulas on R and adjoint of R}
Let $R$ be a transfer operator in $L^2(\la)$ such that $d(\la R)=
Wd \la$.

(1) The domain of $R^*$ contains the dense set $\mathcal 
S^2(\lambda)$.
 The transfer operator  $R$ is a closable in $L^2(\lambda)$.

(2) The following formulas hold:
$$
R^*(f) = (f\circ\sigma) R^*(\mathbf 1),
$$
$$
 (RR^*)(f) = f(RR^*)(1)   = R(W)f.
$$
This means that $RR^*$ is a multiplication operator.

(3) $R^*(f) \in L^1(\lambda)$ for any simple function $f$.

(4) $R^*$ is an isometry if and only if $R(W) =1$.

(5) For every $n\in \N$, the operator $(R^*)^n$ is defined on 
$L^2(\lambda)$ by the formula
$$
(R^*)^n f = (f\circ \sigma^n) W(W\circ \sigma)\cdots (W\circ 
\sigma^{n-1}).
$$
\end{corollary}

\begin{proof}
The proofs of most statements are rather obvious so 
that they can be left for the reader. We show here the proof of (4)
only. Let $f, g \in L^2(\la)$, then we have
\begin{eqnarray*}
\langle R^* f, R^* g\rangle &=& \int_X W(f\cs) W(g\cs)\; d\la\\
\\
&=& \int_X W(fg)\cs) \; d(\la R)\\
\\
&=& \int_X R[W(fg)\cs)] \; d\la \\
\\
&=& \int_X (fg) R(W) \; d\la R\\
\end{eqnarray*}
Hence, $\langle R^* f, R^* g\rangle = \langle f, g\rangle$ if and
only if $R(W) =1$. 
\end{proof}

\begin{proposition}\label{prop R^n f in L2}
Let the conditions of Theorem \ref{thm adjoint of R} hold. Then the 
domain of $R^n$, considered as an operator in $L^2(\lambda)$,  
contains  functions  $f \in \mathcal S^2(\lambda)$ for any $n \geq 
2$.
\end{proposition}

\begin{proof}
Suppose $n=2$. Then, for $f \in \mathcal S^2(\lambda)$, we see that
$$
\int_X R^2(f)\; d\lambda = \int_X R(R(f))\; d\lambda = \int_X R(f) W
\; d\lambda.
$$
Since $R(f)$ and $W$ are in $L^2(\lambda)$ (see Theorem \ref{thm 
adjoint of R}), we conclude that
$$
\int_X R^2(f)\; d\lambda  = \langle R(f), W\rangle_{L^2(\lambda)} < 
\infty.
$$
To prove the statement for any natural $n > 2$, we use induction.
\end{proof}

Proposition \ref{prop R^n f in L2} is important for consideration 
powers of $R$ because it states that $R^n$ has a dense domain in 
$L^2(\la)$ for every $n$. 

\begin{proposition}\label{prop harmonic}
Let $(R, \sigma)$ be a transfer operator acting in the 
$L^2(\la)$-space of measurable functions over a measure space
$(X, \B, \la)$. Suppose that $\la R \ll \la$ and let $Wd\la = d(\la R)$. 
If $h$ is  a  harmonic function $h$ for $R$, then we have
$$
\|R(W)\|_{L^\infty(\la)} \leq 1 \ \ \Longrightarrow \ \ R(W) = 1 \ a.e.
$$
The converse is obviously true.
\end{proposition}

\begin{proof} We use Corollary \ref{cor formulas on R and adjoint of R}
to prove the following relation
$$
\|RR^*\|^2_{L^2(\la)} =\|R(W)\|^2_{L^\infty(\la)}
$$
where $||T||_{L^2(\la)}$ denotes the operator norm of $T$ when 
$T : L^2(\la) \to L^2(\la)$ is a bounded operator. 
It follows then that  $\|R\|_{L^2(\la)} = \|R^*\|_{L^2(\la)} \leq 1$,
 i.e. $R$ and $R^*$ are contractions in $L^2(\la)$. We notice that, 
 in this case,
$$
Rh = h \ \  \Longleftrightarrow \ \ R^* h =h.
$$
(For this, one can show that $\|R^* h - h\|_{L^2(\la)} \leq 0 
\ \Longleftrightarrow \ R^*h =h$).
Therefore, we  use Theorem  \ref{thm adjoint of R} to deduce that
\be\label{eq W coboundary}
Rh = h \ \ \Longrightarrow \ \ h = W (h\circ \sigma),
\ee
 that is $W$ is a $\sigma$-coboundary. When we apply $R$ to the 
 right hand side of (\ref{eq W coboundary}), then we obtain $h = R(W) 
 h$ a.e., hence $R(W) =1$ a.e. 
\end{proof}

\begin{remark}
We observe that, due to the Schwarz inequality, the following 
relation is true.
$$
|R(f)| \leq \sqrt{R(|f|^2)}\sqrt{(R(\mathbf 1))}.
$$
Here we assume that $R$ is integrable and $R(f) \in L^2(\lambda)$. In 
particular, this holds for simple functions.

More generally, we have that, for any $k \in \N$,
$$
|R(f)| \leq \R(|f|^{2^k})^{2^{-k}}R(\mathbf 1)^{\sum_{i=1}^k 2^{-i}}.
$$
\end{remark}

\subsection{More relations between $R$ and $\sigma$}\label{subsect 
relations between R and sigma}
 Let $(X, \B,\lambda)$ be a measure space with a
 surjective  endomorphism $\sigma$.We recall  our assumption about
  the  endomorphism $\sigma$:  it is forward and backward 
  quasi-invariant, i.e., $\lambda(A) =0$ if and only if $
  \lambda(\sigma^{-1}(A)) =0$ if and only if $\lambda(\sigma(A)) 
  =0$. Let $(R, \sigma)$ be a transfer operator such that $\lambda 
  \circ R \ll \lambda$.
We will focus here on the study of relations between  the transfer 
operator $(R, \sigma)$ and the endomorphism $\sigma$.

 For $R$ and $\sigma$,  we recall the definitions of  the  Radon-
 Nikodym derivatives  (see Remark \ref{rem RN derivatives for sigma})
$$
W = \frac{d\lambda R}{d\lambda} \ \ \ \mbox{and} \ \ \ 
\omega_\lambda = \frac{d\lambda\circ \sigma}{d\lambda}.
$$

\begin{proposition}\label{prop two derivatives}
In the above notation, we have
$$
R(\omega_{\lambda R}) = W.
$$
\end{proposition}

\begin{proof} It follows from the definition of the Radon-Nikodym 
derivative $\omega_\lambda$ that, for any $f \in L^1(\lambda)$, one 
has
$$
\int_X f\; d\lambda = \int_X (f\circ\sigma) \omega_\lambda\; d
\lambda
$$
where $\omega_\lambda$ is a uniquely determined function which is 
measurable with respect to $\sigma^{-1}(\B)$. We apply this equality 
to the measure $\lambda R$ and obtain the following sequence of 
equalities:
$$
\int_X (f\circ\sigma) \omega_{\lambda R}\; d(\lambda R) = 
\int_X f \; d(\lambda R)
$$
$$
 \Updownarrow
$$
$$
\int_X R[(f\circ\sigma) \omega_{\lambda R}]\; d\lambda = \int_X 
R(f) \; d\lambda
$$
$$
  \Updownarrow
$$
$$
 \int_X f R(\omega_{\lambda R})\; d\lambda = \int_X f W\; d\lambda.
$$
Since $f $ is an arbitrary function from $L^1(\lambda R)$, we 
obtain the desired result.
\end{proof}

We recall that a measure $\lambda\in \mathcal L(R)$ is  called 
\emph{invariant} with respect to $R$ if $\lambda R = \lambda$, 
i.e., $W(x) = 1$ for $\lambda$-a.e. $x$.

Let $\varphi(x)$ be a Borel non-negative function. Given $\lambda \in 
\mathcal L(R)$, we define the  measure $\mu$:
\be\label{eq defining mu}
d\mu(x) = \varphi(x) d\lambda(x)
\ee
Then, for any measurable $f$,
$$
\int_X f\; d\mu = \int_X f\varphi \; d\lambda.
$$

\begin{theorem}\label{thm mu R-invariance}
Given a transfer operator $(R, \sigma)$, suppose that there exists a  
Borel measure  $\lambda$ on $(X, \B)$ such that $R(\mathbf 1) \in 
L^1(\la)$ and $\lambda  R = \lambda$. Then a Borel measure $\mu$ 
defined by its $\lambda$-density $\varphi(x)$ as in (\ref{eq defining 
mu}) is $R$-invariant if and only if $\varphi\circ \sigma = \varphi$.
\end{theorem}

\begin{proof}
The proof follows from the following chain of equalities. Let $g$ be a 
measurable function, then
\begin{eqnarray*}
  \int_X g\; d(\mu R)  &=&   \int_X R(g)\; d\mu\\
  \\
   &=&  \int_X R(g) \varphi\; d\lambda\\
   \\
   &=& \int_X R[ (\varphi\circ \sigma) g]\; d\lambda \\
   \\
   &=& \int_X (\varphi\circ \sigma) g\; d(\lambda  R)\\
   \\
   & = &  \int_X (\varphi\circ \sigma) g\; d\lambda\\
\end{eqnarray*}
Hence, if $\mu\circ R = \mu$, we get from the above relations that
$$
\int_X \varphi g\; d\lambda = \int_X (\varphi\circ \sigma) g\; d
\lambda,
$$
and $\varphi\circ \sigma = \varphi$. Conversely, if  $\varphi\circ 
\sigma = \varphi$ holds, then
$$
 \int_X g\; d(\mu R) =  \int_X g\; d\mu,
$$
that is $\mu R = \mu$.
\end{proof}

We can easily deduce from Theorem \ref{thm mu R-invariance} (see 
the next statement) that 
if $\sigma$ is ergodic on  $(X, \B, \la)$ , then any two  $R$-invariant
 measures are proportional. 

\begin{corollary}\label{cor ergodic sigma}
Let $(R, \sigma)$ be a transfer operator on $(X, \B, \lambda)$ such 
that $\lambda R = \lambda$. Suppose that $\sigma$ is an ergodic 
endomorphism with respect to the measure $\la$.
Then a Borel measure $\mu\ll \lambda$ is $R$-invariant if and only if 
$\mu = c\lambda$ for some $c \in \R_+$.
\end{corollary}

The following result clarifies the relationship between harmonic 
functions for a transfer operator $R$ and $\sigma$-invariant measures 
$\mu$. 

\begin{proposition}\label{prop harm and invarianca implies sigma 
invarianca}
Let $(R, \sigma)$ be a transfer operator on $(X, \B, \lambda)$ such 
that $\lambda R = \lambda$. Let $h$ be a non-negative Borel 
function. Then $h$ is $R$-harmonic, $R h = h$, if and only if the 
measure $d\mu(x) := h(x)d\lambda(x)$ is $\sigma$-invariant, i.e., $
\mu \circ \sigma^{-1} = \mu$.
\end{proposition}

\begin{proof}
The proof follows from the following argument. Let $g$ be an arbitrary 
Borel function. Then we deduce that the relation $Rh =h$ implies that 
$\mu \circ \sigma^{-1} = \mu$:
\begin{eqnarray*}
\int_X g\; d(\mu\circ \sigma^{-1})   &=& \int_X 
g\circ\sigma\; d\mu\\
\\
 \nonumber
   &=& \int_X (g\circ\sigma) h\; d\lambda\\
   \\
 \nonumber
   &=& \int_X R[(g\circ\sigma) h]\; d\lambda, \qquad (\mbox{recall\ 
   that\ } \lambda R = \lambda)\\
   \\
 \nonumber
   &=& \int_X gR(h) \; d\lambda \\
   \\
 \nonumber
   &=& \int_X g\; d\mu.
\end{eqnarray*}

Conversely, if we assume that $\mu$ is $\sigma$-invariant, then we 
can show, in a similar way, that
$$
\int_X gh\; d\lambda = \int_X gR(h)\; d\lambda
$$
for arbitrary function $g$. Hence, $h$ is harmonic for $R$.
\end{proof}

We recall that if $(R, \sigma)$ is a transfer operator and $k$ is a 
positive Borel function, then one can define a new transfer operator
$(R_k, \sigma)$ where 
$R_k(f) = R(fk)k^{-1}$ (in fact, $k$ can be non-negative but this 
generalization is inessential).  Furthermore,  this operator is 
normalized when $k$ is $R$-harmonic. More details are in Subsection
\ref{subsect harmonic and coboundaries}. 

\begin{lemma}\label{lem measure for R_h} Let $(R, \sigma)$ be  a 
transfer operator and let  $k$ be  a positive  function. Suppose that
$\la \in \mc L(R)$ and denote  by $W$ the corresponding Radon-
Nikodym derivative, $Wd\la 
= d \la R$. Then the measure $\la_k$ such that $d\la_k = k d\la$ is in 
$\mc L(R_k)$ and $W_k = W$ where $W_k d\la_k = d (\la_k R_k)$.
 In particular, if $\la R = \la$, then $\la_k =  \la_k R_k$.
\end{lemma}

\begin{proof} For any integrable function $f$, we get
\begin{eqnarray*}
  \int_X f \; d(\la_k  R_k) &=& \int_X R_k(f) \; d\la_k\\
  \\
   &=& \int_X R(fk) k^{-1} k \; d\la \\
   \\
   &=& \int_X fk W  \; d\la \\
   \\
   &=&  \int_X f W  \; d\la_k.
\end{eqnarray*}
This proves that the Radon-Nikodym derivative $W_k$ of measures $
\la_k R_k$ and $\la_k$  is $W$ for any positive function $k$ and any 
measure $\la\in \mc L(R)$.

\end{proof}

\begin{lemma}\label{lem harmonic omplies isometry}
Let $(R,\sigma)$ be a transfer operator considered on $L^2(\lambda)
$ where $\la \in \mathcal L(R)$. Let $W$ denote the Radon-Nikodym 
derivative $\dfrac{d(\la R)}{d\la}$. Suppose that  $R$ possesses  a  
harmonic function $h$, and we set $d\la_h = hd\la$.  Then the 
operator
$$
V : f \mapsto W^{1/2} (f\circ \sigma)
$$
is an isometry in $L^2(\la_h)$.
\end{lemma}

\begin{proof} We verify by direct calculations that  
$\|Vf\|_{L^2(\la_h)} = \|f\|_{L^2(\la_h)}$:
\begin{eqnarray*}
  \int_X |(Vf)|^2 \; d\la_h &=& \int_X W (|f|^2\circ \sigma) h \; 
  d\la \\
  \\
   &=& \int_X  (|f|^2\circ \sigma) h \; d(\la R) \\
   \\
   &=& \int_X  R[(|f|^2\circ \sigma) h] \; d\la \\
   \\
   &=& \int_X  |f|^2 R(h) \; d\la  \\
   \\
   &=&  \int_X  |f|^2 h \; d\la \\
   \\
 & = &\int_X  |f|^2  \; d\la_h.
\end{eqnarray*}

\end{proof}

Readers coming from other but related areas, may find the following papers/
books useful for background, \cite{BahsonSchmelingVaienti2015, 
FocchiMedrano2016, Matsumoto2017, 
Silvestrov2013, ShiozawaWang2017, Szewczak2017}.
 
 \newpage

\section{Actions of transfer operators on the set of Borel probability 
measures}
\label{sect actions on measures}

Let $(R, \sigma)$ be  a transfer operator defined on the space of Borel 
functions $\mc F(X, \B)$. The main theme of this section is the study 
of a dual action of $R$ on
the set of probability measures $M_1 = M_1(X, \B)$. As a matter of
fact, a big part of our results in this section remains true for any 
sigma-finite measure on $(X, \B)$, but we prefer to work with 
probability measures. The justification of this approach is 
contained in the results of Section \ref{sect L1 and L2} where
we showed that the replacement of a measure  by a probability 
measure does not affect the properties of $R$ described in terms
of measures.  Our main 
assumption for this section is that the  transfer operators 
$R$ are normalized, that is $R(\mathbf 1) = \mathbf 1$.
\medskip

\begin{remark}
We recall that there are classes of transfer operators $R$ for which
the natural action $\la \mapsto \la R$ on the set of measures is 
well defined. For instance, this is true for order continuous transfer 
operators, and for transfer operators defined on locally compact 
Hausdorff space, see Subsection \ref{subsect action 
on measures}. 
The advantage of dealing with probability measures and normalized 
operators $R$ is that, in this case, the 
measure $\la R \in M_1$ is defined  for any measure $\la\in M_1$ 
and for any normalized operator $R$, see Proposition \ref{prop 
measure lambda R}. 
\end{remark}
 
It follows from the above remark that, given a normalized transfer 
operator $(R, \sigma)$, we can associate two maps defined on  
$M_1$. They are
$$
t_R : \la \mapsto \la R, \ \qquad \ s_\sigma : \la \mapsto \la \csi1.
$$
We  call the maps $t_R$ and $s_\sigma$  \emph{actions} of $R$ and 
$\sigma$ on $M_1$, respectively.

For a normalized transfer operator $(R, \sigma)$, we can find out how 
these maps interact. 
We will show that the map $t_R$ is  one-to-one but 
not onto  (this fact is proved below in Theorem 
\ref{thm t_R is 1-1 and more}). Thus, we get the following decreasing 
sequence of subsets:
$$
M_1 \supset M_1R \supset M_1R^2 \supset \cdots .
$$
We will use the notation $K_i(R) = M_1(X)R^i$. Our interest is mostly 
focused on the set $K_1(R)$ (or simply $K_1$ when $R$ is fixed) 
because this set is crucial in our study of the action of $R$ 
on measures.

For a transfer operator $(R, \sigma)$, we also define the set of $R$-
invariant measures and the set of $\sigma$-invariant measures, by 
setting
$$
\mbox{Fix}(R) := \{\la \in M_1 : \la R = \la\},
$$
$$
\mbox{Fix}(\sigma) := \{\la \in M_1 : \la \csi1 = \la\}.
$$

We are interested in the following \emph{question.} Under what 
conditions on a transfer operator $(R, \sigma)$ is the set of 
invariant measures  $\mbox{Fix}(R)$ non-empty? A partial answer 
was given in Theorem \ref{thm mu R-invariance}. 

\begin{remark}\label{rem on partition of M_1}
We recall that, for a fixed transfer operator $R$, we dealt with the
 subset $\mc L(R)$ of the set of all measures $M_1$: by definition, 
 $\la \in \mc L(R)$ if and only if $\la R \ll \la$. In particular, $\la R$ 
 can be equivalent to $\la$. Theorem \ref{thm from measure 
 equivalence} asserts that when  the set of all measures $M_1$ is  
 partitioned into the classes of equivalent measures $[\la] := \{\mu : 
 \mu \sim \la\}$, then the map  $t_R$ preserves the partition into
  these classes  $[\la]$. The same holds for the sets 
$[\la]_{\ll} := \{\nu : \nu \ll \la\}$.  These facts are obviously true 
for the action of $s_\sigma$. Thus, if $M_1(\sim)$ is the set of 
classes of equivalent measures, then $t_R$ and $s_\sigma$ induce 
the maps $t_R(\sim)$ and $s_\sigma(\sim)$, defined on 
$M_1(\sim)$.

These facts will be used in  Section \ref{sect Universal 
HS} in the construction of the universal Hilbert space.
\end{remark}

The action $ s_{\sigma}$ of $\sigma$ on a the measure space $M_1$ 
is used to define the following two subsets naturally related to $\sigma$:
$$
\mc Q_+(\sigma) := \{\lambda \in M_1 :  \la\cs \ll \la\},
 $$
$$
  \mc Q_-(\sigma) := \{\lambda \in M_1 :  \la\csi1 \ll \la\}.
$$

We begin with a simple observation about measures for powers of
a transfer operator $R$.

\begin{lemma} Let $R$ be a transfer operator acting on a functional 
space over $(X, \B)$, and $\mc L(R)  = \{\la \in M_1 : \la R \ll \la\}$. 
Then
$$
\mc L(R) \subset \mc L(R^2) \subset \ \cdots \ \mc L(R^n) \subset 
\mc L(R^{n+1}) \subset \ \cdots,
$$
$$
 \mc Q_-(\sigma) \subset  \mc Q_-(\sigma^2) \subset \cdots
 \subset  \mc Q_-(\sigma^n) \subset  \mc Q_-(\sigma^{n+1}) 
 \cdots
$$
\end{lemma}

\begin{proof}  This fact follows immediately from Theorem 
\ref{thm from measure equivalence} and the discussion in
 Remark \ref{rem on partition of M_1}.
\end{proof}

\begin{lemma}
Suppose that $\la \in \mc L(R)$ and $\mu \ll \la$. Then
$\mu \in \mc L(R)$.
\end{lemma}

\begin{proof}
We need to show that $\mu R \ll \mu$. Since $\la R \ll \la$ and $\mu 
\ll \la$, there exist measurable functions $\varphi $ and $W$ from
 $L^1(\la)$ 
such that
$$
\varphi = \frac{d\mu}{d\la}, \ \ \ \ \ W = \frac{d(\la R)}{d\la}.
$$
Set
$$
Q(x) =
\begin{cases} ((\varphi \cs) W \varphi^{-1})(x), &\mbox{if } x \in A:= 
\{x : \varphi(x) \neq 0\}\\
\\
0, & \mbox{if } x \in A^c:= \{x : \varphi(x) = 0\}
\end{cases}.
$$

Take any measurable function $f$ and compute
\begin{eqnarray*}
\int_X f \; d(\mu R) &=& \int_X R(f)\; d\mu \\
\\
   &=& \int_A R(f) \varphi \; d\la  \\
   \\
   &=& \int_A R(f (\varphi\cs)) \; d\la \\
   \\
   &=&  \int_A f (\varphi\cs) \; d(\la R)\\
   \\
   &=&  \int_A f (\varphi\cs) W\; d\la\\
   \\
 & = & \int_A f (\varphi\cs) \varphi^{-1} W \varphi \; d\la\\
 \\
& = & \int_A f (\varphi\cs)W \varphi^{-1} \; d\mu\\
\\
& = & \int_X f Q \; d\mu.
\end{eqnarray*}
This proves that $\mu R \ll \mu$, and $Q =\dfrac{d(\mu R)}{d\mu}$.

\end{proof}

In the following lemmas we study the relations between the maps 
$t_R, s_\sigma$ and the sets $\mc L(R), \mc Q_-(\sigma), \mc Q_+
(\sigma)$.

\begin{lemma} If $\la \in \mc Q_-(\sigma)$, then $\la \ll \la\cs$.
\end{lemma}

\begin{proof}
For any Borel set, one has $A \subset \sigma^{-1}(\sigma (A))$. 
Therefore, if $(\la\cs)(A) = 0$, then $(\la\csi1)(\sigma (A)) = 0$, 
and then $\la(A) = 0$.
\end{proof}

\begin{lemma}\label{lem support R(chi_A) is sigma(A)}
If $(R, \sigma)$ is a normalized transfer operator, then, for any 
measure $\la$ and Borel set $A$,
$$
\la(\{x \in X : R(\chi_A)(x) > 0\}) = (\la\cs)(A).
$$
\end{lemma}

This fact follows immediately from Lemma \ref{lem R(s) is integrable}.

\begin{theorem}\label{thm L(R) is Q plus (sigma)}
If $(R, \sigma)$ is a normalized transfer operator, then
$$
\mc L(R) = \mc Q_+(\sigma).
$$
\end{theorem}

\begin{proof}
($\ \subset \ $) Suppose that $\la R \ll \la$. We proved in 
Theorem \ref{thm  
lambda R abs cntn wrt lambda} that, for any Borel set $A$,
\be\label{eq lambda R(A) as integral over sigma A}
(\la R)(A) = \int_{\sigma(A)} R(\chi_A)\; d\la.
\ee
If $\la(A) = 0$ implies that $(\la R)(A) = 0$, then, by (\ref{eq lambda 
R(A) as integral over sigma A}) and Lemma \ref{lem support R(chi_A) 
is sigma(A)}, we conclude that $\la(\sigma (A)) = 0$.

($\ \supset \ $)  Conversely, if $\la(A) = 0$ implies that 
$(\la\cs)(A) =0$,  then we 
again use (\ref{eq lambda R(A) as integral over sigma A}) and obtain 
that
$\int_{\sigma(A)} R(\chi_A)\; d\la = 0$, that is $(\la R)(A) = 0$.

\end{proof}

\begin{lemma}\label{lem lambda R equivalent to lambda}
Let $(R,\sigma)$ be  a transfer operator with $R (\mathbf 1) =
\mathbf 1$. Then
$$
\la \in \mc L(R) \bigcap \mc Q_-(\sigma) \ \Longleftrightarrow \ \ 
\la R \sim \la.
$$
\end{lemma}

\begin{proof} We need to show only that $\la \ll \la R$. This is 
equivalent to the statement that $\la(A) > 0$ implies  $(\la R)(A) 
>0$. Since $\la$ is a quasi-invariant measure with respect to $\sigma$ 
and $A \subset \sigma^{-1}(\sigma (A))$, we see that $\la (\sigma 
(A)) > 0$, and therefore
$$
(\la R) (A) =\int_X R(\chi_A) \; d\la > 0.
$$

Conversely, suppose that $\la R \sim \la$. Then we have to show that 
$\la\csi1 \ll \la$. The fact that $R \mathbf 1 =\mathbf 1$ implies 
that $(\la R)\csi1 = \la$ for any measure $\la$. Since the map $s_
\sigma$ preserves the partition of $M_1$ into the classes of 
equivalent measures, we obtain that $\la\csi1 \ll \la$ (in fact we have 
that these measures are equivalent). This proves the statement.
\end{proof}

\begin{lemma}
Let $\nu$ be a measure from $\mc L(R)$. Then for $\la = \nu R$ we 
have the property
$$
\frac{d\la R}{d\la} \in \mc F(X, \sigma^{-1}(\B)).
$$
More generally, $\dfrac{d\la R}{d\la} $ is 
$\sigma^{-i}(\B)$-measurable if $\la = \nu R^i$ and $\nu \in \mc L(R)$
 and $i\in \N$.
\end{lemma}

\begin{proof}
In order to proof the result, it suffices to note that due to Theorem 
\ref{thm from measure equivalence}
$$
\frac{d\la R}{d\la} = \frac{d\nu R^2}{d\nu R} = \frac{d\nu R}{d\nu} 
\cs.
$$
\end{proof}

In the following lemma, we collect a number of results that follow  from 
the  proved lemmas and definitions given in this section.

\begin{theorem}\label{thm t_R is 1-1 and more}
Let $(R, \sigma)$ be a normalized transfer operator acting in the 
space of Borel functions  $\FXB$ such that the dual action 
$\la \mapsto \la R: M_1 \to K_1(R) = M_1R$ is well defined. 
Then the following six statements hold:

(1) A measure $\mu \in K_1(R) $ if and only if $(\mu\csi1) R = \mu$.

(2) For any measure $\mu$, the equation $\mu = \la R$ has  a unique 
solution $\la = \mu\csi1$.

(3) The map $t_R$ is one-to-one on $M_1$.

(4a) Two measures $\la$ and $\la'$ are mutually singular  if and 
only if the measures $\la R$ and $\la' R$ are mutually singular.

(4b) If $\la \in K_1(R)$, then $\la$ and $\la\csi1 $ are mutually 
singular if and only if $\la $ and $\la R$ are mutually singular.

(5)
$$
\mbox{Fix}(R) \subset \bigcap_{i = 0}^\infty M_1R^i.
$$

(6)
$$
K_1(R) \bigcap \mbox{Fix}(\sigma) = \mbox{Fix}(R).
$$
\end{theorem}

\begin{proof} (1) We first recall that the condition $R(\mathbf 1) =
\mathbf 1$ can be written in an equivalent form, namely,
$$
(\la R)\csi1 = \la, \ \ \forall \la \in M_1.
$$
Hence, if $\mu \in K_1(R)$, then $\mu = \la R$ for some $\la \in M_1$, and
$$
[\mu\csi1] R = [(\la R)\csi1] R = \la R = \mu.
$$
The converse is obvious.

(2) This fact follows immediately from statement (1).

(3) Suppose that $\la R = \la' R$. Then condition $R(\mathbf 1) =
\mathbf 1$ implies that
$$
\la = (\la R) \csi1 = (\la' R) \csi1 = \la',
$$
and statement (3) is proved.

(4) Suppose $\la $ and $\la'$ are mutually singular measures. Then 
there exists a set $A$ such that $\la(A) = 1$ and $\la'(A) =0$. Then
$$
(\la R)(\sigma^{-1}(A)) = \int_X \chi_{\sigma^{-1}(A)} d(\la R) = 
\int_X \chi_{A} R(\mathbf 1) d\la = \la(A) =1
$$
and, similarly, we get that $(\la' R)(\sigma^{-1}(A)) = \la'(A) = 0$. 
To see that the converse is true, we observe that if  $\la R$ and
$\la' R$ are singular, then, applying $s_\sigma$ to these measures,
we obtain that  $\la$ and $\la'$ are singular. This proves (4a)

To show that  (4b) holds, we use (4a) and note that if $\la$ and $\la
\csi1 $ are mutually singular, then, applying $t_R$ to these measures, 
we get that $\la$ and $\la R$ are  mutually singular. To see that the 
converse holds we begin with mutually singular measures $\la$ and $
\la R$ and apply $s_\sigma$ to them. Since $R$ is normalized, the 
result follows.

(5) This statement is obvious.

(6) If $\la \in K_1(R) \cap \mbox{Fix}(\sigma)$, then $\la = (\la 
\csi1) R = \la R$.
Conversely, if $\la = \la R$, then  $\la \in K_1(R)$ by (5),  hence $  
(\la \csi1) R  = \la R$. Since $t_R$ is a one-to-one map, we conclude 
that $\la \csi1 = \la$.

\end{proof}

In the next lemma, we continue discussing relations between the maps 
$t_R$ and $s_\sigma$ for a transfer operator $(R, \sigma)$.

\begin{lemma}\label{lem interaction R and sigma on measures}
Let $(R, \sigma)$ be  a transfer operator such that $R(\mathbf 1) =
\mathbf 1$. The following statements hold.

(1) $ s_\sigma t_R = \mbox{id}_{M_1}$ and $t_R s_\sigma = 
\mbox{id}_{K_1}$ where $K_1 = M_1 R$.

(2) If $\la \in K_1$, then
\be\label{eq 1}
\la \csi1 \ll \la \ \ \Longleftrightarrow \ \ \la \ll \la R;
\ee
\be\label{eq 2}
\la \csi1 = \la \ \ \Longleftrightarrow \ \ \la = \la R;
\ee
\be\label{eq 3}
\la \csi1 \gg \la \ \ \Longleftrightarrow \ \ \la \gg \la R.
\ee

(3) Let $T(\la):= t_R s_\sigma(\la)$. Then $T : M_1 \to K_1$ such 
that $T^2 = T$.
Moreover, if $\la_1 = T(\la)$, then
$$
\la_1\csi1 = \la \csi1.
$$
\end{lemma}

\begin{proof} (1) The statement $s_\sigma t_R = \mbox{id}_{M_1}$ 
is a reformulation of the fact that $R(\mathbf 1) = \mathbf 1$ (see, 
for example, Theorem \ref{thm t_R is 1-1 and more} (1)). Let $\la \in 
K_1$, then it follows that $t_R s_\sigma = \mbox{id}
_{K_1(R)}$.

(2) If $\la \csi1 \ll \la$, then, because $R$ possesses the 
``monotonicity'' property ($\mu\ll\nu \Longrightarrow \mu R \ll \nu R
$), we obtain that $\la = (\la\csi1) R \ll \la R$. Conversely, suppose 
that $\la \ll \la R$. Then, applying $\sigma^{-1}$ to this relation,  we 
have $\la\csi1 \ll (\la R) \csi1 = \la$. This proves (\ref{eq 1}).

Relation (\ref{eq 2}) was proved in Theorem \ref{thm t_R is 1-1 and 
more} (6).

To show that (\ref{eq 3}) holds, we again apply $R$ to the both sides 
of $\la\csi1 \gg \la$ and get that $\la R \ll \la$. The converse 
implication, i.e., $\la R \ll \la$ implies $\la\csi1 \gg \la$, follows from 
the fact $R\mathbf 1 = 1$ and application of $\sigma^{-1}$ to $\la R 
\ll \la$. Observe that this implication is true for any measure $\la$.

(3) The fact that $T^2 = T$ follows from the property  $R\mathbf 1 = 
\mathbf 1$ and the corresponding relation $s_\sigma t_R = 
\mathrm{id}_{M_1}$.

Because $\la_1 = T(\la) = (\la\csi1)R$, then, taking into account that 
$R$ is a normalized operator, we obtain
$$
\la_1\csi1 = [(\la\csi1)R]\csi1 = \la\csi1.
$$
\end{proof}

\begin{remark}
We note that for any measure $\la$ in $M(X)$, the following relation 
holds
$$
\la R \ll \la\circ \sigma.
$$
Indeed, this claim easily follows from Lemma \ref{lem R(s) is 
integrable} because the function $R(\chi_A)$ takes zero value on the 
compliment of $\sigma(A)$.
\end{remark}

\newpage

\section{Wold's theorem and automorphic factors of endomorphisms}
\label{sect Wold}

In this section, we discuss Wold's theorem stating the existence of  a 
decomposition of any isometry operator of a Hilbert space in a unitary 
part and a unilateral shift.
The variant of Wold's theorem, we outline below, is a bit more 
geometric than the original result of Wold, which was, in fact,  a 
decomposition theorem stated for stationary stochastic processes. The 
geometric variant is a result that applies to the wider context of any 
isometry in a Hilbert space.  Some of the relevant references include 
\cite{Wold1948, Wold1951, Wold1954, HalmosWallen1969, 
BratteliJorgensen2002}.

\subsection{Hilbert space decomposition defined by an isometry} Let 
$\mc H$ be a real Hilbert space, and let $S$ be an isometry in 
$\mc H$. 
This means that $\|Sx\| = \|x\|$ for every $x\in \mc H$, or 
equivalently, $S^*S = I$ where $I$ denotes the identity operator in $
\mc H$. In general, $S$ is not surjective.

It follows immediately  that the operator $E_1 = SS^*$ is a projection. 
More generally, one can show that $E_n := S^n (S^*)^n$ is a 
projection. Indeed, we use $n$ times the relation $S^*S = I$ and 
obtain
\begin{equation*}
\setlength{\jot}{10pt}
\begin{aligned}
 E_n^2&= S^n (S^*)^{n-1}(S^*S)S^{n-1} (S^*)^n \\
   &= S^n (S^*)^{n-1}S^{n-1} (S^*)^n = \cdots \\
   &=  S^n (S^*)^n\\
   & =  E_n.
   \end{aligned}
\end{equation*}

\begin{lemma}\label{lem sequence of projections} The sequence of 
projections $\{E_n\}$ is decreasing
$$
I \geq E_1 \geq E_2 \geq \cdots ,
$$
 and each $E_n : \mc H \to S^n(\mc H)$ is onto.
\end{lemma}

This result follows from the obvious relation:
$$
\mc H \supset S(\mc H) \supset S^2(\mc H) \supset \cdots .
$$

Let $R_{S} := \{Sx : x \in \mc H\}$ be the range of $S$. Consider the 
kernel of the adjoint operator
$$
N_{S^*} := \{x \in \mc H : S^* x = 0\}.
$$
Then one can see that
\be\label{eq R orthog to N and vice versa}
(N_{S^*})^\bot  = R_S, \quad N_{S^*}  = (R_S)^\bot.
\ee
More generally, if $V$ is a bounded linear operator in $\mc H$, then
$$
ker (V^*) = \mc H \ominus V(\mathcal H).
$$

\begin{lemma}
The sequence $\{S^n N_{S^*}\}$ consists of mutually orthogonal 
subspaces of $\mc H$.
\end{lemma}

The proof of this lemma is clear: for any $k_1, k_2 \in N_{S^*}$ and 
$m \in \N$, we observe that
$$
<S^m k_1, k_2> \  = \ <k_1, (S^*)^m k_2> \ = 0.
$$

\begin{theorem}[Wold's theorem]\label{thm Wold} \index{Wold's theorem} 
Let $S$ be an isometry operator in a Hilbert space $\mc H$. Then the 
following statements hold.

(1) The space $\mc H$ can be decomposed into the orthogonal direct 
sum
$$
\mc H = \mc H_\infty \oplus \mc H_{shift}
$$
where $\mc H_{shift} = N_{S^*} \oplus SN_{S^*} \oplus \cdots 
\oplus S^kN_{S^*} \oplus \cdots$.

(2) The operator $S$ restricted on $\mc H_\infty$ is a unitary 
operator, and $S$ is a unilateral shift \index{unilateral shift} 
in the space $\mc H_{shift}$.

(3) $\mc H_{shift}^{\bot} =  \mc H_\infty$ and  $\mc H_{shift}  =  
\mc H_\infty^{\bot}$

\end{theorem}

\begin{proof} We sketch a proof of this theorem for the reader's 
convenience.

Let a vector $y\in \mc H$ be orthogonal to every subspace 
$S^kN_{S^*}$, $ k =0,1,...$. In particular, it follows from (\ref{eq R 
orthog to N and vice versa}) that
 $$
y \in (N_{S^*})^\bot  \ \Longleftrightarrow \ y \in R_S \ 
\Longleftrightarrow \ E_1y = y.
$$
It turns out that a more general result can be proved.

\begin{lemma} For any $n \in \N$, one has
$$
y \in (S^nN_{S^*})^\bot  \ \ \Longleftrightarrow \ \ E_{n+1}y = y.
$$
\end{lemma}

\begin{proof}
To see that the statement of this lemma is true, we apply the following 
sequence  of equivalences:
\begin{equation*}
\setlength{\jot}{10pt}
\begin{aligned}
  y \in \left( S^n N_{S^*} \right)^\bot & \ \Longleftrightarrow \    
  <y, \ S^n x> = 0 \qquad\qquad \forall x \in N_{S^*}  \\
   & \ \Longleftrightarrow \   <(S^*)^n y, \ x> = 0 \   \qquad \ \  
   \forall x \in N_{S^*}  \\
   & \ \Longleftrightarrow \  (S^*)^n y \in R_S  \qquad \qquad\quad
   \  (\mbox{see\ (\ref{eq R orthog to N and vice versa})}) \\
   & \ \Longleftrightarrow \   \exists x \in \mc H \ \mbox{such\ that}
   \ (S^*)^n y = Sx \\
& \ \Longleftrightarrow \  E_{n+1} y = y.
\end{aligned}
\end{equation*}
The last equivalence follows from the relation
$$
E_{n+1}y = S^{n+1}(S^*)^{n+1} y = S^{n+1}(S^*S)x = S^{n+1}x = y
$$
that proves the lemma.
\end{proof}

We continue the proof of the theorem. It follows from Lemma \ref{lem 
sequence of projections} that the strong limit
$$
\lim_{n\to \infty} E_n  = E_\infty
$$
exists, and is the projection onto the subspace
$$
\mc H_\infty = \bigcap_{n} S^n \mc H = \bigcap_n \left(S^n N_{S^*}
\right)^\bot.
$$

Next, we prove that $S$ and $S^*$ restricted to $\mc H_\infty$ are 
unitary operators. As a corollary, we obtain  a few formulas involving
these operators. For this, we show that
 $$
x \in \mc H_\infty \ \Longleftrightarrow \  \|(S^*)^n x\| = ||x||, \ \ \ 
\forall n \in \N.
$$
Observe first that $||x||^2 = ||E_n x||^2 + ||E_n^\bot x||^2$ where 
$E_n^\bot = I_{\mc H} - E_n$ and $x \in \mc H, n \in \N$. Since $||
E_n  x - E_\infty x|| \to 0$ for all $x \in \mc H$, we obtain that
$$
x \in \mc H_\infty \ \Longleftrightarrow \  E_n^\bot x \to 0 \ (n\to 
\infty)  \Longleftrightarrow \  E_n^\bot x \to 0.
$$
Because $||(S^*)^n x||^2 = ||S^n(S^*)^n x||^2 = ||E_n x||^2$, we 
conclude that
$$
||(S^*)^n x||^2 = ||x|| \ \Longleftrightarrow \ E_n^\bot x =0.
$$
Furthermore,
$$
x \in \mc H_\infty \ \Longleftrightarrow \ E_n^\bot x =0 \ \ \forall n 
 \in \N \ \Longleftrightarrow \ ||(S^*)^n x|| = ||x||.
$$
In particular, this means that $SS^*|_{\mc H_\infty} = I_{\mc H_
\infty}$.

We notice that the subspace $\mc H_\infty$ is invariant with respect 
to $S$ and $S^*$:
$$
S^*(\mc H_\infty) \subset \mc H_\infty \ \Longleftrightarrow \  
S(\mc H_\infty^\bot) \subset \mc H_\infty^\bot.
$$
Indeed, any vector $x$ from $\mc H_\infty^\bot$ has the form
$$
x = k_0 + Sk_1 + \cdots + S^i k_i + \cdots
$$
where all $k_i $ are from $N_{S^*}$, and
$$
S(k_0 + Sk_1 + S^2 k_2 + \cdots ) = 0 + Sk_0 + S^2 k_1 + \cdots .
$$
Thus, the operator $S$ on $\mc H_\infty^\bot = \mc H_{shift}$ is a 
unilateral shift, 
$$
(k_0, k_1, k_2, ... ) \mapsto (0, k_0, k_1, k_2, .... ).
$$
\end{proof}

\subsection{Automorphic factors and exact endomorphisms}

The goal of this subsection is to apply the Wold
theorem to the study of isometries generated by endomorphisms of a
 measure space.
  
We recall first the definition of a factor map. Let $(X, \B)$ and $(Y, 
\mc A)$ be standard Borel spaces, and let $\sigma : X \to X$ and $
\tau : Y \to Y$ be surjective maps.
It is said  that $F : (X, \B, \sigma) \to (Y, \mc A, \tau)$ is a 
\emph{factor map} \index{factor map} if $F$ is measurable with respect 
to the Borel 
sigma-algebras, and $F\circ \sigma = \tau \circ F$. Then  $\tau$ is 
called a \emph{factor} \index{factor} of $\sigma$. If $\tau$ is a Borel 
\textit{automorphism}, 
then the dynamical system $(Y, \mc A, \tau)$ is called an 
\emph{automorphic factor}. \index{automorphic factor} 
These definition can be obviously reformulated in 
the context of measurable dynamical systems when $\sigma$ and $
\tau$ are non-singular (or measure preserving) maps.

Suppose $\zeta $ is a measurable partition of $(X, \B, \mu)$.  Then 
we can define  the quotient space
$$
(Y, \B_\zeta, \mu_\zeta) = (X/\zeta, \B/\zeta, \mu/\zeta)
$$
(see Subection \ref{subsect meas partitions}). 
Let $\phi : X \to Y$ be the natural projection. If, additionally, the 
partition $\zeta$
is invariant with respect to $\sigma$, i.e., $\sigma^{-1} \zeta \preceq 
\zeta$, then $\sigma$ defines an onto endomorphism $\wt \sigma$ 
of $Y$ such that $\phi$ is a factor map: $\phi\sigma = 
\wt\sigma\phi$.

To define an isometry generated by a surjective endomorphism 
$\sigma$, we assume that  $\sigma$ is a finite measure-preserving 
endomorphism of a standard measure space $\sms$, and $\mu \circ 
\sigma^{-1} = \mu$. The assumption about the invariance of $\mu$ is 
not crucial and is made for convenience.  The definition can be
easily  modified to  the case of non-singular endomorphisms.

\begin{theorem}\label{thm S isometry} Let $(X, \B, \mu, \sigma)$ be a measure preserving 
non-invertible dynamical system.  Let  $\mc H = L^2(\mu)$ and define
$$
S : f \mapsto f\circ\sigma : \mc H \to \mc H.
$$
Then $S$ is an isometry.  The adjoint of $S$ is
$$
S^* g = \frac{(g d\mu)\circ\sigma^{-1}}{d\mu}, \ \ g \in \mc H.
$$
\end{theorem}

\begin{proof}
The fact that $S$ is isometry follows from $\sigma$-invariance of $
\mu$. The formula for $S^*$ is deduced as follows:
\begin{equation*}
\setlength{\jot}{10pt}
\begin{aligned}
   \langle S f, \ g\rangle  & =  \int_X (f\circ \sigma) g\; d\mu   \\
   &= \int_X f (g d\mu)\circ\sigma^{-1}  \\
   &= \int_X f \frac{(g d\mu)\circ\sigma^{-1}}{d\mu}\ d \mu \\
   & =   \langle f,\ S^*g \rangle. 
\end{aligned}
\end{equation*}
As was mentioned in Section \ref{sect L1 and 
L2},  the co-isometry $S^*$ is, in fact, a transfer operator $R$ 
corresponding to the endomorphism $\sigma$.
\end{proof}

It follows from this lemma that we can apply the Wold theorem for $S$ 
and construct an orthogonal decomposition of $\mc H = L^2(\mu)$.
It says that $\mc H$ can be decomposed into the orthogonal sum 
$\mc H_\infty \oplus \mc H_\infty^\bot$ where $S$ restricted on 
$\mc H_\infty$ is a unitary operator.
It turns out that  the subspace $\mc H_\infty$ can be  explicitly 
described in terms of the endomorphism $\sigma$.

We recall that, to every endomorphism $\sigma$ of a Borel space 
$(X, \B)$, one can 
associate two partitions of $X$.  Let $\xi$ be the partition of $X$ into 
the $\sigma$-\emph{orbits}, \index{orbit} i.e.,
$\xi = \{ Orb_\sigma(x) : x \in X\}$, where $y \in    Orb_\sigma(x) $ 
if and only if there exist $m, n \in \N$ such that $\sigma^m(x) = 
\sigma^n (y)$. Define also a partition $\eta$  of $X$  into equivalence  
classes such that $x \sim y$ if and only if $\sigma^n(x) = \sigma^n(y)
$ for some $n \in \N$. Then
$$
\eta(x) = \bigcup_n \sigma^{-n}(\sigma^n(x)).
$$
If $\sigma$ is an at most countable-to-one endomorphism, then the 
partitions $\xi$ and $\eta$ define hyperfinite countable Borel  
equivalence relations (see \cite{DoughertyJacksonKechris1994} for 
detail).
Clearly, $\eta$-equivalence classes refine  $\xi$-equivalence classes. 

Suppose that $\mu\in M_1(X)$ is a $\sigma$-invariant measure, so 
that $\sigma$ is considered as a measure preserving 
endomorphism of $\sms$. We denote by $\xi'$ and $\eta'$ the 
measurable hulls of the partitions $\xi$ and $\eta$, respectively.

It is worth noting that there exists a one-to-one correspondence 
between  measurable partitions and complete sigma-subalgebras $\mc 
A$ of $\B$. Let $\mathcal A(\xi')$ and $\mc A(\eta')$ be the 
subalgebras corresponding to $\xi'$ and $\eta'$.

For $(X, \B, \mu, \sigma)$ as above, define
$$
\B_\infty = \bigcap_{n =0}^\infty \sigma^{-n}(\B),
$$
and let $\mc A_\sigma = \{A \in \B : \sigma^{-1}(A) = A\}$ be the 
subalgebra of $\sigma$-invariant subsets of $X$. We recall that $
\sigma$ is called \emph{exact} if $\B_\infty$ is a trivial subalgebra, 
and $\sigma$ is called \emph{ergodic} if $\mc A_\sigma$ is trivial. 
Since $\mc A_\sigma \subset \B_\infty$, exactness implies 
ergodicity.

If $\epsilon$ denotes the partition of $X$ into points, then we have
the sequence of decreasing measurable partitions $\{\sigma^{-i}
\epsilon\}_{i=0}^\infty$:
$$
\epsilon \succeq \sigma^{-1}\epsilon \succeq \sigma^{-2}\epsilon 
\cdots.
$$

The objects, we have defined above, satisfy the following properties.

\begin{lemma} [\cite{Rohlin1961}]\label{lem from Rohlin on exact} In 
the above notation, we have:

(1)
 $$
\xi' \preceq \eta', \qquad \ \ \ \eta' = \bigwedge_n \sigma^{-n}
\epsilon;
$$

(2)
$$
\mc A(\xi') = \mc A_\sigma, \qquad \ \ \ \mc A(\eta') = \B_\infty.
$$
\end{lemma}

In particular, it follows from Lemma \ref{lem from Rohlin on exact} that 
an endomorphism $\sigma$ is ergodic if the partition $\xi'$ is trivial, 
and $\sigma$ is exact if the partition $\eta'$ is trivial (understood in 
terms of  $\mbox{mod}\ 0$ convention).

Since $\eta' $ is a measurable partition, we can define the quotient 
measure space
$(X/\eta', \B/\eta', \mu_{\eta'})$. By Lemma \ref{lem from Rohlin on 
exact}, we see that  $\B/\eta' = \B_\infty$ and
$$
Y := X_\eta' = X_{\bigwedge_n \sigma^{-n}\epsilon}.
$$

\begin{corollary}
(1) Let $\phi : X \to Y$ be the natural projection. Then there exists a 
measure preserving  automorphism $\wt \sigma : (Y, \mu_{\eta'}) \to 
(Y, \mu_{\eta'})$ such that $\wt \sigma$  is an automorphic  factor 
of $\sigma$, i.e.,
$$
\wt \sigma \circ \phi = \phi \circ \sigma.
$$

(2) Let $S : f \to f\circ \sigma$ be the isometry on $\mc H =
 L^2(\mu)$. Then, in the Wold 
decomposition $\mc H = \mc H_\infty \oplus \mc H_\infty^\bot$
for $S$, we have
$$
\mc H_\infty = L^2(Y, \mu_{\eta'}),
$$
and the restriction of $S$ to $\mc H_\infty$ corresponds to the 
unitary operator $U$ defined by $\wt\sigma$, $U(f) = f\circ \wt
\sigma$.
\end{corollary}

\begin{remark} \label{rem representation f is F sigma} Let $\sigma$ 
be an endomorphism of a standard measure space $\sms$ as above. 
Then it follows from the construction of $\B_\infty$ and from the 
definition 
of the partition $\eta'$ that for every $\B_\infty$-measurable 
function $f$ there exists a sequence of functions $(F_n)$ such that 
every $F_n$ is $\B$-measurable and, for every $n \in \N$, 
\be\label{eq f F_n sigma Wold}
f = F_n \circ\sigma^n.
\ee
With some abuse of notation, this relation can be also written as
$$
\mc M_n(X) = \mc M(X) \circ \sigma^n,
$$
where $\mc M_n(X)$ denotes the space of 
$\sigma^{-n}(\B)$-measurable functions.

Moreover, one can show that a $\B_\infty$-measurable function $f$ 
admits a representation $f = F_n\circ \sigma^n$ for every $n \in \N$ 
if and only if $f$ is a constant function on every class of the 
measurable equivalence relation $\eta'$.
\end{remark}

We consider now an application of  relation (\ref{eq f F_n sigma 
Wold}) to transfer operators $R$ defined on $(X, \B, \mu)$ by an 
onto endomorphism $\sigma$.
Suppose that $\sigma$ is not exact, i.e., the subalgebra $\B_\infty$ is 
not trivial. By Remark \ref{rem representation f is F sigma} and relation 
(\ref{eq f F_n sigma Wold}), we can see that, for any $f \in \mc 
M(\B_\infty)$,
$$
R^n(f) = F_n \omega_n, \ \ \ \omega_n := R^n(\mathbf 1).
$$
This fact can be easily proved by induction. Furthermore, since $\mu$ 
is $\sigma$-invariant, one can show that
$$
F_n = (S^*)^n f,
$$
where $S : f \mapsto f\circ \sigma$ is the isometry considered above. 
We leave the details to the reader.

\begin{remark}\label{rem nonsingular sigma in Wald sect}
In this section we have considered the case of measure preserving 
endomorphism $\sigma$. But the proved results are still true 
(mutatis mutandis) in the 
case when $\mu$ is non-singular with respect to $\sigma$. The 
standard method of dealing with non-singular transformations  is 
as follows.

Let $\theta_\mu$ be the Radon-Nikodym derivative, i.e. $\theta_\mu 
= \dfrac{d\mu\circ \sigma^{-1}}{d\mu}$. Then
$$
S : f \mapsto \sqrt{\theta_\mu} (f\circ\sigma) ,\ \ \ f \in L^2(\mu)
$$
is an isometry in $\mc H = L^2(\mu)$. Hence one can use the 
arguments developed
above in this section for the study of the operator $S$. In particular, 
the adjoint of $S$ can be determined by formula
$$
S^* g = \frac{(\sqrt{\theta_\mu} g d\mu)\circ \sigma^{-1}}{d\mu}.
$$

\end{remark}

\newpage

\section{Operators on the universal Hilbert space generated by 
transfer operators}\label{sect Universal HS}

Starting with a fixed transfer operator $(R, \sigma)$ on $(X, \B)$, we 
show below that there is then a naturally induced universal Hilbert 
space $\mc H(X)$  with the property that $(R, \sigma)$  yields 
naturally a corresponding isometry in $\mc H(X)$, i.e., an isometry 
with respect to the inner product from $\mc H(X)$. With this, we then 
obtain a rich spectral theory for the transfer operator at hand, for 
example a setting which may be considered to be an 
infinite-dimensional Perron-Frobenius theory. Our main results are 
Theorems \ref{lem H(la) is S-inv},  \ref{thm K_1(R) tfae}, and 
\ref{thm R is adjoint of S}.

\subsection{Definition of the universal Hilbert space $\mc H(X)$} For 
the reader's convenience, we recall the definition of the \textit{universal 
Hilbert space} \index{universal 
Hilbert space} $\mc H(X)$ where $(X, \B)$ is a standard Borel space as 
usual. We will use \cite{Nelson1969} as a main source. Our analysis
of transfer operators in $\mc H(X)$ is motivated by 
\cite{AlpayJorgensenLewkowicz2016,
 Jorgensen2001, Jorgensen2004}.

Let  $M(X)$ be the set of all Borel measures on $X$.
We write $(f, \mu)$ for a pair consisting of a real-valued function $f 
\in L^2(\mu)$ and a measure $\mu \in M(X)$.

\begin{definition}\label{def pair equivalence}
It is said that two pairs $(f, \mu)$ and $(g, \nu)$ are \index{equivalent 
pairs} \emph{equivalent} if there exists a measure $\la \in M(X)$ such that 
$\mu \ll \la$ and $\nu \ll \la$, and
\be\label{eq pair equivalence}
f\sqrt{\frac{d\mu}{d\la}} = g\sqrt{\frac{d\nu}{d\la}}, \ \ \ \la
\mbox{-a.e.}
\ee

The set of equivalence classes of pairs $(f, \mu)$ is denoted by $\mc 
H(X)$.
\end{definition}

It is not hard to show that, if relation (\ref{eq pair equivalence}) holds 
for some $\la$, then
$$
f\sqrt{\frac{d\mu}{d\la'}} = g\sqrt{\frac{d\nu}{d\la'}}, \ \ \  
\la'\mbox{-a.e.},
$$
for any measure $\la'$ such that $\mu \ll \la'$ and $\nu \ll \la'$  
\cite{Nelson1969}. From this observation, one can conclude  that  
(\ref{eq pair equivalence}) defines an \emph{equivalence relation} on 
the set of pairs $(f, \mu)$. We will denote the equivalence class of a 
pair $(f, \mu)$ by $f \sqrt{d\mu}$.

\begin{remark} \label{rem measure change 1}
(1) We mention an important case of equivalence of two pairs, $(f, 
\mu)$ and $(f', \mu')$. Suppose that $\mu' \ll \mu$ and $d\mu' = 
\varphi d\mu$. Then we can take $\la = \mu$ in Definition \ref{def 
pair equivalence}, so that 
$$
(f, \mu) \sim (f', \mu') \ \Longleftrightarrow \ f = f' \sqrt{\varphi}, \ 
\ \ \mu\mbox{-a.e.},
$$
and these pairs belong to the class $f\sqrt{d\mu}$.

(2) It follows from (1) that any pair $(f,\mu)$ is equivalent to a pair $
(f',\mu')$ with $\mu'(X) = 1$. Hence, one can assume that any 
equivalence class is defined by a probability measure.
\end{remark}

It turns out that $\mc H(X)$ can be endowed with a \textit{Hilbert 
space} structure. To see that $\mc H(X)$ is a vector space, we define 
the sum of elements from $\mc H(X)$ as follows:
$$
f \sqrt{d\mu} + g\sqrt{d\nu} = \left(f\sqrt{\frac{d\mu}{d\la}} +  g
\sqrt{\frac{d\nu}{d\la}}\right)\sqrt{d\la},
$$
where $\mu \ll \la$ and $\nu \ll \la$ for some measure $\la$. The 
definition of the multiplication by a scalar is obvious. Next, an inner 
product on $\mc H(X)$ is defined by
\be\label{eq inner product}
\langle f\sqrt{d\mu},\ g\sqrt{d\nu}\rangle_{\mc H(X)} \  = \int_X fg 
\sqrt{\frac{d\mu}{d\la}} \sqrt{\frac{d\nu}{d\la}} \; d\la 
\ee
where again $\mu \ll \la$ and $\nu \ll \la$ for a measure $\la$. It is a 
simple exercise to show that these operations are well-defined and do 
not depend on the choice of $\la$.

\begin{proposition} [\cite{Nelson1969}] With respect to the 
operations  defined above, $\mc H(X)$ is a Hilbert space.
\end{proposition}

A proof of this assertion can be found in \cite{Nelson1969} or 
\cite{Jorgensen2004}.
\medskip

We will call $\mc H(X)$ the \emph{universal Hilbert space}. \index{universal 
Hilbert space} 

It follows from the definition of the inner product in $\mc H(X)$ that 
for any element $f\sqrt{d\mu}$ of $\mc H(X)$
$$
\|f\sqrt{d\mu}\|^2_{\mc H(X)} = \int_X f^2 \; d\mu = \|f\|
^2_{L^2(\mu)}.
$$
Thus, if $\mu$ is a fixed measure on $X$, then the map
\be\label{eq iota embedding}
{\iota} : f \mapsto f \sqrt{d\mu} : L^2(\mu) \to \mc H(X)
\ee
gives an \emph{isometric embedding} of $L^2(\mu)$ into $\mathcal 
H(X)$.

We denote $\mc H(\mu) := \iota (L^2(\mu))$. The following 
proposition explains why  $\mc H(X)$ is called a universal Hilbert 
space.

\begin{proposition}\label{prop properties of universal H space}
For any two measures $\mu$ and $\nu$ on $(X, \B)$,

(1)  $\mu \ll \nu$ if and only if  $\mc H(\mu) $ is isometrically 
embedded into $\mc H(\nu) $.

(2) $\mu \sim \nu$ if and only if  $\mc H(\mu) = \mc H(\nu)$.

(3)  $\mu$ and $\nu$ are mutually singular if and only if  the 
subspaces $\mc H(\mu)$ and $\mc H(\nu)$ are orthogonal in $\mc 
H(X)$.
\end{proposition}

\begin{proof}
These properties are rather obvious, and can be proved directly. We 
sketch here a proof of (1) to illustrate the used technique.  Let $\psi d
\nu = d\mu$. Set
$$
T(g \sqrt{d\mu}) = g\sqrt{\psi}\sqrt{d\nu},
$$
and show that $T$ implements the isometric embedding. We have
$$
||T(g \sqrt{d\mu})||^2_{\mc H(\nu)} = \int_X g^2 \psi \; d\nu = 
\int_X g^2 \; d\mu = ||g\sqrt{d\mu}||^2_{\mc H(\mu)}.
$$
We  leave the  proof of the other statements to the reader, see details 
in \cite{Nelson1969, Jorgensen2004}.
\end{proof}

\subsection{Transfer operators on $\mc H(X)$} Suppose that we have 
a surjective endomorphism $\sigma$ of a standard Borel space $(X, 
\B)$, and let a transfer operator $(R, \sigma)$ be defined on Borel 
functions on $(X, \B)$. In our further considerations, we will work with  
the transfer operator $R$ acting in the space  $L^2(\lambda)$ where  
$\la$ is a measure from $M(X)$. For given $R$, we divide measures 
into two subsets, $ \mc L(R)$ and $M(X) \setminus \mc L(R)$.

We recall that a measure $\la$ is called \emph{atomic} if there exists 
\index{measure ! atomic} \index{measure ! continuous}
a point in $X$ of positive measure. Non-atomic measures are called 
\emph{continuous}. Let $M_c(X)$ and $M_a(X)$ denote the subsets 
of $M(X)$ formed by continuous and purely atomic measures, 
respectively.  Dealing with vectors $f\sqrt{d\la}$ in the space $\mc 
H(X)$, we distinguish two principal cases: (i) the measure $\la$ is 
continuous or (ii) the measure $\la$ is purely atomic.

 We discuss in the  following statements some  properties of the 
 universal Hilbert space. Note that every measure $\mu $ can be 
 viewed as a vector $\sqrt{d\mu}$ in the space $\mc H(X)$. 

\begin{lemma}\label{lem prop of orthog decompos for universal}
 Let $\mc G$ be a subset of $M(X)$. Denote by $\mc H_{\mc G}$ 
the closure of the subspace spanned by $\mc H(\la), \la \in \mc G$.
Hence,  the universal Hilbert space admits the following orthogonal 
decomposition:
$$
\mc H(X) = \mc H_{\mc G} \oplus (\mc H_{\mc G})^\bot.
$$

Furthermore,  $\mc H_{\mc G}^\bot$ is spanned  by all $\mc H(\nu)$
such that $\nu$ is singular to all measures $\mu$ from $\mc G$. 
\end{lemma}

Lemma \ref{lem prop of orthog decompos for universal} follows 
immediately from Proposition \ref{prop properties of universal H 
space}.

\begin{remark}\label{rem choice of measure in a class}
(1) As mentioned in Remark \ref{rem measure change 1}, we can 
always assume that $\mu$ is a probability measure. Together with the 
assumption that $R(\mathbf 1) =\mathbf 1$, this means that, for any 
probability measure $\mu$, the  measure $\mu R$ is well defined. 
This fact will be repeatedly used below.

(2) For  a specific choice of the set $\mc G$ in Lemma \ref{lem prop 
of orthog decompos  for universal},  we can get  the following  useful
 decompositions of $\mc H(X)$:
\be\label{eq repr in cont and atomic}
\mc H(X) = \mc H_{M_c(X)} \oplus \mc H_{M_a(X)}, \ \ \ 
\mc H(X) = \mc H_{K_1} \oplus (\mc H_{K_1})^\bot, \ \ \ 
\ee
\be\label{eq repr H(X) via L(R)}
\mc H(X) = \mc H_{\mc L(R)} \oplus (\mc H_{\mc L(R)})^\bot
\ee
where $K_1 = M_1(X) R$.

(3) Given  a \emph{nonzero} vector $f\sqrt{d\mu}$ in $\mc H(X)$ 
with a continuous (atomic) measure $\mu$, we remark that the 
 class of equivalent pairs generated by $(f, \mu)$  contains only pairs 
 $(g, \la)$ where $\la$ is a continuous (atomic) measure. This 
 follows  from the following obvious fact: if $\la \ll \nu$ and $\la$ is 
 atomic at  a  point $x_0$, then $\nu$ is atomic at the same point.
This means that  $f\sqrt{d\mu} \in \mc H_{M_c(X)}$ if and only if $
\mu$ is continuous, and $f\sqrt{d\mu} \in \mc H_{M_a(X)}$ if and 
only if $\mu$ is purely atomic. This means   that 
the decomposition (\ref{eq repr in cont and atomic}) is invariant with 
respect to the equivalence of pairs $(f, \mu)$. But the decomposition 
in  (\ref{eq repr H(X) via L(R)}) is not invariant with respect to this 
equivalence relation. 
\end{remark}
\medskip

In what follows we will translate the notion of a transfer operator 
$R$ and its adjoint operator $S$, which are studied in $L^2(\la)$ in 
Section 
\ref{sect L1 and L2},  to the subspace $\mc H(\la)$ of the universal 
Hilbert space $\mc H(X)$. But, in contrast to the pair $(R, S)$, we 
begin with an operator $\wh S$ and show that its adjoint $\wh S^* =
\wh R$ is an analogue of a transfer operator. 
 Our approach is mainly  based on  the application of 
Proposition \ref{prop properties of universal H space} which allows us 
to work with classes of equivalent measures.

In this section, we will deal with a pair of operators $(\wh R, \wh S)$ 
acting in $\mc H(X)$ that are considered as analogous one to the 
symmetric pair of operators $(R, S)$ studied in Section \ref{sect L1 
and L2} where $R$ is a transfer operator obtained as adjoint to the 
composition operator $S$. We first outline our approach to their 
definition. We recall that our main assumption in this context is 
that all considered transfer operators are normalized, 
$R(\mathbf 1) =\mathbf 1$. 

We  define an operator $\wh S$ that acts in the 
set $\mc P$ of all pairs $(f, \mu)$ where $f \in L^2(\mu)$ and $\mu 
\in M_1(X)$. It will be checked that $\wh S$ preserves the partition of 
$\mc P$ into equivalence classes. Therefore this fact allows us to 
consider $\wh S$ as an 
operator acting in $\mc H(X)$. In the next step, we will check that $
\wh S$ is an isometry that leaves every subspace $\mc H(\la)$   
invariant. Hence, the adjoint operator $\wh R = \wh S^*$ exists and is 
a co-isometry. We 
note that it is unclear whether $\wh R$ can be defined directly in 
terms of a transformation on the set $\mc P$ that  preserves the 
equivalence relation on the set of pairs $(f,\la)$. Meantime, there
 exists a particular case when it can be done explicitly and this case 
 will be  studied carefully.

Given a vector $f\sqrt{d\la}$ with continuous measure $\la \in M_c(X)
$, we are going to work with measures $\la R$ and $\la \circ 
\sigma^{-1}$.   We can do it by  virtue of Remark \ref{rem choice of 
measure in a class} (1). In other words, when we deal  with pairs $(f, 
\la)$, we can think that the actions of $R$ and $\sigma$ on the set of 
measures $M_1(X)$ are well defined everywhere. As was explained in 
Remark \ref{rem choice of measure in a class} (3), we can consider the
 two cases of continuous and purely atomic measures independently 
 due to the invariance of the decomposition (\ref{eq repr in cont and 
 atomic}).

As discussed in Section \ref{sect TO on 
measurable spaces}, the assumption that $R(\mathbf 1)
 =\mathbf 1$ is 
not restrictive if a non-trivial harmonic function exists.  On the other 
hand, this property is automatically  true for a wide class of transfer 
operators acting in $L^2(\la)$ for any measure $\la$. 

We begin with the following lemma which is used repeatedly 
below. 

\begin{lemma}\label{lem R 1= 1 - actions on measures}
Let $(R, \sigma)$ be  a transfer operator on Borel functions over $(X,
\B)$. Then
$$
R (\mathbf 1)=  \mathbf 1 \ \ \ \Longleftrightarrow \ \ \ (\mu\circ R)
\sigma^{-1} = \mu \ \ \forall \mu \in M_1(X).
$$
\end{lemma}

This result immediately follows from the the relation
$$
R(\mathbf 1) = \frac{d(\mu\circ R)\sigma^{-1}}{d\mu}
$$
that was proved in Section \ref{sect L1 and L2}, see (\ref{eq R1 R-N 
derivative}).

\begin{definition} Let $\la$ be a continuous probability  measure on $
(X, \B)$,  $R(\mathbf 1) =\mathbf 1$, and $f\in L^2(\la)$. Then we 
define, for any  pair $(f, \la)$,
\be\label{eq def of wh S}
\wh S (f, \la) = (f \circ \sigma, \la R).
\ee
\end{definition}

We first  show that the operator $\wh S$ induces an operator on  the
 space $\mc H(X)$. This fact follows from the following lemma.

\begin{lemma}\label{lem R is defined on classes} Let $f \in L^2(\la)$ 
and $f_1 \in L^2(\la_1)$ where $\la$ and  $\la_1$ are continuous 
probability measures. Then
$$
(f_1, \la_1) \sim (f, \la) \ \Longleftrightarrow \ (f_1\circ \sigma,  
\la_1 R) \sim (f\circ \sigma, \la R)
$$
\end{lemma}

\begin{proof}
By definition of the equivalence relation $\sim$ on the set 
$\mc P$, two  pairs $(f, \la)$ and $(f_1, \la_1)$ are in the same class
if and only if there exists a measure $\mu$ such that $\la \ll \mu$, $
\la_1 \ll \mu$, and
\be\label{eq equiv pairs}
f\sqrt{\frac{d\la}{d\mu}} = f_1\sqrt{\frac{d\la_1}{d\mu}}, \quad 
\mu\mbox{-a.e.}
\ee
In particular, $\mu$ can be chosen as the sum $\la + \la_1$. Then  
$\la R \ll \mu R$ and $ \la_1 R \ll \mu R$ (see Section \ref{sect 
L1 and L2}). It follows from Theorem \ref{thm from measure 
equivalence} (1) that
\be\label{eq add R to la and la_1}
\frac{d\la}{d\mu} \circ \sigma = \frac{d(\la R)}{d(\mu R)}, \quad 
\quad
\frac{d\la_1}{d\mu} \circ \sigma = \frac{d(\la_1 R)}{d(\mu R)}.
\ee
Hence,  we can apply (\ref{eq equiv pairs}), (\ref{eq add R to la and 
la_1})  and conclude that
\begin{eqnarray*}
  (f\circ \sigma) \sqrt{\frac{d(\la R)}{d(\mu R)}}  &=& (f\circ \sigma) 
  \sqrt{\frac{d\la}{d\mu} \circ \sigma}   \\
  \\
   &=& (f_1\circ \sigma)\sqrt{\frac{d\la_1}{d\mu}\circ \sigma}\\
   \\
   &=& (f_1\circ \sigma)\sqrt{\frac{d(\la_1 R)}{d(\mu R)}} \qquad 
   (\mu R)\mbox{-a.e.}
\end{eqnarray*}
This proves that $(f_1\circ \sigma, \la_1 R) \sim (f\circ \sigma, \la R)
$.

Conversely, if we have the fact that the pairs $(f_1\circ \sigma, \la_1
 R)$ and   $(f\circ \sigma, \la R)$ are equivalent, then
$$
 (f\circ \sigma) \sqrt{\frac{d(\la R)}{d(\mu R)}} = (f_1\circ \sigma)
 \sqrt{\frac{d(\la_1 R)}{d(\mu R)}}, \qquad (\mu R)\mbox{-a.e.}
$$
Hence, we can apply the transfer operator $R$ to the both sides of 
this relation, and because $\mu R \ll \mu$, we obtain (\ref{eq equiv 
pairs}).
\end{proof}

\begin{remark}
 If a measure $\la$ is in the set $\mc L(R)$, then, for 
 some measurable function $W$, we have  $d(\la R) = Wd\la$, 
Then the operator $\wh S$ acts in the subspace $\mc H(\la)$ as
follows: 
\be\label{eq modified R 6}
\wh S (f\sqrt{d\la}) = (f\circ\sigma)\sqrt{W} \sqrt{d\la}.
\ee
In order to justify (\ref{eq modified R 6}), we observe that if $(f,\la) 
\sim (f,\la_1)$ with $\la, \la_1 \in \mc L(R)$, then
$$
(\sqrt{W} (f\circ\sigma), \la ) \sim (\sqrt{W_1} (f_1\circ\sigma), 
\la_1)
$$
where $W_1d\la_1 = d(\la_1R)$. This equivalence can be directly 
deduced from the relation $W_1 = (\varphi\circ\sigma) W 
\varphi^{-1}$ where $\varphi d\la = d\la_1$ that was discussed in 
Section \ref{sect L1 and L2}.

\end{remark}

\begin{lemma}\label{thm isometry}
The operator
\be\label{eq wh S on H(X)}
\wh S (f\sqrt{d\la}) = (f\circ\sigma)\sqrt{d(\la R)},
\ee
is well defined in the universal Hilbert space $\mc H(X)$. Furthermore, 
$\wh S$ is bounded if and only if $R(\mathbf 1) \in L^{\infty}(\la)$. 
\end{lemma}

We  use the same notation $\wh S$ for the operators acting on the 
set of pairs $(f, \la)$ and in the Hilbert space $\mc H(X)$. It will be 
clear from the context where $\wh S$ acts.

\begin{proof}
We first need to justify the correctness of  the definition $\wh S$. 
Indeed, this result follows from Lemma  \ref{lem R is defined on 
classes} because if we take any two pairs  $(f, \la)$,  $(f_1, \la_1)$
that belong to the same class, then $\wh S$ maps these pairs
into equivalent pairs $(f \cs, \la R)$ and $(f_1\cs, \la_1 R)$. Hence 
relation ( \ref{eq wh S on H(X)}) defines a transformation in 
$\mc H(X)$. 

To see that $\wh S$ is a linear operator we have to check that $\wh 
S(c_1f\sqrt{d\la} + c_2f_1\sqrt{d\la_1}) = c_1\wh S(f\sqrt{d\la}) + 
c_2\wh S(f_1\sqrt{d\la_1})$. This can be proved again by the choice 
of representatives in the classes $f\sqrt{d\la}$ and $f_1\sqrt{d\la_1}
$ as we did above. The details are left to the reader.
\end{proof}

We recall that $\mc H(\la)$ denotes the subspace of $\mc H(X)$ 
obtained by the isometric  embedding of $L^2(\la)$ into $\mc H(X)$.

\begin{theorem}\label{lem H(la) is S-inv} Let $(R,\sigma)$ be a 
transfer operator on $(X, \B)$.  Then the operator  $\wh 
S$ of $\mc H(X)$ is an isometry if and only if  $R(\mathbf 1) =
\mathbf 1$. Moreover, 
if a measure $\la \in \mc L(R)$, then the subspace $\mc H(\la)$  
is invariant with respect to $\wh S$.
\end{theorem}

\begin{proof}
To see that $\wh S$ is an isometry, we use (\ref{eq inner product}) 
and  calculate
\begin{eqnarray*}
||\wh S (f\sqrt{d\la})||^2_{\mc H(X)} &= & \int_X (f\circ \sigma)^2 
\; d(\la R)\\
\\
& = & \int_X R[(f\circ \sigma)^2] \; d\la \\
\\
& = & \int_X f^2 R(\mathbf 1) \; d\la.
\end{eqnarray*}
Hence, we see that  
$$
||\wh S (f\sqrt{d\la})||^2_{\mc H(X)} =  ||f||^2_{\mc H(X)}
$$
if and only if $R(\mathbf 1) = \mathbf 1$.

To prove the second part of the theorem, we suppose  that $\la\in
 \mc L(R)$, then $d(\la R) = W d\la$. Take 
any element  $f\sqrt{d\mu} \in \mc H(\la)$. By Proposition \ref{prop  
properties of universal H space}, this means that $\mu \sim \la$, 
$d\mu = \va d\la$, and $(f, \mu)$ is equivalent to $(g, \la)$. 
Then $d(\mu R) = (\va\cs) W \va^{-1}d\mu$ 
(see Section \ref{sect L1 and L2}). It follows from this fact that
$$
\wh S (f\sqrt {d\mu}) = (f\cs) \sqrt {d(\mu R)} = [\sqrt{W} 
(f\va)\cs]\sqrt{d\la}. 
$$ 
It follows from the definition of equivalence of pairs $(f,\mu)$ and
$(g, \la)$ that 
$$
\sqrt{W} (f\va)\cs] \in L^2(\la).
$$
 Hence 
$\wh S : \mc H(\la) \to \mc H(\la)$, and the theorem is proved.
\end{proof}

We can  immediately deduce from Theorem \ref{lem H(la) is S-inv} 
several important properties of $\wh S$ and its adjoint.
We recall the notation $K_1 := M_1 R$ that was used in Section 
\ref{sect actions on measures}. Then the subspace $\mc H_{K_1}$ is 
spanned   by $\{\mc H(\la) :  \la \in K_1\}$.  

\begin{corollary}\label{cor kernel adjoint to S}
 (1)  The decomposition $\mc H(X) = 
\mc H_{K_1} \oplus (\mc H_{K_1})^\bot$ implies that 
$\wh S(\mc H(X)) = \mc H_{K_1}$.

(2) The adjoint operator $\wh S^* : \mc H(X) \to \mc H(X)$ 
is well defined and $Ker(\wh S^*) =  (\mc H_{K_1})^\bot$.   
\end{corollary}

We remark that  the adjoint operator $\wh S^*$ is defined 
in terms of the Hilbert  space $\mc H(X)$,  in contrast to the case of 
$\wh S$ where we first defined 
$\wh S$ on the set of pairs $(f, \la)$ and then 
extended to the classes of equivalence that form the
Hilbert space $\mc H(X)$. 

The next result gives an explicit formula for the action of 
$\wh S^*$  when $\la R \ll \la$. We recall that
with this assumption  $\wh S^*$ leaves the subspace $\mc H(\la)$ 
invariant. 

\begin{proposition}\label{cor adjoint of wh S} Let $(R, \sigma)$
 be a normalized transfer operator and $\la \in \mc L(R)$.  Then the
 adjoint operator $\wh S^*$ acts on $\mc H(\la)$ by the formula:
\be\label{eq adjoint of S}
\wh S^* (f\sqrt{d\la}) = R\left(\frac{f}{\sqrt W}\right) \sqrt{d\la}
\ee
where $Wd\la = d(\la R)$.
\end{proposition}

\begin{proof} The result is proved by the following calculation: 
\begin{eqnarray*}
\langle \wh S(f\sqrt{d\la}), g\sqrt{d\la}\rangle_{\mc H(\la)} & =&  
\int_X \sqrt W (f\circ \sigma)g \; d\la \qquad \mbox{(see\ Remark\ 
\ref{rem measure change 1})}\\
\\
&= & \int_X (f\circ \sigma) g \frac{1}{\sqrt W}\; d(\la R)\\
\\
& = & \int_X R\left( (f\circ \sigma)g \frac{1}{\sqrt W}\right)\; d\la \\
\\
& = & \int_X  f R\left( \frac{g}{\sqrt W}\right)\; d\la\\
\\
& = & \langle f\sqrt{d\la}, R( g W^{-1/2})\rangle_{\mc H(\la)}
\end{eqnarray*}
and the proof is complete.
\end{proof}

\begin{corollary} Let $\la \in \mc L(R)$. In the notation of Proposition 
\ref{cor adjoint of
 wh S}, $\wh S\wh S^*$ is the projection in the space $\mc H(\la)$
 which acts by the formula:
$$
\wh S \wh S^* (f\sqrt{d\la}) = 
[R(\frac{f}{\sqrt W})\circ \sigma]\sqrt W \sqrt{d\la}.
$$
\end{corollary}

\begin{proof}
This relation is proved by direct application of (\ref{eq wh S on H(X)}) 
and (\ref{eq adjoint of S}).
\end{proof}

We return to the question about an explicit definition of the adjoint
 operator $\wh S^*$. The key point is that the range of the isometry
 $\wh S$ is the subspace $\mc H_{K_1}$, so that the kernel
 of $\wh S^*$  must be $(\mc H_{K_1})^\bot$. In other words,
  $\wh S^* (f \sqrt{d\la})  =0$ if $\sqrt{\la} \in 
  (\mc H_{K_1})^\bot  $ according to  Corollary  \ref{cor kernel 
  adjoint to S}. Here $\sqrt{\la}$ is considered as a vector in $\mc 
  H(X)$.

To describe the action of $\wh S^*$, we define an operator
 $\wh R$ in the Hilbert space 
$\mc H(X)$ that is generated by the transfer operator  $R$.

\begin{definition}\label{def of wh R} Let $(R, \sigma)$ be a 
normalized transfer operator. We set, for any $f\sqrt{d\la}$, 
$$
\wh R(f\sqrt{d\la}) = \begin{cases}
R(f) \sqrt{d(\la\circ\sigma^{-1})}, \  & \sqrt{\la} \in \mc H_{K_1}\\
\\
0,  & \sqrt{\la}  \in (\mc H_{K_1})^\bot.\\
\end{cases}
$$
\end{definition}

Because $\mc H(X) = \mc H_{K_1} \oplus \mc (H_{K_1})^\bot$, 
the operator  $\wh R$ is  well-defined  in $\mc H(X)$.
With some abuse of notation, we will equally use the relation 
 $\la \in K_1$ in the same meaning as $\sqrt{\la} \in \mc H_{K_1}$.
 
We remark that if $\la \circ\sigma^{-1} = \la$, then the operator 
$\wh R$ sends $f\sqrt{d\la}$ to $R(f) \sqrt{d\la}$, and it can be
 identified  with the transfer operator in $R$ acting in $L^2(\la)$. 
 
 The following theorem is complimentary to the results obtained in 
 Section \ref{sect actions on measures}. This theorem clarifies the 
 role of the subset $K_1 \subset M_(X)$.
 
\begin{theorem}\label{thm K_1(R) tfae} For a normalized transfer 
operator $(R, \sigma)$, the following statements are equivalent:

(1) $\la \in K_1$;

(2) $(\la\csi1)R = \la$;

(3) the map $f \mapsto R(f) \cs |_{L^2(\la)} = \mathbb E_\la(f \ | 
\sigma^{-1}(\B))$ where $\mathbb E_\la ( \cdot \ | \sigma^{-1}(\B))
$ is the conditional expectation \index{conditional expectation} 
on the subalgebra of 
$\sigma^{-1}(\B)$-measurable functions in $L^2(\la)$;

(4)  the operator  $\wh E_1 = \wh S \wh R$ maps $\mc H(\la)$ into 
itself, and
$$
\wh E_1(f\sqrt{d\la}) = \mathbb E_\la (f\; | \; \sigma^{-1}(\B))
\sqrt{d\la}.
$$
\end{theorem}

\begin{proof}
The equivalence of statements (1) and (2) was proved in Theorem 
\ref{thm t_R is 1-1 and more}. Moreover, these two assertions are
equivalent to the fact that the equation $\nu R = \la$ has 
a unique solution for every fixed $\la \in K_1$.

Suppose now that (1) and/or (2) hold. To show that (3) is true, 
we observe that the 
operator $P_\la = f \mapsto R(f) \cs |_{L^2(\la)}$ is obviously a 
projection in $L^2(\la)$ since $P_\la^2 = P_\la$ and $P_\la (g \cs) = 
g\cs$. It remains to show that $P_\la = P_\la^*$ or
$$
\langle P_\la f_1, f_2 \rangle_{L^2(\la)} = \langle   f_1, P_\la f_2 
\rangle_{L^2(\la)}.
$$
To see this, we compute, using that $\la = \nu R$,
\begin{eqnarray}\label{eq Hermitian}
\nonumber
 \langle P_\la f_1, f_2 \rangle_{L^2(\la)}   &=& \int_X (R(f_1)\cs ) 
 f_2\; d\la  \\
 \nonumber
 \\
\nonumber
   &=& \int_X (R(f_1)\cs ) f_2\; d(\nu R) \\
   \\
   \nonumber
   &=& \int_X R[(R(f_1)\cs ) f_2]\; d\nu \\
   \nonumber
   \\
\nonumber
&=& \int_X R(f_1) R(f_2)\; d\nu.
\end{eqnarray}
By symmetry, we see that relation (\ref{eq Hermitian}) gives also
$$
\langle   f_1, P_\la f_2 \rangle_{L^2(\la)} = \int_X R(f_1) R(f_2)\; d
\nu.
$$
Thus, $P_\la$ is self-adjoint. We conclude that $P_\la$ is the
conditional expectation $\mathbb E_\la( \cdot \ | \sigma^{-1}(\B))$.

$(3) \ \Longrightarrow \ (4)$ We apply $\wh E_1$ to a vector
$(f\sqrt{d\la}) \in \mc H(\la)$ and find
\begin{equation*}
\setlength{\jot}{10pt}
\begin{aligned}
(\wh S \wh R)(f\sqrt{d\la}) & = \wh S(R(f) \sqrt{d(\la \csi1)})\\
& = (R(f)\cs) \sqrt{d(\la \csi1)R}\\
& = (R(f)\cs) \sqrt{d\la}.
\end{aligned}
\end{equation*}
The result then follows from (3).

$(4)\ \Longrightarrow \ (1)$ The operator $\wh R$ is nonzero on the 
elements of $f \sqrt{d\la} \in \mc H(X)$ if and only if $\la$ is in 
$K_1$.

\end{proof}

\begin{theorem}\label{thm R is adjoint of S} Let $(R, \sigma)$ be 
a transfer operator such that $R(\mathbf 1) =\mathbf 1$. 
The operators $\wh R$ and $\wh S$ form a symmetric pair 
\index{symmetric pair} in $\mc 
H(X)$, that is  $\wh R = \wh S^*$. 
\end{theorem}

\begin{proof}
We need to show that 
\begin{eqnarray}\label{eq S and R symmetric pair in UHS}
\langle \wh S(f\sqrt{d\nu}), g \sqrt{d\mu}\rangle_{\mc H(X)} =
\langle f\sqrt{d\nu}, R(g) \sqrt{d(\mu \csi1)}\rangle_{\mc H(X)}
\end{eqnarray}
It suffices to assume that $\sqrt{d\mu} \in \mc H_{K_1}$ because 
for $\sqrt{d\mu} \in (\mc H_{K_1})^\bot$ the both parts of 
(\ref{eq S and R symmetric pair in UHS}) are zeros. 

Then, by Theorem \ref{thm K_1(R) tfae}, there exists a measure 
$\la$ such that $d(\nu R) \ll d(\la R)$ and $d\mu = 
d(\mu\csi1 ) R \ll d(\la R)$. We will use in the following 
computation  the formulas that were 
proved in Section \ref{sect L1 and L2}
$$
\frac{d(\nu R)}{d(\la R)} = \frac{d\nu }{d\la } \cs, 
$$
$$
 \frac{d(\mu )}{d(\la R)} =   \frac{d((\mu\csi1) R)}{d(\la R)}
 =  \frac{d\mu \csi1 }{d\la } \cs.
$$
Thus, we have
\begin{eqnarray*}
\langle \wh S(f\sqrt{d\nu}), g \sqrt{d\mu}\rangle_{\mc H(X)}  
& = & \int_X (f\cs) g \sqrt{\frac{d(\nu R)}{d(\la R)}} 
\sqrt{\frac{d(\mu)}{d(\la R)}} \; d(\la R)\\
\\
& = &  \int_X (f\cs) g \sqrt{\frac{d\nu }{d\la } \cs} 
\sqrt{\frac{d(\mu\csi1)}{d\la R} \cs} \; d(\la R)\\
\\
& = &  \int_X f R(g) \sqrt{\frac{d\nu }{d\la } } 
\sqrt{\frac{d(\mu\csi1)}{d\la R} } \; d\la \\
\\
& = &  \langle f\sqrt{d\nu},  R(g) \sqrt{d(\mu \csi1)}
\rangle_{\mc H(X)}.
\end{eqnarray*}
The proof is complete.
\end{proof}

\begin{remark} In this remark, we collect a few facts about the 
operators $\wh S$ and $\wh R$.

(1) If the transfer operator $(R, \sigma)$ is normalized, then It can
 be deduced directly from the definitions of the operators 
$\wh S$ and $\wh R$ that $\wh R \wh S = I_{\mc H(X)}$. 
Indeed, we have
\begin{equation*}
\setlength{\jot}{10pt}
\begin{aligned}
(\wh R\wh S)(f\sqrt{d\la}) &= \wh R( (f\cs )\sqrt{d(\la R)}\\
   &= R(f\cs)\sqrt{d(\la R \csi1)} \\
   &=  f\sqrt{d\la} 
\end{aligned}
\end{equation*}
where statement (2) of Theorem \ref{thm K_1(R) tfae} was used.

(2) If $R$ is not normalized, then  $\wh S$ is bounded in $\mc H(X)$ 
if and only if $R(\mathbf 1) \in L^{\infty}(\la)$ for all $\la$.

To see this, we compute
$$
||\wh S (f\sqrt{d\la})||^2_{\mc H(X)} = \int_X (f\circ\sigma)^2 \; 
d(\lacr) = \int_X f^2 R(\mathbf 1) \; d\la.
$$

(3) The following result which is similar to Theorem 
\ref{lem H(la) is S-inv} can be proved: 

Suppose $h$ is a harmonic function for the transfer operator $R$ 
acting in $L^2(\la)$. Then $\wh S$ is an isometry in $L^2(hd\la)$.

\end{remark}

The next lemma deals with non-normalized transfer operators $(R,
\sigma)$. 

\begin{lemma} Let $(R, \sigma)$ be  a transfer operator. Suppose 
the operators  $\wh R$ and $\wh S$ are defined as above. 
Then the  operator $\wh R\wh S$ is a  multiplication operator in $\mc 
H(X)$.
\end{lemma}

\begin{proof} We obtain that
$$
\wh R \wh S(f\sqrt{\la}) = \wh R((f\circ \sigma) \sqrt{\la R}) = 
(R(\mathbf 1)f \sqrt{(\la R) \csi1 \sigma^{-1}}) = 
(R(\mathbf 1)^{3/2} f  \sqrt{\la}).
$$
We used here relation (\ref{eq R1 R-N derivative}).

\end{proof}

\newpage

\section{Transfer operators with a Riesz property}

A well known theorem (Riesz) in analysis states  that every positive  linear 
functional $L$ on continuous functions is 
represented by a Borel measure. More precisely, let $X$ be  a locally 
compact Hausdorff space and $C_c(X)$ the space of continuous 
functions with compact support. Then the  Riesz theorem says 
that, for every positive linear functional $L$, there exists  a unique regular 
Borel measure $\mu$ on $X$ such that
$$
L(f)  = \int_X f\; d\mu.
$$

We are interesting in a special case of functionals $L_x$ defined on a 
functional space by the formula $L_x(f) = f(x)$.
For Borel functions $\mc F(X, \B)$ over a standard Borel space 
$(X, \B)$,  the Riesz  theorem  is not directly applicable. We introduce
in this section a class of transfer operators $R$ that have the following 
property.

\begin{definition}\label{def Riesz property}
Let $R$ be  a positive operator acting on Borel functions over a 
standard Borel space $(X, \B)$. We say that $R$ has the \emph{Riesz 
property} \index{Riesz property} if, for every $x \in X$, there exists a 
Borel measure $\mu_x$ such that
\be\label{eq def of Riesz prop}
R(f)(x) = \int_X f(y) \; d\mu_x(y), \ \ \qquad\ \ f \in \mc F(X, \B).
\ee
We call $(\mu_x)$ a \emph{Riesz family} \index{Riesz family} of measures corresponding 
to the operator $R$ with Riesz property. 
\end{definition}

In the following remark we present several  facts that 
immediately follow from this definition.

\begin{remark}\label{rem properties of R with Riesz}
(1) If $R(\mathbf 1)(x) = \mathbf 1$ for all $x\in X$, then every 
measure $\mu_x$ is probability, i.e., $\mu_x(X) =1$. In general, 
$\mu_x(X) = R(\mathbf 1)(x)$.

(2) The field of measures $x \mapsto \mu_x$ is Borel in the sense 
that, for any Borel function $f \in \mc F(X, \B)$, the function $ x
\mapsto \mu_x(f)$ is Borel. Indeed, this observation follows from 
(\ref{eq def of Riesz prop}) because $\mu_x(f) = R(f)(x)$.

(3) Given a positive operator $R$, the corresponding Riesz family  $
(\mu_x)$ is uniquely determined.
\end{remark}

Suppose $R$ is a positive operator with the Riesz property. Then any
power $R^k$ also has the Riesz property. So that we can write down 
for $f \in \mc F(X)$
$$
R^k(f) = \int_X f \; d\nu^k_x, \qquad k \in \N.
$$
On the other hand, if we  iterate relation  (\ref{eq def of Riesz prop}), 
then we obtain the following formula
$$
R^k(f)(x) = \int_X\!\cdots\!\int_X  f(y_k)\; d\mu_{y_{k-1}}(y_k)\cdots d
\mu_x(y_1).
$$
By uniqueness of the Riesz family, we conclude that
$$
d\nu_x^k = \int_X\!\cdots\!\int_X \; d\mu_{y_{k-1}}(y_k)\cdots d
\mu_x(y_1).
$$
We will also write $d\mu_x(y) = d\mu(y | x)$ and treat this measure 
as conditional one. This point of view will be used for the case when all 
measures $(\mu_x)$ are pairwise singular.

So far, we have used only the property of positivity of the operator
 $R$. From now on, we will assume that $R$ has the pull-out property,
 i.e., $R$ is a transfer operator on $\mc F(X, \B)$ corresponding 
to an onto endomorphism $\sigma$.

\begin{lemma}\label{lem mu_x for TO} Suppose that $(R, \sigma)$ is 
a transfer operator defined on $\mc F(X, \B)$ such that $R(\mathbf 
1) =  \mathbf 1$. Assume that $R$ has the Riesz property. Then, for 
the Riesz family of measures $(\mu_x)$, we have
$$
\mu_x \circ \sigma^{-1} = \delta_x, \qquad x \in X,
$$
where $\delta_x$ is the Dirac measure.

\end{lemma}

\begin{proof} Since $\delta_x(f) = f(x)$, we note that the relation $
\mu_x \circ \sigma^{-1} = \delta_x$ is equivalent to
$$
\int_X f\; d(\mu_x \circ \sigma^{-1}) = f(x), \qquad \forall f \in \mc 
F(X, \B),
$$
or, in other words, is equivalent to
$$
\int_X (f\circ \sigma)\; d\mu_x  = f(x), \qquad \forall f \in \mc 
F(X, \B).
$$
But by (\ref{eq def of Riesz prop}), we obtain
$$
\int_X (f\circ \sigma)\; d\mu_x = R(f\circ \sigma)(x) = f(x) 
R(\mathbf 1)(x) = f(x),
$$
and we are done.
\end{proof}

The following observation follows directly from this result.

\begin{corollary}\label{cor support of mu_x}
Let $(R, \sigma)$ be a transfer operator acting on  $\mc F(X, \B)$. 
Suppose that $R$ has the Riesz property and $R(\mathbf 1) = 
\mathbf 1 $. Then, for any $x\in X$,
$$
\mathrm{supp} (\{\mu_x\}) = \sigma^{-1}(x),
$$
where $(\mu_x)$ is the Riesz family of measures corresponding to 
$R$.
\end{corollary}

\begin{lemma}\label{lem 3 equivalent properties}
Let $R$ be  a transfer operator with Riesz property such that 
$R(\mathbf 1) =\mathbf 1$. Suppose that 
$$
\int_X f\; d\mu_x = R(f)(x)
$$ 
for all $x\in X$. Take a Borel measure $\la $ on $(X,\B)$. If $\mc 
H(\mu_x)$ is a subspace of the universal Hilbert space $\mc H(X)$, 
then the following statements are equivalent:

(1)
$$
\la \ll \mu_x, \qquad x \in X;
$$

(2)
$$
\mc H(\la) \hookrightarrow \mc H(\mu_x),  \qquad x \in X;
$$

(3)
$$
\int_X f \; d\la = R\left(f \frac{d\la}{d\mu_x}\right)(x).
$$
\end{lemma}

\begin{proof} The equivalence of (1) and (2) is mentioned in 
Proposition \ref{prop properties of universal H space}. The equivalence 
of these statements 
to (3) follows from the definition of the Riesz property.
\end{proof}

\begin{lemma}\label{lem multiplication property}
Let $R$ be a transfer operator with Riesz property such that $R
(\mathbf 1) = \mathbf 1$. If $(\mu_x)$ is the corresponding 
family of measures for $R$, then, for any sets $A, B \in \B$,
$$
\mu_x (\sigma^{-1}(A)\cap B ) = \delta_x(A) \mu_x(B),\qquad 
x \in X.
$$
\end{lemma}

\begin{proof}
To show this, we use Definition \ref{def Riesz property} and Lemma
\ref{lem mu_x for TO}. We compute
\begin{eqnarray*}
\mu_x(\sigma^{-1}(A)\cap B ) & = & \int_X \chi_{\sigma^{-1}(A)}
\chi_B\; d\mu_x\\
\\
 & = & R(\chi_{A}\cs \chi_B )(x)\\
 \\
 & = & \chi_A(x) R(\chi_B)(x)\\
 \\
 & = & \delta_x(A) \mu_x(B)
\end{eqnarray*}
\end{proof}

In a similar way, we can formulate a simple general criterion  for
a positive operator $R$, defined by (\ref{eq def of Riesz prop}), to 
have the pull-out property. 

\begin{lemma}
A positive operator $R$ with Riesz property is a transfer operator with
pull-out property if and only if, for any measurable functions $f, g$ 
from the domain of $R$, 
$$
\int_X (f\cs) g\; d\mu_x = f(x) \int_X g\; d\mu_x.
$$
\end{lemma} 

Let $(R, \sigma)$ be a transfer operator on Borel function $\mc F(X,
\B)$. We recall the construction of the induced transfer operator 
$R_h$ where $h$ is a positive  harmonic function for $R$.
 Then
\be\label{eq R_h in sect R P}
R_h(f) := \frac{R(fh)}{h}, \ \ \qquad \ f \in \mc F(X, \B),
\ee
is a transfer operator such that $R_h(\mathbf 1) = \mathbf 1$.

\begin{proposition}\label{prop mu' and mu}
Let transfer operators $(R, \sigma)$ and $(R_h,\sigma)$  be defined 
 as above and  $Rh =h$. Suppose that $R$ has the Riesz 
property, and let $(\mu_x)$ be the corresponding Riesz family of 
measures. Then $R_h$ also has the Riesz property with respect the 
family $(\mu'_x)$ where the measures $(\mu_x)$ and $(\mu'_x)$
 are related as follows:
$$
d\mu_x(y) = \frac{h(\sigma y)}{h(y)} d\mu'_x(y), \qquad \ \ y \in 
\sigma^{-1}(x),
$$
\end{proposition}

In other words, the statement of the Proposition \ref{prop mu' and
 mu} says that the function $\dfrac{d\mu_x}{d\mu'_x}$ is a 
 $\sigma$-coboundary. 

\begin{proof} 
We need to find the family of measures $(\mu'_x)$ such that
$$
R_h(f)(x) = \int_X f(y) \; d\mu'_x(y).
$$
Since $R_h(f)$ can be found from (\ref{eq R_h in sect R P}),  we can 
write
\begin{eqnarray*}
R_h(f) & =& \frac{R(fh)}{h}(x)\\
\\
& = & R\left(\frac{fh}{h\circ\sigma}\right)(x)\\
\\
&=&\int_X \frac{fh}{h\circ\sigma}(y) \; d\mu_x(y)\\
\\
& =&  \int_X f \; d\mu_x'.
\end{eqnarray*}
Hence, we can take
$$
d\mu_x' (y)= \frac{h}{h\circ\sigma}(y) \; d\mu_x(y).
$$
 We note that $h(x) d\mu_x' (y)= h(y) d\mu_x(y)$ for any $x$ and
 $y \in \sigma^{-1}(x)$.
\end{proof}

Let now $\nu$ be a probability measure on $(X,\B)$. Define a new 
measure $\la$ on $(X,\B)$ by the formula
$$
\la = \int_X \mu_x \; d\nu(x).
$$
This is equivalent to the equality
\be\label{eq lambda via mu_x}
 \int_X f \; d\la = \int_X \left(\int_X f(y) \; d\mu_x(y)\right) 
 \; d\nu(x)
\ee
which is used in the following statement. 

We note that, in the case when a transfer operator $R$ satisfies
the Riesz property,  the family of Riesz measures $(\mu_x)$ can be 
viewed as a system of conditional measures defined $(X, \B, \la)$
 by the measurable partition $\xi = \{\sigma^{-1}(x) | x\in X\}$.

\begin{proposition}\label{prop la R = la} Let a transfer operators 
$(R, \sigma)$ have the Riesz property with the family of measures 
$(\mu_x)$. 

(1) Let $\la$ be  a measure defined by $\nu$ and $(\mu_x)$ as in 
(\ref{eq  lambda via mu_x}). Then $\la = \nu R$. 

(2) Suppose that  $R (\mathbf 1) =\mathbf 1$. Then, for 
  $\nu$ on $(X,\B)$ and  $\la$ as above, we have  $\la R = \la$.
\end{proposition}

\begin{proof} 
(1) It follows from the definition of $\la$ that, for any function $f$,
\begin{eqnarray*}
\int_X f (x)\; d\la & = & \int_X R(f)(x) \; d\nu\\
\\
& = & \int_X f (x) \; d(\nu R)
\end{eqnarray*}
and we are done.

(2) We need to show that for any function $f$ the following 
equality holds
$$
\int_X f\; d \la = \int_X f\; d(\la R) = \int_X R(f) \; d\la.
$$
In the following computation we use relation  (\ref{eq lambda via 
mu_x}) and the fact that $\mu_x$ is a probability measure for all 
$x \in X$. By Definition \ref{def Riesz property}, we have 
\begin{eqnarray*}
\int_X R(f)(x) \; d\la(x) &= & \int_X\left(\int_X R(f)(y) \; d\mu_x(y)
\right) \; d\la(x) \qquad\qquad\mbox{(by \ (\ref{eq lambda via 
mu_x}))}\\
\\
& = &  \int_{(x)}\left[\int_{(z)}  \left( \int_{(y)} f(y) \; d\mu_x(y)
\right) \; d\mu_x(z)\right] \; d\nu(x) \qquad\mbox{(by \ (\ref{eq def 
of Riesz prop}))}\\
\\
& =& \int_{(x)}\left[\int_{(y)} f(y) \left( \int_{(z)}  \; d\mu_x(z)
\right) \; d\mu_x(y)\right] \; d\nu(x)\\
\\
& =& \int_{(x)}\left[\int_{(y)} f(y)  \; d\mu_x(y)\right] \; d\nu(x)\\
\\
& = & \int_X f(x) d\la(x).
\end{eqnarray*}

\end{proof}

\begin{corollary} Let $R$,  $(\mu_x)$, $\nu$ and $\la$ be as in 
Proposition \ref{prop la R = la}.
Suppose $R(\mathbf 1)(x) = W(x)$. Then
$$
W(x) = \frac{d(\la R)}{d\la}(x), \qquad x \in X.
$$
\end{corollary}

\begin{proof} This result follows from the proof of Proposition
 \ref{prop la R = la} in which we will need to use the relation
$$
W(x) = \int_X d\mu_x.
$$
\end{proof}

\section{Transfer operators on the space of densities}
\label{sect TO on densities}

This section is focused on the study of an important class of transfer 
operators. As usual, we fix a non-invertible non-singular dynamical 
system $(X, \B, \mu,\sigma)$. Without loss of generality, we can 
assume that $\mu$ is a probability measure. 

If $\la$ is a a Borel measure such that $\la \ll \mu $, then there exists 
 the Radon-Nikodym derivative $f(x) = \dfrac{d\la}{d\mu}(x)$. 
 Conversely, any nonnegative function $f \in L^1(\mu)$ serves as  a 
 density  function for a measure $d\la = fd\mu$.

\begin{definition}
Define a transfer operator $R_\mu  = (R, \sigma)$ acting on 
$L^1(\mu)$ by the formula
\be\label{eq TO on densities def}
R_\mu(f) (x) = \frac{(f d\mu) \circ \sigma^{-1}}{d\mu}(x), \ \ \ f \in 
L^1(\mu).
\ee
We call $R_\mu $ a transfer operator on the space of densities. 
\index{density}
\end{definition}

In this section, we will work only with transfer operators $R_\mu$ 
defined by (\ref{eq TO on densities def}). 

The following lemma contains main properties of $R =R_\mu$. 
Most of the statements are well known, so that we omit their proofs.

\begin{lemma}\label{lem properties of R_lambda} Let $R$ be 
defined by (\ref{eq TO on densities def}). The following 
statements hold.

(1) $R$ is a positive bounded linear operator with $L^1$-norm equal 
to one. 

(2) The operator $R$ satisfies the pull-out property: 
$R[(f\circ \sigma) g] = f R(g)$. Moreover $R$ is a normalized 
transfer operator if and only if $\mu$ is a probability measure.

(3) The operator $R$ can be defined by the following statement:
 $R(f)$ is a unique element of $L^1(\mu)$ such that, for any function
  $g \in L^{\infty}(\mu)$,
$$
\int_X g (Rf)\; d\mu = \int_X (g\circ \sigma) f\; d\mu.
$$

(4) If $\mu$ is $\sigma$-invariant, then the operator $S : f \to f\circ 
\sigma$ is an isometry in $L^p(\mu)$. In this case, the operator $R$
also preserve the measure $\mu$, $\mu R = \mu$.

(5) If $\mu$ is a probability $\sigma$-invariant measure, 
then $f \mapsto R(f)\cs :
L^1\sms \to L^1(X, \sigma^{-1}(\B), \mu)$ is the conditional 
expectation $\mathbb E_\mu(f | \sB)$. \index{conditional expectation}
\end{lemma}

\begin{proof}
We show only that (5) is true (the other statements are easily 
verified). For this, we observe that (i) $R(R(f)\cs)\cs = R(f) \cs$
and (ii) for any $\sB$-measurable function $g$,
$$
\int_X g R(f)\cs \; d\mu = \int_X g f \; d\mu. 
$$ 
We calculate the left-hand side integral using $\sigma$-invariance of 
$\mu$ and the fact that
$g = h\cs$ for a $\B$-measurable function $h$:
\begin{eqnarray*}
\int_X g R(f)\cs \; d\mu  & = & \int_X (h\cs) R(f)\cs \; d\mu \\
\\
& = & \int_X hR(f)\; d\mu\csi1 \\
\\
& = & \int_X R((h\cs) f)\; d\mu\\
\\
& = & \int_X g f\; d\mu\\
\end{eqnarray*}

\end{proof}

Suppose now $d\la = \varphi d\mu$ where $\varphi$ is a positive 
function from $L^\infty(\mu)$. We will find out how the operators 
$R_\la$ and $R_\mu$ relate.

In the above setting, we define the multiplication operator
$$
M_\varphi (f) = \varphi f : L^1(\mu) \to L^1(\la).
$$

\begin{lemma} For $R_\la, R_\mu, \varphi$, and $M_\varphi$ defined 
 as above, we have
$$
R_\la M_\varphi = M_{\varphi} R_\mu.
$$
\end{lemma}

\begin{proof}
Indeed, we compute, for a function $ f \in L^1(\mu)$,
\begin{eqnarray*}
R_{\la} M_{\va}(f) &=& \frac{(\va f d\la) \csi1}{d\la}\\
\\
& =& \frac{(fd\mu)\csi1}{d\la}\\
\\
& = &  \frac{(fd\mu)\csi1}{d\mu} \frac{d\mu}{d\la}\\
\\
& = & M_{\va} R_{\mu}(f).
\end{eqnarray*}

 \end{proof}

Let $\lambda $ be a Borel measure on $(X, \B)$ which is equivalent to 
$\mu$. Then, as we know from Section \ref{sect L1 and L2},  
$\lambda $ is in $\mathcal L(R)$. We can  find  the 
Radon-Nikodym derivative $W_\la = \dfrac{d(\la R_\mu)}{d\la}$ 
of $R_\mu$ with respect to $\la$.

\begin{lemma} Let $\la \sim \mu$ and $\va = \dfrac{d\la}{d\mu}$.
Then $W_\la$ is a $\sigma$-coboundary, $W_\la = 
(\va\cs) \va^{-1}$.
\end{lemma}

\begin{proof} We use the definition of the Radon-Nikodym derivative 
for the transfer operator and compute
\begin{eqnarray*}
\int_X R_{\mu}(f) \; d\la &=& \int_X R_{\mu}(f)\va \; d\mu\\
\\
&=& \int_X R_{\mu}(f (\va \cs)) \; d\mu\\
\\
&=& \int_X f (\va \cs) \; d\mu\\
\\
&=& \int_X f (\va \cs) \va^{-1}\; d\mu 
\end{eqnarray*}
(we used here that $\mu R_\mu = \mu$). The latter means that
$W_\la = (\va\cs) \va^{-1}$.
\end{proof}

Let $\la$ be a quasi-invariant measure with respect to a surjective 
endomorphism $\sigma$ of $(X, \B)$. Let 
$\dfrac{d(\la \csi1)}{d\la}$. Then $\sigma$ generates an 
operator $S$ on $L^2(\la)$ defined by
$$
S : f \mapsto f\circ \sigma.
$$
It can be seen from Lemma \ref{lem properties of R_lambda} (3) that 
the operators $R_\la$ and $S$, viewed as operators in $L^2(\la)$, 
form a symmetric pair of operators because
$$
\int_X g (R_\la f)\; d\la = \int_X (Sg) f\; d\la, \ \ \ \ f,g \in L^2(\la).
$$
So, we can use the notation $S^*$ for $R_\la$ for consistency.

\begin{lemma}\label{lem RS is multipl operator}
In the above notation, the operator $S^*S$ is an operator of
 multiplication $M_{\theta_\la}$ by the function $\theta_\la$.
\end{lemma}

\begin{proof}
We note that $R_\la$ is not normalized because
$$
R_\la(\mathbf 1) = \frac{d(\la \csi1)}{d\la} = \theta_\la.
$$
Then, using inner product in $L^2(\la)$, we have
\begin{eqnarray*}
\langle S^*S(f), g\rangle_{L^2(\la)} &= & 
\int_X R_\la(f\cs) g\; d\la\\
\\
 &=&\int_X f R_\la(\mathbf 1) g\; d\la\\
 \\
  &=& \langle M_{\theta_\la}(f),  g\rangle_{L^2(\la)}.
\end{eqnarray*}
The result follows.
\end{proof}

\begin{theorem}\label{thm harmonic vs coboundary}
Let  $(X, \B, \la, \sigma)$ be a non-singular dynamical system 
generated by a surjective endomorphism. Let $R_\la$ be the transfer 
operator defined by (\ref{eq TO on densities def}). The following 
statements are equivalent:

(i) there exists a harmonic function $h$ for $R_\la$ such that 
$h$ is $\sB$-measurable;

(ii) the Radon-Nikodym derivative $\theta_\la$ is a 
$\sigma$-coboundary.

\end{theorem}

\begin{proof}
The proof of the theorem is based on Lemma \ref{lem RS is multipl 
operator}. 
We first observe that, since $\la \csi1 \sim \la$,  the Radon-Nikodym 
derivative $\theta_\la$ is positive a.e. Then the fact that 
$\theta_\la$ is a coboundary, $q \theta_\la = q\cs$, implies  that
$q \neq 0$. 

Therefore, to see that (i) implies (ii), we take a harmonic function
for $R_\la$ in the form $h = q\cs$ and obtain by 
Lemma \ref{lem RS is multipl operator} 
\be\label{eq harmonic to coboundary}
 q \cs  = R_\la (q \cs) =  (R_\la S)( q) = \theta_\la q.
\ee

Conversely, if $\theta_\la = (q\cs)q^{-1}$, then
$$
R_\la(q\cs) = \theta_\la q = q\cs 
$$
and the theorem is proved 
\end{proof}

\begin{corollary}\label{cor harmonic vs invariant measure}
For the transfer operator $(R_\mu, \sigma)$ defined on $\sms$, the  
measure $d\la = h d\mu$ is $\sigma$-invariant if and only if $h$ is 
harmonic for $R_\mu $. 
 \end{corollary}

\begin{proof}
This result follows from the equality where we use  Lemma \ref{lem 
properties of R_lambda}. For any measurable function 
$g\in L^\infty(\mu)$, we have
\begin{eqnarray*}
\int_X g \; d\la &=& \int_X gh\; d\mu \\
\\
&=& \int_X gR_\mu(h)\; d\mu \\
\\
&=& \int_X (g\cs) h\; d\mu \\
\\
&=& \int_X g\cs \; d\la .\\
\end{eqnarray*}
Hence, $\la\csi1 =\la$
\end{proof}

Readers coming from other but related areas, may find the following papers/
books useful for background  \cite{Baggett_et_al2010, ArklintRuiz2015}.

\newpage

\section{Piecewise monotone maps and the Gauss endomorphism}
\label{sect Gauss}

The purpose of the next two sections is to outline applications of our 
results to a family of examples of dynamics of endomorphisms, and 
their associated transfer operators. Earlier papers discussing some of 
these examples are as follows \cite{Keane1972, Llibre2015,
Radin1999, Rugh2016} for the case of piecewise 
monotone maps, \cite{AlpayJorgensenLewkowicz2016, 
BektasCanfesDursun2016, ChaoLv2016} for the 
case of Gauss endomorphism (map), and \cite{Hutchinson1981,
JarosMaslankeStrobin2016, YaoLi2016} for the case of iterated 
function systems. Our emphasis is infinite branching systems.

\subsection{Transfer operators for piecewise monotone maps}

In this section, we will discuss invariant measures for  piecewise 
monotone maps $\alpha : I \to I$ of an open interval  $I$ onto itself. 
We also consider the corresponding transfer operators $(R, \alpha)$ 
and show how one can describe $R$-invariant measures on $I$. While 
studying these problems, we assume, for definiteness, that $I = (0,1)$.

 We recall that, by definition, an onto endomorphism $\alpha$ of 
 $(0, 1)$ is called \emph{piecewise monotone} \index{endomorphism !
 piecewise monotone} if $(0,1)$ can be 
 partitioned into a finite or infinite family $(J_k)$ of subintervals $J_k 
 = (t_{k-1}, t_k)$ such that the restriction of $\alpha$ on each $J_k$ 
 is a continuous monotone one-to-one map onto $(0, 1)$ (in many 
 examples, the map $\alpha$ is assumed to be  differentiable on each 
 $J_k$).  Since our main interest is focused on invariant non-atomic 
 measures for piecewise monotone maps, we do not need to define $
 \alpha$ at the points of possible discontinuities $\{t_k : k \in \N\}$.
In the second part of this section, we apply the proved results to the 
Gauss map, which is a famous  example of a piecewise monotone map. 
Moreover, since the Gauss map admits a symbolic representation on a 
product space, we will be able to prove more results about invariant 
measures for the Gauss map.

We notice that the property of piecewise monotonicity of $\alpha$  
means that $\alpha: J_k \to (0,1)$ is a one-to-one map on every 
interval $J_k$. Then, for every $k$, there exists an inverse branch $
\beta_k$ of $\alpha$ such that $\beta_k$ maps $(0, 1)$ onto $J_k$ 
and satisfies the condition
$$
\alpha\circ \beta_k (x) = x, \qquad x \in (0,1).
$$
 We will assume implicitly that the collection of disjoint subintervals $
 (J_k)$  of $(0,1)$ is countable.
\medskip

Let $\alpha$, $(\beta_k : k \in \N)$, and $J_k$ be as above.
Suppose that $\pi = (p_k : k \in \N)$ is a probability infinite-
dimensional positive vector (probability distribution), i.e., $p_k >0$ 
and $\sum_k p_k =1$. 

\begin{definition}
Let  a measure $\mu$ on $X = (0,1)$ satisfy the property
\be\label{eq def mu as ISF}
\mu = \sum_{k =1}^{\infty} p_k \mu\circ \beta_k^{-1}.
\ee
Then $\mu$ is called an \emph{iterated function systems measure (IFS 
measure)}  \index{measure ! iterated function systems (IFS)} 
for the iterated function system \index{iterated function system (IFS)} 
$(\beta_k : k \in \N)$.
\end{definition}

 It is known that a measure $\mu$ satisfying (\ref{eq def mu as ISF}) 
 is uniquely determined and ergodic.
 
The following properties immediately follow from the definitions.

\begin{lemma}\label{lem first props of IFS measures}
(1) Let $\mu$ be an IFS measure for the system 
$(\beta_k : k \in \N)$, defined as in (\ref{eq def mu as ISF}), where 
$\beta_k : (0,1) \to J_k$. Then
$$
\mu(J_k) = p_k, \ \ \ \ k\in \N.
$$

(2) For the IFS measure $\mu$ and $\beta_k, J_k$ as above,
$$
\mu(A \cap J_k) = \mu(J_k) \mu(\beta_k^{-1}(A)).
$$

(3) For $\mu_k := \mu|_{J_k}$, we have $\mu_k \ll \mu$ and $
\mu_k \ll \mu\circ\beta_k^{-1}$. Moreover, the Radon-Nikodym 
derivatives are:
$$
\frac{d\mu_k}{d\mu} = \chi_{J_k}, \ \ \ \ \  \frac{d\mu_k}{d(\mu
\circ \beta_k^{-1})} = p_k \chi_{J_k}.
$$
\end{lemma}

\begin{proof} (1) Since $\beta_l^{-1}$ is defined on $J_l$ only, we 
see that $\mu\circ \beta^{-1}_l (J_k) = 0 $ if $l \neq k$. On the 
other hand, $\mu\circ \beta^{-1}_k (J_k) = 1$ because $\beta^{-1}
_k (J_k) = (0,1)$. Therefore, it follows from (\ref{eq def mu as ISF}) 
that $\mu(J_k) = p_k$.

(2) For an IFS measure $\mu$ such that $\mu = \sum_{k =1}^{\infty} 
p_k \mu\circ \beta_k^{-1}$, where $\beta_k : (0,1) \to J_k$ is a 
one-to one map and all $(J_k)$ are pairwise disjoint, we find that
\begin{eqnarray}\label{eq RN for IFS measures}
 \nonumber
\mu(A \cap J_k) & = & \sum_{i =1}^{\infty} p_i \mu\circ 
\beta_i^{-1}(A \cap J_k) \\
\\
   &=&  p_k \mu(\beta_k^{-1}(A \cap J_k)) \\
   \\
 \nonumber
  &=&  p_k \mu(\beta_k^{-1}(A))\\
  \\
 \nonumber
 & = & \mu(J_k) \mu(\beta_k^{-1}(A))
   \end{eqnarray}

(3) By definition, $\mu_k(A) := \mu(A\cap J_k)$. Then $d\mu_k(x) = 
\chi_{J_k}(x)d\mu(x)$. The other formula in this statement follows 
from (2) and (\ref{eq RN for IFS measures}).

\end{proof}

\begin{lemma}\label{lem ISF measures for alpha}
The IFS measure $\mu$ satisfying (\ref{eq def mu as ISF}) is
$\alpha$-invariant.
\end{lemma}

\begin{proof}
We verify that, for any integrable function $f$ on $X = (0,1)$,
\begin{eqnarray*}
  \int_X f \; d(\mu\circ\alpha^{-1})  &=& \int_X f(\alpha x)\; d\mu \\
  \\
  &=& \sum_{k =1}^{\infty} p_k \int_X (f\circ \alpha)(x) \; d(\mu
  \circ \beta_k^{-1}) \\
  \\
   &=& \sum_{k =1}^{\infty} p_k \int_X (f\circ \alpha) (\beta_k x)\; d
   \mu \\
   \\
   &=& \int_X f \; d\mu.
\end{eqnarray*}
Hence, $\mu\circ \alpha^{-1} = \mu$.
\end{proof}

In the next result we answer the following question. Suppose that a 
piecewise  monotone map $\alpha$ is as above, and let $(\beta_k)$ 
be  the family of inverse branches for $\alpha$.  Let $\mu$ be an 
$\alpha$-invariant IFS measure of the form (\ref{eq def mu as ISF}). 
We address now the following question: how can one determine 
explicitly the entries of the corresponding 
probability distribution $\pi$ in terms of $\alpha$ and the measure 
$\mu$?

\begin{theorem}\label{thm existence of p_k}
Let $\alpha$ be  a piecewise monotone endomorphism of $(0,1)$, and 
let $(J_k : k\in \N)$ be the corresponding collection of the disjoint 
intervals. Suppose that a measure $\mu$ is non-atomic and  satisfies 
relation (\ref{eq def mu as ISF}). Then, the entries $(p_k)$ of the 
probability distribution $\pi = (p_k : k\in \N)$ are determined by 
formula
\be\label{eq formula for p_k}
p_k = \frac{\int_{J_k} \alpha(x) \; d\mu(x) }{\int_0^1 x \; d\mu(x)}
 = \mu(J_k).
\ee
\end{theorem}

\begin{proof}
Without loss of generality, we can assume that $\mu(J_k) >0$ and 
$J_k \cap J_l = \emptyset$ for all $k \neq l$. Let $\beta_k$ be the 
inverse branch of $\alpha$ on the interval $J_k$. Define the collection 
of functions $(f_k : k \in \N)$ on $(0,1)$:
\be\label{eq def of f_k}
f_k(x) := \alpha(x) \chi_{J_k}(x), \ \ \ k\in \N.
\ee
We claim that
\be\label{eq f_k disjointness}
f_k(\beta_l(x)) = x \delta_{k,l}= \begin{cases} x &\mbox{if } k=l \\
0 & \mbox{if } k\neq l\end{cases}.
\ee
Indeed, for any $x \in (0,1)$,
$$
f_k(\beta_k(x)) = \alpha(\beta_k(x)) \chi_{J_k}(\beta_k(x)) = x
$$
because $\beta_k : (0,1) \to J_k$. On the other hand, if $l \neq k$, 
then
$$
f_k(\beta_l(x)) =0,
 $$
since $\beta_l(x) \in J_l$ and $J_l \cap J_k =\emptyset$.

Next, we notice that if $(p_k)$ is defined according to (\ref{eq 
formula for p_k}), then   $\pi = (p_k) $ is a probability distribution:
$$
\sum_{k=1}^\infty \int_{J_k} \alpha(x)\; d\mu(x) = \int_0^1 
\alpha(x) \; d\mu(x) = \int_0^1 x \; d\mu(x)
$$
because $\mu$ is $\alpha$-invariant.

To obtain relation (\ref{eq formula for p_k}), we first check that $p_l 
= \mu(J_l)$.
Indeed, we can use the fact that $\mu\circ \tau_l^{-k}, k \in \N,$ is 
supported by the set $J_k$ and then calculate the measure of $J_l$ 
by formula (\ref{eq def mu as ISF}). We get that the right hand side is 
nonzero only for $k =l$ and $\mu(J_l) = p_l$.

For the other part of (\ref{eq formula for p_k}), we calculate
\begin{eqnarray*}
 \int_{J_k} \alpha(x) \; d\mu(x)  &=&  \int_0^1 f_k(x) \; d\mu(x) \ \ 
 \qquad\qquad (\mbox{by}\ (\ref{eq def of f_k}))\\
 \\
   &=& \sum_l p_l \int_0^1 f_k(\beta_l(x))\; d\mu \qquad 
   (\mbox{by}\ (\ref{eq def mu as ISF})) \\
   \\
   &=& p_k\int_0^1 f_k \circ\beta_k \; d\mu \ \ \qquad \qquad 
   (\mbox{by}\ (\ref{eq f_k disjointness}))\\
   \\
   &=& p_k  \int_0^1x\; d\mu(x),
\end{eqnarray*}
and the result follows.

\end{proof}

The next theorem contains  a converse ( in some sense) statement 
for Theorem \ref{thm existence of p_k}.

\begin{theorem} \label{thm p_k in terms alpha} Let $\alpha$ be  a 
piecewise monotone map of $(0,1)$ onto itself. Let $(J_k  : k\in \N)$ 
be the collection of open subintervals such that the map $\alpha$ is 
monotone on each $J_k$, and let $(\beta_k : (0,1) \to J_k)$ be the 
inverse branches for $\alpha$.
Take an $\alpha$-invariant measure $\mu$  on $(0,1)$, $\mu\circ
\alpha^{-1} =\mu$. Suppose that  $R = R_\pi$ is the transfer operator 
acting on measurable functions such that
$$
R(f)(x) = \sum_{k=1}^\infty p_k f(\beta_k(x)),
$$
where the probability distribution $\pi = (p_k)$ is defined by
$$
p_k := \frac{\int_{J_k} \alpha(x) \; d\mu(x) }{\int_0^1 x \; d\mu(x) } 
= \mu(\beta_k(0,1)). 
$$
Then $\mu$ is $R$-invariant if and only if, for  any $k, m \in \N$,
\be\label{eq condition on mu for R-invariance}
\left(\int_0^1 x\; d\mu(x)\right) \int_{J_k}\alpha(x)^m \; d\mu(x) =
\left(\int_0^1 x^m\; d\mu(x)\right) \int_{J_k}\alpha(x)\; d\mu(x)
\ee
\end{theorem}

\begin{proof} It was shown in the proof of Theorem \ref{thm existence 
of p_k} that the numbers $(p_k)$ defined above satisfy the condition  
$\sum_k p_k =1$. For any integers $m,k$, we set
\be\label{eq def f_k,m}
f_{k,m}(x) := \alpha(x)^m \chi_{J_k}(x).
\ee
Then, for any $k,l,m \in \N$, and $x \in (0,1)$,
$$
f_{k,m}(\beta_l x) = x^m \delta_{k,l}.
$$
We apply the following sequence of equivalences to prove the result:
$$
\int_0^1 Rf\; d\mu = \int_0^1 f \; d\mu, \qquad \qquad \ \ \ \ 
\forall f \in \mc F(X)
$$
$$
\Updownarrow
$$
$$
\int_0^1 R(f_{k,m})\; d\mu = \int_0^1 f_{k,m} \; d\mu, \qquad \ \ \ 
\forall k,m \in \N
$$
$$
\Updownarrow
$$
$$
p_k\int_0^1 x^m \; d\mu(x) = \int_{J_k} \alpha(x)^m \; d\mu(x), 
\qquad \forall k,m \in \N.
$$
This proves the theorem.
\end{proof}

It follows from Theorem \ref{thm p_k in terms alpha} that the left 
hand side of the equality
$$
\frac{\int_{J_k} \alpha(x)^m \; d\mu(x) }{\int_0^1 x^m \; d\mu(x) } 
= \frac{\int_{J_k} \alpha(x) \; d\mu(x) }{\int_0^1 x \; d\mu(x) } = 
p_k
$$
 does not depend on $m$.

\begin{proposition} \label{prop equivalence}
Let $\alpha$ be  a piecewise monotone map of $(0,1)$ onto itself  
such that $\beta_k : (0,1) \to J_k$ is an inverse branch for $\alpha$, 
$k\in \N$. Suppose $\mu$ is an $\alpha$-invariant measure.  The 
following statements are equivalent:

(1)
$$
\mu = \sum_{k=1}^\infty p_k\mu\circ \beta_k^{-1},
$$

(2)
$$
(\chi_{J_k} d\mu)\circ\alpha^{-1} = p_k d\mu,
$$

(3) $\forall f \in \mc F((0,1), \mc B)$,
$$
\int_{J_k} f(\alpha x) \; d\mu(x) = p_k \int_0^1 f(x)\; d\mu(x).
$$

\end{proposition}

\begin{proof} Let $h(x) $ be a measurable function on $X = (0,1)$. Then we note that
the function $f_k(x) = \chi_{J_k}(x) h(\alpha (x))$ satisfies the relation
\be\label{eq f_k vs beta_l}
f_k(\beta_l (x)) = \delta_{k,l} h(x), \qquad \forall k,l \in \N,
\ee
where $\delta_{k,l} $ is the Kronecker delta symbol. Indeed,
\begin{eqnarray*}
f_k(\beta_l (x)) &=& \chi_{J_k}(\beta_l x) h(\alpha (\beta_l(x))) \\
\\
   &=& \begin{cases}  0, &\mbox{if } l \neq k\\
\\   
     \chi_{J_k}(\beta_k x),  & \mbox{if } k = l \end{cases}   \\
\\   
   &=&   \delta_{k,l} h(x).
\end{eqnarray*}
We used the facts that $\beta_k (0,1) = J_k$, and the sets $(J_k : k
\in \N)$ are disjoint.

Suppose now that (1) holds. Then, for any measurable function $
\varphi$, we have
$$
\int_0^1 \varphi \; d\mu = \sum_{k=1}^\infty p_k \int_0^1 \varphi
\circ \beta_k \; d\mu.
$$
Take $\varphi =  \chi_{J_k}(x) h(\alpha (x))$. Hence,
\begin{eqnarray*}
  \int_0^1 \chi_{J_k}(x) h(\alpha (x))  &=&\sum_{k=1}^\infty p_k 
  \int_0^1 \chi_{J_k}(\beta_k x) h(\alpha (\beta_k x))\; d\mu(x)  \\
   &=& p_k \int_0^1 h(x) \; d\mu(x)
 \end{eqnarray*}
This relation can be written in the form
\be\label{eq equivalence reln}
\int_0^1 h \; (\chi_{J_k}d\mu)\circ\alpha^{-1} = p_k \int_0^1 h\; d
\mu
\ee
which is equivalent to statement (2):
$$
p_k = \frac{(\chi_{J_k}d\mu)\circ\alpha^{-1}}{d\mu}.
$$
Simultaneously, we have shown, in the above proof, that (3) holds due 
to relation (\ref{eq equivalence reln}).

To finish the proof, we observe that all implications are reversible so 
that the three statements formulated in the proposition are equivalent.

\end{proof}

\begin{remark} Given $\alpha$ and $\mu$ as above, we can define the 
transfer operator
$$
R\varphi := \frac{(\varphi d\mu)\circ\alpha^{-1}}{d\mu}.
$$
Then $R( \mathbf 1) = \mathbf 1$ since $\mu$ is $\alpha$-invariant.
It follows from Proposition \ref{prop equivalence} that
$$
R(\chi_{J_k}) = \frac{(\chi_{J_k}d\mu)\circ\alpha^{-1}}{d\mu} = 
p_k, \qquad k \in \N.
$$
\end{remark}

\subsection{The Gauss map}

The  famous \emph{Gauss endomorphism (map)} $\sigma$ is an 
example of  a piecewise monotone map with countably many inverse 
branches. This map has been studied in many papers, we refer, for 
example,  to \cite{CornfeldFominSinai1982} for basic definitions and 
facts about $\sigma$. Our study below is motivated by 
\cite{AlpayJorgensenLewkowicz2016}.

Let $\lfloor x \rfloor$ denote  the integer part, and let  $\{x\}$ 
denote the fractional part of $x \in \R$. Then the Gauss map
\index{Gauss map} \index{endomorphism ! Gauss} is defined by the formula
$$
\sigma(x) = \frac{1}{x} - \left\lfloor \frac{1}{x} \right\rfloor = 
\left\{\frac{1}{x}\right\}, \ \ 0< x < 1.
$$
We apply the results provedin the first part of this section  to this map 
and find  explicit formulas for invariant measures.

Since $\sigma$ is a countable-to-one map, which is monotone 
decreasing on the intervals $J_k = (\frac{1}{k+1}, \frac{1}{k})$, we 
can easily  point out the family $\{\tau_k : k \in \N\}$ of inverse 
branches for the Gauss map $\sigma$:
$$
\tau_k (x) = \frac{1}{k +x},\ \ \ \ \ x \in (0,1).
$$
Clearly, $\tau_k$ is a monotone decreasing  map from $(0,1)$ onto 
the subinterval $(\frac{1}{k+1}, \frac{1}{k})$. 
The relation $(\sigma
\circ \tau_k)(x) = x$ holds for any $x \in (0,1)$, so that $\tau_k$ is 
the inverse branch of $\sigma$ for every $k\in \N$.

\begin{remark}
We observe that the composition $\tau_{k_1}\circ \cdots \circ 
\tau_{k_m}$ is also a well defined map from $(0,1)$ onto the open 
subinterval $(\tau_{k_1}\circ \cdots \circ \tau_{k_m}(0), \ 
\tau_{k_1}\circ \cdots \circ \tau_{k_m}(1))$ of $(0,1)$ according to 
the formula:
\be\label{eq composition of tau maps}
\tau_{k_1}\circ \tau_{k_2}\cdots \circ \tau_{k_m}(x) =
\cfrac{1}{k_1+\cfrac{1}{k_2+ \cdots \cfrac{1}{k_m +x}}}
\ee

It follows from the definition of the Gauss map $\sigma$ that the 
endomorphism $\sigma^m$ is a one-to-one map from each subinterval 
$(\tau_{k_1}\circ \cdots \circ \tau_{k_m}(0), \ \tau_{k_1}\circ 
\cdots \circ \tau_{k_m}(1))$ onto $(0,1)$.
\end{remark}

In this section, we are interested in $\sigma$-invariant ergodic 
non-atomic measures. The set of such measures  is uncountable.  As 
follows from Lemma \ref{lem ISF measures for alpha}, every IFS 
measure for $\sigma$ is $\sigma$-invariant (and ergodic as follows 
from Theorem \ref{thm measure mu_pi}).

The following example of an ergodic $\sigma$-invariant measure is 
well known and  goes back to Gauss.

\begin{lemma}[Gauss, \cite{Renyi1957}] The class of measures 
equivalent to the Lebesgue measures $dx$ contains the measure $
\mu_0$ with density
$$
d\mu_0(x) = \frac{1}{\ln 2}\cdot \frac{dx}{(x+1)},\ \ \ x\in (0,1),
$$
such that $\mu_0$ is  $\sigma$-invariant and ergodic.
\end{lemma}

We apply the methods used in the first part of this section and 
consider transfer operators $R$ associated to $\sigma$, probability 
distributions $\pi = (p_k : k \in \N)$ and the corresponding inverse 
branches $(\tau_k)$.

\begin{lemma}\label{lem TO for pi}
Let $\pi = (p_k : k \in \N)$ be  a probability infinite-dimensional 
positive vector.
Then
\be\label{eq TO for pi}
R_\pi f(x) := \sum_{k=1}^{\infty} p_k f(\tau_ k x) = \sum_{k=1}
^{\infty} p_k f\left(\frac{1}{k+x}\right), \ \ x \in (0,1).
\ee
is a normalized transfer operator associated with the Gauss map $
\sigma$. Moreover, $R_\pi$ is normalized, i.e., $R_\pi (\mathbf 1) = 
\mathbf 1$.
\end{lemma}

It is obvious that $R_\pi$ is positive and normalized. The pull-out 
property for $R_\pi$ is verified  by direct computations. We omit the 
details.

\begin{lemma}\label{lem ISF measures for sigma} Suppose that $\la$ 
is a measure on $(0,1)$ such that $\la R_\pi = \la$. Then $\la$ is $
\sigma$-invariant, i.e., $\la\csi1 = \la$.
\end{lemma}

In fact, this is a particular case of the general statement that holds for 
any normalized transfer operator.  A formal proof  was given in Section 
\ref{sect actions on measures}.
\medskip

In what follows we recall a convenient realization of the Gauss map on 
the space of one-sided infinite sequences. Let
$$
\Omega = \prod_{i=1}^\infty \N
$$
be the product space with Borel structure  generated by cylinder sets
$$
C(k_1, ..., k_m):= \{\omega \in \Omega : \omega_1 = k_1, ... , 
\omega_m = k_m\}
$$
where  $\omega = (a_1, a_2, ... )$ denotes an arbitrary point in $
\Omega$. Let $S$ be the one-sided shift in $\Omega$:
$$
S(a_1, a_2, a_3, ... ) = (a_2, a_3, ... ).
$$
Clearly, $S$ is a countable-to-one Borel endomorphism of $\Omega$.
For every $k \in \N$, we define the inverse branch of $S$ by setting
$$
T_k (a_1, a_2, a_3,  ... )  =(k, a_1, a_2,  ... ).
$$
Then $T_k (\Omega) = C(k), k \in \N$, and  $ST_k = \mathrm{id}_
\Omega$.
The collection of sets $(C(k) : k \in \N)$ forms a partition of 
$\Omega$.

Take a positive probability distribution $\pi = (p_1, ... , p_k, ....)$ on 
the set $\N$ and define the probability product measure
$$
\mathbb P = \pi\times\pi\times\cdots
$$
on $\Omega$ so that $\mathbb P(C(k_1, ... , k_m)) = p_{k_1} \cdots 
p_{k_m}$. Clearly, the measure $\mathbb P$ is $S$-invariant, i.e., $
\mathbb P\circ S^{-1} = \mathbb P$.

Define the Borel map $F: \Omega \to (0,1)$ by setting
$$
F(\omega) =
\cfrac{1}{a_1+\cfrac{1}{a_2+\cfrac{1}{a_3+\cdots}}}
$$
where $\omega =  (a_1, a_2, a_3, \cdots) $ is any point from $
\Omega$.

\begin{remark}
It is clear that  $F$ establishes a one-to-one correspondence between 
sequences  from $\Omega$ and all irrational points in the interval $(0,
1)$. We denote by $X$ the set $F(\Omega)$. This means that $X = 
(0,1) \setminus \mathbb Q$. In the sequel, we will use the same 
notation $J_k $ for the interval $(\frac{1}{k+1}, \frac{1}{k})$ with 
removed rational points. As was mentioned above, this alternation does 
not affect continuous $\sigma$-invariant measures which are our main 
object of study.
\end{remark}

It is a simple observation that  $F(C(k)) = J_k$ for any $k\in \N$. 
Moreover, it follows from relation (\ref{eq composition of tau maps}) 
that
\be\label{eq F maps cylinder sets}
F(C(k_1, \cdots, k_m)) = (\tau_{k_1}\circ \cdots \circ \tau_{k_m}
(0), \ \tau_{k_1}\circ \cdots \circ \tau_{k_m}(1)).
\ee

The following statement is well known (see e.g. 
\cite{CornfeldFominSinai1982}). We formulate it for further 
references.

\begin{lemma}\label{lem folklore}
 In the above notation, the map $F$ intertwines the pairs of maps $
 \sigma, S$ and $T_k, \tau_k$:
$$
F \circ S = \sigma\circ F, \ \ \ \ \ F \circ T_k = \tau_k\circ F, \ \ \ \ k 
\in \N.
$$
\end{lemma}

From this lemma and relations (\ref{eq composition of tau maps}) and 
(\ref{eq F maps cylinder sets}), we deduce the following result.

\begin{corollary}
The collection of intervals $\{(\tau_{k_1}\circ \cdots \circ 
\tau_{k_m}(0), \ \tau_{k_1}\circ \cdots \circ \tau_{k_m}(1)) : 
k_1, ... k_m \in \N, m \in \N\}$ generates the sigma-algebra of Borel 
sets on the interval $(0,1)$.
\end{corollary}

\begin{proof} This result obviously follows from the facts that the 
length of subintervals $(\tau_{k_1}\circ \cdots \circ \tau_{k_m}(0), \ 
\tau_{k_1}\circ \cdots \circ \tau_{k_m}(1))$ tends to zero as $m 
\to \infty$, and they separate points in $(0,1)$ because every such 
subinterval is the image of a cylinder set in $\Omega$.
\end{proof}

\begin{theorem}\label{thm measure mu_pi}
For any probability distribution $\pi = (p_k : k \in \N)$, there is a 
unique $\sigma$-invariant measure $\mu_\pi$ on $(0,1)$ such that
$$
\int_0^1 f \; d\mu_\pi = \sum_{k=1}^\infty p_k \int_0^1 f \circ 
\tau_k \; d\mu_\pi,
$$
or equivalently, $\mu_\pi$ is the IFS measure defined by $\pi$ and 
$\{\tau_k\}$ such that
$$
\mu_\pi = \sum_{k=1}^\infty p_k \mu_\pi\circ \tau_k^{-1}.
$$
Conversely, if $\mu$ is an IFS measure on $(0,1)$ with respect to the 
maps $(\tau_k : k \in \N)$, then there exists a product measure $
\mathbb P = \mathbb P_\mu$ on $\Omega$ such that $\mu = 
\mathbb P\circ F^{-1}$.
\end{theorem}

\begin{proof}
 Fix any probability distribution $\pi$ and consider the stationary 
product measure
$$
\mathbb P_\pi = \pi \times \pi \times \cdots
$$
on $\Omega$. The measure $\mathbb P$ is first determined on 
cylinder sets, and then it is extended by Kolmogorov consistency to all 
Borel sets in $\Omega$. In particular, we observe that $\mathbb P_\pi 
(\Omega_k) = p_k$.

Next, we set $\mu_\pi := \mathbb P_\pi \circ F^{-1}$, i.e., $\mu_
\pi(A) = \mathbb P_\pi(F^{-1}(A)), \forall A \in \B(0,1)$. It defines a 
Borel probability measure on $X$.

Let  $\varphi$ be any measurable function on $\Omega$. It follows 
from the above definitions that the following relation holds:
\begin{eqnarray*}
\int_\Omega \varphi(\omega) \; d\mathbb P_\pi(\omega)  &=& 
\sum_{k=1}^\infty \int_{\Omega_k} \varphi(\omega)\; d\mathbb P_
\pi(\omega) \\
   &=& \sum_{k=1}^\infty p_k \int_{\Omega} \varphi(k\omega')\; d
   \mathbb P_\pi(\omega').
\end{eqnarray*}
We used in this calculation the fact that $d\mathbb P_\pi(k\omega') = 
p_kd\mathbb P_\pi(\omega')$.

Since $F^{-1}$ is a one-to-one map from $X$ onto $\Omega$, we see 
that any measurable function $f$ on $X$ is represented as $f = 
\varphi \circ F^{-1}$.
By Lemma \ref{lem folklore}, we obtain that
$$
\varphi(k \omega) = \varphi (k F^{-1}x) = f(\tau_k(x)).
$$
Therefore, we have
\begin{eqnarray*}
\int_0^1 f \; d\mu   &=& \int_\Omega \varphi(\omega) \; d\mathbb 
P_\pi(\omega)\\
\\
   &=& \sum_{k=1}^\infty p_k \int_{\Omega} \varphi(k\omega')\; d
   \mathbb P_\pi(\omega') \\
   \\
   &=& \sum_{k=1}^\infty p_k \int_0^1 f \circ \tau_k\; d\mu.
\end{eqnarray*}
Hence, this  shows that $\mu_\pi$ is an IFS measure. By Lemma 
\ref{lem ISF measures for sigma}, $\mu_\pi$ is $\sigma$-invariant, 
and this fact completes the proof.

In order to prove the converse statement, we begin with an IFS 
measure $\mu = \sum_k p_k \mu\circ \tau_k^{-1}$ and define the 
product measure $\mathbb P$ by setting $\mathbb P = \pi \times \pi 
\times \cdots $ where $\pi = (p_1, p_2, ...)$ as in the definition of $
\mu$. Then, for any $m\in \N$ and $k_1, ... , k_m \in \Z_+$, we have
\be\label{eq measure cylinder sets}
\mathbb P(C(k_1, ... , k_m)) = \mu(\tau_{k_1}\circ\cdots \circ
\tau_{k_m}(0, 1)) = p_1 \cdots p_{k_m}.
\ee
It follows from (\ref{eq measure cylinder sets}) that there is a one-to-
one correspondence between IFS measures on $(0,1)$ and product 
measures on $\Omega$.

\end{proof}

We summarize the previous discussions in the following corollary.

\begin{corollary}\label{cor equivalence to IFS property} Let $\sigma$ 
be the Gauss map, $\mu$ a $\sigma$-invariant measure. The following 
are equivalent:

(i) $\mu$ is an IFS measure, $\mu = \sum_{k=1}^\infty p_k \mu\circ 
\tau_k^{-1}$;

(ii) $\mu = \mathbb P_\mu \circ F^{-1}$ for some product measure $
\mathbb P$ on $\Omega$;

(iii)
\begin{equation*}
\setlength{\jot}{10pt}
\begin{aligned}
  \mathbb P_\mu(C(k_1, ... , C_{k_m})) &= \mu(\tau_{k_1}\circ 
  \cdots \circ \tau_{k_m} (X)) \\
     &= \mu(\tau_{k_1}(0,1)) \ \cdots\  \mu(\tau_{k_m}(0,1))\\
   & = p_{k_1} \cdots p_{k_m},\ \qquad \ \ \ \forall k,m \in \N.
\end{aligned}   
    \end{equation*}
\end{corollary}

\begin{remark}
Suppose that $\pi, \pi'$ are two distinct probability distributions. Then 
the corresponding stationary measures $\mathbb P_\pi$ and $
\mathbb P_{\pi'}$ are mutually singular  by the Kakutani theorem 
\cite{Kakutani1948} for $\pi \neq \pi'$. But it would be interesting 
to find out whether the map 
$F : \Omega \to X$ preserve this property. 
Is it possible to have two $\sigma$-invariant measures $\mu_\pi = 
\mathbb P_\pi\circ F^{-1}$ and $\mu_{\pi'} = \mathbb P_{\pi'}\circ 
F^{-1}$ which are  both equivalent to the Lebesgue measure?
\end{remark}

It turns out that there are $\sigma$-invariant measures on $(0,1)$ 
which are not  generated by product measure on $\Omega$.

\begin{corollary}\label{cor measure mu_0} Let $\sigma$ be the Gauss 
map and let $\mu_0$ be the probability $\sigma$-invariant measure 
on $(0,1)$ given by the density
$$
d\mu_0(x) = (\ln 2)^{-1} \frac{dx}{x+1}
$$
where $dx$ is the Lebesgue measure on $(0,1)$. Then $\mu_0$ is 
not an IFS measure.
\end{corollary}

\begin{proof} We first can directly calculate the measures of the 
intervals $J_k$ on which the Gauss map $\sigma$ is one-to-one:
\begin{eqnarray}\label{eq mu0 measure of J_k}
\nonumber
 \mu_0(J_k) &=& (\ln 2)^{-1} \int_{(k+1)^{-1}}^{k^{-1}} \frac{1}{x
 +1}\; dx \\
   &=& (\ln 2)^{-1}\ln \left( 1 + \frac{1}{k(k+2)}\right).
\end{eqnarray}

In order to prove the formulated statement, we use Corollary \ref{cor 
equivalence to IFS property}. Suppose for contrary that $\mu_0$ is an 
IFS measure. This means that by Theorem \ref{thm measure mu_pi} 
there exists a product measure $\mathbb P_\pi$ such that  $\mu_0  
= \mathbb P_\pi \circ F^{-1}$ for some probability distribution $\pi = 
(p_1, p_2, ...)$. By Corollary \ref{cor equivalence to IFS property} the 
measure $\mu_0$ will then satisfy the property
\be\label{eq property of mu0}
\mu_0(\tau_1\circ\tau_1 (0,1)) = p_1^2 = \mu_0(\tau_1(0,1))^2.
\ee
When we calculate the measures of $\tau_1\circ\tau_1 (0,1)$ and $
\tau_1(0,1)$, we see that
$$
\mu_0(\tau_1(0,1))^2 = \frac{1}{(\ln 2)^2}(\ln (4/3))^2
$$
which is not equal to
$$
\mu_0(\tau_1\circ\tau_1 (0,1)) = \frac{1}{\ln 2} \ln (10/9).
$$
This is a contradiction that shows that $\mu_0$ is not an IFS measure.
\end{proof}

Readers coming from other but related areas, may find the following papers/
books useful for background \cite{ArslanMilousheva2016, Machado2016, 
JarosMaslankeStrobin2016, JiLiuRi2016, GhaneSarizadeh2016, 
HorbaczSleczka2016, YeZouLu2013, SzarekUrbanskiZdunik2013,
JorgensenTian2015, JorgensenPedersenTian2015}.
 
\newpage






\section{Iterated function systems and transfer operators}
\label{sect Example IFS}

\subsection{Iterated function systems and measures} In this section, 
we will discuss an application of general results about transfer 
operators, that were proved in previous sections, to a family of 
examples based on the notion of \emph{ iterated function system 
(IFS)} \cite{Hutchinson1995, HutchinsonRueschendorf1998, Mauldin1998}.

We recall that if an endomorphism $\sigma$ is a finite-to-one map of a
a standard Borel space $(X, \B)$ onto itself, then there exists a family 
of one-to-one maps $\{\tau_i\}_{i=1}^n$  such that $\tau_i : X \to X
$ and $\sigma \circ \tau_i = \mbox{id}_{X}, \forall i$. The maps $
\tau_i$ are called the \emph{inverse branches} \index{inverse branch}
 for $\sigma$. The 
collection of maps $(\tau_i : 1 \leq i \leq N)$ represents  an example 
of iterated function systems. This is a motivating example for the 
concept of iterated function systems which, in general, not need to 
have an endomorphism $\sigma$ but is based on the one-to-one  maps 
$\tau_i : X \to X$ only.

Thus, an IFS consists of a space $X$ and injective maps $\{\tau_i : i 
\in I\}$ of $X$ into itself. The orbit of any point $x\in X$ is formed by 
$(\tau_{i_1} \circ \cdots \circ \tau_{i_k}(x) : i_1, ..., i_k \in I, k \in 
\N)$. The study of properties of an IFS assumes that the underlying 
space $X$ is a complete metric space (or compact space) and the 
maps $\tau_i$ are continuous (or even contractions).

\begin{remark}
We discuss here the case of \emph{finite} iterated function systems 
only. The theory of  \emph{infinite} iterated function systems is more 
complicated and  require additional assumptions (see, for example, the 
expository article \cite{Mauldin1998}). It worth recalling that we have 
already dealt with infinite IFS in Section \ref{sect Gauss} when we 
discussed piecewise monotone maps and the Gauss map.
\end{remark}

Let $p = (p_i : i =1,...,N)$ be a strictly positive probability vector, i.e., 
$\sum_{i=1}^N p_i =1$ and $p_i >0$ for all $i$. A measure $\mu_p$ 
on a Borel space $(X, \B)$ is called an \emph{IFS measure}  for the 
iterated function system $(\tau_i : i =1,...,N)$ if
\be\label{eq IFS measure def}
\mu_p =  \sum_{i=1}^N p_i\; \mu_p\circ \tau_i^{-1},
\ee
or, equivalently,
$$
\int_X f(x)\; d\mu_p(x)  =  \sum_{i=1}^N p_i \int_X f(\tau_i(x))\; d
\mu_p(x), \  \qquad f\in L^1(\mu_p).
$$
In particular, $p$ can be the uniformly distributed probability vector, 
$p_i = 1/N$. Then the corresponding measure $\nu_p$ satisfies the 
property
$$
\nu_p =  N^{-1}\sum_{i=1}^N \nu_p\circ \tau_i^{-1}.
$$

\begin{definition} \label{def operator for IFS} Let $(X; \tau_1, ... , 
\tau_n)$ be a finite iterated function system, and let $p = (p_i)$ be a 
positive probability vector.
Define a positive linear operator acting in the space of Borel functions:
\be\label{eq operator for IFS via W}
R(f)(x) = \sum_{i=1}^N p_i W(\tau_i x)f(\tau_i x)
\ee
where $W$ is a nonnegative Borel function (sometimes it is called a 
\textit{weight}).
If, for all $x$, one has
$$
\sum_{i =1}^N p_i W(\tau_i x) = 1,
$$
then $R$ is a normalized transfer operator in the sense that 
$R(\mathbf 1) = \mathbf 1$.
\end{definition}

To clarify our terminology, we note that $R$ is not, in general,  a 
transfer operator because the maps $(\tau_1, ... , \tau_N)$ do not 
define an endomorphism $\sigma$. But if the IFS $(X; \tau_1, ... , 
\tau_N)$ consists of inverse branches for a finite-to-one onto 
endomorphism $\sigma$, then the operator $R$ is a transfer operator 
related to $\sigma$: for $y_i = \tau_i x$, we have
$$
R(f)(x) = \sum_{y_i : \sigma(y_i) =x} p_{y_i} W(y_i)f(y_i).
$$
In some cases, it is convenient to modify $W$ by considering $\wt 
W(y) = p_y W(y)$.

It is not difficult to see that if $(X; \tau_1, ... , \tau_N)$ is an IFS, 
then the maps $(\tau_i)$ are the inverse branches for an 
endomorphism $\sigma$ if and only if the sets $J_i = \tau_i(X)$ have 
the properties:
$$
X = \bigcup_{i=1}^N J_i, \qquad \ \  J_i \cap J_k =\emptyset, \ \ i 
\neq k.
$$
Indeed, one can then define $\sigma (x) = \tau_i^{-1}(x), x \in J_i, 
1\leq i \leq N$.

If an IFS is generated by inverse branches of an endomorphism $\sigma
$, then we can add more useful relations. For any Borel set $A \subset 
X$, we see that
$$
\sigma^{-1}(A) = \bigcup_i \tau_i (A)
$$
and, more generally,
$$
\sigma^{-k}(A) = \bigcup_{\omega|_k} \tau_{\omega|_k} (A).
$$

As for abstract transfer operators, we define the notion of integrability 
of $R$: we say that $R$ is integrable with respect to a measure $\nu$ 
if $R( \mathbf 1) \in L^1(\nu)$. In this case, $R$ acts on the measure 
$\nu$, $\nu \mapsto \nu R$.
Then we can define the set $\mc L(R)$ of all measures on $(X, \B)$ 
such that $\nu P\ll \nu$.

\begin{lemma} Let $(X; \tau_1, ...., \tau_N)$ be an IFS, and let $p = 
(p_i)$ be a probability distribution on $\{1, ... , N\}$. Suppose that  $R
$ is the operator defined for the IFS by (\ref{eq operator for IFS via 
W}), and the measure $\mu_p$ 
satisfies  (\ref{eq IFS measure def}). Then $R$ is $\mu_p$-integrable 
if and only if $W\in L^1(\mu_p)$. Furthermore, $\mu_p$   belongs to 
$\mc L(R)$ and
$$
\frac{d\mu_p R}{d\mu_p} = W.
$$
\end{lemma}

\begin{proof} We use (\ref{eq operator for IFS via W}) and show that
\begin{eqnarray*}
   \int_X R(f)(x)\; d\mu_p(x) &=& \int_X \sum_{i=1}^N p_i W(\tau_i 
   x) f(\tau_i x)\; d\mu_p(x) \\
  &=& \int_X fW (\sum_{i=1}^N p_i \; d\mu_p\circ \tau_i^{-1}) \\
   &=& \int_X fW \; d\mu_p.
\end{eqnarray*}
This calculation shows that the following  facts hold. Firstly, 
$R(\mathbf 1) $ is $\mu_p$-integrable if and only if $W \in 
L^1(\mu_p)$; secondly, $\mu_p \in \mc L(R)$; thirdly, the Radon-
Nikodym derivative of $\mu_p R$ with respect to $\mu_p$ is $W$.
\end{proof}

The question about the existence of an IFS measure for a given finite 
or infinite iterated function system system $(\tau_i : i \in I)$ is of 
extreme importance. We discuss here a general scheme that leads to 
IFS measures. We do not formulate rigorous statements; instead we 
describe the construction method. This approach works perfectly for 
many specific applications under some additional conditions on $X$ and 
maps $\tau_i$. For definiteness, we assume that $I 
= \{1, ... ,N\}$.

Let $\Omega$ be the product space:
$$
\Omega = \prod_{i =1}^\infty \{1,..., N\}.
$$
 Our goal is to define a map $F$ (a coding map) from  $\Omega$ to 
 $X$. In general, $F(\Omega) $ will be a subsets of $X$ called the 
 attractor of the IFS.

For any  infinite sequence $\omega = (\omega_1, \omega_2, ...) \in 
\Omega$, let $\omega|_n $ denote the finite truncation, i.e., $
\omega|_n$ is the finite word $(\omega_1, ... , \omega_n)$. Then, we 
can use this word $\omega|_n$  to define  a map $\tau_{\omega|_n} 
$ acting on $X$ by the formula:
$$
\tau_{\omega|_n} (x) := \tau_{\omega_1} \cdots \tau_{\omega_n}
(x), \qquad x \in X, \ \  \omega \in \Omega,\ n \in \N.
$$

It is said that $\Omega$ is an \emph{encoding space} \index{encoding space} if, for every $
\omega\in \Omega$,
\be\label{eq singleton}
F(\omega) = \bigcap_{n \geq 1}\tau_{\omega|_n}(X)
\ee
is a singleton. In other words, we have a well defined Borel map $F$ 
from $\Omega$ to $F(\Omega)$ where $x =F(\omega)$ is defined by 
(\ref{eq singleton}). It is worth noting that there are IFS such that 
$F(\Omega) = X$. One of such IFS will be discussed in this section 
below.

There are various sufficient conditions under which there exists a 
\textit{coding map} \index{coding map} $F :\Omega\to X$ for a given IFS.  
For instance, this is the 
case when each $\tau_i$ is a contraction and $X$ is a complete metric 
space. 

Next, we define the following maps on $\Omega$: the left shift $\wt
\sigma$ by setting
$$
\wt\sigma(\omega_1, \omega_2, ... ) = (\omega_2, \omega_3, ... ),
$$
and the inverse branches $\wt\tau_i$ of $\wt\sigma$ by setting
$$
\wt\tau_i (\omega_1, \omega_2, ...) = (i, \omega_1, \omega_2, ...), 
\ \  i = 1,..., N.
$$
Clearly,
$$
\wt \tau_i (\Omega) = C(i) = \{\omega \in \Omega : \omega_1 = i\},
$$
 and the space $\Omega $ is partitioned by the sets $C(i), i  =1,...,N$.

The following statement follows directly from the definitions.

\begin{lemma} \label{lem F intertwines} The map $F : \Omega \to X$
is a factor map, i.e.,
$$
F\circ \wt\tau_i = \tau_i \circ F, \ \ \ \qquad \ \ i = 1,..., N.
$$
\end{lemma}

Let $p = (p_i : i =1,..., N)$ be a positive probability   vector. It defines 
the product measure $\mathbb P$ on $\Omega$,
$$
\mathbb P = p \times p \times \cdots.
$$
Firstly, $\mathbb P$ is defined on the algebra of cylinder sets $(C(i_1, 
... , i_m) : 1 \leq i_1, ... , i_m \leq N, m \in \N)$ by the formula
$$
\mathbb P(C(i_1, ... , i_m)) = p_{i_1} \cdots p_{i_m},
$$
and then $\mathbb P$ is extended  to the sigma-algebra of Borel sets 
on $\Omega$ by the standard procedure.

We observe that the maps $(\wt\tau_1, ... , \wt\tau_N)$ constitute 
an IFS on $\Omega$ such that $\mathbb P$ is an IFS measure:
\be\label{eq P is IFS}
\mathbb P = \sum_{i=1}^N p_i \; \mathbb P\circ \wt\tau_i^{-1}.
\ee

\begin{proposition} Suppose that $(X; \tau_1, ... , \tau_n)$ is an IFS 
that admits a coding map $F : \Omega \to X$. Let $p = (p_i)$ be  a 
probability vector generating the product measure $\mathbb P = p
\times p \times \cdots$. Then the measure
$$
\mu := \mathbb P\circ F^{-1}
$$
is an IFS measure satisfying
$$
\mu = \sum_{i=1}^N p_i \mu\circ \tau_i^{-1}.
$$
Moreover, if $F$ is continuous, then  $\mu$ has full support.
\end{proposition}

\begin{proof} By definition of measure $\mu_p$, we have
$$
\mu_p (A) = \mathbb P(F^{-1}(A)).
$$
Then we use Lemma \ref{lem F intertwines} and (\ref{eq P is IFS}) to 
show that $\mu_p$ is an IFS measure:
\begin{eqnarray*}
 \mu_p (A)   &=& \mathbb P(F^{-1}(A)) \\
   &=& \sum_{i=1}^N p_i\; \mathbb P(\wt \tau_i (F^{-1}(A)))\\
   &=& \sum_{i=1}^N p_i\; \mathbb P(F^{-1}(\tau_i (A))) \\
   &=&  \sum_{i=1}^N p_i\; \mu_p (\tau_i(A)).
\end{eqnarray*}

Let now $C$ be an open subset of $X$. Then there exists a cylinder 
set $C(\omega_1, ... , \omega_m)$ such that
$$
\wt\tau_{\omega_1} \circ \cdots \circ \wt\tau_{\omega_m}
(\Omega) \subset F^{-1}(C).
$$
It follows from this inclusion that
$$
\mu_p(C) = \mathbb P(F^{-1}(C)) \geq \mathbb P(C(\omega_1, ... , 
\omega_m)) = p_{\omega_1} \cdots p_{\omega_m}.
$$
The proof is complete.
\end{proof}

\subsection{Transfer operator for $x \mapsto 2x \; 
\mathrm{mod}\; 1$}

We consider here the one of the most popular endomorphisms, 
 $\sigma:  x \mapsto 2x \mod 1$ defined 
 on the unit interval $[0,1]$ and study its properties related 
to the corresponding transfer operator and iterated function system.

We fix the following notations for this subsection. Let $X = [0,1] = 
\mathbb R/\mathbb Z, \sigma(x) = 2x \mod 1$,
$$
\tau_0 (x) = \frac{x}{2}, \ \ \ \  \tau_1(x) = \frac{x+1}{2}, \qquad x 
\in X,
$$
and let $\la = dx$ denote the Lebesgue measure on $X$. Then $(X; 
\tau_0, \tau_1)$ is an IFS defined by the inverse branches for $
\sigma$.

In this example, we will illustrate the facts about transfer operators $
(R, \sigma)$ by considering specific weights $W$.

We use the formula given in (\ref{eq operator for IFS via W}) to define 
the transfer operator by a weight function $W$.

Take the function 
$W = \cos^2(\pi x)$. Define the transfer operator associated with the 
IFS $(X; \tau_0, \tau_1)$:
$$
R_W(f)(x) = \cos^2(\frac{\pi x}{2}) f(\frac{x}{2}) + 
\cos^2(\frac{\pi (x+1)}{2}) f(\frac{x+1}{2}). 
$$
Since $\cos^2(\frac{\pi (x+1)}{2}) = \sin^2(\frac{\pi x}{2})$, we get 
that
\be\label{eq TO for 2x}
R_W(f )(x) = \cos^2(\frac{\pi x}{2}) f(\frac{x}{2}) + 
 \sin^2(\frac{\pi x}{2})  f(\frac{x+1}{2}).
\ee
It follows from (\ref{eq TO for 2x}) that the transfer operator $R_W$ is 
normalized because
$$
R_W (\mathbf 1) = \cos^2(\frac{\pi x}{2}) + \sin^2(\frac{\pi 
x}{2}).
$$

\begin{lemma}\label{lem la R derivative}
 The Lebesgue measure $\la = dx$ belongs to $\mc L(R_p)$ for any 
 probability vector $p$. If  $p_0 = p_1 = 1/2$, then
$$
\frac{d\la R_{1/2}}{d\la}(x) =  \cos^2(\pi x).
$$
\end{lemma}

\begin{proof} The fact that $\la R_p \ll \la$ will be clear from the 
following computation which are conducted for the case  $p_0 = p_1 = 
1/2$:
\begin{eqnarray*}
 \int_X R(f) \; dx  &=& 2^{-1}\int_X\cos^2(\frac{\pi x}{2}) f(\frac{x}
 {2})\; dx +
2^{-1}\int_X \sin^2(\frac{\pi x}{2})  f(\frac{x+1}{2}) \; dx \\
\\
 & = &  \int_0^{1/2}\cos^2(\pi y) f(y)\; dy + \int_{1/2}^1 
 \cos^2(\pi y)  f(y)\; dy\\
 \\
 &=& \int_X \cos^2(\pi y) f(y) \; dy
\end{eqnarray*}
This means that $2W = \cos^2(\pi x)$ is the Radon-Nikodym 
derivative.
\end{proof}

In what follows we will use the formula
\be\label{eq R via sigma}
R(f)(x) = \sum_{y : \sigma y =x} \cos^2(\pi y) f(y)
\ee
for the transfer operator. The advantage of this definition is that $R$ 
is now normalized, $R (\mathbf 1) = \mathbf 1$. It follows from 
Lemma \ref{lem la R derivative} that
\be\label{eq RN for R}
d (\la R) = 2 \cos^2(\pi x) d \la.
\ee

Let $\delta_0$ denote the atomic Dirac measure concentrated at 
$x=0$.

\begin{corollary}\label{cor delta} (1) The measures $\delta_0$ and $\la$ are 
$\sigma$-invariant.

(2) The measures $\delta_0$ and $\la$ are $R$-invariant,  $\la R$ 
is absolutely continuous with respect to $\la$ but $\delta_0 \notin 
\mc L(R)$.
\end{corollary}

\begin{proof}
The first statement is obvious and well known.

To show that (2) holds, we notice that $R(f)(0) = f(0)$, and this fact 
can be interpreted as
$$
\int_X f \; d(\delta_0 R) = \int_X f \; d\delta_0.
$$
Hence $\delta_0$ is $R$-invariant. Since $\delta_0 R = 1/2(\delta_0 
+ \delta_{1/2})$ it is clear that $\delta_0 R$ is not absolutely 
continuous with respect to $\delta_0$. The result about  the 
Lebesgue measure $\la$ has been proved in Lemma \ref{lem la R 
derivative}.

\end{proof}

We recall that, for given $R, \sigma$, one can define linear operators $
\wh S$ and $\wh R$ in the universal Hilbert space $\mc H(X)$. By 
definition (see Section \ref{sect Universal HS}),
$$
\wh S(f\sqrt{d\mu}) = (f\circ\sigma)\sqrt{d(\la R)}.
$$

\begin{proposition}  The operator $\wh S$ is an isometry in the Hilbert 
space $L^2(\la)$ where $\la$ is the Lebesgue measure on 
$X = [0,1]$.

\end{proposition}

\begin{proof} We first note that the following useful formula holds:
\be\label{eq useful formula for integrat}
\int_X f(\sigma x) g(x) \; dx = \int_X f(x)\frac{g(\tau_0 x) + 
g(\tau_1 x)}{2} \; dx.
\ee
Then we use (\ref{eq R via sigma}) and (\ref{eq RN for R}) to find the 
norm
\begin{eqnarray*}
||\wh S(f\sqrt{d\la})||^2_{L^2(\la)} & = & \int_X f(2x \!\!\!\mod 1)^2 
2\cos^2(\pi x)\; d\la(x) \\
\\
   &=&  \int_X f(x)^2(\cos^2(\pi x/2) + \sin^2 (\pi x/2)\; dx\\
   \\
  &=& ||fd\la||^2_{L^2(\la)}.
\end{eqnarray*}
Hence $\wh S$ is an isometry.
\end{proof}

We recall that the operator $\wh R (f\sqrt{d\mu}) = (Rf) \sqrt{d(\mu
\circ\sigma^{-1})}$ is the adjoint operator $\wh S^*$ for $\wh S$. 
For the Lebesgue measure $\la$, one has $\la\circ \sigma^{-1} =\la$, 
hence
$$
\wh R (f\sqrt{d\la}) = (Rf) \sqrt{d\la}.
$$

\begin{corollary} For the Lebesgue measure $\la$ on $X = [0,1]$, the 
projection $\wh E = \wh S \wh R$ from $\mc H(X)$ onto $\mc H(\la)
$ acts by the formula
$$
\wh E= \sqrt {2} (R(f) \circ \sigma) |\cos (\pi x)|\sqrt{d\la},
$$
where $R(f)\circ \sigma = \mathbb E_\la ( \cdot \; | \; \sigma^{-1}
(\B))$ is the conditional expectation.
\end{corollary}

\begin{example}
In this example, we give several formulas  for the action of the transfer 
operator $R$ on any Dirac measure $\delta_a$. 

For $a =0$, we observe that 
$$
\wh S(f\sqrt{d\delta_0}) = (f\cs) \sqrt{d\delta_0},\ \qquad \wh R(f
\sqrt{d\delta_0}) = R(f)\sqrt{d\delta_0})
$$
because $\delta_0$ is simultaneously $R$-invariant and $\sigma$-
invariant.

By direct computation we find that, for $ a \in (0,1)$,
$$
\wh S(f\sqrt{d\delta_a}) = (f \cs ) \sqrt{\cos^2(\pi a/2)\delta_a + 
\sin^2(\pi a/2)\delta_{(a+1)/2} },
$$
$$
\wh R(f\sqrt{d\delta_a}) = R(f)\sqrt{d\delta_a}),
$$
and 
$$
\wh S \wh R (f\sqrt{d\delta_a})  = R(f) \cs  \sqrt{\cos^2(\pi a)
\delta_a + \sin^2(\pi a)\delta_{a+1/2} }.
$$
In particular, if $a =1/2$, we can show more.

\begin{lemma} For the Dirac measure $\delta_{1/2}$, the following 
relations hold:
$$
(\delta_{1/2} \csi1) R = \delta_0,
$$
$$
\wh E(f\sqrt{d\delta_{1/2}}) = f(0)\sqrt{d\delta_0}) 
$$
where $\wh E = \wh S\wh R$.
\end{lemma}

\begin{proof} These results follow from the above formulas by 
straightforward computations. 
\end{proof}

\end{example}
\newpage

\section{Examples}\label{sect examples}

In this section, we discuss in detail several examples of transfer 
operators that are mentioned in Introduction.

\subsection{Transfer operator and a system of 
conditional measures} \label{ex examples of TO}
For every measurable partition $\xi$ of a probability measure space
$\sms$, there exists a system of conditional measures 
 \cite{Rohlin1949}. We apply this remarkable result to the case
 of a surjective endomorphism. 

\begin{theorem}\label{thm csm}
Let $\sms$ be a standard measure space with finite measure, and let 
$\sigma$ be a surjective homeomorphism on $X$. 
Let $\xi$ be the measurable 
partition into pre-images of $\sigma$, $\xi = \{\sigma^{-1}(x) : x \in X
\}$. Then there exists a system of conditional 
 measures $\{\mu_C\}_{C \in \xi}$ defined uniquely by $\mu$ and 
 $\xi$, see Definition \ref{def system cond measures}. 
 For an onto endomorphism,  $X/\xi$ is identified with $X$, 
and the following disintegration formula holds:
$$
\int_X f(x)\; d\mu(x) = \int_X \left(\int_{C_x} f(y)\; d\mu_{C_x}(y)
\right) \; d\mu_\sigma(x)
$$ 
 where $\mu_\sigma$ is the restriction of $\mu$ to $\sB$, and   $x$
 is identified with $C_x$. 
\end{theorem}

In Example \ref{ex TO by cond meas}, we introduced a  transfer 
operator $(R, \sigma)$ on a standard probability measure 
space  $(X, \B, \mu)$. Here we consider a slightly more
 general construction  by setting
\be\label{eq TO via cond syst meas and W}
R_W(f)(x) := \int_{C_x} f(y)W(y)\; d\mu_{C_x}(y)
\ee
where $C_x$ is the element of $\xi$ containing $x$ and $W$ is a 
\textit{positive} $\mu$-integrable function. It follows then that $f$ is 
$\mu_{C_x}$-integrable functions for a.e. $x$. We note first  that 
the condition $\mu(X) =1$ implies that $\mu_C(X) =1$ for a.e. 
$C \in X/\xi$. Another important fact is that the quotient measure
space defined by $\xi$ is isomorphic to $(X, \sB, \mu_\sigma)$ (see 
Subsection \ref{subsect meas partitions}  for details).

 It is natural to consider the operator  $R_W$ acting either in  
 $L^1(X,  \B, \mu)$ or in $L^2\sms$ in this example.

\begin{proposition}\label{prop R(f) constant on C}
 (1) If $\sigma$ is an onto endomorphism of a probability 
 measure space $\sms$, then  $(R_W, \sigma)$ is a transfer operator 
 on  $\sms$. It is normalized if and only 
if $\int_{C_x} W \; d\mu_{C_x} =1 $ $\mu$-a.e. $x\in X$.

(2) For any measurable function $f$, the function $R_W(f)(x)$ is 
 constant a.e. on every $C_x$ for $\mu$-a.e. $x$. 
\end{proposition} 
 
 \begin{proof}
 (1)  Clearly, $R_W$ is positive. To see 
that the pull-out property holds, we calculate
\begin{eqnarray*}
R_W((f\circ\sigma) g)(x) & = & \int_{C_x} (f\circ\sigma)(y)g(y) 
W(y) \; d\mu_{C_x}\\
\\
&=& f(x) \int_{C_x} g(y)W(y)\; d\mu_{C_x}(y)\\
\\
& = & f(x) R_W(g)(x).
\end{eqnarray*}
Here we used the fact that $f(\sigma(y)) = f(x)$ for $y \in C_x = 
\sigma^{-1}(x)$.

We also obtain from (\ref{eq TO via cond syst meas and W})
 that $R_W$ is normalized if and only if 
$$
R_W(\mathbf 1)(x) = \int_{C_x} W(y)\; d\mu_{C_x}(y) =1
$$
for any $x\in X$.

(2)  Because the value of $R_W(f)$ evaluated at $x$ is the integral of 
$f$ over  the measure space $(C_x, \mu_x)$, we see that the value of
 $R_W(f)$ at $x' \in C_x $ is equal to  $R_W(f)(x)$.  
 
 \end{proof}

Proposition \ref{prop R(f) constant on C} allows us to deduce several 
simple consequences of the proved results.

\begin{corollary} (1) For the transfer operator $(R_W, \sigma)$, the 
following property holds a.e.
$$
R^2_W(f)(x)  = R_W(f)(x)R(\mathbf 1)(x).
$$

(2) If $\int_{C_x} W \; d\mu_{C_x}= 1$ for a.e. $x$, 
then, for any $f$, $R_W(f)$ is a harmonic  function with respect to 
 $R_W$. 
 
 (3) For any $\B$-measurable function $f$, the function $R(f)$ is
 $\sB$-measurable. 
\end{corollary}
 
In particular, the case when $W =1$,  gives a simple straightforward  
example of harmonic functions for  the corresponding transfer
operator $R_1$. 
 
\begin{proof} All these results follow directly from the fact that 
$R_W(f)$ is constant on elements $C$ of the partition $\xi$.
\end{proof}

\begin{proposition}\label{prop W =1}
For the transfer operator $R_W$ defined by (\ref{eq TO via cond syst meas and W}), 
$$
\int_X R_W(f)(x)\; d\mu(x) = \int_X f(x) W(x)\; d\mu(x), \qquad 
f\in L^1(\mu),
$$
that is $W =\dfrac{d\mu R_W}{d\mu}$. 
If $W = 1$, then $R$ is an isometry in the space $L^1(X, \B, \mu)$.
\end{proposition}

\begin{proof} The proof follows from the following calculations based
 on (\ref{eq function integration csm}):
\begin{eqnarray*}
\int_X R_W(f)(x)\; d\mu(x)  &= &\int_X \left(\int_{C_x} f(y)W(y)\; 
d\mu_{C_x}(y)\right) \: d\mu(x)\\
\\
&=& \int_X \int_C\left(\int_{C_x} f(y)W(y)\; d\mu_{C_x}(y)\right) \: d\mu_C(z)\; d\mu_\sigma(C)\\
\\
&=& \int_X \int_{C_x}\left(\int_{C} f(y)W(y)\; d\mu_{C}(z)\right) \: 
d\mu_{C_x}(y)\; d\mu_\sigma(C)\\
\\
&=& \int_X \left( \int_{C_x}f(y)W(y)\; d\mu_{C_x}(y)\right) \: 
d\mu_\sigma(C)\\
\\
&=& \int_X f(x) W(x) \; d\mu(x)
 \end{eqnarray*} 
 We used here that $\mu_C$ is a probability measure for a.e. $C$. 
 This equality shows that $d(\mu R_W) = d\mu$. 
\end{proof}

Let $\sigma$ be an endomorphism of a standard Borel space 
$(X, \B)$. Suppose now that $\xi$ is a $\sigma$-invariant partition of
 $\sms$. This means that $\sigma^{-1}(C)$ is a $\xi$-set for any 
 element $C$ of the partition $\xi$.  
Take a measure $\nu$  on $(X/\xi, \B/\xi)$. Denote by 
$(\nu_C)_{C \in X/\xi}$ a random measure  on $(X, \B)$, i.e., it 
 satisfies the conditions: (i)  $C \to \nu_C(B)$ is measurable for any 
 $B \in \B$, (ii)  $\nu_C(C) \in L^1(X/\xi, \nu)$. We\emph{ define} a  
 measure $\mu$ on $(X, \B)$ by setting
\be\label{eq mu defined by nu_C}
\mu(B) = \int_{X/\xi}  \nu_C(B) \; d\nu(C).
\ee

\begin{corollary}\label{cor W is not 1}
Suppose that the measure $\mu$ on $(X, \B)$ is as in  (\ref{eq mu 
defined by nu_C}). Let the transfer operator $R$ be defined by the 
relation
$$
R(f)(x) := \int_{C_x} f(y)\; d\nu_{C_x}(y).
$$
Then the Radon-Nikodym derivative $W = \dfrac{d\mu  R}{d\mu} = 
\nu_C(C)$.
\end{corollary}

The \emph{proof} is the same as in Proposition \ref{prop W =1}.

For the class of transfer operators which are considered in this example, 
we can easily point out harmonic functions.

\begin{theorem}\label{prop harmonic f-ns}
Let $\sigma$ be an onto endomorphism of a probability standard
measure spaca $\sms$, and $\xi$ is a $\sigma$-invariant measurable 
partition of $X$. Define  a transfer operator $R$ by setting
$$
R(f)(x) = \int_{C_x} f(y)\; d\mu_{C_x}(y)
$$
where $(\mu_{C_x})$ is the system of conditional measures 
\index{system of conditional measures} 
associated to $\xi$.   Then a measurable function $h$ defined on
 $X$ is harmonic with 
respect to $R$ if and only if $h$ is $\xi$-measurable, i.e., $h(x)$ is 
constant on every element $C$ of $\xi$.
\end{theorem}

\begin{proof} We note that $\mu_C$ is a probability measure for
$C \in X/\xi$. Therefore,  if $h(x) = h(C_x)$ for all $x \in X$, then
$$
R(h)(x) = \int_{C_x}h(y) \; d\mu_{C_x}(y) = 
h(C_x)\mu_{C_x}(C_x) = h(x).
$$

Conversely, if for all $x \in X$, we have $R(h)(x) = h(x)$, then $h(x)$ 
satisfies the relation
$$
h(x_1) = \int_{C}h(y) \; d\mu_{C}(y) = h(x_2)
$$
where $x_1$ and $x_2$ are taken from $C$.
\end{proof}

\begin{example}\label{ex cond measures with kernel}
Consider a standard probability measure space $\sms$, and  let $\nu:
x \mapsto \nu_x, x\in X$, be a random measure taken values in 
$M(Y)$ where $(Y, \mc A)$ is a measurable space. 
Consider the product measure space $(X \times Y, m)$ where 
$$
m = \int_X \nu_x \; d\mu 
$$
is a measure on $\B \times \mc A$.

Define an operator $R : \mc F(X\times Y) \to \mc F(X)$:
\be
R(f)(x) = \int_X f(x, y) K(x, y)\; d\nu_x(y)
\ee
where $K(x, y)$ is a non-negative measurable  bounded function.  

Let $\mc F_X$ be the set of functions depending on $x \in X$ only. 
In other words, $\mc F_X = \mc F(X\times Y ) \circ \pi$ where  
$\pi : X \times Y \to X$ is the projection. 
\\

\textit{Claim.} The operator $R$ satisfies the property:
$$
R(fg) = f R(g)
$$
if $f \in \mc F_X$ and $g \in \mc F(X \times Y)$. 
\\

Indeed, we see that 
\begin{eqnarray*}
R((f\circ \pi)g)(x)& = &\int_{X} K(x, y)f(\pi (x,y))g(x, y)\; 
d\nu_x(y)\\
\\
&= & f(x) \int_{X} K(x, y)g(x, y)\; d\nu_x(y) \\
\\
& = & f(x)R(g)(x).
\end{eqnarray*}
\end{example}

\begin{example}\label{ex TO defined by inverse branch of sigma}
In this example, we will work with a countable-to-one (or bounded-to-one) 
endomorphism $\sigma$ of a probability measure space $\sms$.
As was mentioned in Theorem \ref{thm Rohlin partition}, there is a partition 
$(A_i | i \in I)$ of $X$
into measurable sets of positive measure such that $\sigma_i\circ \tau_i = 
\mbox{id}$ on $A_i$.

We define a transfer operator $R_{\tau_i} : \mathcal M(\sigma(A_i)) \to 
\mathcal M(\sigma(A_i))$ by setting
\be\label{TO by tau_i}
(R_if) (x) = f(\tau_i(x)),\ \ \ \ x\in \sigma(A_i).
\ee
Then $R_{\tau_i} $ is positive and $R_{\tau_i} ((f\circ\sigma) g) = f
\circ(\sigma\tau_i) g\circ\tau_i = fR_{\tau_i}(g)$.

This example can be discussed in detail  in the context of stationary 
\textit{Bratteli diagrams}. \index{Bratteli diagram}

\end{example}

\begin{example}[Parry's Jacobian and transfer operator 
\cite{ParryWalters1972}] \label{ex Parry Jacobian}
Let $\sigma$ be a bounded-to-one nonsingular endomorphism of $\sms$, 
and let $(A_i | i \in I)$be the corresponding Rohlin partition.

We define for $x \in A_i$
$$
J_i (x) = \frac{d\mu\tau_i}{d\mu}(x),\ \ \ \ x \in A_i
$$
and let
\be\label{eq Jacobian}
J(x) = \sum_{i \in I} J_i(x) \chi_{A_i}(x),\ \ \ x \in X.
\ee

The function $J(x)= J_\sigma(x)$ is called \emph{Jacobian}
 \index{Jacobian} and was defined in \cite{Parry1969}, 
 \cite{ParryWalters1972}.
One can prove that the function $J(x)$ is independent of the choice of the 
Rohlin partition.

By non-singularity of $\sigma$ we have the following relations
$$
\theta_{\sigma}(x) :=  \frac{d\mu\sigma^{-1}}{d\mu}(x) = \sum_{y \in \sigma^{-1}(x)} \frac{1}{J(x)},
$$
$$
\omega_{\sigma}(x) :=  \frac{d\mu}{d\mu\sigma^{-1}}(x) = \frac{1}{\theta_\sigma(\sigma x)}.
$$
The function $\omega_\sigma$ is called the \emph{Radon-Nikodym derivative of $\sigma$} and satisfies the property
$$
\int_X f\circ\sigma \ \omega_\sigma \; d\mu =\int_X f \; d\mu
$$
for all $f \in L^1(X, \mu)$.

We use the Jacobian to define the transfer operator $R_\sigma$ acting on $\mathcal M(X)$:
$$
(R_\sigma h)(x) = \sum_{ y\in \sigma^{-1}(x)} \frac{h(y)}{ J_\sigma (y)}.
$$
It can be easily verified that $R_\sigma$ satisfies the characteristic property 
for transfer operators.
\end{example}

\newpage

\bibliographystyle{alpha}
\bibliography{bibliography-TO}

\newcommand{\etalchar}[1]{$^{#1}$}
\def\ocirc#1{\ifmmode\setbox0=\hbox{$#1$}\dimen0=\ht0 \advance\dimen0
  by1pt\rlap{\hbox to\wd0{\hss\raise\dimen0
  \hbox{\hskip.2em$\scriptscriptstyle\circ$}\hss}}#1\else {\accent"17 #1}\fi}
\begin{thebibliography}{FMCB{\etalchar{+}}16}

\bibitem[AA01]{AbramovichAliprantis2001}
Y.~A. Abramovich and C.~D. Aliprantis.
\newblock Positive operators.
\newblock In {\em Handbook of the geometry of Banach spaces, {V}ol. {I}}, pages
  85--122. North-Holland, Amsterdam, 2001.

\bibitem[ACKS16]{AlpayColomboKimseySabadini2016}
Daniel Alpay, Fabrizio Colombo, David~P. Kimsey, and Irene Sabadini.
\newblock The spectral theorem for unitary operators based on the
  {$S$}-spectrum.
\newblock {\em Milan J. Math.}, 84(1):41--61, 2016.

\bibitem[AGZ15]{AliliGraczykZak2015}
Larbi Alili, Piotr Graczyk, and Tomasz \.Zak.
\newblock On inversions and {D}oob {$h$}-transforms of linear diffusions.
\newblock In {\em In memoriam {M}arc {Y}or---{S}\'eminaire de {P}robabilit\'es
  {XLVII}}, volume 2137 of {\em Lecture Notes in Math.}, pages 107--126.
  Springer, Cham, 2015.

\bibitem[AJ15]{AlpayJorgensen2015}
Daniel Alpay and Palle Jorgensen.
\newblock Spectral theory for {G}aussian processes: reproducing kernels,
  boundaries, and {$L^2$}-wavelet generators with fractional scales.
\newblock {\em Numer. Funct. Anal. Optim.}, 36(10):1239--1285, 2015.

\bibitem[AJK15]{AlpayJorgensenKimsey2015}
Daniel Alpay, Palle E.~T. Jorgensen, and David~P. Kimsey.
\newblock Moment problems in an infinite number of variables.
\newblock {\em Infin. Dimens. Anal. Quantum Probab. Relat. Top.},
  18(4):1550024, 14, 2015.

\bibitem[AJL13]{AlpayJorgensenLewkowicz2013}
Daniel Alpay, Palle Jorgensen, and Izchak Lewkowicz.
\newblock Parametrizations of all wavelet filters: input-output and
  state-space.
\newblock {\em Sampl. Theory Signal Image Process.}, 12(2-3):159--188, 2013.

\bibitem[AJL16]{AlpayJorgensenLewkowicz2016}
Daniel Alpay, Palle Jorgensen, and Izchak Lewkowicz.
\newblock W-markov measures, transfer operators, wavelets and multiresolutions.
\newblock {\em arXiv:1606.07692}, 2016.

\bibitem[AJLM15]{AlpayJorgensenLewkowiczMartziano2015}
Daniel Alpay, Palle Jorgensen, Izchak Lewkowicz, and Itzik Martziano.
\newblock Infinite product representations for kernels and iterations of
  functions.
\newblock In {\em Recent advances in inverse scattering, {S}chur analysis and
  stochastic processes}, volume 244 of {\em Oper. Theory Adv. Appl.}, pages
  67--87. Birkh\"auser/Springer, Cham, 2015.

\bibitem[AJLV16]{AlpayJorgensenLewkowiczVolok2016}
Daniel Alpay, Palle Jorgensen, Izchak Lewkowicz, and Dan Volok.
\newblock A new realization of rational functions, with applications to linear
  combination interpolation, the {C}untz relations and kernel decompositions.
\newblock {\em Complex Var. Elliptic Equ.}, 61(1):42--54, 2016.

\bibitem[AJS14]{AlpayJorgensenSalomon2014}
Daniel Alpay, Palle Jorgensen, and Guy Salomon.
\newblock On free stochastic processes and their derivatives.
\newblock {\em Stochastic Process. Appl.}, 124(10):3392--3411, 2014.

\bibitem[AJV14]{AlpayJorgensenVolok2014}
Daniel Alpay, Palle Jorgensen, and Dan Volok.
\newblock Relative reproducing kernel {H}ilbert spaces.
\newblock {\em Proc. Amer. Math. Soc.}, 142(11):3889--3895, 2014.

\bibitem[AK13]{AlpayKipnis2013}
Daniel Alpay and Alon Kipnis.
\newblock A generalized white noise space approach to stochastic integration
  for a class of {G}aussian stationary increment processes.
\newblock {\em Opuscula Math.}, 33(3):395--417, 2013.

\bibitem[AK15]{AlpayKipnis2015}
Daniel Alpay and Alon Kipnis.
\newblock Wiener chaos approach to optimal prediction.
\newblock {\em Numer. Funct. Anal. Optim.}, 36(10):1286--1306, 2015.

\bibitem[AL13]{AlpayLewkowicz2013}
Daniel Alpay and Izchak Lewkowicz.
\newblock Convex cones of generalized positive rational functions and the
  {N}evanlinna-{P}ick interpolation.
\newblock {\em Linear Algebra Appl.}, 438(10):3949--3966, 2013.

\bibitem[AM16]{ArslanMilousheva2016}
Kadri Arslan and Velichka Milousheva.
\newblock Meridian surfaces of elliptic or hyperbolic type with pointwise
  1-type {G}auss map in {M}inkowski 4-space.
\newblock {\em Taiwanese J. Math.}, 20(2):311--332, 2016.

\bibitem[AR15]{ArklintRuiz2015}
Sara~E. Arklint and Efren Ruiz.
\newblock Corners of {C}untz-{K}rieger algebras.
\newblock {\em Trans. Amer. Math. Soc.}, 367(11):7595--7612, 2015.

\bibitem[AU15]{AlbeverioUgolini2015}
Sergio Albeverio and Stefania Ugolini.
\newblock A {D}oob h-transform of the {G}ross-{P}itaevskii {H}amiltonian.
\newblock {\em J. Stat. Phys.}, 161(2):486--508, 2015.

\bibitem[Bal00]{Baladi2000}
Viviane Baladi.
\newblock {\em Positive transfer operators and decay of correlations},
  volume~16 of {\em Advanced Series in Nonlinear Dynamics}.
\newblock World Scientific Publishing Co., Inc., River Edge, NJ, 2000.

\bibitem[BB05]{BaillifBaladi2005}
Mathieu Baillif and Viviane Baladi.
\newblock Kneading determinants and spectra of transfer operators in higher
  dimensions: the isotropic case.
\newblock {\em Ergodic Theory Dynam. Systems}, 25(5):1437--1470, 2005.

\bibitem[Bea91]{Beardon1991}
Alan~F. Beardon.
\newblock {\em Iteration of rational functions}, volume 132 of {\em Graduate
  Texts in Mathematics}.
\newblock Springer-Verlag, New York, 1991.
\newblock Complex analytic dynamical systems.

\bibitem[BEK93]{BratteliElliottKishimoto1993}
Ola Bratteli, George~A. Elliott, and Akitaka Kishimoto.
\newblock Quasi-product actions of a compact group on a {$C^*$}-algebra.
\newblock {\em J. Funct. Anal.}, 115(2):313--343, 1993.

\bibitem[B{\'e}n96]{Beneteau1996}
C.~B{\'e}n{\'e}teau.
\newblock A natural extension of a nonsingular endomorphism of a measure space.
\newblock {\em Rocky Mountain J. Math.}, 26(4):1261--1273, 1996.

\bibitem[BER89]{BaladiEckmanRuelle1989}
V.~Baladi, J.-P. Eckmann, and D.~Ruelle.
\newblock Resonances for intermittent systems.
\newblock {\em Nonlinearity}, 2(1):119--135, 1989.

\bibitem[BFMP09]{BaggettFurstMerrillPacker2009}
Lawrence~W. Baggett, Veronika Furst, Kathy~D. Merrill, and Judith~A. Packer.
\newblock Generalized filters, the low-pass condition, and connections to
  multiresolution analyses.
\newblock {\em J. Funct. Anal.}, 257(9):2760--2779, 2009.

\bibitem[BG91]{BezuglyiGolodets1991}
Sergey~I. Bezuglyi and Valentin~Ya. Golodets.
\newblock Weak equivalence and the structures of cocycles of an ergodic
  automorphism.
\newblock {\em Publ. Res. Inst. Math. Sci.}, 27(4):577--625, 1991.

\bibitem[BH09]{BruinHawkins2009}
Henk Bruin and Jane Hawkins.
\newblock Rigidity of smooth one-sided {B}ernoulli endomorphisms.
\newblock {\em New York J. Math.}, 15:451--483, 2009.

\bibitem[BH14]{BezuglyiHandelman2014}
Sergey Bezuglyi and David Handelman.
\newblock Measures on {C}antor sets: the good, the ugly, the bad.
\newblock {\em Trans. Amer. Math. Soc.}, 366(12):6247--6311, 2014.

\bibitem[BHS08]{BarnsleyHutchinsonStenflo2008}
Michael~F. Barnsley, John~E. Hutchinson, and {\"O}rjan Stenflo.
\newblock {$V$}-variable fractals: fractals with partial self similarity.
\newblock {\em Adv. Math.}, 218(6):2051--2088, 2008.

\bibitem[BHS12]{BarnsleyHutchinsonStenflo2012}
Michael Barnsley, John~E. Hutchinson, and {\"O}rjan Stenflo.
\newblock {$V$}-variable fractals: dimension results.
\newblock {\em Forum Math.}, 24(3):445--470, 2012.

\bibitem[BJ97]{BratteliJorgensen1997}
O.~Bratteli and P.~E.~T. Jorgensen.
\newblock Endomorphisms of {${\scr B}({\scr H})$}. {II}. {F}initely correlated
  states on {${\scr O}_n$}.
\newblock {\em J. Funct. Anal.}, 145(2):323--373, 1997.

\bibitem[BJ02]{BratteliJorgensen2002}
Ola Bratteli and Palle Jorgensen.
\newblock {\em Wavelets through a looking glass}.
\newblock Applied and Numerical Harmonic Analysis. Birkh\"auser Boston, Inc.,
  Boston, MA, 2002.
\newblock The world of the spectrum.

\bibitem[BJL96]{BaladiJiangLanford1996}
Viviane Baladi, Yun~Ping Jiang, and Oscar~E. Lanford, III.
\newblock Transfer operators acting on {Z}ygmund functions.
\newblock {\em Trans. Amer. Math. Soc.}, 348(4):1599--1615, 1996.

\bibitem[BJMP05]{BaggettJorgensenMerrillPacker2005}
Lawrence Baggett, Palle Jorgensen, Kathy Merrill, and Judith Packer.
\newblock A non-{MRA} {$C^r$} frame wavelet with rapid decay.
\newblock {\em Acta Appl. Math.}, 89(1-3):251--270 (2006), 2005.

\bibitem[BJP96]{BratteliJorgensenPrice1996}
Ola Bratteli, Palle E.~T. Jorgensen, and Geoffrey~L. Price.
\newblock Endomorphisms of {${\scr B}({\scr H})$}.
\newblock In {\em Quantization, nonlinear partial differential equations, and
  operator algebra ({C}ambridge, {MA}, 1994)}, volume~59 of {\em Proc. Sympos.
  Pure Math.}, pages 93--138. Amer. Math. Soc., Providence, RI, 1996.

\bibitem[BK00]{BratteliKishimoto2000}
Ola Bratteli and Akitaka Kishimoto.
\newblock Homogeneity of the pure state space of the {C}untz algebra.
\newblock {\em J. Funct. Anal.}, 171(2):331--345, 2000.

\bibitem[BK16]{BezuglyiKarpel2016}
S.~Bezuglyi and O.~Karpel.
\newblock Bratteli diagrams: structure, measures, dynamics.
\newblock In {\em Dynamics and numbers}, volume 669 of {\em Contemp. Math.},
  pages 1--36. Amer. Math. Soc., Providence, RI, 2016.

\bibitem[BKLR15]{BischoffKawahigashiLongo2015}
Marcel Bischoff, Yasuyuki Kawahigashi, Roberto Longo, and Karl-Henning Rehren.
\newblock {\em Tensor categories and endomorphisms of von {N}eumann
  algebras---with applications to quantum field theory}, volume~3 of {\em
  Springer Briefs in Mathematical Physics}.
\newblock Springer, Cham, 2015.

\bibitem[BKMS10]{BezuglyiKwiatkowskiMedynetsSolomyak2010}
S.~Bezuglyi, J.~Kwiatkowski, K.~Medynets, and B.~Solomyak.
\newblock Invariant measures on stationary {B}ratteli diagrams.
\newblock {\em Ergodic Theory Dynam. Systems}, 30(4):973--1007, 2010.

\bibitem[BKMS13]{BezuglyiKwiatkowskiMedynetsSolomyak2013}
S.~Bezuglyi, J.~Kwiatkowski, K.~Medynets, and B.~Solomyak.
\newblock Finite rank {B}ratteli diagrams: structure of invariant measures.
\newblock {\em Trans. Amer. Math. Soc.}, 365(5):2637--2679, 2013.

\bibitem[BLP{\etalchar{+}}10]{Baggett_et_al2010}
Lawrence~W. Baggett, Nadia~S. Larsen, Judith~A. Packer, Iain Raeburn, and Arlan
  Ramsay.
\newblock Direct limits, multiresolution analyses, and wavelets.
\newblock {\em J. Funct. Anal.}, 258(8):2714--2738, 2010.

\bibitem[BMPR12]{BaggettMerrillPackerRamsay2012}
Lawrence~W. Baggett, Kathy~D. Merrill, Judith~A. Packer, and Arlan~B. Ramsay.
\newblock Probability measures on solenoids corresponding to fractal wavelets.
\newblock {\em Trans. Amer. Math. Soc.}, 364(5):2723--2748, 2012.

\bibitem[Bog07]{Bogachev2007}
V.~I. Bogachev.
\newblock {\em Measure theory. {V}ol. {I}, {II}}.
\newblock Springer-Verlag, Berlin, 2007.

\bibitem[BRC16]{BandaraRubergCirak2016}
Kosala Bandara, Thomas R{\"u}berg, and Fehmi Cirak.
\newblock Shape optimisation with multiresolution subdivision surfaces and
  immersed finite elements.
\newblock {\em Comput. Methods Appl. Mech. Engrg.}, 300:510--539, 2016.

\bibitem[BsCD16]{BektasCanfesDursun2016}
B.~Bekta\c~s, E.~\"O. Canfes, and U.~Dursun.
\newblock On rotational surfaces in pseudo-{E}uclidean space {$\Bbb E^4_T$}
  with pointwise 1-type {G}auss map.
\newblock {\em Acta Univ. Apulensis Math. Inform.}, (45):43--59, 2016.

\bibitem[BSV15]{BahsonSchmelingVaienti2015}
Wael Bahsoun, J{\"o}rg Schmeling, and Sandro Vaienti.
\newblock On transfer operators and maps with random holes.
\newblock {\em Nonlinearity}, 28(3):713--727, 2015.

\bibitem[CE77]{ChoiEffros1977}
Man~Duen Choi and Edward~G. Effros.
\newblock Injectivity and operator spaces.
\newblock {\em J. Functional Analysis}, 24(2):156--209, 1977.

\bibitem[CFS82]{CornfeldFominSinai1982}
I.~P. Cornfeld, S.~V. Fomin, and Ya.~G. Sina{\u\i}.
\newblock {\em Ergodic theory}, volume 245 of {\em Grundlehren der
  Mathematischen Wissenschaften [Fundamental Principles of Mathematical
  Sciences]}.
\newblock Springer-Verlag, New York, 1982.
\newblock Translated from the Russian by A. B. Sosinski{\u\i}.

\bibitem[CL16]{ChaoLv2016}
Xiaoli Chao and Yusha Lv.
\newblock On the {G}auss map of {W}eingarten hypersurfaces in hyperbolic
  spaces.
\newblock {\em Bull. Braz. Math. Soc. (N.S.)}, 47(4):1051--1069, 2016.

\bibitem[CT16]{CasazzaTremain2016}
Peter~G. Casazza and Janet~C. Tremain.
\newblock Consequences of the {M}arcus/{S}pielman/{S}rivastava solution of the
  {K}adison-{S}inger problem.
\newblock In {\em New trends in applied harmonic analysis}, Appl. Numer.
  Harmon. Anal., pages 191--213. Birkh\"auser/Springer, Cham, 2016.

\bibitem[DH94]{DajaniHawkins1994}
K.~G. Dajani and J.~M. Hawkins.
\newblock Examples of natural extensions of nonsingular endomorphisms.
\newblock {\em Proc. Amer. Math. Soc.}, 120(4):1211--1217, 1994.

\bibitem[DJ06]{DutkayJorgensen2006}
Dorin~E. Dutkay and Palle E.~T. Jorgensen.
\newblock Wavelets on fractals.
\newblock {\em Rev. Mat. Iberoam.}, 22(1):131--180, 2006.

\bibitem[DJ07]{DutkayJorgensen2007}
Dorin~Ervin Dutkay and Palle E.~T. Jorgensen.
\newblock Disintegration of projective measures.
\newblock {\em Proc. Amer. Math. Soc.}, 135(1):169--179, 2007.

\bibitem[DJ15]{DutkayJorgensen2015}
Dorin~Ervin Dutkay and Palle E.~T. Jorgensen.
\newblock Representations of {C}untz algebras associated to quasi-stationary
  {M}arkov measures.
\newblock {\em Ergodic Theory Dynam. Systems}, 35(7):2080--2093, 2015.

\bibitem[DJK94]{DoughertyJacksonKechris1994}
R.~Dougherty, S.~Jackson, and A.~S. Kechris.
\newblock The structure of hyperfinite {B}orel equivalence relations.
\newblock {\em Trans. Amer. Math. Soc.}, 341(1):193--225, 1994.

\bibitem[dlR06]{Rue2006}
Thierry de~la Rue.
\newblock An introduction to joinings in ergodic theory.
\newblock {\em Discrete Contin. Dyn. Syst.}, 15(1):121--142, 2006.

\bibitem[DR07]{DutkayRoysland2007}
Dorin~Ervin Dutkay and Kjetil R{\o}ysland.
\newblock The algebra of harmonic functions for a matrix-valued transfer
  operator.
\newblock {\em J. Funct. Anal.}, 252(2):734--762, 2007.

\bibitem[Dut02]{Dutkay2002}
Dorin~Ervin Dutkay.
\newblock Harmonic analysis of signed {R}uelle transfer operators.
\newblock {\em J. Math. Anal. Appl.}, 273(2):590--617, 2002.

\bibitem[DZ09]{DingZhou2009}
Jiu Ding and Aihui Zhou.
\newblock {\em Nonnegative matrices, positive operators, and applications}.
\newblock World Scientific Publishing Co. Pte. Ltd., Hackensack, NJ, 2009.

\bibitem[ES89]{EigenSilva1989}
Stanley~J. Eigen and Cesar~E. Silva.
\newblock A structure theorem for {$n$}-to-{$1$} endomorphisms and existence of
  nonrecurrent measures.
\newblock {\em J. London Math. Soc. (2)}, 40(3):441--451, 1989.

\bibitem[Fed13]{Fedotov2013}
A.~G. Fedotov.
\newblock On the realization of the generalized solenoid as a hyperbolic
  attractor of sphere diffeomorphisms.
\newblock {\em Math. Notes}, 94(5-6):681--691, 2013.
\newblock Translation of Mat. Zametki {{\bf{9}}4} (2013), no. 5, 733--744.

\bibitem[FGKP16]{FarsiGillaspyKangPacker2016}
Carla Farsi, Elizabeth Gillaspy, Sooran Kang, and Judith~A. Packer.
\newblock Separable representations, {KMS} states, and wavelets for higher-rank
  graphs.
\newblock {\em J. Math. Anal. Appl.}, 434(1):241--270, 2016.

\bibitem[FMCB{\etalchar{+}}16]{FocchiMedrano2016}
Michele Focchi, Gustavo~A. Medrano-Cerda, Thiago Boaventura, Marco Frigerio,
  Claudio Semini, Jonas Buchli, and Darwin~G. Caldwell.
\newblock Robot impedance control and passivity analysis with inner torque and
  velocity feedback loops.
\newblock {\em Control Theory Technol.}, 14(2):97--112, 2016.

\bibitem[Gla03]{Glasner2003}
Eli Glasner.
\newblock {\em Ergodic theory via joinings}, volume 101 of {\em Mathematical
  Surveys and Monographs}.
\newblock American Mathematical Society, Providence, RI, 2003.

\bibitem[GS16]{GhaneSarizadeh2016}
F.~H. Ghane and A.~Sarizadeh.
\newblock Some stochastic properties of topological dynamics of semigroup
  actions.
\newblock {\em Topology Appl.}, 204:112--120, 2016.

\bibitem[GSSY16]{GalicerSagliettiSchmerkinYavicoli2016}
Daniel Galicer, Santiago Saglietti, Pablo Shmerkin, and Alexia Yavicoli.
\newblock {$L^q$} dimensions and projections of random measures.
\newblock {\em Nonlinearity}, 29(9):2609--2640, 2016.

\bibitem[Haw94]{Hawkins1994}
J.~M. Hawkins.
\newblock Amenable relations for endomorphisms.
\newblock {\em Trans. Amer. Math. Soc.}, 343(1):169--191, 1994.

\bibitem[HR00]{HutchinsonRueschendorf1998}
John~E. Hutchinson and Ludger R{\"u}schendorf.
\newblock Selfsimilar fractals and selfsimilar random fractals.
\newblock In {\em Fractal geometry and stochastics, {II}
  ({G}reifswald/{K}oserow, 1998)}, volume~46 of {\em Progr. Probab.}, pages
  109--123. Birkh\"auser, Basel, 2000.

\bibitem[HS91]{HawkinsSilva1991}
Jane~M. Hawkins and Cesar~E. Silva.
\newblock Noninvertible transformations admitting no absolutely continuous
  {$\sigma$}-finite invariant measure.
\newblock {\em Proc. Amer. Math. Soc.}, 111(2):455--463, 1991.

\bibitem[H{\'S}16]{HorbaczSleczka2016}
Katarzyna Horbacz and Maciej {\'S}l{\c{e}}czka.
\newblock Law of large numbers for random dynamical systems.
\newblock {\em J. Stat. Phys.}, 162(3):671--684, 2016.

\bibitem[Hut81]{Hutchinson1981}
John~E. Hutchinson.
\newblock Fractals and self-similarity.
\newblock {\em Indiana Univ. Math. J.}, 30(5):713--747, 1981.

\bibitem[Hut96]{Hutchinson1995}
John~E. Hutchinson.
\newblock Elliptic systems.
\newblock In {\em Instructional {W}orkshop on {A}nalysis and {G}eometry, {P}art
  {I} ({C}anberra, 1995)}, volume~34 of {\em Proc. Centre Math. Appl. Austral.
  Nat. Univ.}, pages 111--120. Austral. Nat. Univ., Canberra, 1996.

\bibitem[HW70]{HalmosWallen1969}
P.~R. Halmos and L.~J. Wallen.
\newblock Powers of partial isometries.
\newblock {\em J. Math. Mech.}, 19:657--663, 1969/1970.

\bibitem[JLR16]{JiLiuRi2016}
You-Qing Ji, Zhi Liu, and Song-il Ri.
\newblock Fixed point theorems of the iterated function systems.
\newblock {\em Commun. Math. Res.}, 32(2):142--150, 2016.

\bibitem[JMS16]{JarosMaslankeStrobin2016}
Patrycja Jaros, \L~ukasz Ma\'slanka, and Filip Strobin.
\newblock Algorithms generating images of attractors of generalized iterated
  function systems.
\newblock {\em Numer. Algorithms}, 73(2):477--499, 2016.

\bibitem[Jon94]{Jones1994}
V.~F.~R. Jones.
\newblock On a family of almost commuting endomorphisms.
\newblock {\em J. Funct. Anal.}, 122(1):84--90, 1994.

\bibitem[Jor01]{Jorgensen2001}
Palle E.~T. Jorgensen.
\newblock Ruelle operators: functions which are harmonic with respect to a
  transfer operator.
\newblock {\em Mem. Amer. Math. Soc.}, 152(720):viii+60, 2001.

\bibitem[Jor04]{Jorgensen2004}
Palle E.~T. Jorgensen.
\newblock Iterated function systems, representations, and {H}ilbert space.
\newblock {\em Internat. J. Math.}, 15(8):813--832, 2004.

\bibitem[JPT15]{JorgensenPedersenTian2015}
Palle Jorgensen, Steen Pedersen, and Feng Tian.
\newblock Spectral theory of multiple intervals.
\newblock {\em Trans. Amer. Math. Soc.}, 367(3):1671--1735, 2015.

\bibitem[JS15]{JorgensenSong2015}
Palle~E. Jorgensen and Myung-Sin Song.
\newblock Filters and matrix factorization.
\newblock {\em Sampl. Theory Signal Image Process.}, 14(3):171--197, 2015.

\bibitem[JT15]{JorgensenTian2015}
Palle Jorgensen and Feng Tian.
\newblock Infinite networks and variation of conductance functions in discrete
  {L}aplacians.
\newblock {\em J. Math. Phys.}, 56(4):043506, 27, 2015.

\bibitem[Kak48]{Kakutani1948}
Shizuo Kakutani.
\newblock On equivalence of infinite product measures.
\newblock {\em Ann. of Math. (2)}, 49:214--224, 1948.

\bibitem[Kar59]{Karlin1959}
Samuel Karlin.
\newblock Positive operators.
\newblock {\em J. Math. Mech.}, 8:907--937, 1959.

\bibitem[Kat07]{Kato2007}
Masahiko Kato.
\newblock Compactly supported framelets and the {R}uelle operators.
\newblock In {\em Applied functional analysis}, pages 177--191. Yokohama Publ.,
  Yokohama, 2007.

\bibitem[Kea72]{Keane1972}
Michael Keane.
\newblock Strongly mixing {$g$}-measures.
\newblock {\em Invent. Math.}, 16:309--324, 1972.

\bibitem[Kec95]{Kechris1995}
Alexander~S. Kechris.
\newblock {\em Classical descriptive set theory}, volume 156 of {\em Graduate
  Texts in Mathematics}.
\newblock Springer-Verlag, New York, 1995.

\bibitem[KFB16]{KutzFuBrunton2016}
J.~Nathan Kutz, Xing Fu, and Steven~L. Brunton.
\newblock Multiresolution {D}ynamic {M}ode {D}ecomposition.
\newblock {\em SIAM J. Appl. Dyn. Syst.}, 15(2):713--735, 2016.

\bibitem[Lli15]{Llibre2015}
Jaume Llibre.
\newblock Brief survey on the topological entropy.
\newblock {\em Discrete Contin. Dyn. Syst. Ser. B}, 20(10):3363--3374, 2015.

\bibitem[LM94]{LasotaMackey1994}
Andrzej Lasota and Michael~C. Mackey.
\newblock {\em Chaos, fractals, and noise}, volume~97 of {\em Applied
  Mathematical Sciences}.
\newblock Springer-Verlag, New York, second edition, 1994.
\newblock Stochastic aspects of dynamics.

\bibitem[Lon89]{Longo1989}
Roberto Longo.
\newblock Index of subfactors and statistics of quantum fields. {I}.
\newblock {\em Comm. Math. Phys.}, 126(2):217--247, 1989.

\bibitem[LP13]{LatremoliereFredericPacker2013}
Fr{\'e}d{\'e}ric Latr{\'e}moli{\`e}re and Judith~A. Packer.
\newblock Noncommutative solenoids and their projective modules.
\newblock In {\em Commutative and noncommutative harmonic analysis and
  applications}, volume 603 of {\em Contemp. Math.}, pages 35--53. Amer. Math.
  Soc., Providence, RI, 2013.

\bibitem[LP15]{LatremoliereFredericPacker2016}
Fr{\'e}d{\'e}ric Latr{\'e}moli{\`e}re and Judith~A. Packer.
\newblock Explicit construction of equivalence bimodules between noncommutative
  solenoids.
\newblock In {\em Trends in harmonic analysis and its applications}, volume 650
  of {\em Contemp. Math.}, pages 111--140. Amer. Math. Soc., Providence, RI,
  2015.

\bibitem[Mai13]{Maier2013}
Daniel Maier.
\newblock Realizations of rotations on {$a$}-adic solenoids.
\newblock {\em Math. Proc. R. Ir. Acad.}, 113A(2):131--141, 2013.

\bibitem[Mat17]{Matsumoto2017}
Kengo Matsumoto.
\newblock Uniformly continuous orbit equivalence of {M}arkov shifts and gauge
  actions on {C}untz--{K}rieger algebras.
\newblock {\em Proc. Amer. Math. Soc.}, 145(3):1131--1140, 2017.

\bibitem[Mau95]{Mauldin1998}
R.~Daniel Mauldin.
\newblock Infinite iterated function systems: theory and applications.
\newblock In {\em Fractal geometry and stochastics ({F}insterbergen, 1994)},
  volume~37 of {\em Progr. Probab.}, pages 91--110. Birkh\"auser, Basel, 1995.

\bibitem[MdF16]{Machado2016}
Guilherme Machado~de Freitas.
\newblock Submanifolds with homothetic {G}auss map in codimension two.
\newblock {\em Geom. Dedicata}, 180:151--170, 2016.

\bibitem[MSS15]{MarkusSpielmanSrivastava2015}
Adam~W. Marcus, Daniel~A. Spielman, and Nikhil Srivastava.
\newblock Interlacing families {II}: {M}ixed characteristic polynomials and the
  {K}adison-{S}inger problem.
\newblock {\em Ann. of Math. (2)}, 182(1):327--350, 2015.

\bibitem[MU10]{MayerUrbanski2010}
Volker Mayer and Mariusz Urba{\'n}ski.
\newblock Thermodynamical formalism and multifractal analysis for meromorphic
  functions of finite order.
\newblock {\em Mem. Amer. Math. Soc.}, 203(954):vi+107, 2010.

\bibitem[MU15]{MayerUrbanski2015}
Volker Mayer and Mariusz Urba{\'n}ski.
\newblock Countable alphabet random subhifts of finite type with weakly
  positive transfer operator.
\newblock {\em J. Stat. Phys.}, 160(5):1405--1431, 2015.

\bibitem[Nel69]{Nelson1969}
Edward Nelson.
\newblock {\em Topics in dynamics. {I}: {F}lows}.
\newblock Mathematical Notes. Princeton University Press, Princeton, N.J.;
  University of Tokyo Press, Tokyo, 1969.

\bibitem[NR82]{NicoloRadin1982}
Francesco Nicol\`o and Charles Radin.
\newblock A first-order phase transition between crystal phases in the shift
  model.
\newblock {\em J. Statist. Phys.}, 28(3):473--478, 1982.

\bibitem[Par69]{Parry1969}
William Parry.
\newblock {\em Entropy and generators in ergodic theory}.
\newblock W. A. Benjamin, Inc., New York-Amsterdam, 1969.

\bibitem[Pow99]{Powers1999}
Robert~T. Powers.
\newblock Induction of semigroups of endomorphisms of {$\germ B(\germ H)$} from
  completely positive semigroups of {$(n\times n)$} matrix algebras.
\newblock {\em Internat. J. Math.}, 10(7):773--790, 1999.

\bibitem[PP93]{PowersPrice1993}
Robert~T. Powers and Geoffrey~L. Price.
\newblock Binary shifts on the hyperfinite {${\rm II}_1$} factor.
\newblock In {\em Representation theory of groups and algebras}, volume 145 of
  {\em Contemp. Math.}, pages 453--464. Amer. Math. Soc., Providence, RI, 1993.

\bibitem[PU10]{PrzytyckiUrbanski2010}
Feliks Przytycki and Mariusz Urba{\'n}ski.
\newblock {\em Conformal fractals: ergodic theory methods}, volume 371 of {\em
  London Mathematical Society Lecture Note Series}.
\newblock Cambridge University Press, Cambridge, 2010.

\bibitem[PW72]{ParryWalters1972}
William Parry and Peter Walters.
\newblock Endomorphisms of a {L}ebesgue space.
\newblock {\em Bull. Amer. Math. Soc.}, 78:272--276, 1972.

\bibitem[Rad99]{Radin1999}
Charles Radin.
\newblock {\em Miles of tiles}, volume~1 of {\em Student Mathematical Library}.
\newblock American Mathematical Society, Providence, RI, 1999.

\bibitem[R{\'e}n57]{Renyi1957}
A.~R{\'e}nyi.
\newblock Representations for real numbers and their ergodic properties.
\newblock {\em Acta Math. Acad. Sci. Hungar}, 8:477--493, 1957.

\bibitem[Ren87]{Renault1987}
Jean Renault.
\newblock Repr\'esentation des produits crois\'es d'alg\`ebres de groupo\"\i
  des.
\newblock {\em J. Operator Theory}, 18(1):67--97, 1987.

\bibitem[RG16]{Reveles2016}
Ferm{\'{\i}}n~Omar Reveles-Gurrola.
\newblock Homeomorphisms of a solenoid isotopic to the identity and its second
  cohomology groups.
\newblock {\em C. R. Math. Acad. Sci. Paris}, 354(9):879--886, 2016.

\bibitem[Roh49]{Rohlin1949}
V.~A. Rohlin.
\newblock On the fundamental ideas of measure theory.
\newblock {\em Mat. Sbornik N.S.}, 25(67):107--150, 1949.

\bibitem[Roh61]{Rohlin1961}
V.~A. Rohlin.
\newblock Exact endomorphisms of a {L}ebesgue space.
\newblock {\em Izv. Akad. Nauk SSSR Ser. Mat.}, 25:499--530, 1961.

\bibitem[Rud90]{Rudolph1990}
Daniel~J. Rudolph.
\newblock {\em Fundamentals of measurable dynamics}.
\newblock Oxford Science Publications. The Clarendon Press, Oxford University
  Press, New York, 1990.
\newblock Ergodic theory on Lebesgue spaces.

\bibitem[Rue78]{Ruelle1978}
David Ruelle.
\newblock {\em Thermodynamic formalism}, volume~5 of {\em Encyclopedia of
  Mathematics and its Applications}.
\newblock Addison-Wesley Publishing Co., Reading, Mass., 1978.
\newblock The mathematical structures of classical equilibrium statistical
  mechanics, With a foreword by Giovanni Gallavotti and Gian-Carlo Rota.

\bibitem[Rue89]{Ruelle1989}
David Ruelle.
\newblock The thermodynamic formalism for expanding maps.
\newblock {\em Comm. Math. Phys.}, 125(2):239--262, 1989.

\bibitem[Rue92]{Ruelle1992}
David Ruelle.
\newblock Thermodynamic formalism for maps satisfying positive expansiveness
  and specification.
\newblock {\em Nonlinearity}, 5(6):1223--1236, 1992.

\bibitem[Rue02]{Ruelle2002}
David Ruelle.
\newblock Dynamical zeta functions and transfer operators.
\newblock {\em Notices Amer. Math. Soc.}, 49(8):887--895, 2002.

\bibitem[Rug16]{Rugh2016}
Hans~Henrik Rugh.
\newblock The {M}ilnor-{T}hurston determinant and the {R}uelle transfer
  operator.
\newblock {\em Comm. Math. Phys.}, 342(2):603--614, 2016.

\bibitem[SG16]{SuarezGhosal2016}
Adam~Justin Suarez and Subhashis Ghosal.
\newblock Bayesian clustering of functional data using local features.
\newblock {\em Bayesian Anal.}, 11(1):71--98, 2016.

\bibitem[Sil88]{Silva1988}
Cesar~E. Silva.
\newblock On {$\mu$}-recurrent nonsingular endomorphisms.
\newblock {\em Israel J. Math.}, 61(1):1--13, 1988.

\bibitem[Sil13]{Silvestrov2013}
Sergei Silvestrov.
\newblock Dynamics, wavelets, commutants and transfer operators satisfying
  crossed product type commutation relations.
\newblock In {\em Operator algebra and dynamics}, volume~58 of {\em Springer
  Proc. Math. Stat.}, pages 273--293. Springer, Heidelberg, 2013.

\bibitem[Sim12]{Simmons2012}
David Simmons.
\newblock Conditional measures and conditional expectation; {R}ohlin's
  disintegration theorem.
\newblock {\em Discrete Contin. Dyn. Syst.}, 32(7):2565--2582, 2012.

\bibitem[Sto12]{Stoyanov2012}
Luchezar Stoyanov.
\newblock Regular decay of ball diameters and spectra of {R}uelle operators for
  contact {A}nosov flows.
\newblock {\em Proc. Amer. Math. Soc.}, 140(10):3463--3478, 2012.

\bibitem[Sto13]{Stoyanov2013}
Luchezar Stoyanov.
\newblock Ruelle operators and decay of correlations for contact {A}nosov
  flows.
\newblock {\em C. R. Math. Acad. Sci. Paris}, 351(17-18):669--672, 2013.

\bibitem[Sur16]{Sureson2016}
Claude Sureson.
\newblock {$\Pi_1^1$}-{M}artin-{L}\"of random reals as measures of natural open
  sets.
\newblock {\em Theoret. Comput. Sci.}, 653:26--41, 2016.

\bibitem[SUZ13]{SzarekUrbanskiZdunik2013}
Tomasz Szarek, Mariusz Urba{\'n}ski, and Anna Zdunik.
\newblock Continuity of {H}ausdorff measure for conformal dynamical systems.
\newblock {\em Discrete Contin. Dyn. Syst.}, 33(10):4647--4692, 2013.

\bibitem[SW17]{ShiozawaWang2017}
Yuichi Shiozawa and Jian Wang.
\newblock Rate {F}unctions for {S}ymmetric {M}arkov {P}rocesses via {H}eat
  {K}ernel.
\newblock {\em Potential Anal.}, 46(1):23--53, 2017.

\bibitem[Sze17]{Szewczak2017}
Zbigniew~S. Szewczak.
\newblock Berry--{E}ss\'een theorem for sample quantiles of asymptotically
  uncorrelated non reversible {M}arkov chains.
\newblock {\em Comm. Statist. Theory Methods}, 46(8):3985--4003, 2017.

\bibitem[Ver94]{Vershik1994}
A.~M. Vershik.
\newblock Theory of decreasing sequences of measurable partitions.
\newblock {\em Algebra i Analiz}, 6(4):1--68, 1994.

\bibitem[Ver00]{Vershik2000}
A.~M. Vershik.
\newblock Dynamic theory of growth in groups: entropy, boundaries, examples.
\newblock {\em Uspekhi Mat. Nauk}, 55(4(334)):59--128, 2000.

\bibitem[Ver01]{Vershik2001}
A.~M. Vershik.
\newblock V. {A}. {R}okhlin and the modern theory of measurable partitions.
\newblock In {\em Topology, ergodic theory, real algebraic geometry}, volume
  202 of {\em Amer. Math. Soc. Transl. Ser. 2}, pages 11--20. Amer. Math. Soc.,
  Providence, RI, 2001.

\bibitem[Ver05]{Vershik2005}
A.~M. Vershik.
\newblock Polymorphisms, {M}arkov processes, and quasi-similarity.
\newblock {\em Discrete Contin. Dyn. Syst.}, 13(5):1305--1324, 2005.

\bibitem[VF85]{Vershik_Fedorov1985}
A.~M. Vershik and A.~L. F{\"e}dorov.
\newblock Trajectory theory.
\newblock In {\em Current problems in mathematics. {N}ewest results, {V}ol.\
  26}, Itogi Nauki i Tekhniki, pages 171--211, 260. Akad. Nauk SSSR, Vsesoyuz.
  Inst. Nauchn. i Tekhn. Inform., Moscow, 1985.

\bibitem[Wol48]{Wold1948}
Herman O.~A. Wold.
\newblock On prediction in stationary time series.
\newblock {\em Ann. Math. Statistics}, 19:558--567, 1948.

\bibitem[Wol51]{Wold1951}
Herman O.~A. Wold.
\newblock Stationary time series.
\newblock {\em Trabajos Estad\'\i stica}, 2:3--74, 1951.

\bibitem[Wol54]{Wold1954}
Herman Wold.
\newblock {\em A study in the analysis of stationary time series}.
\newblock Almqvist and Wiksell, Stockholm, 1954.
\newblock 2d ed, With an appendix by Peter Whittle.

\bibitem[YL16]{YaoLi2016}
Yuanyuan Yao and Wenxia Li.
\newblock Generating iterated function systems for the {V}icsek snowflake and
  the {K}och curve.
\newblock {\em Amer. Math. Monthly}, 123(7):716--721, 2016.

\bibitem[YLZ99]{YanLiuZhu1999}
Dejun Yan, Xiangdong Liu, and Weiyong Zhu.
\newblock A study of {M}andelbrot and {J}ulia sets generated from a general
  complex cubic iteration.
\newblock {\em Fractals}, 7(4):433--437, 1999.

\bibitem[YZL13]{YeZouLu2013}
Ruisong Ye, Yuru Zou, and Jian Lu.
\newblock Chaotic dynamical systems on fractals and their applications to image
  encryption.
\newblock In {\em Recent advances in applied nonlinear dynamics with numerical
  analysis}, volume~15 of {\em Interdiscip. Math. Sci.}, pages 279--304. World
  Sci. Publ., Hackensack, NJ, 2013.

\bibitem[ZJ15]{ZhangJorgensen2015}
Zhihua Zhang and Palle E.~T. Jorgensen.
\newblock Modulated {H}aar wavelet analysis of climatic background noise.
\newblock {\em Acta Appl. Math.}, 140:71--93, 2015.

\end{thebibliography}

\printindex{}

\end{document}